\pdfoutput=1
\documentclass{amsart}
\usepackage[utf8]{inputenc}

\usepackage[margin=1in]{geometry}
\usepackage[expansion=false]{microtype}

\usepackage{amsmath, amsfonts, amssymb, amsthm, amsopn, amscd}
\usepackage{eucal}

\usepackage{booktabs,longtable,graphicx}

\usepackage{tikz}

\usepackage[pdfusetitle,unicode,hidelinks]{hyperref}
\usepackage{verbatim}

\usetikzlibrary{matrix,arrows,decorations}

\makeatletter
\def\@tipextend{
	\arrowsize=.3pt
	\advance\arrowsize by .5\pgflinewidth
	\pgfarrowsleftextend{-3\arrowsize-.5\pgflinewidth}
	\pgfarrowsrightextend{.5\pgflinewidth}
}
\def\@tipoptions{
	\pgfsetdash{}{0pt}
	\pgfsetroundcap
	\pgfsetroundjoin
}

\pgfarrowsdeclare{tip}{tip}{\@tipextend}{
	\@tipoptions
	\pgfmathsetlength{\pgfutil@tempdima}{\pgflinewidth}
	\pgfpathmoveto {\pgfpoint{-3.5\pgfutil@tempdima}{6.2\pgfutil@tempdima}}
	\pgfpathcurveto{\pgfpoint{-2.8\pgfutil@tempdima}{2\pgfutil@tempdima}}
	               {\pgfpoint{1\pgfutil@tempdima}{0.05\pgfutil@tempdima}}
	               {\pgfpoint{2\pgfutil@tempdima}{0pt}}
	\pgfpathcurveto{\pgfpoint{1\pgfutil@tempdima}{-0.05\pgfutil@tempdima}}
	               {\pgfpoint{-2.8\pgfutil@tempdima}{-2\pgfutil@tempdima}}
	               {\pgfpoint{-3.5\pgfutil@tempdima}{-6.2\pgfutil@tempdima}}
	\pgfusepathqstroke
}
\pgfarrowsdeclare{top}{top}{\@tipextend}{
	\@tipoptions
	\pgfmathsetlength{\pgfutil@tempdima}{\pgflinewidth}
	\pgfpathmoveto {\pgfpoint{-3.5\pgfutil@tempdima}{6.2\pgfutil@tempdima}}
	\pgfpathcurveto{\pgfpoint{-2.8\pgfutil@tempdima}{2\pgfutil@tempdima}}
	               {\pgfpoint{1\pgfutil@tempdima}{0.05\pgfutil@tempdima}}
	               {\pgfpoint{2\pgfutil@tempdima}{0pt}}
	\pgfpathcurveto{\pgfpoint{1\pgfutil@tempdima}{-0.05\pgfutil@tempdima}}
	               {\pgfpoint{-1.95\pgfutil@tempdima}{-1.8\pgfutil@tempdima}}
	               {\pgfpoint{-2.85\pgfutil@tempdima}{-4.2\pgfutil@tempdima}}
	\pgfusepathqstroke
}
\pgfarrowsdeclare{bot}{bot}{\@tipextend}{
	\@tipoptions 
	\pgfmathsetlength{\pgfutil@tempdima}{\pgflinewidth}
	\pgfpathmoveto {\pgfpoint{-2.85\pgfutil@tempdima}{4.2\pgfutil@tempdima}}
	\pgfpathcurveto{\pgfpoint{-1.95\pgfutil@tempdima}{1.8\pgfutil@tempdima}}
	               {\pgfpoint{1\pgfutil@tempdima}{0.05\pgfutil@tempdima}}
	               {\pgfpoint{2\pgfutil@tempdima}{0pt}}
	\pgfpathcurveto{\pgfpoint{1\pgfutil@tempdima}{-0.05\pgfutil@tempdima}}
	               {\pgfpoint{-2.8\pgfutil@tempdima}{-2\pgfutil@tempdima}}
	               {\pgfpoint{-3.5\pgfutil@tempdima}{-6.2\pgfutil@tempdima}}
	\pgfusepathqstroke
}
\pgfarrowsdeclare{mid}{mid}{\@tipextend}{
	\@tipoptions 
	\pgfmathsetlength{\pgfutil@tempdima}{\pgflinewidth}
	\pgfpathmoveto {\pgfpoint{-2.85\pgfutil@tempdima}{4.2\pgfutil@tempdima}}
	\pgfpathcurveto{\pgfpoint{-1.95\pgfutil@tempdima}{1.8\pgfutil@tempdima}}
	               {\pgfpoint{1\pgfutil@tempdima}{0.05\pgfutil@tempdima}}
	               {\pgfpoint{2\pgfutil@tempdima}{0pt}}
	\pgfpathcurveto{\pgfpoint{1\pgfutil@tempdima}{-0.05\pgfutil@tempdima}}
	               {\pgfpoint{-1.95\pgfutil@tempdima}{-1.8\pgfutil@tempdima}}
	               {\pgfpoint{-2.85\pgfutil@tempdima}{-4.2\pgfutil@tempdima}}
	\pgfusepathqstroke
}

\expandafter\let\csname pgf@ar@means@|\endcsname=\relax
\expandafter\let\csname pgf@arrow@code@|\endcsname=\relax
\pgfarrowsdeclare{|}{|}{
	\pgfarrowsleftextend{-0.25\pgflinewidth}
	\pgfarrowsrightextend{.75\pgflinewidth}
}{
	\@tipoptions 
	\pgfmathsetlength{\pgfutil@tempdima}{\pgflinewidth}
	\pgfpathmoveto{\pgfpoint{0pt}{6\pgfutil@tempdima}}
	\pgfpathlineto{\pgfpoint{0pt}{-6\pgfutil@tempdima}}
	\pgfusepathqstroke
}
\makeatother

\pgfdeclaredecoration{normal}{initial}{
	\state{initial}[width=\pgfdecoratedpathlength-1sp]{\pgfpathmoveto{\pgfpointorigin}} 
	\state{final}{\pgfpathlineto{\pgfpointorigin}} 
}

\pgfarrowsdeclarealias{c}{c}{left hook}{right hook}

\def\rowsep{2.5em}
\def\colsep{2.5em}

\def\dbl{1.7}

\def\diagram{\matrix[diagram]}
\def\arrows{\path[->,font=\scriptsize]}

\newenvironment{tikzmath}{\displaymath\tikzpicture}{\endtikzpicture\enddisplaymath}
\newenvironment{tikzequation}{\equation\tikzpicture}{\endtikzpicture\endequation}

\tikzset{
	every picture/.style={
		baseline=(current bounding box.center),
		>=tip,
		line cap=round 
	},
	diagram/.style={
		inner sep=0,
		execute at begin cell=\node\bgroup\math\displaystyle,
		execute at end cell={
			\endmath\egroup;
			\coordinate (-\the\pgfmatrixcurrentrow\the\pgfmatrixcurrentcolumn) at ([yshift=2.6] \the\pgfmatrixcurrentrow\the\pgfmatrixcurrentcolumn.base west);
			\coordinate (\the\pgfmatrixcurrentrow\the\pgfmatrixcurrentcolumn -) at ([yshift=2.6] \the\pgfmatrixcurrentrow\the\pgfmatrixcurrentcolumn.base east);
		},
		cells={anchor=base},
		nodes={
			name=\the\pgfmatrixcurrentrow\the\pgfmatrixcurrentcolumn,
			inner sep=.3333em 
		},
		row sep=\rowsep,
		column sep=\colsep
	},
	vshift/.style={
		decoration={normal,raise=#1},
		decorate 
	}
}

\newcommand{\<}{\discretionary{}{}{}}

\numberwithin{equation}{section}

\theoremstyle{plain}
\newtheorem*{theorem*}{Theorem}
\newtheorem{theorem}[equation]{Theorem}

\newtheorem{proposition}[equation]{Proposition}
\newtheorem{lemma}[equation]{Lemma}
\newtheorem{corollary}[equation]{Corollary}

\theoremstyle{definition}
\newtheorem{definition}[equation]{Definition}
\newtheorem{example}[equation]{Example}
\newtheorem{notation}[equation]{Notation}

\newtheorem{remark}[equation]{Remark}

\let\scr=\mathcal
\let\bb=\mathbb
\let\rm=\mathrm
\def\Z{\bb Z}
\def\R{\bb R}
\def\Q{\bb Q}
\def\C{\bb C}
\def\A{\bb A}
\def\P{\bb P}
\def\V{\bb V}
\def\1{\mathbf 1}
\def\h{\mathrm h}

\def\G{\mathbb G}
\def\ph{\mathord-}
\makeatletter
\def\pt{{\mathpalette\pt@{.75}}}
\def\pt@#1#2{\mathord{\scalebox{#2}{$\m@th#1\bullet$}}}
\makeatother

\def\eff{\mathrm{eff}}
\def\veff{\mathrm{veff}}

\def\L{\mathrm{L}{}}
\def\Spi{\mathrm{Spi}}

\newcommand{\comp}{\wedge}

\let\from=\leftarrow
\let\into=\hookrightarrow
\let\onto=\twoheadrightarrow
\let\tens=\otimes
\def\suchthat{\:\vert\:}

\DeclareMathOperator{\Sym}{Sym}
\DeclareMathOperator{\NSym}{NSym}
\def\id{\mathrm{id}}

\def\Hom{\mathrm{Hom}}

\def\Aut{\mathrm{Aut}}
\def\Map{\mathrm{Map}}
\def\MAP{\mathbf{Map}}

\DeclareMathOperator{\Spec}{Spec}

\def\Th{\mathrm{Th}}
\def\Pic{\mathrm{Pic}}
\def\GW{\mathrm{GW}}
\def\W{\mathrm W}
\def\CH{\mathrm{CH}}
\def\rB{\mathrm{Re_B}}
\def\SingB{\mathrm{Sing_B}}
\def\ul{\underline}
\def\M{\mathrm{M}}
\def\Weil{\mathrm{R}}
\def\B{\mathrm{B}}
\def\HH{\mathrm{H}}
\let\sect=\S
\def\S{\mathrm{S}}
\def\E{\mathrm{E}}
\def\Norm{\mathrm{N}}
\def\Def{\mathrm{Def}}
\def\Orb{\mathrm{Orb}}
\def\noe{\mathrm{noe}}
\def\FM{\mathrm{FM}}
\def\r{r}
\def\cst{\mathrm{const}}

\DeclareMathOperator{\fib}{fib}
\DeclareMathOperator{\cofib}{cofib}

\DeclareMathOperator{\rk}{rk}
\def\Nis{\mathrm{Nis}}
\def\Zar{\mathrm{Zar}}
\def\cdh{\mathrm{cdh}}
\def\mot{\mathrm{mot}}
\def\pre{\mathrm{pre}}
\def\et{\mathrm{\acute et}}

\DeclareMathOperator{\Char}{char}
\DeclareMathOperator{\Tr}{Tr}
\def\tr{\mathrm{tr}}

\def\QCoh{\mathrm{QCoh}{}}

\def\Ex{\mathrm{Ex}}
\def\Dist{\mathrm{Dis}}
\DeclareMathOperator{\Bl}{Bl}

\let\cat=\mathrm

\def\Gr{\mathrm{Gr}{}}

\def\MGL{\mathrm{MGL}}
\def\KGL{\mathrm{KGL}}
\def\K{\mathrm{K}{}}
\def\KH{\mathrm{KH}{}}
\def\MSL{\mathrm{MSL}}
\def\MSp{\mathrm{MSp}}
\def\SL{\mathrm{SL}}
\def\Sp{\mathrm{Sp}}

\def\Sph{\cat S\mathrm{ph}{}}

\def\H{\mathcal H}
\def\SH{\mathcal S\mathcal H}
\def\DM{\mathcal D\mathcal M}
\def\Mod{\cat{M}\mathrm{od}{}}
\def\Sch{\cat{S}\mathrm{ch}{}}
\def\SmQP{\mathrm{SmQP}{}}
\def\QP{\mathrm{QP}{}}
\def\QPCor{\mathrm{QPCor}{}}
\def\Sm{{\cat{S}\mathrm{m}}}

\def\SmQPCor{\mathrm{SmQPCor}{}}
\def\Et{\cat{E}\mathrm{t}}
\def\FEt{\mathrm{FEt}{}}
\def\Fin{\cat F\mathrm{in}}
\def\op{\mathrm{op}}
\def\oop{\textnormal{1-op}}

\def\Vect{\cat{V}\mathrm{ect}{}}

\def\Sp{\mathrm{Sp}}
\def\Pr{\mathcal{P}\mathrm{r}}
\def\St{\mathrm{St}}

\def\CAlg{\mathrm{CAlg}{}}
\def\NAlg{\mathrm{NAlg}{}}
\def\Ind{\mathrm{Ind}{}}

\def\ev{\mathrm{ev}}

\def\minus{\smallsetminus}
\def\MW{\mathrm{MW}}

\def\spi{\underline{\pi}}

\def\longinto{\lhook\joinrel\longrightarrow}

\def\Cat{\mathcal{C}\mathrm{at}{}}
\def\Set{\mathcal{S}\mathrm{et}{}}
\def\Ab{\mathcal{A}\mathrm{b}}
\def\CW{\mathcal{CW}}
\def\PSh{\mathcal{P}}
\def\Shv{\mathrm{Shv}}
\def\Aff{\mathrm{Aff}}
\def\GL{\mathrm{GL}}

\def\Tw{\mathrm{Tw}{}}
\def\fp{\mathrm{fp}}
\def\Fun{\mathrm{Fun}}

\def\flf{\mathrm{f{}lf}}
\def\fet{\mathrm{f\acute et}}
\def\all{\mathrm{all}}
\def\lleft{\mathrm{left}}
\def\rright{\mathrm{right}}

\def\fold{\mathrm{fold}}
\def\Span{\mathrm{Span}}
\def\Sect{\mathrm{Sect}}
\def\nc{\mathrm{nc}}
\def\fin{\mathrm{fin}}

\def\inj{\mathrm{inj}}

\def\equi{\mathrm{equi}}
\def\SmSep{\mathrm{SmSep}{}}
\def\Chow{\mathcal C\mathrm{how}{}}

\def\Perf{\mathrm{Perf}{}}

\def\cg{\mathrm{cg}}
\def\ft{\mathrm{ft}}
\def\SmAff{\mathrm{SmAff}{}}
\def\SmNC{\mathrm{SmNC}{}}

\def\exc{\mathrm{exc}}
\def\wnc{\mathrm{wnc}}
\def\wexc{\mathrm{wexc}}
\def\gp{\mathrm{gp}}

\def\Pos{\mathcal P\mathrm{os}}

\def\Pro{\mathrm{Pro}}
\def\Gpd{\mathrm{Gpd}}
\def\FinGpd{\mathrm{FinGpd}}

\def\fincov{\mathrm{fcov}}
\def\Top{\mathcal{T}\mathrm{op}{}}

\let\amalg=\sqcup
\def\Bar{\mathrm{Bar}{}}

\newcommand{\adj}{\rightleftarrows}

\newcommand{\sslash}{\mathbin{/\mkern-6mu/}}

\let\lim=\relax
\DeclareMathOperator*{\lim}{lim}

\DeclareMathOperator*{\colim}{colim}

\def\f{\mathrm{f}}
\def\s{\mathrm{s}}
\let\what=\relax

\title{Norms in motivic homotopy theory}
\date{\today}

\author{Tom Bachmann}
\address{Mathematisches Institut\\
LMU München\\
Theresienstr. 39\\
80333 München\\
Germany}
\email{\href{mailto:tom.bachmann@zoho.com}{tom.bachmann@zoho.com}}
\urladdr{\url{http://tom-bachmann.com}}
\author{Marc Hoyois}
\address{Fakultät für Mathematik\\
Universität Regensburg\\
Universitätsstr. 31\\
93040 Regensburg\\
Germany}
\email{\href{mailto:marc.hoyois@ur.de}{marc.hoyois@ur.de}}
\urladdr{\url{http://www.mathematik.ur.de/hoyois/}}

\begin{document}
	
\maketitle

\begin{abstract}
If $f\colon S' \to S$ is a finite locally free morphism of schemes, we construct a symmetric monoidal ``norm'' functor $f_\otimes\colon \H_\pt(S') \to \H_\pt(S)$, where $\H_\pt(S)$ is the pointed unstable motivic homotopy category over $S$. 
If $f$ is finite étale, we show that it stabilizes to a functor $f_\otimes\colon \SH(S') \to \SH(S)$, where $\SH(S)$ is the $\P^1$-stable motivic homotopy category over $S$. 
Using these norm functors, we define the notion of a \emph{normed motivic spectrum}, which is an enhancement of a motivic $\E_\infty$-ring spectrum.
The main content of this text is a detailed study of the norm functors and of normed motivic spectra, and the construction of examples. In particular: we investigate the interaction of norms with Grothendieck's Galois theory, with Betti realization, and with Voevodsky's slice filtration; we prove that the norm functors categorify Rost's multiplicative transfers on Grothendieck–Witt rings; and we construct normed spectrum structures on the motivic cohomology spectrum $\HH\Z$, the homotopy $\K$-theory spectrum $\KGL$, and the algebraic cobordism spectrum $\MGL$. 
The normed spectrum structure on $\HH\Z$ is a common refinement of Fulton and MacPherson's mutliplicative transfers on Chow groups and of Voevodsky's power operations in motivic cohomology.
\end{abstract}

\tableofcontents

\section{Introduction}
\label{sec:intro}

The goal of this paper is to develop the formalism of norm functors in motivic homotopy theory, to study the associated notion of normed motivic spectrum, and to construct many examples.

\subsection{Norm functors}
\label{sub:norm-functors}

Norm functors in unstable motivic homotopy theory were first introduced by Voevodsky in \cite[\sect5.2]{DeligneNote}, in order to extend the symmetric power functors to the category of motives \cite[\sect 2]{MEMS}. In this work, we will be mostly interested in norm functors in \emph{stable} motivic homotopy theory. For every finite étale morphism of schemes $f\colon T\to S$, we will construct a norm functor
\[
f_\otimes\colon \SH(T) \to \SH(S),
\]
where $\SH(S)$ is Voevodsky's $\infty$-category of motivic spectra over $S$ \cite[Definition 5.7]{Voevodsky:1998}.
Another functor associated with $f$ is the pushforward $f_*\colon \SH(T)\to\SH(S)$. For comparison, if $\nabla\colon S\amalg S\to S$ is the fold map, then
\[
\nabla_*\colon \SH(S\amalg S)\simeq \SH(S)\times\SH(S)\to\SH(S)
\]
is the direct sum functor, whereas
\[
\nabla_\otimes\colon \SH(S\amalg S)\simeq \SH(S)\times\SH(S)\to\SH(S)
\]
is the smash product functor. This example already shows one of the main technical difficulties in the theory of norm functors: unlike $f_*$, the functor $f_\otimes$ preserves neither limits nor colimits. 

The norm functors are thus a new basic functoriality of stable motivic homotopy theory, to be added to the list $f_\sharp$, $f^*$, $f_*$, $f_!$, $f^!$. The three covariant functors $f_\sharp$, $f_*$, and $f_\otimes$ should be viewed as parametrized versions of the sum, product, and tensor product in $\SH(S)$:
\begin{center}
	\begin{tabular}{cc}\addlinespace[5pt]
		 categorical structure & parametrized extension  \\ \midrule 
		sum & $f_\sharp$  for $f$ smooth \\
		product & $f_*$ for $f$ proper \\
		tensor product & $f_\otimes$ for $f$ finite étale \\
		\midrule\addlinespace[5pt]
	\end{tabular}
\end{center}
Of course, the sum and the product coincide in $\SH(S)$, and much more generally the functors $f_\sharp$ and $f_*$ coincide up to a ``twist'' when $f$ is both smooth and proper.

The basic interactions between sums, products, and tensor products extend to this parametrized setting. For example, the \emph{associativity} and \emph{commutativity} of these operations are encoded in the compatibility of the assignments $f\mapsto f_\sharp,f_*,f_\otimes$ with composition and base change. The \emph{distributivity} of tensor products over sums and products has a parametrized extension as well. 
To express these properties in a coherent manner, we are led to consider \emph{$2$-categories of spans}\footnote{All the $2$-categories that appear in this introduction are $(2,1)$-categories, i.e., their $2$-morphisms are invertible.}, which will be ubiquitous throughout the paper. The most fundamental example is the $2$-category $\Span(\Fin)$ of spans of finite sets: its objects are finite sets, its morphisms are spans $X\from Y\to Z$, and composition of morphisms is given by pullback. The one-point set in this $2$-category happens to be the \emph{universal commutative monoid}: if $\scr C$ is any $\infty$-category with finite products, then a finite-product-preserving functor $\Span(\Fin)\to\scr C$ is precisely a commutative monoid in $\scr C$ (we refer to Appendix~\ref{app:spans} for background on $\infty$-categories of spans and for a proof of this fact). 
For example, the $\infty$-category $\SH(S)$ is itself a commutative monoid in the $\infty$-category $\Cat_\infty$ of $\infty$-categories, under either the direct sum or the smash product. These commutative monoid structures correpond to functors
\[
\SH(S)^\oplus, \SH(S)^\otimes\colon \Span(\Fin) \to \Cat_\infty,\quad I\mapsto \SH(S)^I.
\]
Using that $\SH(S)^I\simeq \SH(\coprod_I S)$, one can splice these functors together as $S$ varies to obtain
\[
\SH^\oplus,\SH^\otimes\colon \Span(\Sch,\all,\fold) \to \Cat_\infty, \quad S\mapsto \SH(S).
\]
Here, $\Span(\Sch,\all,\fold)$ is a $2$-category whose objects are schemes and whose morphisms are spans $X\from Y\to Z$ where $X\from Y$ is arbitrary and $Y\to Z$ is a sum of fold maps (i.e., there is a finite coproduct decomposition $Z=\coprod_i Z_i$ and isomorphisms $Y\times_ZZ_i \simeq Z_i^{\sqcup n_i}$ over $Z_i$). It turns out that the functors $\SH^\oplus$ and $\SH^\otimes$ encode exactly the same information as the presheaves of symmetric monoidal $\infty$-categories $S\mapsto \SH(S)^\oplus$ and $S\mapsto\SH(S)^\otimes$. 
The basic properties of parametrized sums, products, and tensor products are then encoded by extensions of $\SH^\oplus$ and $\SH^\otimes$ to larger $2$-categories of spans:
\begin{align*}
	\SH^\amalg\colon &\Span(\Sch,\all,\mathrm{smooth}) \to \Cat_\infty,\quad S\mapsto \SH(S),\quad (X\xleftarrow f Y\xrightarrow p Z)\mapsto p_\sharp f^*,\\
	\SH^\times\colon &\Span(\Sch,\all,\mathrm{proper}) \to \Cat_\infty,\quad S\mapsto \SH(S),\quad (X\xleftarrow f Y\xrightarrow p Z)\mapsto p_* f^*,\\
	\SH^\otimes\colon &\Span(\Sch,\all,\fet) \to \Cat_\infty,\quad S\mapsto \SH(S),\quad (X\xleftarrow f Y\xrightarrow p Z)\mapsto p_\otimes f^*,
\end{align*}
where ``$\fet$'' is the class of finite étale morphisms.
A crucial difference between the first two functors and the third one is that the former are completely determined by their restriction to $\Sch^\op$ and the knowledge that the functors $f_\sharp$ and $f_*$ are left and right adjoint to $f^*$. The existence of the extensions $\SH^\amalg$ and $\SH^\times$ is then a formal consequence of the smooth base change theorem (which holds in $\SH$ by design) and the proper base change theorem (formulated by Voevodsky \cite{Voevodsky6functors} and proved by Ayoub \cite{Ayoub}).
By contrast, the functor $\SH^\otimes$ must be constructed by hand.

We will not have much use for the compactly supported pushforward $f_!$ and its right adjoint $f^!$. For completeness, we note that there is a functor
\[
\Span(\Sch,\all,\mathrm{lft}) \to \Cat_\infty,\quad S\mapsto \SH(S),\quad (X\xleftarrow f Y\xrightarrow p Z)\mapsto p_! f^*,
\]
where ``lft'' is the class of morphisms locally of finite type,
that simultaneously extends $\SH^\times$ and a twisted version of $\SH^\amalg$. This can be regarded as a vast generalization of the identification of finite sums and finite products in $\SH(S)$, which also subsumes Poincaré duality.

Any finite-product-preserving functor
\[
\Span(\Sch,\all,\fet) \to \Cat_\infty,
\]
lifts uniquely to the $\infty$-category of symmetric monoidal $\infty$-categories. Such a ``presheaf of $\infty$-categories with finite étale transfers'' can therefore be viewed as a presheaf of symmetric monoidal $\infty$-categories, where the symmetric monoidal structures are enhanced with norm functors along finite étale maps. 
In addition to the case of stable motivic homotopy theory $\SH(S)$, we will construct norm functors for several other theories, such as:
\begin{itemize}
	\item Morel and Voevodsky's unstable motivic homotopy theory $\H(S)$ and its pointed version $\H_\pt(S)$;
	\item Voevodsky's theory of motives $\DM(S)$ and some of its variants;
	\item Robalo's theory of noncommutative motives $\SH_\nc(S)$ and some of its variants.
\end{itemize}
In the case of $\H(S)$ and $\H_\pt(S)$, we even have well-behaved norm functors $f_\otimes$ for all finite locally free morphisms $f$, but in the other cases we do not know how to construct norm functors in this generality.

It is worth noting here that any étale sheaf of symmetric monoidal $\infty$-categories on the category of schemes \emph{extends uniquely} to a functor
\[
\Span(\Sch,\all,\fet) \to \Cat_\infty.
\]
This is because finite étale maps are fold maps (i.e., of the form $S^{\sqcup n}\to S$) locally in the étale topology, and so one may use descent and the symmetric monoidal structure to define norm functors $f_\otimes$.
For example, the étale version of stable motivic homotopy theory $\SH_\et(S)$, the theory of rational motives $\DM(S,\Q)$, and the theory of $\ell$-adic sheaves automatically acquire norm functors in this way.
All the examples mentioned above are crucially \emph{not} étale sheaves, which is why the construction of norm functors is nontrivial and interesting.

\subsection{Normed motivic spectra}
Given a symmetric monoidal $\infty$-category
\[
\scr C^\otimes\colon \Span(\Fin) \to \Cat_\infty, \quad I\mapsto\scr C^I, \quad (I\xleftarrow fJ\xrightarrow pK)\mapsto p_\otimes f^*,
\]
we may consider the cocartesian fibration $q\colon \int\scr C^\otimes \to \Span(\Fin)$ classified by $\scr C^\otimes$, which is an $\infty$-categorical version of the Grothendieck construction.
Every object $A\in\scr C$ induces a section of $q$ over $\Fin^\op\subset \Span(\Fin)$ that sends the finite set $I$ to the constant family $p_I^*(A)\in\scr C^I$, where $p_I\colon I\to *$.
An \emph{$\E_\infty$-algebra structure} on $A$ is an extension of this section to $\Span(\Fin)$. In particular, for every finite set $I$, such an extension provides a multiplication map
\[
\mu_{p_I}\colon p_{I\otimes} p_I^*(A) =\bigotimes_{i\in I}A \to A,
\]
and the functoriality on $\Span(\Fin)$ encodes the associativity and commutativity of this multiplication.

Consider the functor
\[
\SH^\otimes\colon \Span(\Sch,\all,\fet) \to \Cat_\infty,\quad S\mapsto \scr \SH(S),
\]
encoding the norms in stable motivic homotopy theory.
If $\scr C$ is any full subcategory of $\Sch_S$ that contains $S$ and is closed under finite sums and finite étale extensions, one can show that an $\E_\infty$-ring spectrum in $\SH(S)$ is equivalent to a section of $q\colon \int\SH^\otimes \to \Span(\Sch,\all,\fet)$ over $\Span(\scr C,\all,\fold)$ that sends backward maps to cocartesian edges.
This immediately suggests a notion of \emph{normed spectrum} in $\SH(S)$: it is a section of $q$ over $\Span(\scr C,\all,\fet)$ that sends backward maps to cocartesian edges.
In fact, we obtain notions of normed spectra of varying strength depending on $\scr C$: the weakest with $\scr C=\FEt_S$ and the strongest with $\scr C=\Sch_S$. The intermediate case $\scr C=\Sm_S$ seems to be the most relevant for applications, but several of our examples will be constructed with $\scr C=\Sch_S$.
In this introduction, we will only consider the case $\scr C=\Sm_S$ for simplicity, and we refer to the text for more general statements.

Concretely, a normed spectrum over $S$ is a motivic spectrum $E\in \SH(S)$ equipped with maps
\[
\mu_f\colon f_\otimes E_Y\to E_X
\]
in $\SH(X)$ for all finite étale maps $f\colon Y\to X$ in $\Sm_S$, subject to some coherence conditions that in particular make $E$ an $\E_\infty$-ring spectrum.
The $\infty$-category of normed motivic spectra over $S$ is denoted by $\NAlg_\Sm(\SH(S))$. It is monadic over $\SH(S)$ as well as monadic and comonadic over the $\infty$-category $\CAlg(\SH(S))$ of motivic $\E_\infty$-ring spectra over $S$.

A motivic space $E\in\H(S)$ has an underlying cohomology theory 
\[
E^{0,0}\colon \Sm_S^\op \to \Set,\quad X\mapsto [\1_X,E_X].
\]
If $E\in\SH(S)$, this cohomology theory acquires \emph{additive transfers} $\tau_f\colon E^{0,0}(Y)\to E^{0,0}(X)$ along finite étale maps $f\colon Y\to X$, which define an extension of $E^{0,0}$ to the category of spans $\Span(\Sm_S,\all,\fet)$. 
If $E\in\NAlg_\Sm(\SH(S))$,
the cohomology theory $E^{0,0}$ also acquires \emph{multiplicative transfers} $\nu_f\colon E^{0,0}(Y)\to E^{0,0}(X)$ along finite étale maps $f\colon Y\to X$, which define another extension of $E^{0,0}$ to $\Span(\Sm_S,\all,\fet)$. Together, these additive and multiplicative extensions form a \emph{Tambara functor} on $\Sm_S$, in the sense of \cite[Definition 8]{bachmann-gwtimes}.
In fact, all of this structure exists at the level of the space-valued cohomology theory 
\[
\Omega^\infty E\colon\scr \Sm_S^\op\to \scr S,\quad X\mapsto\Map(\1_X,E_X).
\]
For every $X\in\Sm_S$, $\Omega^\infty E(X)$ is an $\E_\infty$-ring space, and for every finite étale map $f\colon Y\to X$ in $\Sm_S$, we have an additive $\E_\infty$-map $\tau_f\colon \Omega^\infty E(Y)\to \Omega^\infty E(X)$ and a multiplicative $\E_\infty$-map $\nu_f\colon \Omega^\infty E(Y)\to \Omega^\infty E(X)$, which together form a space-valued Tambara functor on $\Sm_S$.
This Tambara functor contains much of the essential structure of the normed motivic spectrum $E$, though in general $E$ also includes the data of \emph{nonconnective} infinite deloopings of the spaces $\Omega^\infty E(X)$.

Normed spectra also have \emph{power operations} generalizing Voevodsky's power operations in motivic cohomology (which arise from a normed structure on the motivic Eilenberg–Mac Lane spectrum $\HH \Z$). For example, if $E$ is a normed spectrum over $S$ and $X\in\Sm_S$, the $n$th power map $\Omega^\infty E(X)\to\Omega^\infty E(X)$ in cohomology admits a refinement
\[
\rm P_n\colon \Omega^\infty E(X) \to \Omega^\infty E(X\times\B_\et\Sigma_n),
\]
where $\B_\et\Sigma_n$ is the classifying space of the symmetric group $\Sigma_n$ in the étale topology.

Finally, we will show that normed spectra are closed under various operations. If $E\in\SH(S)$ is a normed spectrum, then the following spectra inherit normed structures: the localization $E[1/n]$ and the completion $E_n^\comp$ for any integer $n$, the rationalization $E_\Q$, the effective cover $\f_0E$, the very effective cover $\tilde\f_0E$, the zeroth slice $\s_0E$, the generalized zeroth slice $\tilde\s_0E$, the zeroth effective homotopy module $\spi_0^\eff(E)$, the associated graded for the slice filtration $\bigvee_{n\in \Z}\s_nE$ and for the generalized slice filtration $\bigvee_{n\in \Z}\tilde\s_nE$, the pullback $f^*(E)$ for any pro-smooth morphism $f\colon S'\to S$, and the pushforward $f_*(E)$ for any morphism $f\colon S\to S'$.

\subsection{Examples of normed spectra}
Besides developing the general theory, the main goal of this paper is to provide many examples and techniques of construction of normed motivic spectra. 
Let us first review some examples of multiplicative transfers in algebraic geometry:
\begin{itemize}
	\item In \cite{FultonMacPherson}, Fulton and MacPherson construct norm maps on Chow groups: for $f\colon Y\to X$ a finite étale map between smooth quasi-projective schemes over an algebraically closed field, they construct a multiplicative transfer $\CH^*(Y)\to \CH^*(X)$, where
	\[
	\CH^*(X) = \bigoplus_{n\geq 0}\CH^n(X).
	\]
	\item In \cite{Joukhovitski}, Joukhovitski constructs norm maps on $\K_0$-groups: for $f\colon Y\to X$ a finite étale map between quasi-projective schemes over a field, he constructs a multiplicative transfer $\K_0(Y)\to \K_0(X)$, where $\K_0(X)$ is the Grothendieck group of the exact category of vector bundles on $X$.
	\item In \cite{rost2003multiplicative}, Rost constructs norm maps on Grothendieck–Witt groups of field: for $F\subset L$ a finite separable extension of fields, he constructs a multiplicative transfer $\GW(L)\to \GW(F)$. This transfer was further studied by Wittkop in \cite{wittkop2006multiplikative}.
\end{itemize}
Each of these examples is the underlying cohomology theory of a motivic spectrum: the Chow group $\CH^*(X)$ is the underlying set of the ``periodized'' motivic Eilenberg–Mac Lane spectrum $\bigvee_{n\in\Z}\Sigma^{2n,n}\HH\Z_X$, the group $\K_0(X)$ is the underlying set of Voevodsky's $\K$-theory spectrum $\KGL_X$ (assuming $X$ regular), and the Grothendieck–Witt group $\GW(F)$ of a field $F$ is the underlying set of the motivic sphere spectrum $\1_F$.
We will show that the above norm maps are all induced by normed structures on the corresponding motivic spectra (assuming $\Char(F)\neq 2$ in the last case).

Being the unit object for the smash product, the motivic sphere spectrum has a unique structure of normed spectrum, so the task in this case is to compute the effect of the norm functors on the endomorphisms of the sphere spectrum.
Our strategy for constructing norms on $\HH\Z$ and $\KGL$ is \emph{categorification}: as mentioned in~\sect\ref{sub:norm-functors}, we construct norm functors on the $\infty$-categories $\DM(S)$ and $\SH_\nc(S)$. 
We also promote the canonical functors $\SH(S)\to \DM(S)$ and $\SH(S)\to\SH_\nc(S)$ to natural transformations
\[
\SH^\otimes \to \DM^\otimes,\; \SH^\otimes \to \SH_\nc^\otimes\colon \Span(\Sch,\all,\fet) \to \Cat_\infty.
\]
The right adjoints of these functors send the unit objects to $\HH\Z$ and $\KGL$, respectively.
Just as the right adjoint of a symmetric monoidal functor preserves $\E_\infty$-algebras, it follows formally that $\HH\Z$ and $\KGL$ are normed spectra.
In the case of $\HH\Z$, some amount of work is needed to show that this structure recovers the Fulton–MacPherson norms, since for this model of $\HH\Z$ the isomorphism $\CH^n(X)\simeq [\1_X,\Sigma^{2n,n}\HH\Z_X]$ is highly nontrivial.
We will also show that $\HH R$ is a normed spectrum for every commutative ring $R$.

We obtain other interesting examples using the compatibility of norms with the slice filtration and with the effective homotopy $t$-structure. Over a Dedekind domain of mixed characteristic, we show that Spitzweck's $\E_\infty$-ring spectrum $\HH\Z^\Spi$ representing Bloch–Levine motivic cohomology is in fact a normed spectrum. In particular, this extends the Fulton–MacPherson norms to Chow groups in mixed characteristic. Over a field, we show that the spectrum $\HH\tilde\Z$ representing Milnor–Witt motivic cohomology is a normed spectrum. In particular, we obtain a lift of the Fulton–MacPherson norms to Chow–Witt groups.

Our next major example comes from the theory of motivic Thom spectra. In particular, Voevodsky's algebraic cobordism spectrum $\MGL$ and its periodization $\bigvee_{n\in\Z}\Sigma^{2n,n}\MGL$ are normed spectra, and similarly for other standard families of algebraic groups instead of the $\GL$-family. While it is well-known that $\MGL$ is an $\E_\infty$-ring spectrum, the existing constructions of this $\E_\infty$-ring structure rely on point-set models for the smash product of motivic spectra. To realize $\MGL$ as a normed spectrum, we will need a new description of $\MGL$ itself.
 In topology, the complex cobordism spectrum $\mathrm{MU}$ is the Thom spectrum (i.e., the colimit) of the $\rm J$-homomorphism $j\colon \B\mathrm{U} \to \Sp$, where $\Sp$ is the $\infty$-category of spectra. From this perspective, the $\E_\infty$-ring structure on $\mathrm{MU}$ comes from the fact that $j\colon \B\mathrm{U} \to \Sp$ is an $\E_\infty$-map for the smash product symmetric monoidal structure on $\Sp$. We will establish a similar picture in motivic homotopy theory. The motivic analog of the $\rm J$-homomorphism is a natural transformation 
 $j\colon \K^\circ\to \SH$,
 where $\K^\circ$ is the rank $0$ summand of the Thomason–Trobaugh $\K$-theory space, and we show that $\MGL_S\in\SH(S)$ is the \emph{relative colimit} of $j|\Sm_S$. We moreover extend $j$ to the $2$-category of spans $\Span(\Sch,\all,\fet)$ using the pushforward in $\K$-theory and the norm functors on $\SH(-)$. It then follows more or less formally that $\MGL_S$ is a normed spectrum. 
 
In general, the \emph{motivic Thom spectrum functor} $\M_S$ sends a presheaf $A\in\PSh(\Sm_S)$ equipped with a natural transformation $\phi\colon A\to\SH$ to a motivic spectrum $\M_S(\phi)\in\SH(S)$. As in topology, this functor is a powerful tool to construct structured motivic spectra, including normed spectra. We also use $\M_S$ to give a formula for the free normed spectrum on $E\in\SH(S)$: it is the motivic Thom spectrum of a certain natural transformation $\coprod_{n\geq 0} \B_\et\Sigma_n\to\SH$ that refines the smash powers $E^{\wedge n}$ with their $\Sigma_n$-action.

Our list of examples of normed motivic spectra is of course not exhaustive. An important missing example is the Hermitian $\K$-theory spectrum $\mathrm{KO}$ \cite{Hornbostel}: we expect that it admits a canonical structure of normed spectrum, but we do not attempt to construct it in this paper. 
Finally, it is worth noting that not every motivic $\E_\infty$-ring spectrum admits a normed structure (see Examples \ref{ex:hill-hopkins}, \ref{ex:eta-localization}, and~\ref{ex:h-completion}). Witt theory is a concrete example of a cohomology theory that is represented by an $\E_\infty$-ring spectrum but does not have norms.

\subsection{Norms in other contexts}
Many of the ideas presented so far are not specific to the category of schemes. 
For example, if $f\colon Y\to X$ is a finite covering map of topological spaces, there is a norm functor $f_\otimes$ between the $\infty$-categories of sheaves of spectra.
 However, the construction of these norm functors is essentially trivial because finite covering maps are locally fold maps. Moreover, normed spectra in this context are nothing more than sheaves of $\E_\infty$-ring spectra (see Remark~\ref{rmk:SH-top} for justifications of these claims). This also means that the cohomology theory represented by any $\E_\infty$-ring spectrum automatically has norms along finite covering maps. In the case of ordinary cohomology, these norm maps were first studied by Steiner \cite{Steiner}, expanding on earlier work of Evens \cite{Evens} for the cohomology of groups.
 
The story becomes much more interesting when we consider stable \emph{equivariant} homotopy theory. Norm functors in this context were constructed by Hill, Hopkins, and Ravenel and played an important role in their resolution of the Kervaire invariant one problem \cite{HHR}. 
To put this theory in our framework, one must first note that the $\infty$-category of $G$-spectra, for $G$ a finite group, does not fully depend on the group $G$ but only on the \emph{groupoid} $\B G$; we will write it as $\SH(\B G)$. Being a presheaf of symmetric monoidal $\infty$-categories under the smash product, $\B G\mapsto \SH(\B G)$ extends to a finite-product-preserving functor
\[
\SH^\otimes\colon \Span(\FinGpd,\all,\fold) \to\Cat_\infty,
\]
where $\FinGpd$ is the $2$-category of finite groupoids. The theory of norm functors is then encoded by an extension of $\SH^\otimes$ to a functor
\[
\SH^\otimes\colon \Span(\FinGpd,\all,\fincov) \to \Cat_\infty,\quad X\mapsto \SH(X),\quad (X\xleftarrow f Y\xrightarrow p Z)\mapsto p_\otimes f^*
\]
where ``$\fincov$'' is the class of finite covering maps. As in the motivic case, this extension is nontrivial because the presheaf $\SH\colon \FinGpd^\op\to \Cat_\infty$ does not satisfy descent with respect to finite covering maps. If $G$ is a finite group and $H\subset G$ is a subgroup, the induced map $p\colon \B H\to \B G$ is a finite covering map and $p_\otimes\colon \SH(\B H)\to \SH(\B G)$ is the Hill–Hopkins–Ravenel norm $\Norm_H^G$.

The category $\Fin_G$ of finite $G$-sets can be identified with the full subcategory of $\FinGpd_{/\B G}$ spanned by the finite covering maps. A \emph{normed $G$-spectrum} is a section of the cocartesian fibration $\int\SH^\otimes\to \Span(\FinGpd,\all,\fincov)$ over $\Span(\Fin_G)$ that sends backward maps to cocartesian edges. One can show that this recovers the classical notion of $G$-$\E_\infty$-ring spectrum or $G$-commutative ring spectrum.

The norm functors in stable equivariant homotopy theory can be constructed in essentially the same way as the norm functors in stable motivic homotopy theory. Moreover, the two constructions are directly related via Grothendieck's Galois theory of schemes. Indeed, the profinite étale fundamental groupoid can be promoted to a functor
\[
\widehat\Pi_1^\et\colon \Span(\Sch,\all,\fet) \to \Span(\Pro(\FinGpd),\all,\fincov),
\]
allowing us to view stable equivariant homotopy theory (extended in an obvious way to profinite groupoids) as a presheaf of $\infty$-categories on schemes with finite étale transfers.
Galois theory then provides a functor from stable equivariant homotopy theory to stable motivic homotopy theory (see \cite{HellerOrmsby} in the case of a base field), which is compatible with the norm functors. More precisely, there is a natural transformation
\[
\SH^\otimes \circ \widehat\Pi_1^\et \to \SH^\otimes \colon \Span(\Sch,\all,\fet)\to\Cat_\infty
\]
that sends the suspension spectrum of a finite $\widehat\Pi_1^\et(S)$-set to the suspension spectrum of the corresponding finite étale $S$-scheme.

This Galois correspondence turns out to be useful in the study of motivic norms. For example, we use it to show that one can invert integers in normed motivic spectra and that the motivic norm functors preserve effective spectra. On the other hand, one can use the Galois correspondence to construct $G$-$\E_\infty$-ring spectra from algebraic geometry. For example, if $S$ is smooth over a field and $S'\to S$ is an étale Galois cover with Galois group $G$, then Bloch's cycle complex $z^*(S,*)$ can be promoted to a $G$-$\E_\infty$-ring spectrum (whose underlying spectrum is a connective $\HH\Z$-module). Since the $\E_\infty$-ring structure is already very difficult to construct, this is a particularly nontrivial example.

Another connection between motivic and equivariant homotopy theory is the real Betti realization functor $\SH(\R)\to \SH(\B\rm C_2)$ (defined unstably in \cite[\sect 3.3]{MV} and stably in \cite[\sect 4.4]{HellerOrmsby}). To prove that this functor is compatible with norms, we will give another construction of the norm functors in stable equivariant homotopy theory based on a topological model for $G$-spectra.

\subsection{Norms vs.\ framed transfers}

In ordinary homotopy theory, an important property of an $\E_\infty$-ring spectrum $E$ is that its space of units $(\Omega^\infty E)^\times$ is canonically the zeroth space of a connective spectrum. This is because $(\Omega^\infty E)^\times$ is a grouplike $\E_\infty$-space, and by a theorem of Segal $\Omega^\infty$ induces an equivalence between connective spectra and grouplike $\E_\infty$-spaces \cite[Proposition 3.4]{segal1974categories}. This is already interesting for the sphere spectrum, which is an $\E_\infty$-ring spectrum for trivial reasons. Indeed, the resulting spectrum of units classifies stable spherical fibrations and plays a crucial role in surgery theory.

In equivariant homotopy theory, the space of units of a \emph{normed} $G$-spectrum is similarly the zeroth space of a connective $G$-spectrum. Again, this is because connective $G$-spectra may be identified with grouplike ``normed $G$-spaces'', after a theorem of Guillou–May \cite[Theorem 0.1]{GuillouMay} and Nardin \cite[Theorem A.4]{BDGNS4}.

An analogous result in motivic homotopy theory is highly desirable. Evidence for such a result was provided by the first author in \cite{bachmann-gwtimes}, where it is shown that the units in $\underline{\GW}\simeq \Omega^\infty\spi^\eff_0(\1)$ are the zeroth space of a motivic spectrum. Unfortunately, our current understanding of stable motivic homotopy theory is not sufficient to formulate a more general result. On the one hand, if $E$ is a normed spectrum, then $(\Omega^\infty E)^\times$ is a motivic space with finite étale transfers. On the other hand, working over a perfect field, there is an equivalence between very effective motivic spectra and grouplike motivic spaces with \emph{framed finite syntomic} transfers, which are more complicated than finite étale transfers \cite{EHKSY}. 
We would therefore need to bridge the gap between finite étale transfers and framed finite syntomic transfers. This could be achieved in one of two ways:
\begin{itemize}
	\item by showing that a motivic space with finite étale transfers automatically has framed finite syntomic transfers;
	\item by enhancing the theory of normed spectra so as to entail framed finite syntomic transfers on spaces of units.
\end{itemize}
Either approach seems very difficult.

\subsection{Summary of the construction}
We give a quick summary of the construction of the norm functor
\[
f_\otimes\colon \SH(T)\to \SH(S)
\]
for a finite étale morphism $f\colon T\to S$,
which is the content of Section~\ref{sec:stablenorms}.
Let $\SmQP_{S+}$ denote the category of pointed $S$-schemes of the form $X_+ = X \amalg S$, where $X \to S$ is smooth and quasi-projective. 
It is a symmetric monoidal category under the \emph{smash product} $X_+ \wedge Y_+ = (X \times Y)_+$. 
Alternatively, one can view $\SmQP_{S+}$ as the category whose objects are smooth quasi-projective $S$-schemes and whose morphisms are partially defined maps with clopen domains of definition.
If $f\colon T \to S$ is finite locally free (i.e., finite, flat, and of finite presentation), the pullback functor $f^*\colon \SmQP_{S} \to \SmQP_{T}$ has a right adjoint $\Weil_f$ called \emph{Weil restriction} \cite[\sect7.6]{NeronModels}. Moreover, $\Weil_f$ preserves clopen immersions and therefore extends in a canonical way to a symmetric monoidal functor $f_\otimes\colon \SmQP_{T+} \to \SmQP_{S+}$.

The symmetric monoidal $\infty$-category $\SH(S)$ can be obtained from $\SmQP_{S+}$ in three steps:
\begin{enumerate}
	\item sifted cocompletion $\SmQP_{S+}\to \PSh_\Sigma(\SmQP_S)_\pt$;
	\item motivic localization $\PSh_\Sigma(\SmQP_S)_\pt\to \H_\pt(S)$;
	\item $\P^1$-stabilisation $\H_\pt(S)\to\SH(S)$.
\end{enumerate}
Here, $\PSh_\Sigma(\SmQP_S)_\pt$ is the $\infty$-category of presheaves of pointed spaces on $\SmQP_{S}$ that transform finite coproducts into finite products (called \emph{radditive presheaves} by Voevodsky \cite{Voevodsky:2010b}).
Crucially, each of these steps is described by a universal construction in the $\infty$-category of symmetric monoidal $\infty$-categories with \emph{sifted} colimits. For step (3), this is a minor refinement of the universal property of $\P^1$-stabilization proved by Robalo \cite[\sect 2]{Robalo}.
It follows that an arbitrary symmetric monoidal functor $F\colon \SmQP_{T+}\to \SmQP_{S+}$ has at most one extension to a symmetric monoidal functor $\SH(T)\to \SH(S)$ that preserves sifted colimits, in the sense that the space of such extensions is either empty or contractible. More precisely:
\begin{enumerate}
	\item $F$ extends unconditionally to a functor $\PSh_\Sigma(\SmQP_T)_\pt\to \PSh_\Sigma(\SmQP_S)_\pt$.
	\item $F$ extends to $\H_\pt(T)\to\H_\pt(S)$ if and only if:
	\begin{itemize}
		\item for every étale map $U\to X$ in $\SmQP_T$ that is a covering map for the Nisnevich topology, the augmented simplicial diagram
	\begin{tikzmath}
		\def\colsep{.9em}
		\diagram{
		\dotsb & F((U\times_XU)_+) & F(U_+) & F(X_+) \\
		};
		\arrows
		(11-) edge[-top,vshift=2*\dbl] (-12) edge[-mid] (-12) edge[-bot,vshift=-2*\dbl] (-12)
		(12-) edge[-top,vshift=\dbl] (-13) edge[-bot,vshift=-\dbl] (-13) (13-) edge (-14);
	\end{tikzmath}
	is a colimit diagram in $\H_\pt(S)$;
		\item for every $X\in\SmQP_T$, the map $F((\A^1\times X)_+) \to F(X_+)$ induced by the projection $\A^1\times X\to X$ is an equivalence in $\H_\pt(S)$.
	\end{itemize} 
	\item $F$ extends further to $\SH(T)\to \SH(S)$ if and only if the cofiber of $F(\infty)\colon F(T_+) \to F(\P^1_{T+})$ is invertible in $\SH(S)$.
\end{enumerate}
Now if $f\colon T\to S$ is finite locally free, we show that the functor $f_\otimes\colon \SmQP_{T+}\to\SmQP_{S+}$ satisfies condition (2) (see Theorem \ref{thm:norm}), which gives us the unstable norm functor
\[
f_\otimes\colon \H_\pt(T)\to \H_\pt(S).
\]
Unfortunately, $f_\otimes$ does not satisfy condition (3) in general (see Remark~\ref{rmk:norms-dont-stabilize}), but we show that it does when $f$ is finite étale (see Lemma~\ref{lem:spheres}).
This requires a detailed study of the unstable norm functors, which is the content of Section~\ref{sec:pointednorms}. 
In Section~\ref{sec:coherence}, we explain how to assemble the norm functors into a functor
\[
\SH^\otimes\colon \Span(\Sch,\all,\fet) \to \Cat_\infty,
\]
which allows us to define the $\infty$-categories of normed spectra in Section~\ref{sec:normedspectra}.

\subsection{Summary of results}
Beyond the construction of the norm functors,
our main results are as follows.
\begingroup\renewcommand{\descriptionlabel}[1]{\hspace{\labelsep}\textit{#1.}}
\begin{description}
\item[Basic properties of norms (Section \ref{sec:functoriality})] The norm functors $f_\otimes$ interact with functors of the form $g_\sharp$ and $h_*$ via so-called \emph{distributivity laws}, and they are compatible with the purity and ambidexterity equivalences \cite[\sect 1.6.3 and \sect 1.7.2]{Ayoub}. If $f\colon T\to S$ is finite étale of degree $\leq n$, then $f_\otimes\colon \SH(T) \to \SH(S)$ is an $n$-excisive functor \cite[Definition 6.1.1.3]{HA}.
\item[Coherence of norms (Section \ref{sec:coherence})] We construct the functor $\SH^\otimes\colon \Span(\Sch,\all,\fet)\to \Cat_\infty$, and we give general criteria for norms to preserve subcategories or to be compatible with localizations.
\item[Basic properties of normed spectra (Section \ref{sec:normedspectra})] The $\infty$-category $\NAlg_\Sm(\SH(S))$ is presentable and is both monadic and comonadic over $\CAlg(\SH(S))$ \cite[Definition 4.7.3.4]{HA}. If $f\colon S' \to S$ is an arbitrary morphism, then $f_*\colon \SH(S') \to \SH(S)$ preserves normed spectra, and if $f$ is pro-smooth, then $f^*\colon \SH(S) \to \SH(S')$ preserves normed spectra.
If $E$ is a normed spectrum, its underlying cohomology theory $E^{0,0}(\ph)$ is a Tambara functor \cite[Definition 8]{bachmann-gwtimes}.
\item[Norm–pullback–pushforward adjunctions (Section \ref{sec:f-tens-*-adj})] If $f\colon S' \to S$ is finite étale, there is an adjunction \[f_\otimes: \NAlg_\Sm(\SH(S')) \adj \NAlg_\Sm(\SH(S)): f^*.\] 
If $f\colon S' \to S$ is pro-smooth, there is an adjunction \[f^*: \NAlg_\Sm(\SH(S)) \adj \NAlg_\Sm(\SH(S')): f_*.\]
\item[Equivariant spectra (Section \ref{sec:galois-equiv})] We construct a functor $\SH^\otimes\colon \Span(\Pro(\Fin\Gpd),\all,\fp)\to\Cat_\infty$ that encodes norms and geometric fixed points in stable equivariant homotopy theory \cite{lewis1986equivariant,HHR}.
\item[Grothendieck's Galois theory (Section \ref{sec:ggt})] 
If $S$ is a scheme, there is a canonical symmetric monoidal functor $c_S\colon \SH(\widehat\Pi_1^\et(S)) \to \SH(S)$, where $\widehat\Pi_1^\et(S)$ is the profinite completion of the étale fundamental groupoid of $S$. It is compatible with norms, and its right adjoint induces a functor $\NAlg_{\Sm}(\SH(S))\to \NAlg(\SH(\widehat\Pi_1^\et(S)))$. 
If $S$ is a regular local scheme over a field of characteristic $\neq 2$ and $f\colon T\to S$ is finite étale, the map $f_\otimes\colon [\1_T,\1_T]\to [\1_S,\1_S]$ coincides with Rost's multiplicative transfer $\GW(T)\to\GW(S)$ \cite{rost2003multiplicative}.
\item[Betti realization (Section \ref{sec:betti})] The $\rm C_2$-equivariant Betti realization functor $\rB\colon \SH(\R) \to \SH(\B\rm C_2)$ \cite[\sect4.4]{HellerOrmsby} is compatible with norms and induces a functor $\NAlg_\Sm(\SH(\R)) \to \NAlg(\SH(\B\rm C_2))$.
\item[Localization (Section \ref{sec:localization})] If $E$ is a normed spectrum and $\alpha \in \pi_{-n,-m}(E)$, we give criteria for 
\begin{align*}
E[1/\alpha] &= \colim(E \xrightarrow{\alpha} E \wedge \S^{n,m} \xrightarrow{\alpha} E \wedge \S^{2n,2m} \xrightarrow{\alpha} \dotsb)\quad \text{and}\\
E_\alpha^\comp &= \lim(\dotsb \rightarrow E/\alpha^3 \rightarrow E/\alpha^2 \rightarrow E/\alpha)
\end{align*}
to be normed spectra, and we provide positive and negative examples. In particular we show that $E[1/n]$ and $E_n^\comp$ are normed spectra for any integer $n$.
\item[Slices (Section \ref{sec:slices})] The norm functors are compatible with the slice filtration \cite[\sect 2]{Voevodsky:2002}. If $E$ is a normed spectrum, then $\f_0 E$, $\s_0 E$, $\bigvee_{n\in\Z}\f_n E$, and $\bigvee_{n\in\Z}\s_n E$ are normed spectra, where $\f_n E$ is the $n$-effective cover of $E$ and $\s_n E$ is the $n$th slice of $E$. In fact, there is a notion of graded normed motivic spectrum, and $\f_*E$ and $\s_* E$ are $\Z$-graded normed motivic spectra. We obtain analogous results for the generalized slice filtration \cite{Bachmann-slices}, and we also show that the $0$th effective homotopy module $\spi_0^\eff(E)$ is a normed spectrum. We construct normed structures on Spitzweck's spectrum $\HH\Z^\Spi$ representing Bloch–Levine motivic cohomology in mixed characteristic \cite{SpitzweckHZ}, and on the generalized motivic cohomology spectrum $\HH\tilde\Z$ representing Milnor–Witt motivic cohomology \cite{CalmesFasel}.
\item[Motives (Section \ref{sec:PST})] We construct norm functors $f_\otimes\colon \DM(T, R) \to \DM(S, R)$ for $f\colon T\to S$ finite étale and $S$ noetherian, where $\DM(S, R)$ is Voevodsky's $\infty$-category of motives over $S$ with coefficients in a commutative ring $R$. The canonical functor $\SH(S)\to \DM(S,R)$ is compatible with norms and induces an adjunction 
\[\NAlg_\Sm(\SH(S)) \adj \NAlg_\Sm(\DM(S,R)).\] 
In particular, the motivic Eilenberg–Mac Lane spectrum $\HH R_S$ is a normed spectrum. The induced norm maps on Chow groups coincide with the Fulton–MacPherson construction \cite{FultonMacPherson}.
\item[Noncommutative motives (Section \ref{sec:dgCat})] We prove similar results for noncommutative motives (in the sense of Robalo \cite{Robalo}) 
in place of ordinary motives, and we deduce that the homotopy $\K$-theory spectrum $\KGL_S$ is a normed spectrum. We also show that the nonconnective $\K$-theory spectrum of a scheme $S$ can be promoted to a normed $\widehat\Pi_1^\et(S)$-spectrum.
\item[Thom spectra (Section \ref{sec:thomspectra})] We define a general motivic Thom spectrum functor and show that it is compatible with norms. Among other things, this allows us to construct canonical normed spectrum structures on the algebraic cobordism spectra $\MGL$, $\MSL$, $\MSp$, $\mathrm{MO}$, and $\mathrm{MSO}$, and to give a formula for free normed spectra. Over a field, we show that $\bigvee_{n\in\Z}\Sigma^{2n,n}\HH\Z$ and $\bigvee_{n\in\Z}\Sigma^{4n,2n}\HH\tilde\Z$ are normed spectra and describe their norms in terms of Thom isomorphisms.
\end{description}
\endgroup

\subsection{Guide for the reader}
The reader should glance at the table of notation at the end if they find themselves confronted with unexplained notation.
Sections \ref{sec:prelim} to~\ref{sec:normedspectra} build up the definition of the $\infty$-category of normed motivic spectra and are fundamental for essentially everything that follows.
Subsection~\ref{sub:coherence} is especially important, because it explains in detail some techniques that are used repeatedly in the later sections.
 The remaining sections are mostly independent of one another, except that Sections~\ref{sec:ggt} (on Grothendieck's Galois theory) and~\ref{sec:betti} (on equivariant Betti realization) depend on Section~\ref{sec:galois-equiv} (on equivariant homotopy theory). Some of the proofs in Sections~\ref{sec:localization} (on localization), \ref{sec:slices} (on the slice filtration), and~\ref{sec:dgCat} (on noncommutative motives) also make use of the results of Section~\ref{sec:ggt}.
There are four appendices, which are used throughout the text and may be referred to as needed.

\subsection{Remarks on \texorpdfstring{$\infty$}{∞}-categories}
We freely use the language of $\infty$-categories throughout this paper, as set out in \cite{HTT, HA}. This is not merely a cosmetic choice, as we do not know how to construct the norm functors without $\infty$-category theory. In any case, the framework of model categories would not be adapted to the study of such functors since they cannot be Quillen functors. Although it might be possible to construct motivic norms using suitable categories with weak equivalences, as is done in \cite{HHR} in the case of equivariant homotopy theory, it would be prohibitively difficult to prove even the most basic properties of norms in such a framework.

This text assumes in particular that the reader is familiar with the basics of $\infty$-category theory, notably the notions of limits and colimits \cite[\sect 1.2.13]{HTT}, adjunctions \cite[\sect 5.2]{HTT}, Kan extensions \cite[\sect 4.3]{HTT}, (co)cartesian fibrations \cite[Chapter 2]{HTT}, sheaves \cite[\sect 6.2.2]{HTT}, and commutative algebras \cite[\sect 2.1.3]{HA}.

Furthermore, we always use the language of $\infty$-category theory in a model-independent manner, which means that some standard terminology is used with a slightly different meaning than usual (even for ordinary category theory). For example, a small $\infty$-category means an essentially small $\infty$-category, a full subcategory is always closed under equivalences, a cartesian fibration is any functor that is equivalent to a cartesian fibration in the sense of \cite[\sect2.4]{HTT}, etc.
We also tacitly regard categories as $\infty$-categories, namely those with discrete mapping spaces.

In Section \ref{sec:f-tens-*-adj} and Appendices \ref{app:spans} and \ref{sec:app-adjunctions}, we also use a small amount of $(\infty,2)$-category theory, as set out in \cite[Appendix A]{GRderalg}. However, the only result that uses $(\infty,2)$-categories in an essential way is Theorem~\ref{thm:norm-pullback}. Elsewhere, $(\infty,2)$-categories are only used to clarify the discussion and could easily be avoided.

\subsection{Standing assumptions}
Except in Appendix \ref{app:nisnevich}, and unless the context clearly indicates otherwise, all schemes are assumed to be \emph{quasi-compact and quasi-separated} (qcqs). We note however that essentially all results generalize to arbitrary schemes: often the same proofs work, and sometimes one needs a routine Zariski descent argument to reduce to the qcqs case. For example, our construction of the norm functors does not apply to non-qcqs schemes because the $\infty$-category $\SH(X)$ may fail to be compactly generated, but nevertheless the norm functors immediately extend with all their properties to arbitrary schemes by descent (using Proposition~\ref{prop:span-RKE}). It seemed far too tedious to systematically include this generality, as additional (albeit trivial) arguments would be required in many places.

\subsection{Acknowledgments}
The main part of this work was completed during a joint stay at Institut Mittag-Leffler as part of the research program ``Algebro-Geometric and Homotopical Methods''. We are grateful to the Institute and to the organizers Eric Friedlander, Lars Hesselholt, and Paul Arne Østvær for this opportunity.
We thank Denis Nardin for numerous helpful comments regarding $\infty$-categories, Akhil Mathew for telling us about polynomial functors and several useful discussions, and Lorenzo Mantovani for comments on a draft version.
Finally, we thank the anonymous referee for many suggestions that improved the readability of the text.
T.B. was partially supported by the DFG under SPP 1786.
M.H. was partially supported by the NSF under grants DMS-1508096 and DMS-1761718.

\section{Preliminaries}
\label{sec:prelim}

\subsection{Nonabelian derived \texorpdfstring{$\infty$}{∞}-categories}

Let $\scr C$ be a small $\infty$-category with finite coproducts. We denote by $\PSh_\Sigma(\scr C)\subset\PSh(\scr C)$ the full subcategory of presheaves that transform finite coproducts into finite products, also known as the nonabelian derived $\infty$-category of $\scr C$. We refer to \cite[\sect5.5.8]{HTT} for basic properties of this $\infty$-category (another treatment using the language of simplicial presheaves is \cite{Voevodsky:2010b}). In particular, recall that $\PSh_\Sigma(\scr C)$ is the $\infty$-category freely generated by $\scr C$ under sifted colimits \cite[Proposition 5.5.8.15]{HTT}. We denote by
\[
\L_\Sigma\colon \PSh(\scr C)\to\PSh_\Sigma(\scr C)
\]
the left adjoint to the inclusion.

If $\scr C$ has a final object $*$, we denote by $\scr C_\pt=\scr C_{*/}$ the $\infty$-category of pointed objects of $\scr C$, and we let $\scr C_+$ be the full subcategory of $\scr C_\pt$ spanned by objects of the form $X_+=X\amalg *$.
The following lemma is the $\infty$-categorical version of \cite[Lemma 3.2]{Voevodsky:2010b}.

\begin{lemma}\label{lem:C+}
	Let $\scr C$ be a small $\infty$-category with finite coproducts and a final object. Then the Yoneda embedding $\scr C_+\into \PSh_\Sigma(\scr C)_\pt$ induces an equivalence $\PSh_\Sigma(\scr C_+)\simeq \PSh_\Sigma(\scr C)_\pt$.
\end{lemma}

\begin{proof}
	This is a straightforward application of \cite[Proposition 5.5.8.22]{HTT}.
\end{proof}

Let $\scr C$ be a small $\infty$-category with finite coproducts, and suppose that $\scr C$ is equipped with a symmetric monoidal structure. Then $\PSh_\Sigma(\scr C)$ acquires a symmetric monoidal structure that preserves sifted colimits in each variable, called the Day convolution, and it has a universal property as such \cite[Proposition 4.8.1.10]{HA}; it is cartesian if the symmetric monoidal structure on $\scr C$ is cartesian.
If the tensor product in $\scr C$ distributes over finite coproducts, then $\PSh_\Sigma(\scr C)$ is \emph{presentably symmetric monoidal}, i.e., it is a commutative algebra in $\Pr^\mathrm{L}$ (equipped with the Lurie tensor product \cite[\sect 4.8.1]{HA}).
In this case, we get an induced symmetric monoidal structure on $\PSh_\Sigma(\scr C)_\pt \simeq \PSh_\Sigma(\scr C)\otimes \scr S_\pt$; if $\scr C$ has a final object, it restricts to a symmetric monoidal structure on the full subcategory $\scr C_+$.

\begin{lemma}\label{lem:Day}
	Let $\scr C$ be a small symmetric monoidal $\infty$-category with finite coproducts and a final object, whose tensor product distributes over finite coproducts.
	Then the Yoneda embedding $\scr C_+\into \PSh_\Sigma(\scr C)_\pt$ induces an equivalence of symmetric monoidal $\infty$-categories $\PSh_\Sigma(\scr C_+) \simeq \PSh_\Sigma(\scr C)_\pt$.
\end{lemma}

\begin{proof}
	The universal property of the Day convolution provides a symmetric monoidal functor $\PSh_\Sigma(\scr C_+) \to \PSh_\Sigma(\scr C)_\pt$, which is an equivalence by Lemma~\ref{lem:C+}.
\end{proof}

In general, $\PSh_\Sigma(\scr C)$ is far from being an $\infty$-topos. For example, if $\scr C$ is the category of finitely generated free abelian groups, then $\PSh_\Sigma(\scr C)$ is the $\infty$-category of connective $\HH\Z$-module spectra.
We can nevertheless consider the topology on $\scr C$ generated by finite coproduct decompositions; we denote by $\Shv_\amalg(\scr C)\subset\PSh(\scr C)$ the $\infty$-topos of sheaves for this topology.

\begin{definition}\label{def:extensive}
	An $\infty$-category is called \emph{extensive} if it admits finite coproducts, if coproducts are disjoint (i.e., for every objects $X$ and $Y$, $X\times_{X\sqcup Y}Y$ is an initial object), and if finite coproduct decompositions are stable under pullbacks.
\end{definition}

 For example, if $S$ is a scheme, any full subcategory of $\Sch_S$ that is closed under finite coproducts and summands is extensive.

\begin{lemma}\label{lem:topos}
	Let $\scr C$ be a small extensive $\infty$-category.
	Then $\PSh_\Sigma(\scr C)=\Shv_\amalg(\scr C)$.
	In particular, $\L_\Sigma$ is left exact and $\PSh_\Sigma(\scr C)$ is an $\infty$-topos. 
\end{lemma}

\begin{proof}
	Under these assumptions on $\scr C$, the topology of finite coproduct decompositions is generated by an obvious cd-structure satisfying the assumptions of \cite[Theorem 3.2.5]{AHW}.\footnote{The statement of \cite[Theorem 3.2.5]{AHW} assumes that $\scr C$ is a $1$-category. However, the proof works for any $\infty$-category $\scr C$, provided that the vertical maps in the cd-structure are truncated (i.e., their iterated diagonals are eventually equivalences).} The conclusion is that a presheaf $F$ on $\scr C$ is a sheaf if and only if $F(\emptyset)$ is contractible and $F(U\amalg V)\simeq F(U)\times F(V)$, i.e., if and only if $F\in\PSh_\Sigma(\scr C)$.
\end{proof}

\begin{remark}
A key property of $\infty$-topoi that we will use several times is \emph{universality of colimits} \cite[Theorem 6.1.0.6(3)(ii)]{HTT}:
if $f\colon V\to U$ is a morphism in an $\infty$-topos $\scr X$, the pullback functor $f^*\colon \scr X_{/U}\to \scr X_{/V}$ preserves colimits.
\end{remark}

The fact that the $\infty$-topos of Lemma~\ref{lem:topos} is hypercomplete will be useful. Recall that a presentable $\infty$-category $\scr C$ is \emph{Postnikov complete} if the canonical functor
\[
\scr C\to \lim_{n\to\infty} \tau_{\leq n}\scr C
\]
is an equivalence, where $\tau_{\leq n}\scr C\subset\scr C$ is the full subcategory of $n$-truncated objects. A Postnikov complete $\infty$-topos is in particular hypercomplete (i.e., the above functor is conservative).

\begin{lemma}\label{lem:postnikov}
	Let $\scr C$ be a small $\infty$-category with finite coproducts. Then $\PSh_\Sigma(\scr C)$ is Postnikov complete.
\end{lemma}

\begin{proof}
	Since the truncation functors $\tau_{\leq n}$ preserve finite products of spaces, the inclusion $\PSh_\Sigma(\scr C)\subset \PSh(\scr C)$ commutes with truncation. Postnikov completeness of $\PSh_\Sigma(\scr C)$ then follows from that of $\PSh(\scr C)$.
\end{proof}

Next we recall some well-known techniques for computing colimits in $\infty$-categories.
If $\scr C$ is an $\infty$-category with finite colimits and $X\from Y\to Z$ is a span in $\scr C$, the bar construction provides a simplicial object $\Bar_Y(X,Z)_\bullet$ with
	\[
	\Bar_Y(X,Z)_n = X\amalg Y^{\amalg n}\amalg Z
	\] 
(this can be regarded as a special case of \cite[Construction 4.4.2.7]{HA} for the cocartesian monoidal structure on $\scr C$).

\begin{lemma}\label{lem:bar-construction}
	Let $\scr C$ be an $\infty$-category with finite coproducts and let $X\from Y\to Z$ be a span in $\scr C$. 
	Then the pushout $X\amalg_YZ$ exists if and only if the geometric realization $\lvert \Bar_Y(X,Z)_\bullet\rvert$ exists, in which case there is a canonical equivalence
	\[
	X\amalg_YZ \simeq \lvert \Bar_Y(X,Z)_\bullet\rvert.
	\]
\end{lemma}

\begin{proof}
	This is a direct application of \cite[Corollary 4.2.3.10]{HTT}, using the functor $\Delta^\op\to (\Set_\Delta)_{/\Lambda^2_0}$, $[n]\mapsto \Delta^{\{0,1\}}\amalg (\Delta^{\{0\}})^{\amalg n}\amalg \Delta^{\{0,2\}}$.
\end{proof}

\begin{lemma}\label{lem:sifted+coprod}
	Let $\scr C$ be an $\infty$-category admitting sifted colimits and finite coproducts. Then $\scr C$ admits small colimits. Moreover, if $f\colon \scr C\to\scr D$ is a functor that preserves sifted colimits and finite coproducts, then $f$ preserves small colimits.
\end{lemma}

\begin{proof}
	If $\scr C$ has pushouts and coproducts, the analogous statements are \cite[Propositions 4.4.2.6 and 4.4.2.7]{HTT}.
	An arbitrary coproduct is a filtered colimit of finite coproducts, and a pushout can be written as a geometric realization of finite coproducts by Lemma~\ref{lem:bar-construction}. Since filtered colimits and geometric realizations are sifted colimits, the result follows. 
\end{proof}

Let $\scr C$ be an $\infty$-category with small colimits. Recall that a class of morphisms in $\scr C$ is \emph{strongly saturated} if it is closed under 2-out-of-3, cobase change, and small colimits in $\Fun(\Delta^1,\scr C)$ \cite[Definition 5.5.4.5]{HTT}.

\begin{lemma}\label{lem:strong-saturation}
	Let $\scr C$ be an $\infty$-category with small colimits and $E$ a class of morphisms in $\scr C$. Then $E$ is strongly saturated if and only if it contains all equivalences and is closed under 2-out-of-3 and small colimits in $\Fun(\Delta^1,\scr C)$.
\end{lemma}

\begin{proof}
	An equivalence is a cobase change of the initial object in $\Fun(\Delta^1,\scr C)$, so strongly saturated classes contain all equivalences. Conversely, if $g\colon C\to D$ is a cobase change of $f\colon A\to B$, then $g=f\amalg_{\id_A}\id_C$ in $\Fun(\Delta^1,\scr C)$.
\end{proof}

The following lemma with $\scr C=\PSh_\Sigma(\scr C_0)$ is an $\infty$-categorical version of \cite[Corollary 3.52]{Voevodsky:2010b}.

\begin{lemma}
	\label{lem:voevodsky}
	Let $\scr C$ be an $\infty$-category with small colimits and $\scr C_0\subset\scr C$ a full subcategory that generates $\scr C$ under sifted colimits and is closed under finite coproducts. Let $E$ be a class of morphisms in $\scr C$ containing the equivalences in $\scr C_0$. Then the strong saturation of $E$ is generated under 2-out-of-3 and sifted colimits by morphisms of the form $f\amalg\id_X$ with $f\in E$ and $X\in\scr C_0$.
\end{lemma}

\begin{proof}
	Let $E^\amalg$ be the class of morphisms of $\scr C$ of the form $f\amalg\id_X$ with $f\in E$ and $X\in\scr C_0$.
	Since $\scr C_0$ is closed under finite coproducts, $E^\amalg$ contains $E$ and is closed under $(\ph)\amalg \id_X$ for every $X\in\scr C_0$.
	Let $E'$ be the class of morphisms generated under 2-out-of-3 and sifted colimits by $E^\amalg$. We must show that $E'$ is strongly saturated. Note that $E'$ contains all equivalences. By Lemma~\ref{lem:strong-saturation}, it remains to show that $E'$ is closed under small colimits in $\Fun(\Delta^1,\scr C)$. By Lemma~\ref{lem:sifted+coprod}, it suffices to show that $E'$ is closed under binary coproducts.
	If $f\in E^\amalg$, the class of all $X\in \scr C$ such that $f\amalg \id_X\in E'$ contains $\scr C_0$ and is closed under sifted colimits (because sifted simplicial sets are weakly contractible and $\amalg$ preserves weakly contractible colimits in each variable), hence it equals $\scr C$. For $X\in\scr C$, the class of morphisms $f$ in $\scr C$ such that $f\amalg \id_X\in E'$ contains $E^\amalg$ and is closed under 2-out-of-3 and sifted colimits, hence it contains $E'$. We have shown that if $f\in E'$ and $X\in\scr C$, then $f\amalg\id_X\in E'$. By 2-out-of-3, if $f,g\in E'$, then $f\amalg g\in E'$, as desired.
\end{proof}

\subsection{Unstable motivic homotopy theory}

Let $S$ be a qcqs scheme. The category $\Sm_S$ of finitely presented smooth $S$-schemes is extensive, so that $\PSh_\Sigma(\Sm_S) = \Shv_\amalg(\Sm_S)$ by Lemma~\ref{lem:topos}.
We refer to Appendix~\ref{app:nisnevich} for a discussion of the Nisnevich topology on $\Sm_S$; in particular, it is quasi-compact (Proposition~\ref{prop:nisnevich-definition}) and its points are given by henselian local schemes (Proposition~\ref{prop:nisnevich-properties}). Since coproduct decompositions are Nisnevich coverings, it follows that
\[
\Shv_\Nis(\Sm_S)\subset \PSh_\Sigma(\Sm_S).
\]
We let $\PSh_{\A^1}(\Sm_S)\subset \PSh(\Sm_S)$ be the full subcategory of $\A^1$-invariant presheaves. The $\infty$-category $\H(S)$ of motivic spaces over $S$ is defined by
\[\H(S)=\Shv_\Nis(\Sm_S)\cap \PSh_{\A^1}(\Sm_S)\subset \PSh(\Sm_S).\]
We denote by $\L_\Nis$, $\L_{\A^1}$, and $\L_\mot$ the corresponding localization functors on $\PSh(\Sm_S)$, and we say that a morphism $f$ in $\PSh(\Sm_S)$ is a \emph{Nisnevich equivalence}, an \emph{$\A^1$-equivalence}, or a \emph{motivic equivalence} if $\L_\Nis(f)$, $\L_{\A^1}(f)$, or $\L_\mot(f)$ is an equivalence. We recall that $\L_\Nis$ is left exact and that $\L_{\A^1}$ and $\L_\mot$ preserve finite products \cite[Proposition C.6]{HoyoisGLV}.
By \cite[Proposition 5.5.4.15]{HTT} and \cite[Corollary C.2]{HoyoisGLV}, the class of Nisnevich equivalences (resp.\ of $\A^1$-equivalences) in $\PSh(\Sm_S)$ is the strongly saturated class generated by the finitely generated Nisnevich covering sieves (resp.\ by the projections $X\times\A^1\to X$ for $X\in\Sm_S$). Together these classes generate the strongly saturated class of motivic equivalences.

For any morphism of schemes $f\colon T\to S$, the base change functor $f^*\colon \Sm_S\to\Sm_T$ preserves finite coproducts, Nisnevich squares, and $\A^1$-homotopy equivalences. It follows that the pushforward functor $f_*\colon \PSh(\Sm_T) \to \PSh(\Sm_S)$ restricts to functors
\[
f_*\colon \PSh_\Sigma(\Sm_T) \to \PSh_\Sigma(\Sm_S),\quad f_*\colon \Shv_\Nis(\Sm_T)\to\Shv_\Nis(\Sm_S),\quad f_*\colon \H(T)\to\H(S).
\]
The first preserves sifted colimits and the last two preserve filtered colimits, since they are computed pointwise in each case.

\begin{proposition}\label{prop:integralsifted}
	Let $p\colon T\to S$ be an integral morphism of schemes. Then the functor
	\[
	p_*\colon \PSh_\Sigma(\Sm_T)\to \PSh_\Sigma(\Sm_S)
	\]
	preserves Nisnevich and motivic equivalences.
\end{proposition}

\begin{proof}
	It is clear that $p_*$ preserves $\A^1$-homotopy equivalences, since it preserves products and $p^*(\A^1_S)\simeq \A^1_T$. Note that if $f$ is either a Nisnevich equivalence or an $\A^1$-homotopy equivalence, so is $f \amalg \id_X$ for any $X\in\Sm_T$. Thus by Lemma~\ref{lem:voevodsky}, it suffices to show that $p_*$ preserves Nisnevich equivalences.
	Again by Lemma~\ref{lem:voevodsky}, the class of Nisnevich equivalences in $\PSh_\Sigma(\Sm_T)$ is generated under 2-out-of-3 and sifted colimits by morphisms of the form $\L_\Sigma U\into X$, where $U\into X$ is a finitely generated Nisnevich sieve.
	Note that if $U\into X$ is the sieve generated by $U_1,\dotsc,U_n$ then $\L_\Sigma U\into X$ is the sieve generated by $U_1\amalg\dotsb\amalg U_n$.
	It therefore remains to show the following: if $X\in\Sm_T$ and $U\into X$ is a Nisnevich sieve generated by a single map $X'\to X$, then $p_*U\into p_*X$ is a Nisnevich equivalence. Since this is a morphism between $0$-truncated objects, it suffices to check that it is an equivalence on stalks. If $V$ is the henselian local scheme of a point in a smooth $S$-scheme, we must show that every $T$-morphism $V\times_ST\to X$ lifts to $X'$. We can write $T=\lim_\alpha T_\alpha$ where $T_\alpha$ is finite over $S$ \cite[Tag 09YZ]{Stacks}. 
	Since $X'\to X\to T$ are finitely presented, we can assume that $p\colon T\to S$ is finite. Then $V\times_ST$ is a finite sum of henselian local schemes \cite[Proposition 18.5.10]{EGA4-4}, and hence every $T$-morphism $V\times_ST\to X$ lifts to $X'$.
\end{proof}

\begin{corollary}\label{cor:integralsifted}
	Let $p\colon T\to S$ be an integral morphism of schemes. Then the functor
	\[
	p_*\colon \H(T)\to \H(S)
	\]
	preserves sifted colimits.
\end{corollary}

\begin{proof}
	Let $i\colon \H(T)\into \PSh(\Sm_T)$ be the inclusion, with left adjoint $\L_\mot$.
If $X\colon K \to \H(T)$ is a diagram, then $\colim_K X \simeq \L_\mot \colim_K iX$. Since $p_*\colon \PSh_\Sigma(\Sm_T) \to \PSh_\Sigma(\Sm_S)$ preserves motivic equivalences and sifted colimits, the result follows.
\end{proof}

We shall denote by $\H_\pt(S)$ the $\infty$-category of pointed objects in $\H(S)$, that is, the full subcategory of $\PSh_\Sigma(\Sm_S)_\pt$ spanned by the $\A^1$-invariant Nisnevich sheaves. Since $\L_\mot$ preserves finite products, the localization functor $\L_\mot\colon \PSh_\Sigma(\Sm_S)_\pt\to\H_\pt(S)$ is symmetric monoidal with respect to the smash product.
Moreover, since finite products distribute over finite coproducts in $\Sm_S$, Lemmas~\ref{lem:C+} and \ref{lem:Day} apply: we have $\PSh_\Sigma(\Sm_S)_\pt\simeq \PSh_\Sigma(\Sm_{S+})$ and the smash product symmetric monoidal structure on the former agrees with the Day convolution symmetric monoidal structure on the latter.

\subsection{Weil restriction}
Let $p\colon T\to S$ be a morphism of schemes and $X$ a scheme over $T$. If the image of $X$ by the pushforward functor $p_*\colon \PSh(\Sch_T) \to \PSh(\Sch_S)$ is representable by an $S$-scheme, this scheme is called the \emph{Weil restriction} of $X$ along $p$ and denoted by $\Weil_pX$. Explicitly, $\Weil_pX$ is characterized by the following universal property: there is a bijection
\[
\Map_S(Y, \Weil_p X) \simeq \Map_T(Y\times_ST, X)
\]
natural in the $S$-scheme $Y$. We refer to \cite[\sect 7.6]{NeronModels} for basic properties of Weil restriction and for existence criteria. In particular, if $p$ is finite locally free and $X$ is quasi-projective over $T$, then $\Weil_p X$ exists \cite[\sect 7.6, Theorem 4]{NeronModels}.

\begin{lemma}\label{lem:Weil}
	Let $S$ be an arbitrary scheme, $p\colon T\to S$ a finite locally free morphism, and $X$ a quasi-projective $T$-scheme. Then the Weil restriction $\Weil_pX$ exists and is quasi-projective over $S$.
\end{lemma}

\begin{proof}
	The proof of \cite[Proposition A.5.8]{CGP} shows that $\Weil_pX$ has a relatively ample line bundle.
	By \cite[\sect7.6, Proposition 5(c)]{NeronModels}, $\Weil_pX\to S$ is locally of finite type, so it remains to show that it is quasi-compact. 
	Locally on $S$, $X\to T$ is the composition of a closed immersion and a finitely presented morphism. Both types of morphisms are preserved by Weil restriction \cite[\sect7.6, Propositions 2(ii) and 5(e)]{NeronModels}, so $\Weil_pX\to S$ is quasi-compact.
\end{proof}

\begin{remark}
	Quasi-projectivity assumptions are often made in the sequel to ensure that Weil restrictions exist as schemes.
	They can be removed at the mild cost of extending $\H_\pt(-)$ and $\SH(-)$ to algebraic spaces, since Weil restrictions along finite locally free morphisms always exist as algebraic spaces \cite[Theorem 3.7]{RydhHilb}.
	By \cite[Proposition 5.7.6]{GrusonRaynaud}, Nisnevich sheaves such as $\H_\pt(-)$ and $\SH(-)$ extend uniquely to the category of qcqs algebraic spaces.
\end{remark}

\begin{lemma}\label{lem:Weil-smooth}
	Let $S$ be an arbitrary scheme, $p\colon T\to S$ a finite locally free morphism, and $X$ a smooth $T$-scheme. If $\Weil_p X$ exists, then $\Weil_pX$ is smooth over $S$.
\end{lemma}

\begin{proof}
	It follows immediately from the universal property of Weil restriction that $\Weil_p X$ is formally smooth. Moreover, $\Weil_pX$ is locally of finite presentation by \cite[\sect 7.6, Proposition 5(d)]{NeronModels}, hence smooth.
\end{proof}

\section{Norms of pointed motivic spaces}
\label{sec:pointednorms}

Let $p\colon T\to S$ be a morphism of schemes. Our goal is to construct a ``multiplicative pushforward'' or ``parametrized smash product'' functor
\[
p_\otimes\colon \H_\pt(T) \to \H_\pt(S).
\]
More precisely, we ask for a functor $p_\otimes$ satisfying the following properties:
\begin{enumerate}
	\item[(N1)] $p_\otimes$ is a symmetric monoidal functor.
	\item[(N2)] $p_\otimes$ preserves sifted colimits.
	\item[(N3)] There is a natural equivalence $p_\otimes(X_+)\simeq (p_*X)_+$.
	\item[(N4)] If $T=S^{\amalg n}$ and $p$ is the fold map, then $p_\otimes$ is the $n$-fold smash product.
\end{enumerate}
It is unreasonable to expect that such a functor $p_\otimes$ exists for arbitrary $p$. For example, it is not at all clear how to make sense of the smash product of an infinite family of objects in $\H_\pt(S)$. Conditions (N2) and (N3) hint that we should at least assume $p$ integral, as this guarantees that $p_*$ preserves sifted colimits (Corollary~\ref{cor:integralsifted}). The main result of this section is that there is a canonical definition of $p_\otimes$ satisfying conditions (N1)–(N4) above, provided that $p$ is \emph{integral} and \emph{universally open}, for example if $p$ is finite locally free, or if $p$ is $\Spec L\to\Spec K$ for an algebraic field extension $L/K$.
 
 Our strategy is to first construct a functor $p_\otimes\colon \PSh_\Sigma(\Sm_T)_\pt \to \PSh_\Sigma(\Sm_S)_\pt$ satisfying the conditions (N1)–(N4) and then prove that it preserves motivic equivalences.
 It might seem at first glance that such a functor $p_\otimes$ should be determined by conditions (N2) and (N3), because $\PSh_\Sigma(\Sm_T)_\pt$ is generated under sifted colimits by objects of the form $X_+$. However, the maps in these colimit diagrams are not all in the image of the functor $X\mapsto X_+$, so neither existence nor uniqueness of $p_\otimes$ is evident. The typical example of a pointed map that is not in the image of that functor is the map $f\colon (X\amalg Y)_+\to X_{+}$ that collapses $Y$ to the base point. Intuitively, the induced map $p_\otimes(f)\colon p_*(X\amalg Y)_+ \to p_*(X)_+$ should collapse $p_*(Y)$ and all the ``cross terms'' in $p_*(X\amalg Y)$ to the base point: this is what happens when $p$ is a fold map, according to condition (N4). 
Lemma~\ref{lem:crossterms} below shows that there is a well-defined notion of ``cross terms'' when $p$ is both universally closed and universally open.

\subsection{The unstable norm functors}

Let $p\colon T\to S$ be a morphism of schemes and let $X\in\PSh(\Sm_T)$. 
By a \emph{subpresheaf} $Y\subset X$ we will mean a monomorphism of presheaves $Y\to X$, i.e., a morphism $Y\to X$ such that for every $U\in\Sm_T$, $Y(U)\to X(U)$ is an inclusion of connected components. We will say that a morphism $Y\to X$ is \emph{relatively representable} if for every $U\in \Sm_T$ and every map $U\to X$, the presheaf $U\times_XY$ is representable.
Let $Y_1,\dotsc,Y_k\to X$ be relatively representable morphisms.
For $U\in\Sm_S$, let
\[
p_*(X|Y_1,\dotsc,Y_k)(U) = \{s\colon U\times_ST \to X\suchthat s^{-1}(Y_i) \to U\text{ is surjective for all }i\}.
\]
It is clear that $p_*(X|Y_1,\dotsc,Y_k)$ is a subpresheaf of $p_*(X)$, since surjective maps are stable under base change. Moreover, if $X\in\PSh_\Sigma(\Sm_T)$, then $p_*(X|Y_1,\dotsc,Y_k)\in\PSh_\Sigma(\Sm_S)$.

We say that a morphism of schemes is \emph{clopen} if it is both closed and open.

\begin{lemma}\label{lem:crossterms}
	Let $p\colon T\to S$ be a universally clopen morphism, let $X\in\PSh_\Sigma(\Sm_T)$, and let $Y_1,\dotsc,Y_k\to X$ be relatively representable morphisms.
	 For every coproduct decomposition $X\simeq X'\amalg X''$ in $\PSh_\Sigma(\Sm_T)$,
 there is a coproduct decomposition
	\[
	p_*(X|Y_1,\dotsc,Y_k)\simeq p_*(X'|Y_1',\dotsc,Y_k')\amalg p_*(X|X'',Y_1,\dotsc,Y_k)
	\]
	in $\PSh_\Sigma(\Sm_S)$, where $Y_i'=Y_i\times_X X'$.
\end{lemma}

\begin{proof}
	Let $\phi\colon p_*(X')\amalg p_*(X|X'') \to p_*(X)$ be the morphism induced by the inclusions $p_*(X')\subset p_*(X)$ and $p_*(X|X'')\subset p_*(X)$.
	We first note that 
	\begin{align*}
	&p_*(X'|Y_1',\dotsc,Y_k')=p_*(X')\cap p_*(X|Y_1,\dotsc,Y_k)\\
	\text{and}\quad & p_*(X|X'',Y_1,\dotsc,Y_k)=p_*(X|X'')\cap p_*(X|Y_1,\dotsc,Y_k)
	\end{align*}
	as subpresheaves of $p_*(X)$. By universality of colimits in $\PSh_\Sigma(\Sm_S)$, we obtain a cartesian square
\begin{tikzmath}
	\def\colsep{1.5em}
	\diagram{p_*(X'|Y_1',\dotsc,Y_k')\amalg p_*(X|X'',Y_1,\dotsc,Y_k) & p_*(X|Y_1,\dotsc,Y_k)\\
	p_*(X')\amalg p_*(X|X'') & p_*(X)\rlap, \\};
	\arrows (11-) edge (-12) (21-) edge node[above]{$\phi$} (-22) (11) edge[c->] (21) (12) edge[c->] (22);
\end{tikzmath}
and we may therefore assume that $k=0$.
Note that $p_*(X')\times_{p_*(X)}p_*(X|X'')$ has no sections over nonempty schemes, hence is the initial object of $\PSh_\Sigma(\Sm_S)$. By universality of colimits, it follows that the diagonal of $\phi$ is an equivalence, i.e., that $\phi$ is a monomorphism.
	 
	It remains to show that $\phi$ is objectwise an effective epimorphism.
	Given $U\in\Sm_S$ and $s \in p_*(X)(U)$, let $U'=\{x\in U\suchthat p_U^{-1}(x)\subset s^{-1}(X')\}$ and let $U''$ be its complement. Then $U''=p_U(s^{-1}(X''))$, which is a clopen subset of $U$ since $p_U$ is clopen. We thus have a coproduct decomposition $U=U'\amalg U''$. By construction, the restriction of $s$ to $U'$ belongs to $p_*(X')(U')$, and its restriction to $U''$ belongs to $p_*(X|X'')(U'')$. Hence, these restrictions define a section of $p_*(X')\amalg p_*(X|X'')$ over $U$, which is clearly a preimage of $s$ by $\phi_U$.
 \end{proof}

 \begin{example}\label{ex:addition-formula}
 	Let $p\colon T\to S$ be a universally clopen morphism and let $X,Y\in\PSh_\Sigma(\Sm_T)$. Then we have the decomposition
 	\[
 	p_*(X\amalg Y) \simeq p_*(\emptyset)\amalg p_*(X|X) \amalg p_*(Y|Y) \amalg p_*(X\amalg Y|X,Y)
 	\]
 	in $\PSh_\Sigma(\Sm_S)$ (apply Lemma~\ref{lem:crossterms} three times). Note that $p_*(\emptyset)$ is represented by the complement of the image of $p$, while $p_*(X|X)$ is the restriction of $p_*(X)$ to the image of $p$. If $p$ is surjective, then $p_*(\emptyset)=\emptyset$ and $p_*(X|X)=p_*(X)$, in which case we get
 	\[
 	p_*(X\amalg Y) \simeq p_*(X)\amalg p_*(Y) \amalg p_*(X\amalg Y|X,Y).
 	\]
 \end{example}
 
 In the situation of Lemma~\ref{lem:crossterms}, we obtain in particular a map
 \begin{equation*}\label{eqn:collapse}
p_*(X)_+ \to p_*(X')_+
 \end{equation*}
 in $\PSh_\Sigma(\Sm_S)_\pt$ by sending $p_*(X|X'')$ to the base point; it is a retraction of the inclusion $p_*(X')_+\into p_*(X)_+$.

\begin{theorem}\label{thm:norm}
	Let $p\colon T\to S$ be a universally clopen morphism. Then there is a unique symmetric monoidal functor
	\begin{equation*}\label{eqn:normPSigma}
	p_\otimes\colon \PSh_\Sigma(\Sm_T)_\pt \to \PSh_\Sigma(\Sm_S)_\pt
	\end{equation*}
	such that:
	\begin{enumerate}
		\item $p_\otimes$ preserves sifted colimits;
		\item there is given a symmetric monoidal natural equivalence $p_\otimes(X_+)\simeq p_*(X)_+$;
		\item for every $f\colon Y_+\to X_+$ with $X,Y\in\PSh_\Sigma(\Sm_T)$, the map $p_\otimes(f)$ is the composite
		\[
		p_*(Y)_+ \to p_*(f^{-1}(X))_+ \xrightarrow{f} p_*(X)_+,
		\]
		where the first map collapses $p_*(Y|Y\minus f^{-1}(X))$ to the base point.
	\end{enumerate}
	Moreover:
	\begin{enumerate}\setcounter{enumi}{3}
		\item if $p$ is integral, then $p_\otimes$ preserves Nisnevich and motivic equivalences;
		\item if $p$ is a universal homeomorphism, then $p_\otimes\simeq p_*$;
		\item if $p\colon S^{\amalg n}\to S$ is the fold map, then
		\[
		p_\otimes\colon \PSh_\Sigma(\Sm_{S^{\amalg n}})_\pt \simeq (\PSh_\Sigma(\Sm_S)_\pt)^{\times n} \to \PSh_\Sigma(\Sm_S)_\pt
		\]
		is the $n$-fold smash product.
	\end{enumerate}
\end{theorem}

\begin{proof}
	Recall from Lemma~\ref{lem:Day} that $\PSh_\Sigma(\Sm_S)_\pt\simeq \PSh_\Sigma(\Sm_{S+})$, where $\Sm_{S+}$ is the full subcategory of pointed objects of $\Sm_S$ of the form $X\amalg S$, and that the smash product on $\PSh_\Sigma(\Sm_S)_\pt$ corresponds to the Day convolution product on $\PSh_\Sigma(\Sm_{S+})$. Hence, defining a symmetric monoidal functor $p_\otimes$ that preserves sifted colimits is equivalent to defining a symmetric monoidal functor
	\[
	\Sm_{T+}\to \PSh_\Sigma(\Sm_S)_\pt.
	\]
	Moreover, symmetric monoidal equivalences $p_\otimes((\ph)_+)\simeq p_*(\ph)_+$ are determined by their restriction to $\Sm_{T}$.
	It is straightforward to check that (3) with $X,Y\in\Sm_T$ uniquely defines such a symmetric monoidal functor and equivalence.
	To prove that (3) holds in general, we must compute $p_\otimes$ on the collapse map $Y_+\to f^{-1}(X)_+$, which we can do by writing $f^{-1}(X)$ and $Y\minus f^{-1}(X)$ as sifted colimits of representables.
	
	Any morphism $f$ in $\PSh_\Sigma(\Sm_T)_\pt$ is a simplicial colimit of morphisms of the form $f_{+\dotsb +}$.
	As both functors in the adjunction $\PSh_\Sigma(\Sm_T)\rightleftarrows \PSh_\Sigma(\Sm_T)_\pt$ preserve Nisnevich and motivic equivalences, if $f$ is such an equivalence, so is $f_{+\dotsb +}$. 
	Since $p_\otimes$ preserves simplicial colimits and $p_\otimes(g_+)\simeq p_*(g)_+$, we deduce that $p_\otimes$ preserves Nisnevich or motivic equivalences whenever $p_*$ does. By Proposition~\ref{prop:integralsifted}, this is the case if $p$ is integral, which proves (4).
	
	To prove (5), it suffices to check that $p_\otimes$ and $p_*$ agree on $\Sm_{T+}$.
	Since $p$ is a universal homeomorphism, the functor $p_*$ preserves finite sums. Hence, for $f\colon Y_+\to X_+$ in $\Sm_{T+}$, the map $p_\otimes(f)$ described in (3) coincides with $p_*(f)$.
	
	To prove (6), it suffices to note that the $n$-fold smash product preserves sifted colimits and has the functoriality described in (3).
\end{proof}

By Theorem~\ref{thm:norm}(4) and the universal property of localization, if $p\colon T\to S$ is integral and universally open, we obtain norm functors
\begin{align*}
	p_\otimes &\colon \Shv_\Nis(\Sm_T)_\pt \to \Shv_\Nis(\Sm_S)_\pt,\\
p_\otimes &\colon \H_\pt(T)\to\H_\pt(S)
\end{align*}
with the desired properties (N1)–(N4). 

The following lemma provides a more explicit description of $p_\otimes$ when $S$ is noetherian (in which case every smooth $S$-scheme is the sum of its connected components).

\begin{lemma}\label{lem:wedge}
	Let $p\colon T\to S$ be a universally clopen morphism, let $X\in\PSh_\Sigma(\Sm_T)_\pt$, and let $U$ be a pro-object in $\Sm_S$ corepresented by a connected $S$-scheme.
	Then $U\times_ST$ has a finite set of connected components $\{V_i\}_{i\in I}$, and there is an equivalence
	\[
	p_\otimes(X)(U) \simeq \bigwedge_{i\in I}X(V_i)
	\]
	natural in $X$ and in $U$.
\end{lemma}

\begin{proof}
	Since $p_U$ is quasi-compact, the fact that $U\times_ST$ has finitely many connected components follows from \cite[Tag 07VB]{Stacks}.
	With $U$ fixed, both sides are functors of $X$ that preserve sifted colimits. Hence, it suffices to define a natural equivalence for $X=Y_+\in \Sm_{T+}$. In that case, both sides are naturally identified with the pointed set $\Map_T(U\times_ST,Y)_+$.
\end{proof}

\begin{example}\label{ex:degree2}
	 Let $p\colon T\to S$ be finite étale of degree $2$. Using Lemma~\ref{lem:wedge}, we find 
	 \[
	 p_\otimes(\S^1) \simeq \S^1 \wedge \Sigma T
	 \]
	 in $\PSh_\Sigma(\Sm_S)_\pt$, where $\Sigma T=S\amalg_{T}S$.
\end{example}

\begin{lemma}\label{lem:0truncated}
	Let $p\colon T\to S$ be a universally clopen morphism. Then the functor $p_\otimes\colon \PSh_\Sigma(\Sm_T)_\pt\to \PSh_\Sigma(\Sm_S)_\pt$ preserves $0$-truncated objects (i.e., set-valued presheaves).
\end{lemma}

\begin{proof}
	Let $X\in\PSh_\Sigma(\Sm_T)_\pt$ be $0$-truncated.
	For a scheme $U$, let $\mathrm{Clop}_U$ denote the poset of clopen subsets of $U$. Note that a presheaf on $\Sm_S$ is $0$-truncated if and only if its restriction to $\mathrm{Clop}_U$ is $0$-truncated for all $U\in\Sm_S$.
	By Lemmas \ref{lem:topos} and~\ref{lem:postnikov}, $\PSh_\Sigma(\mathrm{Clop}_U)$ is a hypercomplete $\infty$-topos. By Deligne's completeness theorem \cite[Theorem A.4.0.5]{SAG}, it suffices to show that the stalks of $p_\otimes(X)|\mathrm{Clop}_U$ are $0$-truncated. The points of the $\infty$-topos $\PSh_\Sigma(\mathrm{Clop}_U)$ are given by the pro-objects in $\mathrm{Clop}_U$ whose limits are connected schemes (see for example \cite[Remark A.9.1.4]{SAG}). 
	The formula of Lemma~\ref{lem:wedge} then concludes the proof.
\end{proof}

\subsection{Norms of quotients}

Let $p\colon T\to S$ be a morphism of schemes, let $X\in\PSh(\Sm_T)$, and let $Y\subset X$ be a subpresheaf.
For $U\in\Sm_S$, let
\[
p_*(X\Vert Y)(U) = \{s\colon U\times_ST \to X\suchthat \text{$s$ sends a clopen subset covering $U$ to $Y$}\}.
\]
Note that $p_*(X\Vert Y)$ is a subpresheaf of $p_*(X)$, and it is in $\PSh_\Sigma$ if $X$ and $Y$ are.
Moreover, if $Y\subset X$ is relatively representable, then $p_*(X\Vert Y)\subset p_*(X|Y)$. 

\begin{proposition}\label{prop:quotients}
	Let $p\colon T\to S$ be a universally clopen morphism, let $X\in\PSh_\Sigma(\Sm_T)$, and let $Y\subset X$ be a subpresheaf in $\PSh_\Sigma$. Then there is a natural equivalence
	\[
	p_\otimes(X/Y) \simeq p_*(X)/p_*(X\Vert Y)
	\]
	in $\PSh_\Sigma(\Sm_S)_\pt$.
\end{proposition}

\begin{proof}
	The quotient $X/Y$ is the colimit of the following simplicial diagram in $\PSh_\Sigma(\Sm_T)_\pt$:
	\begin{tikzmath}
		\def\colsep{.9em}
		\diagram{
		\dotsb & (X\amalg Y\amalg Y)_+ & (X\amalg Y)_+ & X_+\rlap. \\
		};
		\arrows
		(11-) edge[-top,vshift=3*\dbl] (-12) edge[-mid,vshift=\dbl] (-12) edge[-mid,vshift=-\dbl] (-12) edge[-bot,vshift=-3*\dbl] (-12)
		(12-) edge[-top,vshift=2*\dbl] (-13) edge[-mid] (-13) edge[-bot,vshift=-2*\dbl] (-13)
		(13-) edge[-top,vshift=\dbl] (-14) edge[-bot,vshift=-\dbl] (-14);
	\end{tikzmath}
	Indeed, this is the standard bar construction for the pushout $X\amalg_Y*$ (Lemma~\ref{lem:bar-construction}).
	Hence, $p_\otimes(X/Y)$ is the colimit of the induced simplicial diagram:
	\begin{tikzmath}
		\def\colsep{.9em}
		\diagram{
		\dotsb & p_*(X\amalg Y\amalg Y)_+ & p_*(X\amalg Y)_+ & p_*(X)_+\rlap. \\
		};
		\arrows
		(11-) edge[-top,vshift=3*\dbl] (-12) edge[-mid,vshift=\dbl] (-12) edge[-mid,vshift=-\dbl] (-12) edge[-bot,vshift=-3*\dbl] (-12)
		(12-) edge[-top,vshift=2*\dbl] (-13) edge[-mid] (-13) edge[-bot,vshift=-2*\dbl] (-13)
		(13-) edge[-top,vshift=\dbl] (-14) edge[-bot,vshift=-\dbl] (-14);
	\end{tikzmath}
	The quotient $p_*(X)/p_*(X\Vert Y)$ may be written as the colimit of a similar simplicial diagram.
		By successive applications of Lemma~\ref{lem:crossterms}, we obtain the decomposition
		\[
		p_*(X\amalg Y^{\amalg n}) = p_*(X) \amalg p_*(X\amalg Y|Y) \amalg\dotsb \amalg p_*(X\amalg Y^{\amalg n}|Y).
		\]
		Note that the map $p_*(X\amalg Y^{\amalg i})\to p_*(X)$ induced by the fold map sends
		$p_*(X\amalg Y^{\amalg i}|Y)$ to $p_*(X\Vert Y)$. By inspection, these maps fit together in a natural transformation of simplicial diagrams
		inducing a natural map 
		\begin{equation}\label{eqn:X/Y}
		p_\otimes(X/Y) \to p_*(X)/p_*(X\Vert Y)
		\end{equation}
		 in the colimit.
		 Suppose that $X$ is the sifted colimit of presheaves $X_i\in\PSh_\Sigma(\Sm_T)$ and define $Y_i=Y\times_XX_i\subset X_i$. By universality of colimits, we have $Y\simeq\colim_i Y_i$.
		 As $p_\otimes$ preserves sifted colimits, we get $p_\otimes(X/Y)\simeq\colim_i p_\otimes(X_i/Y_i)$. On the other hand, we have
		\[
		p_*(X_i\Vert Y_i)=p_*(X\Vert Y)\times_{p_*(X)} p_*(X_i)
		\]
		as subpresheaves of $p_*(X_i)$. By universality of colimits, we deduce that $p_*(X\Vert Y)\simeq\colim_i p_*(X_i\Vert Y_i)$.
		Hence, to check that~\eqref{eqn:X/Y} is an equivalence, we may assume that $X\in \Sm_{T}$, and in particular that $X$ is $0$-truncated.
	In that case, we know that $p_\otimes(X/Y)$ is $0$-truncated by Lemma~\ref{lem:0truncated}, so we deduce that $p_\otimes(X/Y)$ is the quotient of $p_*(X)$ by the image of the fold map $p_*(X\amalg Y|Y)\to p_*(X)$.
	It follows from the definitions that this image is exactly $p_*(X\Vert Y)$.
\end{proof}

\begin{remark}\label{rmk:cofiber}
	Any morphism $f\colon Y\to X$ in $\PSh_\Sigma(\Sm_T)$ can be written as a simplicial colimit of monomorphisms between discrete presheaves, by using the injective model structure on simplicial presheaves. As $p_\otimes$ preserves simplicial colimits, Proposition~\ref{prop:quotients} effectively gives a formula for $p_\otimes(\cofib(f))$. In particular, any $X\in\PSh_\Sigma(\Sm_T)_\pt$ is the colimit of a pointed simplicial presheaf $X_\bullet$, and we get
	\[
	p_\otimes(X) \simeq \colim_{n\in\Delta^\op} p_*(X_n)/p_*(X_n\Vert *).
	\]
\end{remark}

Finally, we investigate the behavior of $p_\otimes$ with respect to the Nisnevich topology.

\begin{lemma}\label{lem:comparison}
	Let $p\colon T\to S$ be an integral morphism, let $X\in\PSh(\Sm_T)$, and let $Y\subset X$ be an open subpresheaf.
	Then the inclusion
	$p_*(X\Vert Y) \into p_*(X|Y)$
	in $\PSh(\Sm_S)$ is a Nisnevich equivalence.
\end{lemma}

\begin{proof}
	Let $U$ be the henselian local scheme of a point in a smooth $S$-scheme. 
	If $p$ is finite, $U\times_ST$ is a sum of local schemes, so any open subset of $U\times_ST$ covering $U$ contains a clopen subset covering $U$.
	In general, $p$ is a limit of finite morphisms and any quasi-compact open subset of $U\times_ST$ is defined at a finite stage, so the same result holds.
	 This implies that the inclusion
	\[
	p_*(X\Vert Y)(U) \into p_*(X|Y)(U)
	\]
	is an equivalence.
	By Proposition~\ref{prop:nisnevich-properties}(1), we deduce that the inclusion $p_*(X\Vert Y) \into p_*(X|Y)$ becomes $\infty$-connective in $\Shv_\Nis(\Sm_S)$, and in particular an effective epimorphism. Since it is also a monomorphism, it is a Nisnevich equivalence by \cite[Proposition 6.2.3.4]{HTT}.
\end{proof}

\begin{corollary}\label{cor:quotients}
	Let $p\colon T\to S$ be an integral universally open morphism, let $X\in\PSh_\Sigma(\Sm_T)$, and let $Y\subset X$ be an open subpresheaf. Then there is a natural equivalence
	\[
	p_\otimes(X/Y)\simeq p_*(X)/p_*(X|Y)
	\]
	in $\Shv_\Nis(\Sm_S)_\pt$.
\end{corollary}

\begin{proof}
	Combine Proposition~\ref{prop:quotients} and Lemma~\ref{lem:comparison}.
\end{proof}

\begin{remark}
	The only property of $\Sm_S$ that we have used so far is that it is an admissible category in the sense of \cite[\sect0]{MEMS}.
	If we use instead the category $\mathrm{QP}_S$ of quasi-projective $S$-schemes, we can compare our functor $p_\otimes p^*$ with the smash power $X\mapsto X^{\wedge T}$ from \cite[\sect5.2]{DeligneNote}, defined when $p\colon T\to S$ is finite locally free.
	For $X\in\PSh_\Sigma(\mathrm{QP}_S)_\pt$ discrete, $X^{\wedge T}$ is the quotient of $p_*(X_T)$ by a certain subpresheaf $(X,*)^T_1$. If the inclusion $*\into X$ is open, then $(X,*)^T_1=p_*(X_T|*)$. By Corollary~\ref{cor:quotients}, we then have a canonical equivalence $p_\otimes(X_T)\simeq X^{\wedge T}$ in $\Shv_{\Nis}(\QP_S)_\pt$. The definition of $(X,*)^T_1$ when $*\into X$ is not open appears to be wrong, as it does not make \cite[Lemma 20]{DeligneNote} true; presumably, the intended definition is $(X,*)^T_1=p_*(X_T\Vert *)$.
\end{remark}

\begin{proposition}\label{prop:pairs}
	Let $p\colon T\to S$ be a finite étale morphism, let $X\in\Sm_T$, and let $Z\subset X$ be a closed subscheme. Suppose that the Weil restriction $\Weil_pX$ exists, e.g., that $X$ is quasi-projective over $T$.
	Then 
	\[
	p_\otimes\left(\frac X{X\minus Z}\right)\simeq \frac{\Weil_pX}{\Weil_pX\minus \Weil_pZ}
	\]
	in $\Shv_\Nis(\Sm_S)_\pt$.
\end{proposition}

\begin{proof}
	By Lemma~\ref{lem:Weil-smooth}, $\Weil_pX$ is a smooth $S$-scheme.
	Since $p$ is finite locally free, $\Weil_pZ$ is a closed subscheme of $\Weil_pX$ \cite[\sect7.6, Proposition 2(ii)]{NeronModels}, so $\Weil_pX\minus \Weil_pZ$ is a well-defined smooth $S$-scheme.
	Note that $p_*(X|X\minus Z)\subset \Weil_pX\minus \Weil_pZ$. By Corollary~\ref{cor:quotients}, it will suffice to show that this inclusion is an isomorphism. Let $U$ be a smooth $S$-scheme.
	A morphism $s\colon U\to \Weil_pX$ factors through $\Weil_pX\minus \Weil_pZ$ if and only if, for every $x\in U$, the adjoint morphism $p_U^{-1}(x)\to X$ does not factor through $Z$. As $p_U^{-1}(x)$ is reduced, this is the case if and only if the image of $p_U^{-1}(x) \to X$ intersects $X\minus Z$, i.e., if and only if $s$ factors through $p_*(X|X\minus Z)$.
\end{proof}

\begin{example}\label{ex:tangent}
	Proposition~\ref{prop:pairs} fails if we only assume $p$ finite locally free. In fact, if $p$ is a universal homeomorphism, then $p_*(X|X\minus Z)=\Weil_p(X\minus Z)$ which is usually a strict open subset of $\Weil_pX\minus \Weil_pZ$. For example, let $S[\epsilon]=S\times\Spec \Z[t]/(t^2)$ and let $p\colon S[\epsilon]\to S$ be the projection, so that $\Weil_p(X[\epsilon])$ is the tangent bundle $\rm T_{X/S}$ for every $S$-scheme $X$. If $Z\subset X$ is reduced and has positive codimension, then the canonical map
	\[
	p_\otimes\left(\frac{X[\epsilon]}{X[\epsilon]\minus Z[\epsilon]}\right) \simeq \frac{\rm T_{X/S}}{\rm T_{X\minus Z/S}}\to \frac{\rm T_{X/S}}{\rm T_{X/S}\minus \rm T_{Z/S}}
	\]
	in $\Shv_\Nis(\Sm_S)_\pt$ is not an equivalence. If $Z\subset X$ is $S\subset \A^1_S$, it is even nullhomotopic in $\H_\pt(S)$.
\end{example}

\section{Norms of motivic spectra}
\label{sec:stablenorms}

\subsection{Stable motivic homotopy theory}

If $V$ is a vector bundle over a scheme $S$, we denote by $\S^V\in\H_\pt(S)$ the motivic sphere defined by
\[
\S^V= V/(V\minus 0) \simeq \P(V\times\A^1)/\P(V),
\]
and we write $\Sigma^V=\S^V\wedge (\ph)$ and $\Omega^V=\Hom(\S^V,\ph)$.

The $\infty$-category $\SH(S)$ of motivic spectra over $S$ is the presentably symmetric monoidal $\infty$-category obtained from $\H_\pt(S)$ by inverting the motivic sphere $\S^{\A^1}$ \cite[\sect2]{Robalo}. We will denote by
\[
\Sigma^\infty : \H_\pt(S) \adj \SH(S) : \Omega^\infty
\] 
the canonical adjunction. By \cite[Corollary 2.22]{Robalo}, the underlying $\infty$-category of $\SH(S)$ is the limit of the tower
\[
\dotsb \to \H_\pt(S)\xrightarrow{\Omega^{\A^1}}\H_\pt(S)\xrightarrow{\Omega^{\A^1}}\H_\pt(S).
\]
Recall that, for every vector bundle $V$ on $S$, the motivic sphere $\S^V$ becomes invertible in $\SH(S)$ \cite[Corollary 2.4.19]{CD}; we denote by $\S^{-V}$ its inverse.

Because the unstable norm functor $p_\otimes\colon \H_\pt(T)\to\H_\pt(S)$ does not preserve colimits, we cannot extend it to a symmetric monoidal functor $p_\otimes\colon \SH(T)\to\SH(S)$ using the universal property of $\Sigma^\infty\colon \H_\pt(T)\to\SH(T)$ in $\CAlg(\Pr^\mathrm{L})$. However, $\Sigma^\infty$ has a slightly more general universal property, as a special case of the following lemma:

\begin{lemma}\label{lem:inversion}
	Let $\scr C$ be a compactly generated symmetric monoidal $\infty$-category whose tensor product preserves compact objects and colimits in each variable. Let $X$ be a compact object of $\scr C$ such that the cyclic permutation of $X^{\otimes n}$ is homotopic to the identity for some $n\geq 2$, and let $\Sigma^\infty\colon \scr C \to \scr C[X^{-1}]$ be the initial functor in $\CAlg(\Pr^\mathrm{L})$ sending $X$ to an invertible object.
	Let $\scr K$ be a collection of simplicial sets containing filtered simplicial sets and $\scr D$ a symmetric monoidal $\infty$-category admitting $\scr K$-indexed colimits and whose tensor product preserves $\scr K$-indexed colimits in each variable. Composition with $\Sigma^\infty$ then induces a fully faithful embedding
	\[
	\Fun^{\otimes,\scr K}(\scr C[X^{-1}],\scr D) \into \Fun^{\otimes,\scr K}(\scr C,\scr D),
	\]
	where $\Fun^{\otimes,\scr K}$ denotes the $\infty$-category of symmetric monoidal functors that preserve $\scr K$-indexed colimits,
	whose essential image consists of those functors $F\colon\scr C\to\scr D$ such that $F(X)$ is invertible.
\end{lemma}

\begin{proof}
	Suppose first that $\scr K$ is the collection of filtered simplicial sets. To prove the statement in that case, it suffices to show that the symmetric monoidal functor $\scr C^\omega \to \scr C[X^{-1}]$ induces an equivalence $\Ind(\scr C^\omega[X^{-1}])\simeq\scr C[X^{-1}]$, where $\scr C^\omega \to \scr C^\omega[X^{-1}]$ is the initial functor in $\CAlg(\Cat_\infty)$ sending $X$ to an invertible object. By \cite[Proposition 2.19]{Robalo}, since $X$ is a symmetric object, $\scr C^\omega[X^{-1}]$ has underlying $\infty$-category the filtered colimit
	\[
	\scr C^\omega \xrightarrow{X\otimes\ph} \scr C^\omega \xrightarrow{X\otimes\ph} \scr C^\omega \to\dotsb
	\]
	in $\Cat_\infty$. By \cite[Proposition 5.5.7.11]{HTT}, $\scr C^\omega[X^{-1}]$ is also the colimit of this sequence in $\Cat_\infty^\mathrm{rex}$.
	 On the other hand, by \cite[Corollary 2.22]{Robalo}, $\scr C[X^{-1}]$ has underlying $\infty$-category the filtered colimit
	\[
	\scr C \xrightarrow{X\otimes\ph} \scr C \xrightarrow{X\otimes\ph} \scr C \to\dotsb
	\]
	in $\Pr^\mathrm{L}$. The claim now follows from the fact that $\Ind\colon \Cat_\infty^\mathrm{rex}\to\Pr^\mathrm{L}$ preserves colimits, which follows from \cite[Propositions 5.5.7.10, 5.5.7.6, and 5.5.3.18]{HTT}.
	
	It remains to prove the following: if a symmetric monoidal functor $F\colon \scr C[X^{-1}]\to\scr D$ preserves filtered colimits and if $F\circ\Sigma^\infty$ preserves colimits of shape $\scr I$, then $F$ preserves colimits of shape $\scr I$. 
	Under the equivalence $\Pr^\mathrm{L,\op}\simeq\Pr^\mathrm{R}$, $\scr C[X^{-1}]$ becomes the limit of the tower
	\[
	\dotsb \to \scr C \xrightarrow{\Hom(X,\ph)} \scr C \xrightarrow{\Hom(X,\ph)} \scr C
	\]
	in $\Pr^\mathrm{R}$, whence in $\Cat_\infty$ \cite[Proposition 5.5.3.18]{HTT}.
	Let $A\colon\scr I\to \scr C[X^{-1}]$ be a diagram and let $A_n\colon \scr I \to \scr C$ be its $n$th component in the above limit. Then $A\simeq\colim_{n} (X^{\otimes(-n)}\otimes\Sigma^\infty A_n)$ by \cite[Lemma 6.3.3.7]{HTT}, so it suffices to show that $F$ preserves the colimit of $\Sigma^\infty A_n$, which is true by assumption.
\end{proof}

\begin{remark}\label{rmk:inversion}
	Lemma~\ref{lem:inversion} also holds if we replace the object $X$ by an arbitrary set of such objects (see \cite[\sect6.1]{Hoyois} for a discussion of $\scr C[X^{-1}]$ in this case). In fact, the following observation reduces the proof to the single-object case. If $D$ is a filtered diagram of compactly generated $\infty$-categories whose transition functors preserve colimits and compact objects, then the colimit of $D$ in $\Pr^\mathrm{L}$ is also the colimit of $D$ in the larger $\infty$-category of $\infty$-categories with $\scr K$-indexed colimits (and functors preserving those colimits), where $\scr K$ is any collection of simplicial sets containing filtered simplicial sets.
\end{remark}

\begin{remark}\label{rmk:C-module}
	With $\scr C$, $X$, and $\scr K$ as in Lemma~\ref{lem:inversion}, the functor $\Sigma^\infty\colon \scr C\to \scr C[X^{-1}]$ also has a universal property as a $\scr C$-module functor that preserves $\scr K$-indexed colimits. The proof is essentially the same.
\end{remark}

\subsection{The stable norm functors}

If $p\colon T\to S$ is a finite étale morphism and $V\to T$ is a vector bundle, its Weil restriction $\Weil_p V\to S$ has a canonical structure of vector bundle. In fact, if $V=\V(\scr E)$, then $\Weil_p V \simeq \V(p_*(\scr E))$ by comparison of universal properties, using that $p_*$ is also \emph{left} adjoint to $p^*$ (by Lemma~\ref{lem:frobenius}).

\begin{lemma}\label{lem:spheres}
	Let $p\colon T\to S$ be a finite étale morphism and let $V$ be a vector bundle over $T$. Then \[p_\otimes(\S^V)\simeq \S^{\Weil_pV}\] in $\H_\pt(S)$.
\end{lemma}

\begin{proof}
	This follows immediately from Proposition~\ref{prop:pairs}.
\end{proof}

\begin{proposition}\label{prop:stablenorm}
	Let $p\colon T\to S$ be a finite étale morphism. Then the functor $\Sigma^\infty p_\otimes\colon \H_\pt(T)\to\SH(S)$ has a unique symmetric monoidal extension
	\[
	p_\otimes\colon\SH(T)\to\SH(S)
	\]
	that preserves filtered colimits. Moreover, it preserves sifted colimits.
\end{proposition}

\begin{proof}
	Combine Lemmas \ref{lem:inversion} and~\ref{lem:spheres} with the fact that $\Sigma^\infty\S^{\Weil_p\A^1}\in\SH(S)$ is invertible.
\end{proof}

\begin{remark}
	Let $p\colon T\to S$ be finite étale and let $E\in\SH(T)$. We can write $E\simeq \colim_n \Sigma^{-\A^n}\Sigma^\infty E_n$, where $E_n$ is the $n$th space of $E$. Hence, by Lemma~\ref{lem:spheres},
	\[
	p_\otimes(E) \simeq \colim_n \Sigma^{-\Weil_p\A^n} \Sigma^\infty p_\otimes(E_n).
	\]
\end{remark}

\begin{remark}\label{rmk:non-etale-norm}
	The assumption that $p$ is finite étale in Proposition~\ref{prop:stablenorm} is not always necessary. For example, if $p\colon T\to S$ is a universal homeomorphism with a section, then $p_\otimes(\S^{\A^1})\simeq p_*(\S^{\A^1})\simeq \S^{\A^1}$ by Theorem~\ref{thm:norm}(5) and so the extension $p_\otimes\colon \SH(T)\to \SH(S)$ exists. In fact, it is equivalent to $p_*$ and is an equivalence of $\infty$-categories.
\end{remark}

\begin{remark}\label{rmk:norms-dont-stabilize}
The assumption that $p$ is finite étale in Proposition~\ref{prop:stablenorm} cannot be dropped in general. Let $S$ be the spectrum of a discrete valuation ring over a field of characteristic $\neq 2$, with closed point $i\colon \{x\} \to S$ and generic point $j\colon \{\eta\} \to S$. Let $p\colon T \to S$ be a finite locally free morphism with generic fiber étale of degree $n>1$ and the reduction of the special fiber mapping isomorphically to $x$. We claim that $E = \Sigma^\infty p_\otimes(\S^{\A^1}) \in \SH(S)$ is not invertible. Since $p_\otimes$ commutes with base change (Proposition~\ref{prop:BC}), we have $j^* E \simeq \S^{\A^n}$ by Lemma \ref{lem:spheres}, whereas $i^* E \simeq \S^{\A^1}$ by Remark~\ref{rmk:non-etale-norm}. Consider the gluing triangle
\[ j_! j^* E \to E \to i_! i^* E \xrightarrow{\alpha} j_!j^* E[1]. \]
It suffices to show that $\alpha \simeq 0$, since for invertible $E\in\SH(S)$ we have $[E, E] \simeq [\1_S, \1_S] \simeq \GW(S)$ (Theorem \ref{thm:GW}), which has no idempotents by \cite[Theorem 2.4]{gille2015quadratic} and \cite[Proposition II.2.22]{wittrings}. It thus suffices to show that $[\S^{\A^1}, i^! j_! \S^{\A^n}[1]] = 0$.
This follows from the gluing exact sequences
\begin{gather*}
[\1_\eta, \S^{\A^{n-1}}]\to [\1_x, i^! j_! \S^{\A^{n-1}}[1]] \to [\1_S, j_! \S^{\A^{n-1}}[1]],\\
[\1_x, \S^{\A^{n-1}}]\to [\1_S, j_! \S^{\A^{n-1}}[1]]\to [\1_S, \S^{\A^{n-1}}[1]],
\end{gather*}
and Morel's stable $\A^1$-connectivity theorem \cite[Theorem 6.1.8]{Morel:2005}.
\end{remark}

\begin{remark}\label{rmk:Pic}
	Let $p\colon T\to S$ be finite étale. Since $p_\otimes\colon\SH(T)\to\SH(S)$ is symmetric monoidal, it induces a morphism of grouplike $\E_\infty$-spaces
	\[
	p_\otimes\colon \Pic(\SH(T)) \to \Pic(\SH(S)).
	\]
	Moreover, as the equivalence of Lemma~\ref{lem:spheres} is manifestly natural and symmetric monoidal as $V$ varies in the groupoid of vector bundles, there is a commutative square of grouplike $\E_\infty$-spaces
	\begin{tikzmath}
		\diagram{ \K^\oplus(T) & \Pic(\SH(T)) \\ \K^\oplus(S) & \Pic(\SH(S))\rlap, \\};
		\arrows (11-) edge (-12) (21-) edge (-22) (11) edge node[left]{$p_*$} (21) (12) edge node[right]{$p_\otimes$} (22);
	\end{tikzmath}
	where $\K^\oplus$ is the direct-sum $\K$-theory of vector bundles and the horizontal maps are induced by $V\mapsto \S^{V}$.
	In fact, since this square is compatible with base change in $S$ (see \sect\ref{sub:coherence}) and $\SH(\ph)$ is a Zariski sheaf, one can replace $\K^\oplus$ by the Thomason–Trobaugh $\K$-theory. We will discuss this in more detail in \sect\ref{sub:Jhomomorphism}.
\end{remark}

\begin{example}
	If $p\colon T\to S$ is finite étale of degree $2$, it follows from Proposition~\ref{prop:stablenorm} and Example~\ref{ex:degree2} that $\Sigma^\infty\Sigma T\in \SH(S)$ is an invertible object. When $S$ is the spectrum of a field $k$ of characteristic $\neq 2$ and $T=\Spec k[\sqrt a]$ for some $a\in k^\times$, Hu gave an explicit description of the inverse \cite[Proposition 1.1]{Hu}, which implies that $p_\otimes(\S^{-1})\simeq \Sigma^{-2}\Sigma^{-1,-1}\Sigma^\infty\Spec k[x,y]/(x^2-ay^2-1)$. By \cite[Proposition 3.4]{Hu}, if $a$ is not a square, then $p_\otimes(\S^{1})\not\simeq \S^{p,q}$ in $\SH(k)$ for all $p,q\in\Z$.
\end{example}

\section{Properties of norms}
\label{sec:functoriality}

In this section, we investigate the compatibility of finite étale norms with various features of the $\infty$-category $\SH(S)$. We will see that norms commute with arbitrary base change and \emph{distribute} over smoothly parametrized sums and properly parametrized products, and we will describe precisely the effect of norms on the purity and ambidexterity equivalences.

\subsection{Composition and base change}

\begin{proposition}[Composition]
	\label{prop:composition}
	Let $p\colon T\to S$ and $q\colon S\to R$ be universally clopen morphisms. Then there is a symmetric monoidal natural equivalence
	\[
	(qp)_\otimes \simeq q_\otimes p_\otimes \colon \PSh_\Sigma(\Sm_T)_\pt \to \PSh_\Sigma(\Sm_R)_\pt.
	\]
	Hence, if $p$ and $q$ are integral and universally open (resp.\ are finite étale), the same holds in $\H_\pt$ (resp.\ in $\SH$).
\end{proposition}

\begin{proof}
	By the universal property of $\PSh_\Sigma$, it suffices to construct such an equivalence between the restrictions of these functors to $\Sm_{T+}$. By Lemma~\ref{lem:0truncated}, this is now a $1$-categorical task, which is straightforward using properties (2) and (3) of Theorem~\ref{thm:norm}.
\end{proof}

Consider a cartesian square of schemes
\begin{tikzequation}\label{eqn:pfsquare}
	\diagram{T' & T \\ S' & S \\};
	\arrows (11-) edge node[above]{$g$} (-12) (11) edge node[left]{$q$} (21) (21-) edge node[below]{$f$} (-22) (12) edge node[right]{$p$} (22);
\end{tikzequation}
with $p$ universally clopen. Then there is a unique symmetric monoidal natural transformation
\[
\Ex^*_\otimes\colon f^*p_\otimes \to q_\otimes g^* \colon \PSh_\Sigma(\Sm_T)_\pt \to \PSh_\Sigma(\Sm_{S'})_\pt
\]
given on $X_+\in\Sm_{T+}$ by the exchange transformation $f^*p_*(X)_+ \to q_*g^*(X)_+$.
One only needs to check that the latter is natural with respect to collapse maps $(X\amalg Y)_+\to X_+$, which is straightforward.
If $p$ is moreover integral, then $f^*p_\otimes$ and $q_\otimes g^*$ preserve motivic equivalences (Theorem~\ref{thm:norm}(4)), and by the universal property of $\L_\mot\colon \PSh_\Sigma(\Sm_T)_\pt\to \H_\pt(T)$ we obtain an induced symmetric monoidal natural transformation
\[
	\Ex^*_\otimes\colon f^*p_\otimes \to q_\otimes g^* \colon \H_\pt(T) \to \H_\pt(S').
\]
If $p$ is finite étale, then by the universal property of $\H_\pt(T)\to\SH(T)$ we further obtain an induced symmetric monoidal natural transformation
\[
	\Ex^*_\otimes\colon f^*p_\otimes \to q_\otimes g^* \colon \SH(T) \to \SH(S').
\]

\begin{proposition}[Base change]
	\label{prop:BC}
	Consider the cartesian square of schemes~\eqref{eqn:pfsquare}, with $p$ universally clopen. 
	Let $\scr C\subset\Sm_T$ be a full subcategory and let $X\in\PSh_\Sigma(\scr C)_\pt$.
	Assume that either of the following conditions hold:
	\begin{enumerate}
		\item $f$ is smooth;
		\item for every $U\in\scr C$, the Weil restriction $\Weil_pU$ is a smooth $S$-scheme; for example, $p$ is finite locally free and $\scr C=\SmQP_T$.
	\end{enumerate}
	Then the transformation $\Ex^*_\otimes\colon f^*p_\otimes(X) \to q_\otimes g^*(X)$ is an equivalence. Hence, if $p$ is finite locally free (resp.\ finite étale) and $f$ is arbitrary, then $\Ex^*_\otimes\colon f^*p_\otimes \to q_\otimes g^*$ is an equivalence in $\H_\pt$ (resp.\ in $\SH$).
\end{proposition}

\begin{proof}
	By definition of $\Ex^*_\otimes$, it suffices to check that, for every $U\in\scr C$, the map $f^*p_*(U)\to q_*g^*(U)$ is an equivalence in $\PSh_\Sigma(\Sm_{S'})$.
	This is clear if $f$ is smooth, as $f^*p_*\to q_*g^*$ is then the mate of the equivalence $g_\sharp q^* \simeq p^*f_\sharp$. 
	Under assumption (2), the result holds because Weil restriction commutes with arbitrary base change.
	The claim in (2) follows from Lemmas~\ref{lem:Weil} and \ref{lem:Weil-smooth}.
	The last statement for $\H_\pt$ follows from (2) because $\SmQP_{T+}$ generates $\H_\pt(T)$ under sifted colimits. 
	Finally, the last statement for $\SH$ follows because $\SH(T)$ is generated by suspension spectra under filtered colimits, smash products, and smash inverses.
\end{proof}

\begin{remark}
	Given two instances of~\eqref{eqn:pfsquare} glued horizontally, we obtain three exchange transformations of the form $\Ex^*_\otimes$ that fit in an obvious commutative diagram. Similarly, given two instances of~\eqref{eqn:pfsquare} glued vertically, the three induced exchange transformations fit in a commutative diagram, which now also involves the equivalence of Proposition~\ref{prop:composition}. Given a more general grid of such squares, we can ask what compatibilities exist between these commutative diagrams. We will answer this coherence question in \sect\ref{sub:coherence}.
\end{remark}

\begin{remark}\label{rmk:mates}
	The transformation $\Ex_\otimes^*\colon f^*p_\otimes\to q_\otimes g^*$ induces by adjunction
	\[
	\Ex_{\otimes*}\colon p_\otimes g_* \to f_* q_\otimes.
	\]
	If $f$ is smooth, we also have
	\[
	\Ex_{\sharp\otimes}\colon f_\sharp q_\otimes\to p_\otimes g_\sharp,
	\]
	since $\Ex_\otimes^*$ is invertible by Proposition~\ref{prop:BC}(1). Since $\Ex^*_\otimes$ is a symmetric monoidal transformation, $\Ex_{\otimes*}$ and $\Ex_{\sharp\otimes}$ are $\SH(T)$-linear transformations.
\end{remark}

\subsection{The distributivity laws}

\begin{lemma}
	\label{lem:distributivity}
	Let $p\colon T\to S$ be a finite locally free morphism and $h\colon U\to T$ a quasi-projective morphism. Consider the diagram
	\begin{tikzequation}\label{eqn:exponential}
		\diagram{ U & \Weil_pU\times_ST & \Weil_pU \\ & T & S\rlap, \\};
		\arrows (11-) edge[<-] node[above]{$e$} (-12) (12-) edge node[above]{$q$} (-13) (11) edge node[below left]{$h$} (22) (12) edge node[right]{$g$} (22) (22-) edge node[above]{$p$} (-23) (13) edge node[right]{$f$} (23);
	\end{tikzequation}
	where $e$ is the counit of the adjunction, $q$ and $g$ are the canonical projections, and $f = \Weil_p(h)$. Then the natural transformation
	\[
	\Dist_{\sharp *}\colon f_\sharp q_*e^* \xrightarrow{\Ex_{\sharp *}} p_* g_\sharp e^* \xrightarrow\epsilon p_* h_\sharp \colon \QP_U \to \QP_S
	\]
	is an isomorphism, where $\epsilon\colon g_\sharp e^*\simeq h_\sharp e_\sharp e^*\to h_\sharp$ is the counit map.
\end{lemma}

\begin{proof}
	This is a general result that holds in any category with pullbacks, as soon as $p_*$ exists.
	We leave the details to the reader.
\end{proof}

\begin{remark} \label{rem:big-pentagon-1}
	In the setting of Lemma~\ref{lem:distributivity}, suppose given a quasi-projective morphism $V\to U$, and form the diagram
	\begin{tikzmath}
		\diagram[column sep={6em,between origins}]{V & W & \Weil_pV\times_ST & \Weil_pV \\ & U & \Weil_pU\times_ST & \Weil_pU \\ & & T & S \\};
		\arrows (11-) edge[<-] (-12) (12-) edge[<-] (-13) (13-) edge (-14)
		(22-) edge[<-] node[above]{$e$} (-23) (23-) edge node[above]{$q$} (-24) (33-) edge node[above]{$p$} (-34)
		(11) edge (22) (12) edge (23) (13) edge (23) (14) edge (24)
		(22) edge (33) (23) edge (33) (24) edge (34);
	\end{tikzmath}
	in which all three quadrilaterals are cartesian; in particular $W = V \times_U (\Weil_p U \times_S T)=e^*(V)$. The lower pentagon and the boundary are both instances of~\eqref{eqn:exponential}. Lemma~\ref{lem:distributivity} states that the upper right pentagon is also an instance of~\eqref{eqn:exponential}, i.e., that $\Weil_p V \simeq \Weil_q W$.
\end{remark}

\begin{remark}
	In the diagram~\eqref{eqn:exponential}, if $h$ is formally smooth or finitely presented, so are $f$, $g$, and $e$.
\end{remark}

Note that when $p$ and $h$ are fold maps, Lemma~\ref{lem:distributivity} expresses the distributivity of products over sums in the category of schemes.
Given the diagram~\eqref{eqn:exponential} with $p$ finite locally free (resp.\ finite étale) and $h$ quasi-projective, we can consider the transformations
\begin{align*}
	\Dist_{\sharp \otimes}&\colon f_\sharp q_\otimes e^* \xrightarrow{\Ex_{\sharp\otimes}} p_\otimes g_\sharp e^* \xrightarrow{\epsilon} p_\otimes h_\sharp, \\
	\Dist_{\otimes *}&\colon p_\otimes h_* \xrightarrow{\eta} p_\otimes g_* e^* \xrightarrow{\Ex_{\otimes*}} f_* q_\otimes e^*,
\end{align*}
in $\H_\pt$ (resp.\ in $\SH$), the former assuming $h$ smooth. These transformations measure the distributivity of parametrized smash products over parametrized sums and products, respectively. It follows from Remark~\ref{rmk:mates} that $\Dist_{\sharp\otimes}$ and $\Dist_{\otimes *}$ are $\SH(T)$-linear transformations.

\begin{proposition}[Distributivity]
	\label{prop:distributivity}
	Consider the diagram~\eqref{eqn:exponential} with $h$ quasi-projective.
	\begin{enumerate}
		\item If $p$ is finite locally free (resp.\ finite étale) and $h$ is smooth, then $\Dist_{\sharp\otimes}$ is an equivalence in $\H_\pt$ (resp.\ in $\SH$).
		\item If $p$ is finite étale and $h$ is a closed immersion (resp.\ is proper), then $\Dist_{\otimes*}$ is an equivalence in $\H_\pt$ (resp.\ in $\SH$).
	\end{enumerate}
\end{proposition}

\begin{proof}
	(1) As all functors involved preserve sifted colimits, it suffices to check that $\Dist_{\sharp\otimes}$ is an equivalence on $X_+$ (resp.\ on $\S^{-\A^n}\wedge \Sigma^\infty_+X$) for $X\in\SmQP_U$. Via the smooth projection formulas, the morphism $\Dist_{\sharp\otimes}(h^*A\wedge B)$ is identified with $p_\otimes A\wedge \Dist_{\sharp\otimes}(B)$. Since moreover all functors commute with $\Sigma^\infty$, it suffices to treat the unstable case.
	The result then follows from Lemma~\ref{lem:distributivity}.
	
	(2) Suppose first that $h$ is a closed immersion. Then $h_*$ is fully faithful, and as all functors involved preserve sifted colimits, it suffices to show that $\Dist_{\otimes *}h^*$ is an equivalence on $X_+$ (resp.\ on $\S^{-\A^n}\wedge \Sigma^\infty_+X$) for $X\in\SmQP_{T}$. Via the closed projection formulas, the morphism $\Dist_{\otimes*}(h^*A\wedge B)$ is identified with $p_\otimes A\wedge \Dist_{\otimes*}(B)$. Since moreover all functors commute with $\Sigma^\infty$, it suffices to treat the unstable case.
	 By the gluing theorem \cite[Proposition C.10]{HoyoisGLV}, the unit map $\id\to h_*h^*$ induces an equivalence
	\[
	\frac{X}{X\minus X_U} \simeq h_*h^*(X_+),
	\]
	and similarly for 
	$g_*g^*(X_+)$ and $f_*f^*(\Weil_pX_+)$. The transformation $\Dist_{\otimes *}h^*$ on $X_+$ is thereby identified with the collapse map
	\[
	p_\otimes\left(\frac{X}{X\minus X_U}\right) \to \frac{\Weil_pX}{\Weil_pX\minus \Weil_p(X_U)}.
	\]
	As $p$ is finite étale, this map is an equivalence by Proposition~\ref{prop:pairs}.
	
	It remains to prove that $\Dist_{\otimes*}$ is an equivalence in $\SH$ when $h$ is proper.
	If $h$ is also smooth, this follows from (1) and Proposition~\ref{prop:alpha} below.\footnote{The proof of Proposition~\ref{prop:alpha} uses part (2) of the current proposition, but only for $h$ a closed immersion.}
	If $a\colon S'\to S$ is a smooth morphism, it is easy to show that the transformation $a^*\Dist_{\otimes*}$ can be identified with $\Dist_{\otimes*}a^*$. In particular, the question is Nisnevich-local on $S$ and we may assume that $S$ is affine. In this case, $T$ is also affine and $h$ factors as $h''\circ h'$ where $h'$ is a closed immersion and $h''$ is smooth and projective \cite[Tag 087S]{Stacks}. The transformation $\Dist_{\otimes*}$ for $h$ is then the composition of the transformation $\Dist_{\otimes *}$ for $h''$, an exchange transformation $\Ex_*^*$, and the transformation $\Dist_{\otimes *}$ for a pullback of $h'$. The exchange transformation is an equivalence by proper base change, so we are done.
\end{proof}

\begin{remark}\label{rem:pairs}
	Proposition~\ref{prop:distributivity} provides a functorial refinement of Proposition~\ref{prop:pairs} as follows. Suppose that $p\colon T\to S$ is finite locally free, that $h\colon X\to T$ is smooth and quasi-projective, and that $u\colon Z\into X$ is a closed immersion. As in Remark \ref{rem:big-pentagon-1}, we may form the diagram
	\begin{tikzequation}\label{eqn:closedexponential}
		\diagram[column sep={5em,between origins},row sep={4em,between origins}]{
		Z & \bullet & \bullet & \Weil_pZ \\ & X & \bullet & \Weil_pX \\ & & T & S\rlap. \\
		};
		\arrows (11-) edge[<-] node[above]{$e'$} (-12) (12-) edge[<-] node[above]{$d$} (-13) (13-) edge node[above]{$r$} (-14)
		(22-) edge[<-] node[above]{$e$} (-23) (23-) edge node[above]{$q$} (-24) (33-) edge node[above]{$p$} (-34)
		(11) edge[c->] node[below left]{$u$} (22) (12) edge[c->] node[circle,inner sep=1,fill=white]{$u'$} (23) (13) edge[c->] node[right]{$t$} (23) (14) edge[c->] node[right]{$s$} (24)
		(22) edge node[below left]{$h$} (33) (23) edge node[right]{$g$} (33) (24) edge node[right]{$f$} (34);
	\end{tikzequation}
	Then we have a zig-zag
	\[
	p_\otimes h_\sharp u_* \xleftarrow{\Dist_{\sharp\otimes}} f_\sharp q_\otimes e^*u_* \xrightarrow{\Ex^*_*} f_\sharp q_\otimes u'_*e^{\prime*} \xrightarrow{\Dist_{\otimes*}} f_\sharp s_*r_\otimes d^*e^{\prime*}
	\]
	in $\Fun(\H_\pt(Z),\H_\pt(S))$, where the first two maps are equivalences, and the third one is an equivalence if $p$ is finite étale. Evaluated on $\1_Z$, this zig-zag can be identified with the canonical map
	\[
	p_\otimes\left(\frac X{X\minus Z}\right)\to \frac{\Weil_pX}{\Weil_pX\minus \Weil_pZ}
	\]
	in $\H_\pt(S)$.
\end{remark}

It will be useful to have an explicit form of the distributivity law for binary coproducts. Let $p\colon T \to S$ be a \emph{surjective} finite locally free morphism, and consider the diagram
\begin{tikzmath}
	\def\colsep{1.5em}
	\diagram{ T\amalg T & \Weil_p(T\amalg T)_T & \Weil_p(T\amalg T) \\ & T & S\rlap, \\};
	\arrows (11-) edge[<-] node[above]{$e$} (-12) (12-) edge node[above]{$q$} (-13) (11) edge node[below left]{$h$} (22) (12) edge node[right]{$g$} (22) (22-) edge node[above]{$p$} (-23) (13) edge node[right]{$f$} (23);
\end{tikzmath}
where $h$ is the fold map.
By Example~\ref{ex:addition-formula}, we can write $\Weil_p(T \amalg T) = S \amalg C \amalg S$, where the two maps $S \to \Weil_p(T \amalg T)$ correspond via adjunction to the two canonical maps $T \to T \amalg T$, and $C$ consists of the ``cross terms''. This induces a decomposition $\Weil_p(T \amalg T)_T = T \amalg C_T \amalg T$. The restriction of $e$ to $C_T$ can be further decomposed as
\[
C_T = L\amalg R \xrightarrow{e_l\amalg e_r} T\amalg T.
\]
Write $q_l\colon L \to C$ and $q_r\colon R\to C$ for the restrictions of $q$ and $c\colon C \to S$ for the restriction of $f$.

\begin{corollary} \label{cor:distrib-explicit}
Let $p\colon T \to S$ be a surjective finite locally free (resp.\ finite étale) morphism and let $E,F\in\H_\pt(T)$ (resp.\ $E, F \in \SH(T)$). With the notation as above, we have a natural equivalence
\[ p_\otimes(E \vee F) \simeq p_\otimes(E) \vee c_\sharp(q_{l\otimes}e_l^*(E) \wedge q_{r\otimes}e_r^*(F)) \vee p_\otimes(F). \]
Moreover,  $q_l$ and $q_r$ are also surjective.
\end{corollary}

\begin{proof}
The first claim is an immediate consequence of Proposition \ref{prop:distributivity}(1), together with the transitivity of $(\ph)_\sharp$ and $(\ph)_\otimes$ and their identification for fold maps.
For the second claim, note that an $S$-morphism $X\to C$ with $X$ connected is the same thing as a $T$-morphism $X_T\to T\amalg T$ that hits both copies of $T$ (by construction of $C$). It follows that $X_T\to C_T$ hits both $L$ and $R$. Taking $X$ to have a single point, we deduce that $q_l$ and $q_r$ are surjective.
\end{proof}

\begin{corollary}\label{cor:Nthom}
	Let $p\colon T\to S$ be finite étale, let $X\in\SmQP_T$, and let $\xi$ be a vector bundle on $X$ (resp.\ let $\xi\in \K(X)$). Then the distributivity transformation $\Dist_{\sharp\otimes}$ induces an equivalence
	\[
	p_\otimes\Th_X(\xi) \simeq \Th_{\Weil_pX}(q_*e^*\xi)
	\]
	in $\H_\pt(S)$ (resp.\ in $\SH(S)$), where $X \stackrel e\from \Weil_pX\times_ST \stackrel q\to \Weil_pX$.
\end{corollary}

\begin{proof}
	If $h\colon X\to T$ and $f\colon \Weil_pX\to S$ are the structure maps, we have
	\[
	p_\otimes h_\sharp \S^\xi \xleftarrow{\Dist_{\sharp\otimes}} f_\sharp q_\otimes e^* \S^\xi \simeq f_\sharp q_\otimes \S^{e^*\xi} \simeq f_\sharp \S^{q_*e^*\xi},
	\]
	where $\Dist_{\sharp\otimes}$ is an equivalence by Proposition~\ref{prop:distributivity}(1), and the last equivalence follows from Lemma~\ref{lem:spheres} (resp.\ from Remark~\ref{rmk:Pic}).
\end{proof}

\subsection{The purity equivalence}

Let $f\colon X\to S$ be a smooth morphism and $s\colon Z\into X$ a closed immersion such that $fs$ is smooth.
Recall that the purity equivalence $\Pi_s$ is a natural equivalence
\[
\Pi_s\colon f_\sharp s_*\simeq (fs)_\sharp\Sigma^{\rm N_s} \colon \H_\pt(Z) \to \H_\pt(S),
\]
where $\rm N_s$ is the normal bundle of $s$. 

\begin{lemma}\label{lem:Weil-normal}
	Consider the diagram~\eqref{eqn:closedexponential} where $p$ is finite étale and $u$ is a closed immersion in $\SmQP_T$. Then there is a canonical isomorphism of vector bundles
	\[
	\rm N_s \simeq r_*d^*e^{\prime *}(\rm N_u).
	\]
\end{lemma}

\begin{proof}
	Let $\scr N_i$ be the conormal sheaf of a closed immersion $i$, so that $\rm N_i=\V(\scr N_i)$.
	By the functoriality of conormal sheaves, there is a canonical map
	\[
	\scr N_u \to e'_* d_* \scr N_t.
	\]
	Since $r$ is finite étale, the functor $r_*$ is left adjoint to $r^*$ (see Lemma~\ref{lem:frobenius}). Using the isomorphism $\scr N_t\simeq r^*\scr N_s$, we obtain by adjunction a canonical map
	\[
	r_*d^*e^{\prime *}(\scr N_u) \to \scr N_s.
	\]
	To check that this map is an isomorphism, we can assume that $p$ is a fold map, in which case it is a straightforward verification.
\end{proof}

\begin{proposition}[Norms vs.\ purity]\label{prop:purity}
	Let $p\colon T\to S$ be a finite étale map and $u\colon Z\into X$ a closed immersion in $\SmQP_T$. Consider the induced diagram~\eqref{eqn:closedexponential}. Then the following diagram of equivalences commutes in $\H_\pt$ and $\SH$:
\begin{tikzmath}
	\diagram{
	f_\sharp s_*r_\otimes d^*e^{\prime*} & f_\sharp q_\otimes u'_*e^{\prime*} & f_\sharp q_\otimes e^*u_* & p_\otimes h_\sharp u_* \\
	(fs)_\sharp \Sigma^{\rm N_s}r_\otimes d^*e^{\prime*} & (fs)_\sharp \Sigma^{r_*d^*e^{\prime*}(\rm N_u)}r_\otimes d^*e^{\prime*} & (fs)_\sharp r_\otimes d^*e^{\prime*}\Sigma^{\rm N_u} & p_\otimes(hu)_\sharp\Sigma^{\rm N_u}\rlap. \\
	};
	\arrows (11-) edge[<-] node[above]{$\Dist_{\otimes *}$} (-12) (12-) edge[<-] node[above]{$\Ex^*_*$} (-13) (13-) edge node[above]{$\Dist_{\sharp\otimes}$} (-14) (11) edge node[left]{$\Pi_s$} (21) (14) edge node[right]{$\Pi_u$} (24)
	(21-) edge node[below]{$\simeq$} (-22) (22-) edge node[below]{$\simeq$} (-23) (23-) edge node[below]{$\Dist_{\sharp\otimes}$} (-24);
\end{tikzmath}
\end{proposition}

\begin{proof}
	Since $u_*\colon \H_\pt(Z)\to\H_\pt(X)$ is fully faithful \cite[Proposition C.10]{HoyoisGLV}, $\H_\pt(Z)$ is a localization of $\PSh_\Sigma(\SmQP_{X+})$.
	Since moreover all functors on display preserve sifted colimits, it suffices to show that the image of the given diagram in $\Fun(\SmQP_{X+},\H_\pt(S))$ can be made commutative.
	
	For notational simplicity, let us first consider the evaluation of this diagram on $\1_Z$.
	Recall that the purity equivalence $\Pi_u$ is then given by a zig-zag of equivalences
	\[
	\frac{X}{X\minus Z} \to \frac{\Def(X,Z)}{\Def(X,Z)\minus (Z\times\A^1)} \from \frac {\rm N_u}{\rm N_u\minus Z}
	\]
	in $\H_\pt(T)$, where $\Def(X,Z)=\Bl_{Z\times 0}(X\times\A^1)\minus \Bl_{Z\times 0}(X\times 0)$ is Verdier's deformation space.
	Using Proposition~\ref{prop:pairs}, we see that $p_\otimes(\Pi_u)$ is the zig-zag of equivalences
	\[
	\frac{\Weil_pX}{\Weil_pX\minus \Weil_pZ} \to \frac{\Weil_p\Def(X,Z)}{\Weil_p\Def(X,Z)\minus \Weil_p(Z\times\A^1)} \from \frac {\rm N_{s}}{\rm N_{s}\minus \Weil_pZ}
	\]
	in $\H_\pt(S)$. On the other hand, $\Pi_{s}$ is the zig zag of equivalences
	\[
	\frac{\Weil_pX}{\Weil_pX\minus \Weil_pZ} \to \frac{\Def(\Weil_pX,\Weil_pZ)}{\Def(\Weil_pX,\Weil_pZ)\minus (\Weil_pZ\times\A^1)} \from \frac {\rm N_{s}}{\rm N_{s}\minus \Weil_pZ}.
	\]
	By Remark~\ref{rem:pairs}, the commutativity of the diagram in this case states that these two zig-zags are equivalent in $\H_\pt(S)$. To produce such an equivalence, it suffices to find a morphism between their middle terms that makes both triangles commutes. Such a morphism is provided by the canonical map
	\[
	\Def(\Weil_pX,\Weil_pZ) \to \Weil_p\Def(X,Z)
	\]
	over the unit map $\A^1_S\to \Weil_p\A^1_T$.
	
	In the previous discussion, one can replace $X$ by any smooth quasi-projective $X$-scheme $Y$ (and $Z$ by $Y_Z$). The equivalence constructed at the end is then natural in $Y_+\in\SmQP_{X+}$. This shows that the restriction of the given diagram to $\SmQP_{X+}$ can be made commutative, as desired.
\end{proof}

\begin{corollary}\label{cor:purity}
	Let $p\colon T\to S$ be finite étale and let $s\colon Z\into X$ be a closed immersion in $\SmQP_T$. Then, under the identification of Proposition~\ref{prop:pairs},
	\[
	p_\otimes(\Pi_s) \simeq \Pi_{\Weil_ps} \colon \frac{\Weil_pX}{\Weil_pX\minus \Weil_pZ} \simeq \frac{\rm N_{\Weil_ps}}{\rm N_{\Weil_ps}\minus \Weil_pZ}
	\]
	in $\H_\pt(S)$.
\end{corollary}

\begin{proof}
	Evaluate the diagram of Proposition~\ref{prop:purity} on $\1_Z$.
\end{proof}

\subsection{The ambidexterity equivalence}

If $f\colon Y\to X$ is a smooth separated morphism, there is a canonical transformation
\[
\alpha_f\colon f_\sharp \to f_*\Sigma^{\rm T_f}
\]
in $\H_\pt$ and $\SH$: it is adjoint to the composition
\[
f^*f_\sharp \stackrel{\Ex^*_\sharp}\simeq \pi_{2\sharp}\pi_1^* \stackrel{\eta}\to \pi_{2\sharp}\delta_*\delta^*\pi_1^* \simeq \pi_{2\sharp}\delta_* \stackrel{\Pi_\delta}\simeq \Sigma^{\rm T_f},
\]
where $\delta\colon Y\to Y\times_XY$ is the diagonal, which is a closed immersion.\footnote{One must also use a specific identification of $\rm N_\delta$ with $\rm T_f$, lest an undesirable automorphism be introduced, but this will not be visible in our arguments. The correct choice can be found in \cite[(5.20)]{Hoyois}.} 

\begin{lemma}\label{lem:Weil-tangent}
	Consider the diagram~\eqref{eqn:exponential} with $p\colon T\to S$ finite étale and $h\colon U\to T$ smooth quasi-projective. Then there is a canonical isomorphism of vector bundles
	\[
	\rm T_f\simeq q_*e^*\rm T_h.
	\]
\end{lemma}

\begin{proof}
	This follows from Lemma~\ref{lem:Weil-normal} applied with $u$ the diagonal of $h$.
\end{proof}

\begin{proposition}[Norms vs.\ ambidexterity]\label{prop:alpha}
	Consider the diagram~\eqref{eqn:exponential} with $p\colon T\to S$ finite étale and $h\colon U\to T$ smooth quasi-projective. Then the following diagram commutes in $\H_\pt$ and $\SH$:
	\begin{tikzmath}
		\diagram{f_\sharp q_\otimes e^* & f_*\Sigma^{\rm T_f}q_\otimes e^* & f_*\Sigma^{q_*e^*\rm T_h}q_\otimes e^* \\
		p_\otimes h_\sharp & p_\otimes h_*\Sigma^{\rm T_h} & f_*q_\otimes e^* \Sigma^{\rm T_h}\rlap. \\};
		\arrows (11-) edge node[above]{$\alpha_f$} (-12) (12-) edge node[above]{$\simeq$} (-13) 
		(11) edge node[left]{$\Dist_{\sharp\otimes}$} (21) (21-) edge node[below]{$\alpha_h$} (-22)
		(22-) edge node[below]{$\Dist_{\otimes *}$} (-23) (13) edge node[right]{$\simeq$} (23);
	\end{tikzmath}
\end{proposition}

\begin{proof}
	We write $\delta$, $\zeta$, and $\theta$ for the diagonals of $f$, $g$, and $h$, and $\pi_i$, $\rho_i$, and $\sigma_i$ ($i=1,2$) for their canonical retractions. We will make use of the following diagram:
\begin{tikzmath}
	\diagram[column sep={5em,between origins}]{ U & \bullet & \bullet & \bullet & \Weil_pU \\ & U\times_TU & \bullet & \bullet & \Weil_pU\times_S\Weil_pU \\ & & U & \bullet & \Weil_pU \\ & & & T & S\rlap. \\};
	\arrows 
	(14-) edge node[above]{$q$} (-15)
	(11-) edge[<-] node[above]{$e$} (-12) (12-) edge[<-] node[below]{$d'$} (-13) (13-) edge[<-] node[below]{$c$} (-14)
	(11) edge[c->] node[fill=white]{$\theta$} (22) (12) edge[c->] node[fill=white]{$\theta'$} (23) (13) edge[c->] node[fill=white]{$\theta''$} (24) (14) edge[c->] node[right]{$\zeta$} (24) (15) edge[c->] node[right]{$\delta$} (25)
	(22-) edge[<-] node[below]{$e'$} (-23) (23-) edge[<-] node[below]{$d$} (-24) (24-) edge node[above]{$r$} (-25)
	(33-) edge[<-] node[above]{$e$} (-34) (34-) edge node[above]{$q$} (-35) (44-) edge node[above]{$p$} (-45)
	(22) edge node[fill=white]{$\sigma_2$} (33) (23) edge node[fill=white]{$\sigma_2'$} (34) (24) edge node[right]{$\rho_2$} (34) (25) edge node[right]{$\pi_2$} (35)
	(33) edge  node[fill=white]{$h$} (44) (34) edge node[right]{$g$} (44) (35) edge node[right]{$f$} (45);
	\draw[->,font=\scriptsize] (14.north west) to[out=165,in=15] node[above]{$\id$} (12.north east);
\end{tikzmath}
By Lemma~\ref{lem:distributivity}, the six similarly shaped pentagons in this diagram are all instances of~\eqref{eqn:exponential}.
There is of course a similar diagram having the first projections in the middle row: all the other maps are the same, except $e'$, $d$, $d'$, and $\theta'$. However, the composition $e'\circ d=e\times e$ and hence its pullback $e\circ d'$ are the same in both diagrams.

To prove the proposition, we divide the given rectangle as follows:
	\begin{tikzmath}
		\diagram{f_\sharp q_\otimes e^* & f_*\pi_{2\sharp}\pi_1^*q_\otimes e^* & f_*\pi_{2\sharp}\delta_*q_\otimes e^* & f_*\Sigma^{\rm T_f}q_\otimes e^* & f_*\Sigma^{q_*e^*\rm T_h}q_\otimes e^* \\
		p_\otimes h_\sharp & p_\otimes h_*\sigma_{2\sharp}\sigma_1^* & f_*q_\otimes e^*\sigma_{2\sharp}\sigma_1^* & f_*q_\otimes e^*\sigma_{2\sharp}\theta_* & f_*q_\otimes e^* \Sigma^{\rm T_h}\rlap. \\};
		\arrows (11-) edge (-12) (12-) edge node[above]{$\eta$} (-13) (13-) edge node[below]{$\Pi$} node[above]{$\simeq$} (-14) (14-) edge node[above]{$\simeq$} (-15) 
		(11) edge node[left]{$\Dist_{\sharp\otimes}$} (21) (21-) edge (-22)
		(22-) edge  node[below]{$\Dist_{\otimes *}$} (-23) (23-) edge node[below]{$\eta$} (-24)
		(24-) edge node[below]{$\Pi$} node[above]{$\simeq$} (-25) (15) edge node[right]{$\simeq$} (25)
		(12) edge[dashed] (23) (13) edge[dashed] (24)
		(12) edge[draw=none] node[font=\normalsize,left=1.5em]{$(1)$} (22)
		(13) edge[draw=none] node[font=\normalsize]{$f_*(2)$} (23)
		(14) edge[draw=none] node[font=\normalsize,right=1.5em]{$f_*(3)$} (24);
	\end{tikzmath}
The first dashed arrow is $f_*$ of the composition
\[
\pi_{2\sharp}\pi_1^* q_\otimes e^*\xrightarrow{\Ex^*_\otimes}
\pi_{2\sharp} r_\otimes\rho_1^*e^*\xrightarrow{\Ex_{\sharp\otimes}}
q_\otimes\rho_{2\sharp}\rho_1^*e^*\simeq
q_\otimes\rho_{2\sharp}(e\times e)^*\sigma_1^*\xrightarrow{\Ex^*_\sharp}
q_\otimes e^*\sigma_{2\sharp}\sigma_1^*
\]
(note that the last exchange transformation is associated with a square that is not cartesian and is not an equivalence).
The second dashed arrow is $f_*$ of the composite equivalence
\[
\pi_{2\sharp}\delta_*q_\otimes e^* \xleftarrow{\Dist_{\otimes*}} 
\pi_{2\sharp}r_\otimes\theta''_*d^{\prime*}e^* \xleftarrow{\Ex_*^*} 
\pi_{2\sharp}r_\otimes d^*e^{\prime*}\theta_*\xrightarrow{\Dist_{\sharp\otimes}}
q_\otimes\sigma'_{2\sharp}e^{\prime*}\theta_*\xrightarrow{\Ex_\sharp^*} 
q_\otimes e^*\sigma_{2\sharp}\theta_*.
\]

The pentagon (1) can be decomposed as follows (every diagonal exchange transformation is an equivalence):
\begin{tikzmath}
	\diagram[column sep={8em,between origins}]{
	f_\sharp q_\otimes e^* & f_*f^*f_\sharp q_\otimes e^* & & \\
	p_\otimes g_\sharp e^* & f_*f^*p_\otimes g_\sharp e^* & f_*\pi_{2\sharp}\pi_1^* q_\otimes e^* & \\
	p_\otimes h_\sharp & p_\otimes g_*g^*g_\sharp e^* & f_*q_\otimes g^* g_\sharp e^* & f_*\pi_{2\sharp} r_\otimes\rho_1^*e^* \\
	p_\otimes h_* h^* h_\sharp & p_\otimes g_* g^* h_\sharp & p_\otimes g_*\rho_{2\sharp}\rho_1^*e^* & f_*q_\otimes\rho_{2\sharp}\rho_1^*e^* \\
	& p_\otimes h_*\sigma_{2\sharp}\sigma_1^* & p_\otimes g_* e^* \sigma_{2\sharp}\sigma_1^* & f_*q_\otimes e^*\sigma_{2\sharp}\sigma_1^*\rlap. \\
	};
	\arrows
	(11-) edge node[above]{$\eta$} (-12)
	(21-) edge node[above]{$\eta$} (-22)
	(21) edge node[fill=white]{$\eta$} (32)
	(31) edge node[fill=white]{$\eta$} (42)
	(41-) edge node[above]{$\eta$} (-42)
	(31) edge node[left]{$\eta$} (41)
	(52-) edge node[below]{$\eta$} (-53)
	(11) edge node[left]{$\Ex_{\sharp\otimes}$} (21)
	(12) edge node[left]{$\Ex_{\sharp\otimes}$} (22)
	(34) edge node[right]{$\Ex_{\sharp\otimes}$} (44)
	(21) edge node[left]{$\epsilon$} (31)
	(32) edge node[left]{$\epsilon$} (42)
	(12) edge node[fill=white]{$\Ex^*_\sharp$} (23)
	(33) edge node[fill=white]{$\Ex^*_\sharp$} (44)
	(32) edge node[fill=white]{$\Ex^*_\sharp$} (43)
	(42) edge node[fill=white]{$\Ex^*_\sharp$} (53)
	(41) edge node[fill=white]{$\Ex^*_\sharp$} (52)
	(22) edge node[fill=white]{$\Ex^*_\otimes$} (33)
	(23) edge node[fill=white]{$\Ex^*_\otimes$} (34)
	(32-) edge node[above]{$\Ex_{\otimes*}$} (-33)
	(43-) edge node[below]{$\Ex_{\otimes*}$} (-44)
	(53-) edge node[below]{$\Ex_{\otimes*}$} (-54)
	(43) edge node[right]{$\Ex^*_\sharp$} (53)
	(44) edge node[right]{$\Ex^*_\sharp$} (54)
	;
\end{tikzmath}
Each face in this diagram commutes either by naturality or by a routine verification, using only triangle identities and the definitions of derived exchange transformations. 

The square (2) can be decomposed as follows (instances of $\delta^*\pi_1^*$, $\zeta^*\rho_1^*$, and $\theta^*\sigma_1^*$ in the target of unit maps have been erased):
\begin{tikzmath}
	\def\rowsep{1.5em}
	\def\colsep{1em}
	\diagram{
	\pi_{2\sharp}\pi_1^* q_\otimes e^* & \pi_{2\sharp}\delta_* q_\otimes e^* & & \pi_{2\sharp}\delta_* q_\otimes e^* \\
	& \pi_{2\sharp}\delta_*\delta^*r_\otimes\rho_1^*e^* & \pi_{2\sharp}r_\otimes\zeta_*e^* & \\
	\pi_{2\sharp} r_\otimes\rho_1^*e^* & & & \pi_{2\sharp} r_\otimes\theta''_*d^{\prime*}e^*  \\
	q_\otimes\rho_{2\sharp}\rho_1^*e^* & & q_\otimes\rho_{2\sharp}\theta''_*d^{\prime*}e^* & \pi_{2\sharp} r_\otimes d^*e^{\prime*}\theta_* \\
	q_\otimes\rho_{2\sharp}d^*e^{\prime*}\sigma_1^* & & q_\otimes\rho_{2\sharp}d^*e^{\prime*}\theta_* & \\
	& q_\otimes \sigma'_{2\sharp}e^{\prime*}\sigma_1^* & & q_\otimes \sigma'_{2\sharp}e^{\prime*}\theta_* \\
	q_\otimes e^*\sigma_{2\sharp}\sigma_1^* & & & q_\otimes e^*\sigma_{2\sharp}\theta_*\rlap. \\
	};
	\arrows 
	(11) edge node[left]{$\Ex^*_\otimes$} (31) (31) edge node[left]{$\Ex_{\sharp\otimes}$} (41) (41) edge node[left]{$\simeq$} (51) (51) edge node[left]{$\Ex^*_\sharp$} (71) (12) edge node[left]{$\Ex^*_\otimes$} (22) (14) edge[<-] node[right]{$\Dist_{\otimes*}$} (34) (34) edge[<-] node[right]{$\Ex^*_*$} (44) (43) edge[<-] node[left]{$\Ex^*_*$} (53) (44) edge node[right]{$\Dist_{\sharp\otimes}$} (64) (64) edge node[right]{$\Ex^*_\sharp$} (74)
	(11-) edge node[above]{$\eta$} (-12) (12-) edge node[above]{$\id$} (-14)
	(31-) edge node[below]{$\eta$} (-34) (51-) edge node[below]{$\eta$} (-53)
	(62-) edge node[below]{$\eta$} (-64) (71-) edge node[below]{$\eta$} (-74)
	(31) edge node[fill=white,inner sep=1]{$\eta$} (22) (22) edge node[fill=white,inner sep=1]{$\Ex^*_\otimes$} (14) 
	(31) edge node[fill=white,inner sep=1]{$\eta$} (23) (23) edge node[fill=white,inner sep=1]{$\Ex_{\otimes*}$} (14) (34) edge node[fill=white,inner sep=1]{$\eta$} (23)
	(41-) edge node[below]{$\eta$} (-43) (34) edge node[fill=white,inner sep=1]{$\Ex_{\sharp\otimes}$} (43) (44) edge node[fill=white,inner sep=1]{$\Ex_{\sharp\otimes}$} (53) (53) edge node[fill=white,inner sep=1]{$\epsilon$} (64) (51) edge node[fill=white,inner sep=1]{$\epsilon$} (62) (62) edge node[fill=white,inner sep=1]{$\Ex^*_\sharp$} (71)
	;
\end{tikzmath}
Finally, the commutativity of (3) follows from Proposition~\ref{prop:purity} (or trivially if $h$ is étale).
\end{proof}

Recall that $\alpha_f$ is an equivalence in $\SH$ if $f\colon Y\to X$ is smooth and proper (see \cite[\sect2.4.d]{CD}). If $X$ is moreover smooth over $S$, $\alpha_f^{-1}$ induces a transfer map
\[
\tau_f\colon \Sigma^\infty_+X \to \Th_Y(-\rm T_f)
\]
in $\SH(S)$.

\begin{corollary}
	\label{cor:transfers}
	Let $p\colon T\to S$ be finite étale and let $f\colon Y\to X$ be a smooth proper morphism in $\SmQP_T$. Then, under the identification of Corollary~\ref{cor:Nthom},
	\[
	p_\otimes(\tau_f)\simeq \tau_{\Weil_pf}\colon \Sigma^\infty_+ \Weil_pX \to \Th_{\Weil_pY}(-\rm T_{\Weil_pf})
	\]
	in $\SH(S)$.
\end{corollary}

\begin{proof}
	Using distributivity, we may assume that $X=T$, so that we are in the following situation:
	\begin{tikzmath}
		\diagram{ Y & \Weil_pY\times_ST & \Weil_pY \\ & T & S\rlap. \\};
		\arrows (11-) edge[<-] node[above]{$e$} (-12) (12-) edge node[above]{$q$} (-13) (11) edge node[below left]{$f$} (22) (12) edge node[right]{$v$} (22) (22-) edge node[above]{$p$} (-23) (13) edge node[right]{$u$} (23);
	\end{tikzmath}
	We must then show that the boundary of the following diagram commutes (evaluated on $\1_T$):
	\begin{tikzmath}
		\diagram{p_\otimes & p_\otimes f_*f^* & p_\otimes f_\sharp\Sigma^{-\rm T_f}f^* \\
		u_*u^*p_\otimes & u_*q_\otimes v^* & u_\sharp\Sigma^{-\rm T_u}q_\otimes v^*\rlap. \\};
		\arrows (11) edge node[left]{$\eta$} (21) (11-) edge node[above]{$\eta$} (-12) (12-) edge node[above]{$\alpha_f^{-1}$} (-13)
		(12) edge node[right]{$\Dist_{\otimes *}$} (22) (21-) edge node[below]{$\Ex_\otimes^*$} (-22)
		(22-) edge node[below]{$\alpha_u^{-1}$} (-23) (13) edge[<-] node[right]{$\Dist_{\sharp\otimes}$} (23);
	\end{tikzmath}
	The right square commutes by Proposition~\ref{prop:alpha}. 
	Unwinding the definition of $\Dist_{\otimes *}$, the commutativity of the left square follows from a triangle identity for the adjunction $v^*\dashv v_*$.
\end{proof}

\subsection{Polynomial functors}

We conclude this section with another application of the distributivity law.
The following definition is due to Eilenberg and Mac Lane in the case $\scr C=\scr D=\Ab$ \cite[\sect 9]{EilenbergMacLane}.
We learned it and its relationship with excisive functors (see Lemma~\ref{lem:poly-cubes}) from Akhil Mathew.

\begin{definition}\label{def:polynomial}
	Let $f\colon \scr C\to\scr D$ be a functor between pointed $\infty$-categories with finite colimits and let $n\in\Z$.
	We say that $f$ is \emph{polynomial of degree $\leq n$} if
	\begin{itemize}
		\item $n\leq -1$ and $f$ is the zero functor, or
		\item $n\geq 0$ and, for every $X\in \scr C$, the functor
		\[
		\rm D_X(f)\colon\scr C\to\scr D, \quad Y\mapsto \cofib(f(Y) \to f(X\vee Y)),
		\]
		is polynomial of degree $\leq n-1$.
	\end{itemize}
\end{definition}

Recall that a $n$-cube $(\Delta^1)^n\to\scr C$ is \emph{strongly cocartesian} if its $2$-dimensional faces are cocartesian squares \cite[Definition 6.1.1.2]{HA}. For $n\geq -1$, a functor $f\colon \scr C\to\scr D$ is \emph{$n$-excisive} if it sends strongly cocartesian $(n+1)$-cubes to cartesian cubes \cite[Definition 6.1.1.3]{HA}.
We shall say that a simplicial diagram $\Delta^\op\to\scr C$ is \emph{finite} if it is left Kan extended from the subcategory $\Delta_{\leq k}^\op$ for some $k$. The colimit of such a diagram exists whenever $\scr C$ has finite colimits, since it is the same as the colimit of its restriction to $\Delta_{\leq k}^\op$.

\begin{lemma}\label{lem:poly-cubes}
	Let $f\colon \scr C\to\scr D$ be a functor between pointed $\infty$-categories with finite colimits and let $n\geq -1$. 
	Consider the following assertions:
	\begin{enumerate}
		\item $f$ is polynomial of degree $\leq n$;
		\item $f$ sends strongly cocartesian $(n+1)$-cubes whose edges are summand inclusions to cubes with contractible total cofibers;
		\item $f$ sends strongly cocartesian $(n+1)$-cubes to cubes with contractible total cofibers;
		\item $f$ sends strongly cocartesian $(n+1)$-cubes to cocartesian cubes;
		\item $f$ is $n$-excisive.
	\end{enumerate}
	In general, $(4)\Rightarrow (3)\Rightarrow (2)\Leftrightarrow (1)$. If $f$ preserves finite simplicial colimits, then $(2)\Leftrightarrow (3)$. If $\scr D$ is prestable, then $(3)\Leftrightarrow (4)\Rightarrow (5)$. If $\scr D$ is stable, then $(4)\Leftrightarrow (5)$
\end{lemma}

\begin{proof}
	The implications $(4)\Rightarrow (3)\Rightarrow (2)$ are clear.
	Consider a strongly cocartesian $(n+1)$-cube $C$ in $\scr C$ whose edges are summand inclusions. Let $X$ be the initial vertex of $C$ and let $X\to X\vee X_i$ be the initial edges with $0\leq i\leq n$. Then the total cofiber of $f(C)$ is $\rm D_{X_0}\dots\rm D_{X_n}(f)(X)$. From this observation we immediately deduce that $(1)\Leftrightarrow (2)$.
	If $\scr D$ is prestable, then a morphism in $\scr D$ is an equivalence if and only if its cofiber is contractible \cite[Corollary C.1.2.5]{SAG}, hence a cube in $\scr D$ is cocartesian if and only if its total cofiber is contractible, which proves $(3)\Leftrightarrow (4)$. 
	Recall that a cube in a stable $\infty$-category is cocartesian if and only if it is cartesian \cite[Proposition 1.2.4.13]{HA}. This proves $(4)\Leftrightarrow (5)$ when $\scr D$ is stable. If $\scr D$ is prestable, it follows from \cite[Corollary C.1.2.3]{SAG} that cocartesian cubes in $\scr D$ are cartesian, hence $(4)\Rightarrow (5)$.
	Finally, suppose that $(2)$ holds and that $f$ preserves finite simplicial colimits. Let $C$ be an arbitrary strongly cocartesian $(n+1)$-cube with initial edges $X\to Y_i$. By Lemma~\ref{lem:bar-construction}, we can write $C$ as the simplicial colimit of the strongly cocartesian $(n+1)$-cubes $C_n$ with initial edges $X\to \Bar_X(X,Y_i)_n$. Since the edges of $C_n$ are summand inclusions, $f(C_n)$ has contractible total cofiber. By assumption, $f$ preserves the colimit of the simplicial object $\Bar_X(X,Y_i)_\bullet$ (which is left Kan extended from $\Delta_{\leq 1}^\op$). Hence, $f(C)$ is the colimit of $f(C_\bullet)$ and therefore has contractible total cofiber. This proves $(2)\Rightarrow (3)$.
\end{proof}

\begin{remark}\label{rmk:excisive}
Let $f\colon \scr C\to\scr D$ be a functor between pointed $\infty$-categories with finite colimits. We only expect the notion of polynomial functor from Definition~\ref{def:polynomial} to be well-behaved when $f$ preserves finite simplicial colimits and $\scr D$ is prestable. A better definition in general would be condition (4) from Lemma~\ref{lem:poly-cubes}.
\end{remark}

\begin{lemma}\label{lem:polynomial}
	Let $\scr C$, $\scr D$, and $\scr E$ be pointed $\infty$-categories with finite colimits.
	\begin{enumerate}
		\item If $f\colon \scr C\to \scr D$ is constant, then $f$ is polynomial of degree $\leq 0$.
		\item If $f\colon \scr C\to\scr D$ preserves binary sums, then $f$ is polynomial of degree $\leq 1$.
		\item If $f\colon \scr C\to\scr D$ and $g\colon\scr C\to\scr D$ are polynomial of degree $\leq n$, then $f\vee g\colon\scr C\to\scr D$ is polynomial of degree $\leq n$.
		\item If $f\colon \scr C\to\scr D$ preserves binary sums and $g\colon\scr D\to\scr E$ is polynomial of degree $\leq n$, then $g\circ f\colon\scr C\to\scr E$ is polynomial of degree $\leq n$.
		\item If $f\colon \scr C\to\scr D$ is polynomial of degree $\leq n$ and $g\colon\scr D\to\scr E$ preserves finite colimits, then $g\circ f\colon\scr C\to\scr E$ is polynomial of degree $\leq n$.
	\end{enumerate}
\end{lemma}

\begin{proof}
	(1) and (2) are clear. (3), (4), and (5) are proved by induction on $n$ using the formulas $\rm D_X(f\vee g)\simeq\rm D_X(f)\vee \rm D_X(g)$, $\rm D_X(g\circ f)\simeq\rm D_{f(X)}(g)\circ f$, and $\rm D_X(g\circ f)\simeq g\circ\rm D_X(f)$, respectively.
\end{proof}

The distributivity of norms over finite sums has the following consequence:

\begin{proposition}\label{prop:polynomial}
	Let $p\colon T\to S$ be finite locally free (resp.\ finite étale) of degree $\leq n$. Then the functor $p_\otimes\colon\H_\pt(T)\to \H_\pt(S)$ (resp.\ $p_\otimes\colon \SH(T)\to\SH(S)$) is polynomial of degree $\leq n$.
\end{proposition}

\begin{proof}
	We proceed by induction on $n$. If $n\leq -1$, then $S=\emptyset$ and the statement is trivial.
	If $n= 0$, then $p_\otimes$ is constant with value $\1_S$, hence polynomial of degree $\leq 0$.
	Suppose that $S=\coprod_i S_i$ and let $p_i\colon p^{-1}(S_i)\to S_i$ be the restriction of $p$. Then it is clear that $p_\otimes$ is polynomial of degree $\leq n$ if and only if $p_{i\otimes}$ is polynomial of degree $\leq n$ for all $i$. We may therefore assume that $p$ has constant degree $n\geq 1$. Let $X\in \H_\pt(T)$ (resp.\ $X\in\SH(T)$). Borrowing the notation from Corollary~\ref{cor:distrib-explicit}, we have
	\[
	\rm D_X(p_\otimes) \simeq p_\otimes(X) \vee c_\sharp(q_{l\otimes}e_l^*(X) \wedge q_{r\otimes}e_r^*(\ph)).
	\]
	Moreover, since $q_l$ and $q_r$ are both surjective onto $C$, $q_r$ has degree $\leq n-1$. By the induction hypothesis, $q_{r\otimes}$ is polynomial of degree $\leq n-1$. By Lemma~\ref{lem:polynomial}, we conclude that $\rm D_X(p_\otimes)$ is polynomial of degree $\leq n-1$, as desired.
\end{proof}

\begin{corollary}
	Let $p\colon T\to S$ be finite locally free (resp.\ finite étale) of degree $\leq n$. Then the functor $\Sigma^\infty p_\otimes\colon\H_\pt(T)\to \SH(S)$ (resp.\ $p_\otimes\colon \SH(T)\to\SH(S)$) is $n$-excisive.
\end{corollary}

\begin{proof}
	Since these functors preserve sifted colimits, this follows from Proposition~\ref{prop:polynomial} and Lemma~\ref{lem:poly-cubes}.
\end{proof}

\section{Coherence of norms}
\label{sec:coherence}

\subsection{Functoriality on the category of spans}
\label{sub:coherence}

Our goal in this subsection is to construct the functor
\[
\SH^\otimes\colon \Span(\Sch,\all,\fet) \to \CAlg(\what\Cat{}_\infty),\quad S\mapsto \SH(S),\quad
(U\stackrel f\from T\stackrel p\to S) \mapsto p_\otimes f^*.
\]
Here, $\Span(\Sch,\all,\fet)$ is the $2$-category whose objects are schemes, whose $1$-morphisms are spans $U\stackrel f\from T\stackrel p\to S$ with $p$ finite étale, and whose $2$-morphisms are isomorphisms of spans (see Appendix~\ref{app:spans}).
The functor $\SH^\otimes$ encodes several features of norms:
\begin{itemize}
	\item if $p$ and $q$ are composable finite étale maps, then $(qp)_\otimes\simeq q_\otimes p_\otimes$ (cf.\ Proposition~\ref{prop:composition});
	\item given a cartesian square of schemes
\begin{tikzmath}
	\diagram{\bullet & \bullet \\ \bullet & \bullet \\};
	\arrows (11-) edge node[above]{$g$} (-12) (11) edge node[left]{$q$} (21) (21-) edge node[below]{$f$} (-22) (12) edge node[right]{$p$} (22);
\end{tikzmath}
with $p$ finite étale, $f^*p_\otimes\simeq q_\otimes g^*$ (cf.\ Proposition~\ref{prop:BC});
\item coherence of the above equivalences.
\end{itemize}

The strategy, which we will use several times in the course of this paper, is to decompose the construction of $\SH(S)$ into the steps
\[
\SmQP_{S+} \rightsquigarrow \PSh_\Sigma(\SmQP_{S})_\pt \rightsquigarrow \H_\pt(S) \rightsquigarrow \SH(S),
\]
each of which has a universal property in a suitable $\infty$-category of $\infty$-categories that contains the norm functors.
We start with $\SmQP_S$ instead of $\Sm_S$ because of the following fact (see Lemmas~\ref{lem:Weil} and \ref{lem:Weil-smooth}): if $p\colon T\to S$ is finite locally free and $X$ is a smooth quasi-projective $T$-scheme, its Weil restriction $\Weil_pX$ is smooth and quasi-projective over $S$. In other words, the functor $p_*\colon \PSh(\SmQP_T)\to \PSh(\SmQP_S)$ sends $\SmQP_T$ to $\SmQP_S$, and hence the norm $p_\otimes$ sends $\SmQP_{T+}$ to $\SmQP_{S+}$. Note that this restriction to smooth quasi-projective $S$-schemes is harmless: as every smooth $S$-scheme admits an open cover by quasi-projective $S$-schemes, the inclusion $\SmQP_S\subset\Sm_S$ induces an equivalence
\[
\Shv_\Nis(\SmQP_S) \simeq \Shv_\Nis(\Sm_S).
\]
In particular, $\H_\pt(S)$ is a localization of $\PSh_\Sigma(\SmQP_{S+})$.

To begin with, we construct a functor of $2$-categories
\begin{equation*}\label{eqn:SmQP+}
\SmQP_+^\otimes\colon \Span(\Sch,\all,\flf) \to \CAlg(\Cat_1),\quad S\mapsto \SmQP_{S+}, \quad (U\stackrel f\from T\stackrel p\to S)\mapsto p_\otimes f^*,
\end{equation*}
where ``$\flf$'' denotes the class of finite locally free morphisms.
This is reasonably easy to do by hand, but we can proceed more cogently as follows. 
If $p\colon T\to S$ is finite locally free, the Weil restriction functor $\Weil_p$ is by definition right adjoint to the pullback functor $p^*\colon \SmQP_S\to\SmQP_S$. Using Barwick's unfurling construction \cite[Proposition 11.6]{BarwickMackey}, we obtain a functor
\[
\Span(\Sch,\all,\flf) \to \Cat_1,\quad S\mapsto \SmQP_{S}, \quad (U\stackrel f\from T\stackrel p\to S)\mapsto \Weil_p f^*.
\]
There is an obvious equivalence of categories $\SmQP_{S+}\simeq \Span(\SmQP_S,\mathrm{clopen},\all)$, where ``clopen'' is the class of clopen immersions. 
Both $f^*$ and $\Weil_p$ preserve clopen immersions (see \cite[\sect7.6, Proposition 2]{NeronModels} for the latter). Using the functoriality of $2$-categories of spans described in Proposition~\ref{prop:spans-extended-functoriality}, we obtain the functor $\SmQP_+^\otimes\colon \Span(\Sch,\all,\flf) \to \Cat_1$, which lifts uniquely to $\CAlg(\Cat_1)$ by Proposition~\ref{prop:automatic-calg}.

\begin{remark}\label{rmk:fet-norm-uniqueness}
	The restriction of $\SmQP_+^\otimes$ to $\Span(\Sch,\all,\fet)$ is in fact the \emph{unique} extension of the functor $\Sch^\op\to\CAlg(\Cat_1)$, $S\mapsto \SmQP_{S+}$. This follows from the fact that the latter functor is a sheaf for the finite étale topology (see Lemma~\ref{lem:corr-etale-descent}) and Corollary~\ref{cor:automatic-norms}.
\end{remark}

We then invoke the basic fact that every functor $F\colon\scr C\to\scr D$ has a partial left adjoint defined on the full subcategory of $\scr D$ spanned by the objects $d$ such that $\Map(d,F(\ph))$ is corepresentable \cite[Lemma 5.2.4.1]{HTT}.
For example, there is a functor\footnote{The notation $\PSh_\Sigma$ is slightly abusive, since $\PSh_\Sigma(\scr C)$ is only defined when $\scr C$ admits finite coproducts.}
\[
\PSh_\Sigma\colon\CAlg(\Cat_\infty) \to \CAlg(\what\Cat{}_\infty^\mathrm{sift})
\]
that sends a symmetric monoidal $\infty$-category $\scr C$ to its sifted cocompletion equipped with the Day convolution symmetric monoidal structure, this being a left adjoint to the forgetful functor
\[
\CAlg(\what\Cat{}_\infty^\mathrm{sift})\to \CAlg(\what\Cat{}_\infty)
\]
\cite[Proposition 4.8.1.10]{HA}. Here, $\what{\Cat}{}_\infty^\mathrm{sift}$ is the $\infty$-category of sifted-cocomplete $\infty$-categories and sifted-colimit-preserving functors between them, equipped with the cartesian symmetric monoidal structure.
 If we compose $\SmQP_+^\otimes$ with the inclusion $\CAlg(\Cat_1)\subset \CAlg(\Cat_\infty)$ and with $\PSh_\Sigma$, we obtain the functor
\begin{equation}\label{eqn:PSigma}
\PSh_\Sigma(\SmQP)_\pt^\otimes\colon \Span(\Sch,\all,\flf) \to \CAlg(\what\Cat{}_\infty^\mathrm{sift}), \quad S\mapsto \PSh_\Sigma(\SmQP_S)_\pt.
\end{equation}

Next we introduce two auxiliary $\infty$-categories:
\begin{itemize}
	\item $\scr O\Cat_\infty$ is the $\infty$-category of $\infty$-categories equipped with a collection of equivalence classes of objects;
	\item $\scr M\Cat_\infty$ is the $\infty$-category of $\infty$-categories equipped with a collection of equivalence classes of arrows.
\end{itemize}
More precisely, $\scr O\Cat_\infty$ and $\scr M\Cat_\infty$ are defined by the cartesian squares
\begin{tikzmath}
	\diagram{\scr M\Cat_\infty & \scr O\Cat_\infty & \scr E \\ \Cat_\infty & \Cat_\infty & \Pos\rlap, \\};
	\arrows (11-) edge (-12) (11) edge (21) (21-) edge node[above]{$(\ph)^{\Delta^1}$} (-22) (12) edge (22) (12-) edge (-13) (22-) edge (-23) (13) edge (23);
\end{tikzmath}
where $\scr E\to\Pos$ is the universal cocartesian fibration in posets (i.e., the cocartesian fibration classified by the inclusion $\Pos=\Cat_{0}\subset\Cat_\infty$) and the second bottom arrow sends an $\infty$-category to the poset of subsets of its set of equivalence classes of objects.

Since $f^*$ and $p_\otimes$ preserve motivic equivalences and the latter are stable under the smash product, we can lift~\eqref{eqn:PSigma} to a functor
\[
\Span(\Sch,\all,\flf) \to \CAlg(\what{\scr M\Cat}{}_\infty^\mathrm{sift}), \quad S\mapsto (\PSh_\Sigma(\SmQP_S)_\pt, \text{motivic equivalences}).
\]
By the universal property of localization of symmetric monoidal $\infty$-categories (\cite[Proposition 5.2.7.12]{HTT} and \cite[Proposition 4.1.7.4]{HA}), $\H_\pt(S)$ is the image of $(\PSh_\Sigma(\SmQP_S)_\pt,\<\text{motivic equivalences})$ by a partial left adjoint to the functor
\[
\CAlg(\what\Cat{}_\infty^\mathrm{sift}) \to \CAlg(\what{\scr M\Cat}{}_\infty^\mathrm{sift}), \quad \scr C\mapsto (\scr C,\text{equivalences}).
\]
We therefore obtain
\begin{equation*}
\H_\pt^\otimes \colon \Span(\Sch,\all,\flf) \to \CAlg(\what\Cat{}_\infty^\mathrm{sift}),\quad S\mapsto \H_\pt(S).
\end{equation*}
The functor $f^*$ sends a motivic sphere $\S^V$ to $\S^{f^*V}$ and, if $p$ is finite étale, $p_\otimes$ sends $\S^V$ to $\S^{\Weil_pV}$ (Lemma~\ref{lem:spheres}). Hence, we get a functor
\[
\Span(\Sch,\all,\fet) \to \CAlg(\what{\scr O\Cat}{}_\infty^\mathrm{sift}),\quad S\mapsto (\H_\pt(S), \{\S^V\}_{V/S}),
\]
By Lemma~\ref{lem:inversion}, $\SH(S)$ is the image of $(\H_\pt(S),\{\S^V\}_{V/S})$ by a partial left adjoint to the functor
\[
\CAlg(\what{\Cat}{}_\infty^\mathrm{sift}) \to\CAlg(\what{\scr O\Cat}{}_\infty^\mathrm{sift}), \quad \scr C\mapsto (\scr C,\pi_0\Pic(\scr C)).
\]
Hence, we finally obtain
\begin{equation*}
\SH^\otimes\colon \Span(\Sch,\all,\fet) \to \CAlg(\what\Cat{}_\infty^\mathrm{sift}),\quad S\mapsto \SH(S).
\end{equation*}

\begin{remark}\label{rmk:H->SH}
	Using the same strategy, starting with $\SmQP$ instead of $\SmQP_+$, we can define
	\[
	\H^\otimes\colon \Span(\Sch,\all,\flf) \to \CAlg(\what\Cat{}_\infty^\mathrm{sift}),\quad S\mapsto \H(S), 
	\quad (U\stackrel f\from T\stackrel p\to S) \mapsto p_* f^*.
	\]
	The three functors $\H^\otimes$, $\H_\pt^\otimes$, and $\SH^\otimes$ are related by natural transformations
	\[
	\H^\otimes \xrightarrow{(\ph)_+} \H_\pt^\otimes \xrightarrow{\Sigma^\infty} \SH^\otimes,
	\]
	the latter being defined on $\Span(\Sch,\all,\fet)$. 
	These natural transformations can be defined in the same way as the functors themselves, replacing $\Span(\Sch,\all,\flf/\fet)$ with $\Span(\Sch,\all,\flf/\fet)\times\Delta^1$. For example, the last step in the construction of $\Sigma^\infty$ uses the functor
	\[
	\Span(\Sch,\all,\fet)\times\Delta^1 \to \CAlg(\what{\scr O\Cat}{}_\infty^\mathrm{sift}),\quad (S,0\to 1)\mapsto ((\H_\pt(S), \{\1_S\})\to(\H_\pt(S), \{\S^V\}_{V/S})).
	\]
\end{remark}

\begin{remark}\label{rmk:symmmon}
	If we forget that $\SH^\otimes$ takes values in symmetric monoidal $\infty$-categories, we can recover the symmetric monoidal structure of $\SH(S)$ as the composition
	\[
	\Fin_\pt \simeq \Span(\Fin,\inj,\all) \into \Span(\Fin) \to \Span(\FEt_S) \to \Span(\Sch,\all,\fet) \xrightarrow{\SH^\otimes} \what\Cat{}_\infty,
	\]
	where the middle functor is induced by $\Fin\to \FEt_S$, $X\mapsto \coprod_X S$. 
\end{remark}

\subsection{Normed \texorpdfstring{$\infty$}{∞}-categories}
\label{sub:normed-categories}

Motivated by the example of $\SH^\otimes$, we introduce the notion of a \emph{normed $\infty$-category}: it is a generalization of a symmetric monoidal $\infty$-category that includes norms along finite étale maps. We will encounter several other examples in the sequel. The main results of this subsection are criteria for verifying that norms preserve certain subcategories (Proposition~\ref{prop:normed-subcategory}) or are compatible with certain localizations (Proposition~\ref{prop:normed-localization}).

Let $S$ be a scheme.
We will write $\scr C\subset_\fet\Sch_S$ if $\scr C$ is a full subcategory of $\Sch_S$ that contains $S$ and is closed under finite coproducts and finite étale extensions. Under these assumptions, we can form the $2$-category of spans $\Span(\scr C,\all,\fet)$, and the functor $\scr C^\op \into \Span(\scr C,\all,\fet)$ preserves finite products (Lemma~\ref{lem:Span-semiadditive}). 

\begin{definition}\label{dfn:normed-category}
	Let $S$ be a scheme and $\scr C\subset_\fet\Sch_S$. A \emph{normed $\infty$-category} over $\scr C$ is a functor
	\[
	\scr A\colon \Span(\scr C,\all,\fet) \to \Cat_\infty, \quad (X\stackrel f\leftarrow Y\stackrel p\rightarrow Z) \mapsto p_\otimes f^*,
	\]
	that preserves finite products. We say that $\scr A$ is \emph{presentably normed} if:
	\begin{enumerate}
		\item for every $X\in\scr C$, $\scr A(X)$ is presentable;
		\item for every finite étale morphism $h\colon Y\to X$, $h^*\colon \scr A(X)\to\scr A(Y)$ has a left adjoint $h_\sharp$;
		\item for every morphism $f\colon Y\to X$, $f^*\colon \scr A(X)\to\scr A(Y)$ preserves colimits;
		\item for every cartesian square
		\begin{tikzmath}
			\diagram{Y' & Y \\ X' & X \\};
			\arrows (11-) edge node[above]{$g$} (-12) (11) edge node[left]{$h'$} (21) (21-) edge node[below]{$f$} (-22) (12) edge node[right]{$h$} (22);
		\end{tikzmath}
		 with $h$ finite étale, the exchange transformation
		\[
		\Ex_\sharp^*\colon h'_\sharp g^* \to f^* h_\sharp\colon \scr A(Y)\to \scr A(X')
		\]
		is an equivalence;
		\item for every finite étale map $p\colon Y\to Z$, $p_\otimes\colon \scr A(Y)\to \scr A(Z)$ preserves sifted colimits;
		\item for every diagram
		\begin{tikzmath}
			\diagram{ U & \Weil_pU\times_ZY & \Weil_pU \\ & Y & Z\rlap, \\};
			\arrows (11-) edge[<-] node[above]{$e$} (-12) (12-) edge node[above]{$q$} (-13) (11) edge node[below left]{$h$} (22) (12) edge node[right]{$g$} (22) (22-) edge node[above]{$p$} (-23) (13) edge node[right]{$f$} (23);
		\end{tikzmath}
		 with $p$ and $h$ finite étale, the distributivity transformation
		\[
		\Dist_{\sharp\otimes}\colon f_\sharp q_\otimes e^* \to p_\otimes h_\sharp\colon \scr A(U)\to \scr A(Z)
		\]
		is an equivalence. 
	\end{enumerate}
	The notion of a \emph{nonunital} (presentably) normed $\infty$-category over $\scr C$ is defined in the same way, with the class of finite étale maps replaced by that of \emph{surjective} finite étale maps (and the additional assumption that $p$ is surjective in (5) and (6)).
\end{definition}

	By Proposition~\ref{prop:automatic-calg}, a normed $\infty$-category $\scr A\colon \Span(\scr C,\all,\fet)\to\Cat_\infty$ lifts uniquely to $\CAlg(\Cat_\infty)$. Explicitly, the $n$-ary tensor product in $\scr A(X)$ is
	\[
	\scr A(X)^{\times n} \simeq \scr A(X^{\amalg n}) \xrightarrow{\nabla_\otimes} \scr A(X), 
	\]
	where $\nabla\colon X^{\amalg n}\to X$ is the fold map. Moreover, if $\scr A$ is presentably normed, then each $\scr A(X)$ is presentably symmetric monoidal, i.e., its tensor product distributes over colimits. Indeed, condition (5) implies that $\otimes$ distributes over sifted colimits, and condition (6) applied to the diagram of fold maps
	\begin{tikzmath}
		\diagram{ X\amalg X^{\amalg n} & X^{\amalg n} \amalg X^{\amalg n} & X^{\amalg n} \\ & X\amalg X & X \\};
		\arrows (11-) edge[<-] (-12) (12-) edge (-13) (11) edge (22) (12) edge (22) (22-) edge (-23) (13) edge (23);
	\end{tikzmath}
	shows that $\otimes$ distributes over finite coproducts.

\begin{remark}
	One may regard the functors $h_\sharp$ from condition (2) as a new type of colimits, generalizing finite coproducts. With this interpretation, conditions (3) and (4) state that base change preserves colimits, and conditions (5) and (6) state that norms distribute over colimits.
\end{remark}

\begin{example}
	The functors
	\[
	\SH^\otimes,\; \SH^\otimes\circ \widehat\Pi_1^\et,\; \DM^\otimes,\; \SH^\otimes_\nc\colon \Span(\Sch,\all,\fet)\to\Cat_\infty
	\]
	constructed in \sect\ref{sub:coherence}, \sect\ref{sub:galois}, \sect\ref{sub:PST}, and \sect\ref{sub:nc-motives} are all presentably normed $\infty$-categories. In fact, for $\SH^\otimes$, $\DM^\otimes$, and $\SH_\nc^\otimes$, conditions (2), (4), and (6) of Definition~\ref{dfn:normed-category} hold for $h$ any smooth morphism (with a quasi-projectivity assumption in (6)).
\end{example}

\begin{example}
	For $n\geq 1$, $n$-effective and very $n$-effective motivic spectra form nonunital normed subcategories of $\SH^\otimes$ (see \sect\ref{sub:zero-slice}).
\end{example}

\begin{remark}\label{rmk:unital-subcategory}
	Let $\scr A$ be a normed $\infty$-category over $\scr C\subset_\fet\Sch_S$ and let $\scr B\subset\scr A$ be a nonunital normed subcategory. Then $\scr B$ is a normed subcategory of $\scr A$ if and only if $\1_S\in\scr B(S)$.
\end{remark}

\begin{remark}\label{rmk:polynomial}
	Let $\scr A$ be a presentably normed $\infty$-category over $\scr C\subset_\fet\Sch_S$ such that each $\scr A(X)$ is pointed, and let $p\colon Y\to Z$ be a finite étale map in $\scr C$. The proof of Proposition~\ref{prop:polynomial} shows that, if $p$ has degree $\leq n$, then $p_\otimes\colon \scr A(Y)\to\scr A(Z)$ is polynomial of degree $\leq n$.
\end{remark}

\begin{lemma}\label{lem:polynomial-extensions}
	Let $\scr C$ be a pointed $\infty$-category with finite colimits, $\scr C'\subset\scr C$ a full subcategory closed under binary sums, $\scr D$ a stable $\infty$-category, and $\scr D_{\geq 0}\subset \scr D$ the nonnegative part of a $t$-structure on $\scr D$. Let $f\colon\scr C\to\scr D$ be a functor such that:
	\begin{enumerate}
		\item $f$ is polynomial of degree $\leq n$ for some $n$;
		\item $f$ preserves simplicial colimits;
		\item $f(\scr C')\subset \scr D_{\geq 0}$.
	\end{enumerate}
	If $A\to B\to C$ is a cofiber sequence in $\scr C$ with $A,C\in\scr C'$, then $f(B)\in\scr D_{\geq 0}$.
\end{lemma}

\begin{proof}
	We proceed by induction on $n$. If $n=-1$, then $f$ is the zero functor and the result is trivial.
	By Lemma~\ref{lem:bar-construction}, we have a simplicial colimit diagram
	\begin{tikzmath}
		\def\colsep{.9em}
		\diagram{
		\dotsb & A\vee A\vee B & A\vee B & B & C\rlap. \\
		};
		\arrows
		(11-) edge[-top,vshift=3*\dbl] (-12) edge[-mid,vshift=\dbl] (-12) edge[-mid,vshift=-\dbl] (-12) edge[-bot,vshift=-3*\dbl] (-12)
		(12-) edge[-top,vshift=2*\dbl] (-13) edge[-mid] (-13) edge[-bot,vshift=-2*\dbl] (-13)
		(13-) edge[-top,vshift=\dbl] (-14) edge[-bot,vshift=-\dbl] (-14) (14-) edge (-15);
	\end{tikzmath}
	Since $f$ preserves simplicial colimits, there is an induced colimit diagram
	\begin{tikzequation}\label{eqn:extension}
		\def\colsep{.9em}
		\diagram{
		\dotsb & \rm D_{A\vee A}(f)(B)\vee f(B) & \rm D_A(f)(B)\vee f(B) & f(B) & f(C)\rlap. \\
		};
		\arrows
		(11-) edge[-top,vshift=3*\dbl] (-12) edge[-mid,vshift=\dbl] (-12) edge[-mid,vshift=-\dbl] (-12) edge[-bot,vshift=-3*\dbl] (-12)
		(12-) edge[-top,vshift=2*\dbl] (-13) edge[-mid] (-13) edge[-bot,vshift=-2*\dbl] (-13)
		(13-) edge[-top,vshift=\dbl] (-14) edge[-bot,vshift=-\dbl] (-14) (14-) edge (-15);
	\end{tikzequation}
	For any $X\in \scr C'$, the functor $\rm D_X(f)\colon\scr C\to\scr D$ is polynomial of degree $\leq n-1$ and satisfies conditions (2) and (3), hence $\rm D_X(f)(B)\in \scr D_{\geq 0}$ by the induction hypothesis. Applying the truncation functor $\tau_{<0}$ to~\eqref{eqn:extension}, we obtain a colimit diagram in $\scr D_{<0}$:
	\begin{tikzmath}
		\def\colsep{.9em}
		\diagram{
		\dotsb & \tau_{<0}f(B) & \tau_{<0}f(B) & \tau_{<0}f(B) & 0\rlap. \\
		};
		\arrows
		(11-) edge[-top,vshift=3*\dbl] (-12) edge[-mid,vshift=\dbl] (-12) edge[-mid,vshift=-\dbl] (-12) edge[-bot,vshift=-3*\dbl] (-12)
		(12-) edge[-top,vshift=2*\dbl] (-13) edge[-mid] (-13) edge[-bot,vshift=-2*\dbl] (-13)
		(13-) edge[-top,vshift=\dbl] (-14) edge[-bot,vshift=-\dbl] (-14) (14-) edge (-15);
	\end{tikzmath}
	It follows that $\tau_{<0}f(B)=0$, i.e., $f(B)\in\scr D_{\geq 0}$.
\end{proof}

\begin{proposition}\label{prop:normed-subcategory}
	Let $\scr C\subset_\fet\Sch_S$ and let $\scr A\colon \Span(\scr C,\all,\fet)\to\Cat_\infty$ be a presentably normed $\infty$-category (resp.\ a presentably normed $\infty$-category such that $\scr A(X)$ is stable for every $X\in\scr C$).
	For each $X\in\scr C$, let $\scr B_0(X)\subset \scr A(X)$ be a collection of objects, and let $\scr B(X)\subset\scr A(X)$ be the full subcategory generated by $\scr B_0(X)$ under colimits (resp.\ under colimits and extensions).
	Suppose that the following conditions hold:
	\begin{enumerate}
		\item For every morphism $f\colon Y\to X$, we have $f^*(\scr B_0(X)) \subset \scr B(Y)$.
		\item For every finite étale morphism $h\colon Y\to X$, we have $h_\sharp(\scr B_0(Y))\subset \scr B(X)$.
		\item For every surjective finite étale morphism $p\colon Y\to Z$, we have $p_\otimes(\scr B_0(Y))\subset \scr B(Z)$.
		\item For every $X\in\scr C$ and $A,B\in\scr B_0(X)$, we have $A\otimes B\in \scr B(X)$.
	\end{enumerate}
	Then $\scr B$ is a nonunital normed subcategory of $\scr A$.
	Moreover, if $\1_S\in\scr B(S)$, then $\scr B$ is a normed subcategory of $\scr A$.
\end{proposition}

\begin{proof}
	The last statement follows from Remark~\ref{rmk:unital-subcategory}.
	Since $f^*$ and $h_\sharp$ preserve colimits (and hence extensions), we deduce from (1) and (2) that $f^*(\scr B(X))\subset \scr B(Y)$ and $h_\sharp(\scr B(Y))\subset \scr B(X)$. 
	Given $X_1,\dotsc,X_n\in\scr C$, the equivalence
	\[
	\scr A(X_1\amalg \dotsb\amalg X_n) \to \scr A(X_1)\times\dotsb\times\scr A(X_n)
	\]
	has inverse $(A_1,\dotsc,A_n)\mapsto i_{1\sharp}A_1\amalg \dotsb \amalg i_{n\sharp}A_n$, and it follows that $\scr B\colon \scr C^\op\to\Cat_\infty$ preserves finite products. 
	Since the tensor product in $\scr A(X)$ preserves colimits in each variable, condition (4) implies that $\scr B(X)$ is closed under binary tensor products.
	
	Let $p\colon Y\to Z$ be a surjective finite étale morphism. We need to show that $p_\otimes(\scr B(Y))\subset \scr B(Z)$.
	Since $Z$ is quasi-compact, $p$ has degree $\leq n$ for some integer $n$, and we proceed by induction on $n$. If $n=1$, then $p$ is the identity and the result is trivial. By (3), $(p_\otimes)^{-1}(\scr B(Z))$ contains $\scr B_0(Y)$. We claim that $(p_\otimes)^{-1}(\scr B(Z))$ is closed under colimits in $\scr A(Y)$. For sifted colimits, this follows from the assumption that $p_\otimes$ preserves sifted colimits. Since $p$ is surjective, $\Weil_p(\emptyset)$ is empty and the distributivity law with $h\colon \emptyset\to Y$ implies that $p_\otimes(\emptyset)\simeq\emptyset$, so that $\emptyset \in (p_\otimes)^{-1}(\scr B(Z))$.
	It remains to show that $(p_\otimes)^{-1}(\scr B(Z))$ is closed under binary sums. This follows from the explicit form of the distributivity law (Corollary~\ref{cor:distrib-explicit}) and the inductive hypothesis (note that $q_l$ and $q_r$ have degree $<n$, since they are both surjective), using what we have already established about $f^*$, $h_\sharp$, and binary tensor products.
	
	In case each $\scr A(X)$ is stable and $\scr B(X)$ is the closure of $\scr B_0(X)$ under colimits and extensions, we further need to show that $(p_\otimes)^{-1}(\scr B(Z))$ is closed under extensions. 
	Given a cofiber sequence $A\to B\to C$ with $A,C\in (p_\otimes)^{-1}(\scr B(Z))$, let $\scr D\subset (p_\otimes)^{-1}(\scr B(Z))$ be a small subcategory containing $A$ and $C$ and closed under binary sums, and let $\scr E\subset \scr B(Z)$ be the full subcategory generated under colimits and extensions by $p_\otimes(\scr D)$.
	By \cite[Proposition 1.4.4.11]{HA}, $\scr E$ is the nonnegative part of a $t$-structure on $\scr A(Z)$. Since $p_\otimes\colon \scr A(Y)\to \scr A(Z)$ is a polynomial functor of degree $\leq n$ by Remark~\ref{rmk:polynomial}, Lemma~\ref{lem:polynomial-extensions} implies that $p_\otimes(B)\in \scr E\subset \scr B(Z)$, as desired.
\end{proof}

We now discuss localizations of normed $\infty$-categories. If $L\colon \scr A\to\scr A$ is a localization functor \cite[\sect 5.2.7]{HTT}, we say that a morphism $f$ in $\scr A$ is an \emph{$L$-equivalence} if $L(f)$ is an equivalence.

\begin{definition}
	Let $\scr C\subset_\fet\Sch_S$ and let $\scr A\colon \Span(\scr C,\all,\fet)\to\Cat_\infty$ be a normed $\infty$-category over $\scr C$. A family of localization functors $L_X\colon \scr A(X)\to\scr A(X)$ for $X\in\scr C$ is called \emph{compatible with norms} if $L$-equivalences form a normed subcategory of $\Fun(\Delta^1,\scr A)$, or equivalently if:
	\begin{enumerate}
		\item the functors $f^*$ and $p_\otimes$ preserve $L$-equivalences;
		\item under the equivalence $\scr A(X\amalg Y)\simeq \scr A(X)\times\scr A(Y)$, we have $L_{X\amalg Y}=L_X\times L_Y$.
	\end{enumerate}
\end{definition}

\begin{proposition}\label{prop:compatible-with-norms}
	Let $\scr C\subset_\fet\Sch_S$, let $\scr A\colon \Span(\scr C,\all,\fet)\to\Cat_\infty$ be a normed $\infty$-category over $\scr C$, and let $L_X\colon \scr A(X)\to\scr A(X)$, $X\in\scr C$, be a family of localization functors that is compatible with norms. Then the subcategories $L_X\scr A(X)\subset \scr A(X)$ assemble into a normed $\infty$-category $L\scr A$ over $\scr C$ and the functors $L_X$ assemble into a natural transformation
	\[
	L\colon \scr A\to L\scr A\colon \Span(\scr C,\all,\fet)\to\Cat_\infty.
	\]
\end{proposition}

\begin{proof}
	The existence of the functor $L\scr A\colon \Span(\scr C,\all,\fet)\to\Cat_\infty$ and the natural transformation $L$ follows from Proposition~\ref{prop:localization-abstract}.
	The fact that $L_{X\amalg Y}=L_X\times L_Y$ implies that $L\scr A$ preserves finite products, so that it is a normed $\infty$-category.
\end{proof}

\begin{proposition}\label{prop:normed-localization}
	Let $\scr C\subset_\fet\Sch_S$ and let $\scr A\colon \Span(\scr C,\all,\fet)\to\Cat_\infty$ be a presentably normed $\infty$-category. For each $X\in\scr C$, let $W(X)\subset \Fun(\Delta^1,\scr A(X))$ be a class of morphisms, and let $\overline W(X)$ be the strong saturation of $W(X)$. Suppose that the following conditions hold:
	\begin{enumerate}
		\item For every morphism $f\colon Y\to X$, we have $f^*(W(X)) \subset \overline W(Y)$.
		\item For every finite étale morphism $h\colon Y\to X$, we have $h_\sharp(W(Y))\subset \overline W(X)$.
		\item For every surjective finite étale morphism $p\colon Y\to Z$, we have $p_\otimes(W(Y))\subset \overline W(Z)$.
		\item For every $X\in\scr C$ and $A\in\scr A(X)$, we have $\id_A\otimes W(X)\subset \overline W(X)$.
	\end{enumerate}
	  Then $\overline W$ is a normed subcategory of $\Fun(\Delta^1,\scr A)$.
\end{proposition}

\begin{proof}
	The proof is exactly the same as that of Proposition~\ref{prop:normed-subcategory}, using that $\overline W(X)$ is generated under 2-out-of-3 and colimits by $W(X)$ and $\id_A$ for $A\in\scr A(X)$ (Lemma~\ref{lem:strong-saturation}).
\end{proof}

If $\scr A$ is a presentable $\infty$-category and $\scr B\subset \scr A$ is an accessible full subcategory closed under colimits,
 we can form the cofiber sequence
\[
\scr B\into \scr A \to \scr A/\scr B
\]
in $\Pr^\L$. The functor $\scr A\to\scr A/\scr B$ has a fully faithful right adjoint identifying $\scr A/\scr B$ with the subcategory of objects $A\in\scr A$ such that $\Map(B,A)\simeq *$ for all $B\in\scr B$. We denote by $\L^0_\scr B\colon \scr A\to\scr A$ the corresponding localization functor.

\begin{definition}\label{def:normed-ideal}
	Let $\scr A$ be a nonunital normed $\infty$-category over $\scr C\subset_\fet\Sch_S$. A \emph{normed ideal} in $\scr A$ is a nonunital normed full subcategory $\scr B\subset\scr A$ such that, for every $X\in\scr C$, $\scr B(X)\subset\scr A(X)$ is a tensor ideal (i.e., for every $A\in\scr A(X)$ and $B\in\scr B(X)$, we have $A\otimes B\in \scr B(X)$).
\end{definition}

\begin{corollary}\label{cor:normed-nullification}
	Let $\scr C\subset_\fet\Sch_S$, let $\scr A\colon \Span(\scr C,\all,\fet)\to\Cat_\infty$ be a presentably normed $\infty$-category, and let $\scr B\subset \scr A$ be a normed ideal. Suppose that:
	\begin{enumerate}
		\item for every $X\in\scr C$, $\scr B(X)$ is accessible and closed under colimits in $\scr A(X)$;
		\item for every finite étale map $h\colon Y\to X$, we have $h_\sharp(\scr B(Y))\subset \scr B(X)$.
	\end{enumerate}
	Then the localization functors $\L^0_{\scr B(X)}$ are compatible with norms. Consequently, they assemble into a natural transformation of normed $\infty$-categories
	\[
	\L^0_{\scr B} \colon \scr A \to \L^0_{\scr B}\scr A \colon \Span(\scr C,\all,\fet)\to\Cat_\infty.
	\]
\end{corollary}

\begin{proof}
	For $X\in\scr C$, let $W(X)\subset \Fun(\Delta^1,\scr A(X))$ be the class of morphisms of the form $\emptyset \to B$ with $B\in\scr B(X)$.
	Then the strong saturation of $W(X)$ is the class of $\L^0_{\scr B(X)}$-equivalences. The corollary now follows from Propositions~\ref{prop:normed-localization} and~\ref{prop:compatible-with-norms}.
\end{proof}

\section{Normed motivic spectra}
\label{sec:normedspectra}

\subsection{Categories of normed spectra}

If $\scr A\colon\scr C\to\Cat_\infty$ is a functor classifying a cocartesian fibration $p\colon\scr E\to\scr C$, a \emph{section} of $\scr A$ will mean a section $s\colon\scr C\to\scr E$ of $p$. Thus, for every $c\in\scr C$, $s(c)$ is an object of $\scr A(c)$, and for every morphism $f\colon c\to c'$ in $\scr C$, $s(f)$ is a morphism $\scr A(f)(s(c))\to s(c')$ in $\scr A(c')$.
 We will write
\[
	\int\scr A = \scr E \quad\text{and}\quad \Sect(\scr A) = \Fun_{\scr C}(\scr C,\scr E).
\]
These $\infty$-categories are respectively the left-lax colimit and left-lax limit of $\scr A$, although we will not need a precise definition of these terms.

\begin{definition}\label{def:normedspectrum}
	Let $S$ be a scheme and let $\scr C\subset_\fet\Sch_S$.
	\begin{enumerate}
	\item A \emph{normed spectrum} over $\scr C$ is a section of
	$\SH^\otimes$ over $\Span(\scr C,\all,\fet)$ that is cocartesian over $\scr C^\op$.
	\item An \emph{incoherent normed spectrum} over $\scr C$ is a section of $\h\SH^\otimes$ over $\Span(\scr C,\all,\fet)$ that is cocartesian over $\scr C^\op$.
	\end{enumerate}
\end{definition}

\begin{notation}
We denote by 
\[\NAlg_{\scr C}(\SH)\subset \Sect(\SH^\otimes|\Span(\scr C,\all,\fet))\]
the full subcategory of normed spectra over $\scr C$.
More generally, given a functor $\scr A^\otimes\colon \Span(\scr C,\all,\fet)\to \Cat_\infty$, we denote by $\NAlg_{\scr C}(\scr A)$ the $\infty$-category of sections of $\scr A^\otimes$ that are cocartesian over $\scr C^\op$.

In applications, the category $\scr C$ is often $\Sm_S$, $\Sch_{S}$, or $\FEt_S$. To avoid double subscripts, we will usually write $\NAlg_{\Sm}(\SH(S))$ instead of $\NAlg_{\Sm_S}(\SH)$, and similarly in the other cases.
\end{notation}

\begin{remark}
	If $\scr C\subset_\fet\Sch_S$ and $\scr A^\otimes\colon \Span(\scr C,\all,\fet)\to \Cat_\infty$, a section of $\scr A^\otimes$ is cocartesian over $\scr C^\op$ if and only if, for every $X\in\scr C$, it sends the structure map $X\to S$ to a cocartesian edge. This follows from \cite[Proposition 2.4.1.7]{HTT}.
\end{remark}

Let us spell out more explicitly the definition of an incoherent normed spectrum over $\scr C$. It is a spectrum $E\in\SH(S)$ equipped with maps
\[
	\mu_p\colon p_\otimes E_{V} \to E_{U}
\]
in $\SH(U)$, for all finite étale maps $p\colon V\to U$ in $\scr C$, such that:
\begin{itemize}
	\item if $p$ is the identity, then $\mu_p$ is an instance of the equivalence $\id_\otimes\simeq\id$;
	\item for every composable finite étale maps $q\colon W\to V$ and $p\colon V\to U$ in $\scr C$, the following square commutes up to homotopy:
	\begin{tikzmath}
		\diagram{p_\otimes q_\otimes E_W & p_\otimes E_V \\
		(pq)_\otimes E_W & E_U\rlap;\\};
		\arrows (11-) edge node[above]{$p_\otimes\mu_q$} (-12) (12) edge node[right]{$\mu_p$} (22) (11) edge node[left]{$\simeq$} (21) (21-) edge node[below]{$\mu_{pq}$} (-22);
	\end{tikzmath}
	\item for every cartesian square
	\begin{tikzmath}
		\diagram{V' & V \\ U' & U \\};
		\arrows (11-) edge node[above]{$g$} (-12) (11) edge node[left]{$q$} (21) (21-) edge node[below]{$f$} (-22) (12) edge node[right]{$p$} (22);
	\end{tikzmath}
	 in $\scr C$ with $p$ finite étale, the following pentagon commutes up to homotopy:
	\begin{tikzequation} \label{eq:incoherent-compat-square}
		\def\rowsep{10pt}
		\diagram{f^*p_\otimes E_V & f^*E_{U} \\
		q_\otimes g^*E_V & \\
		q_\otimes E_{V'} & E_{U'}\rlap.\\};
		\arrows (11-) edge node[above]{$f^*\mu_p$} (-12) (11) edge node[left]{$\simeq$} (21) (21) edge node[left]{$\simeq$} (31) (12) edge node[right]{$\simeq$} (32) (31-) edge node[below]{$\mu_{q}$} (-32);
	\end{tikzequation}
\end{itemize}
It follows in particular that the map $\mu_p\colon p_\otimes E_V \to E_U$ is equivariant, up to homotopy, for the action of $\Aut(V/U)$ on $p_\otimes E_V$. 
By contrast, a normed spectrum includes the data of homotopies making the above diagrams commute, as well as coherence data for these homotopies. In particular, if $E$ is a normed spectrum, the map $\mu_p$ induces
\[
\mu_p\colon (p_\otimes E_V)_{h\Aut(V/U)} \to E_U.
\]

By Remark~\ref{rmk:symmmon}, there is a forgetful functor
\[
\NAlg_{\scr C}(\SH) \to \CAlg(\SH(S)),
\]
given by restricting a section from $\Span(\scr C,\all,\fet)$ to $\Span(\Fin,\inj,\all)\simeq \Fin_\pt$.
Definition~\ref{def:normedspectrum} is motivated in part by the following observation, where ``fold'' denotes the class of morphisms of schemes that are finite sums of fold maps $S^{\amalg n}\to S$:

\begin{proposition}\label{prop:NAlgtoCAlg}
	Let $S$ be a scheme and let $\scr C\subset_\fet\Sch_S$. The functor 
	\[
	\Fin_\pt\to \Span(\scr C,\all,\fold), \quad X_+\mapsto \coprod_XS,
	\]
	 induces an equivalence between $\CAlg(\SH(S))$ and the $\infty$-category of sections of $\SH^\otimes$ over $\Span(\scr C,\all,\fold)$ that are cocartesian over $\scr C^\op$.
\end{proposition}

\begin{proof}
	This follows from Corollary~\ref{cor:fold-spans}.
\end{proof}

We record some categorical properties of normed spectra:

\begin{proposition} \label{prop:categorical-props}
	Let $S$ be a scheme and $\scr C\subset_\fet\Sch_{S}$.
	\begin{enumerate}
		\item The $\infty$-category $\NAlg_\scr C(\SH)$ has colimits and finite limits. If $\scr C$ is small, it is presentable and hence has all limits.
		\item The forgetful functor $\NAlg_\scr C(\SH)\to\SH(S)$ is conservative and preserves sifted colimits and finite limits. If $\scr C\subset\Sm_S$, it preserves limits and hence is monadic.
		\item The forgetful functor $\NAlg_\scr C(\SH)\to\CAlg(\SH(S))$ is conservative and preserves colimits and finite limits. If $\scr C\subset\Sm_S$, it preserves limits and hence is both monadic and comonadic.
		\item If $A\in\NAlg_\scr C(\SH)$, the symmetric monoidal $\infty$-category $\Mod_{A(S)}(\SH(S))$ can be promoted to a functor \[\Mod_A(\SH)^\otimes\colon \Span(\scr C,\all,\fet)\to\CAlg(\what\Cat{}_\infty^\mathrm{sift}).\]
		Moreover, there is an equivalence of $\infty$-categories
		\[
		\NAlg_\scr C(\SH)_{A/} \simeq \NAlg_\scr C(\Mod_A(\SH)).
		\]
		\item Let $\scr C_0\subset_\fet\scr C$ be a subcategory such that $\SH\colon\scr C^\op\to\what\Cat{}_\infty$ is the right Kan extension of its restriction to $\scr C_0^\op$ (e.g., the adjunction $\PSh(\scr C)\rightleftarrows\PSh(\scr C_0)$ restricts to an equivalence between the subcategories of sheaves for some topologies coarser than the induced cdh topologies). Then the inclusion $\scr C_0\subset\scr C$ induces an equivalence $\NAlg_{\scr C}(\SH)\simeq\NAlg_{\scr C_0}(\SH)$.
		\item Suppose that $\scr C\subset \Sch_S^\fp$ and let $\scr C\subset \scr C'\subset_\fet \Sch_S$ be such that every $S$-scheme in $\scr C'$ is the limit of a cofiltered diagram in $\scr C$ with affine transition maps. Then the inclusion $\scr C\subset\scr C'$ induces an equivalence $\NAlg_\scr C(\SH)\simeq \NAlg_{\scr C'}(\SH)$.
	\end{enumerate}
	Let $f\colon S'\to S$ be a morphism and let $\scr C'\subset_\fet\Sch_{S'}$.
	\begin{enumerate}\setcounter{enumi}{6}
		\item  Suppose $f_\sharp(\scr C')\subset\scr C$. Then the pullback functor $f^*$ preserves normed spectra. More precisely, if $A\in\NAlg_\scr C(\SH)$, the restriction of $A$ to $\Span(\scr C',\all,\fet)$ is a normed spectrum over $\scr C'$.
		\item Suppose $f^*(\scr C)\subset\scr C'$.
		If the pushforward functor $f_*\colon \SH(S')\to\SH(S)$ is compatible with any base change $T\to S$ in $\scr C$ (e.g., $\scr C\subset\Sm_S$ or $f$ is proper),
		then it preserves normed spectra. More precisely, if $A\in\NAlg_{\scr C'}(\SH)$, the assignment $X\mapsto f_*A(X\times_SS')$ is a normed spectrum over $\scr C$.
	\end{enumerate}
\end{proposition}

\begin{proof}
	(1)--(3) The conservativity assertions are obvious. The functor $\SH^\otimes\colon \Span(\Sch,\all,\fet)\to\Cat_\infty$ lands in the $\infty$-category of accessible sifted-cocomplete $\infty$-categories, since $f^*$ and $p_\otimes$ both preserve sifted colimits. By \cite[Proposition 5.4.7.11]{HTT}, $\NAlg_\scr C(\SH)$ admits sifted colimits that are preserved by the forgetful functor $\NAlg_\scr C(\SH) \to \SH(S)$, and it is accessible if $\scr C$ is small.
	Since $f^*$ preserves finite limits (being a stable functor), $\NAlg_\scr C(\SH)$ is closed under finite limits in the larger $\infty$-category of sections of $\SH^\otimes$ over $\Span(\scr C,\all,\fet)$, where limits are computed objectwise \cite[Proposition 5.1.2.2]{HTT}.
	If $\scr C\subset\Sm_S$, the same holds for arbitrary limits since $f^*$ preserves limits when $f$ is smooth.
	 It remains to show that $\NAlg_\scr C(\SH)$ has finite coproducts that are preserved by the forgetful functor $u\colon\NAlg_\scr C(\SH)\to\CAlg(\SH(S))$. The commutative diagram
	 \begin{tikzmath}
	 	\diagram{\Fin_\pt \times \Fin_\pt & \Span(\scr C,\all,\fet)\times\Fin_\pt & \Span(\scr C,\all,\fet) \\ \Fin_\pt & \Span(\scr C,\all,\fet) & \\};
		\arrows (11-) edge (-12) (11) edge node[left]{$\wedge$} (21) (21-) edge (-22) (12) edge node[right]{$\wedge$} (22) (-13) edge node[above]{$(\id,\1)$} (12-) (13) edge node[below right]{$\id$} (22);
	 \end{tikzmath}
	 induces a commutative diagram
	 \begin{tikzmath}
	 	\diagram{\CAlg(\CAlg(\SH(S))) & \CAlg(\NAlg_\scr C(\SH)) & \NAlg_\scr C(\SH)\rlap. \\ \CAlg(\SH(S)) & \NAlg_\scr C(\SH) & \\};
		\arrows (11-) edge[<-] node[above]{$\CAlg(u)$} (-12) (11) edge[<-] (21) (21-) edge[<-] node[above]{$u$} (-22) (12) edge[<-] (22) (-13) edge[<-] (12-) (13) edge[<-] node[below right]{$\id$} (22);
	 \end{tikzmath}
	 The left vertical functor is an equivalence by \cite[Example 3.2.4.5]{HA}, and the upper right horizontal functor is an equivalence by Corollary~\ref{cor:automatic-calg}. Moreover, by \cite[Proposition 3.2.4.7]{HA}, any $\infty$-category of the form $\CAlg(\scr A)$ has finite coproducts, and any functor of the form $\CAlg(u)$ preserves them. This concludes the proof.
	
	(4) We can regard $A$ as a section of $\CAlg(\SH^\otimes)$ over $\Span(\scr C,\all,\fet)$ (Corollary~\ref{cor:automatic-calg}), and then the first claim follows from the functoriality of $\Mod_R(\scr D)$ in the pair $(\scr D,R)$ \cite[\sect4.8.3]{HA}. For the second claim, we observe that a section of $\CAlg(\SH^\otimes)$ under $A$ is equivalently a section of $\CAlg(\SH^\otimes)_{A/}\simeq \CAlg(\Mod_A(\SH)^\otimes)$, which is in turn equivalently a section of $\Mod_A(\SH)^\otimes$, by Corollary~\ref{cor:automatic-calg}.
	
	(5) This is a special case of Corollary~\ref{cor:span-RKE}. The parenthetical statement follows from the fact that $\SH\colon\Sch^\op\to\what\Cat{}_\infty$ is a cdh sheaf \cite[Proposition 6.24]{Hoyois}.
	
	(6) By (2), the restriction functor $\NAlg_{\scr C'}(\SH)\to \NAlg_{\scr C}(\SH)$ is conservative. For every $X\in\scr C'$, the overcategory $\Span(\scr C,\all,\fet)_{/X}$ admits finite sums and hence is sifted. It follows from \cite[Corollary 4.3.1.11]{HTT} that the restriction functor 
	\[
	\Sect(\SH^\otimes|\Span(\scr C',\all,\fet)) \to \Sect(\SH^\otimes|\Span(\scr C,\all,\fet))
	\]
	has a fully faithful left adjoint $L$ given by relative left Kan extension, and it remains to show that $L$ preserves normed spectra. Let $E$ be a section of $\SH^\otimes$ over $\Span(\scr C,\all,\fet)$. Then the section $L(E)$ is given by the formula
	\[
	L(E)_X = \colim_{Z\stackrel f\from Y\stackrel p\to X} p_\otimes f^*(E_Z),
	\]
	where the indexing category is $\Span(\scr C,\all,\fet)_{/X}$. We claim that the inclusion $(\scr C^\op)_{/X} \into \Span(\scr C,\all,\fet)_{/X}$ is cofinal. By \cite[Theorem 4.1.3.1]{HTT}, it suffices to show that the comma category
	\[
	(\scr C^\op)_{/X} \times_{\Span(\scr C,\all,\fet)_{/X}} (\Span(\scr C,\all,\fet)_{/X})_{(Z\stackrel f\from Y\stackrel p\to X)/}
	\]
	is weakly contractible. An object in this category is a commutative diagram
	\begin{tikzmath}
		\diagram{Z & Y & X \\ & Y_0 & X_0\rlap,\\};
		\arrows (11-) edge[<-] node[above]{$f$} (-12) (12-) edge node[above]{$p$} (-13) (12) edge (22) (13) edge (23)
		(11) edge[<-] (22) (22-) edge node[below]{$p_0$} (-23);
	\end{tikzmath}
	where $X_0\in\scr C$, $p_0$ is finite étale, and the square is cartesian.
	Since $\scr C\subset\Sch_S^\fp$, it follows from \cite[Proposition 8.13.5]{EGA4-3} that $\scr C'$ is identified with a full subcategory of $\Pro(\scr C)$. Moreover, by \cite[Théorèmes 8.8.2(ii) and 8.10.5(x)]{EGA4-3} and \cite[Proposition 17.7.8(ii)]{EGA4-4}, every finite étale map in $\scr C'$ is the pullback of a finite étale map in $\scr C$. This implies that the above category is in fact filtered, which proves our claim.
	 We therefore have a natural equivalence
	\[
	L(E)_X \simeq \colim_{(f\colon X\to Z)\in\scr C_{X/}} f^*(E_Z).
	\]
	If $E$ is cocartesian over backward morphisms, the right-hand side is the colimit of a constant diagram, and we deduce that $L(E)$ is also cocartesian over backward morphisms, as desired.
	
	(7) It is clear that the given section is cocartesian over $\scr C^{\prime\op}$.
	
	(8) The given section is defined more precisely as follows. 
	The assignment $X\mapsto A(X\times_SS')$ is a section of the cocartesian fibration over $\Span(\scr C,\all,\fet)$ classified by the composition
	\[
	\Span(\scr C,\all,\fet) \xrightarrow{f^*} \Span(\scr C',\all,\fet) \xrightarrow{\SH^\otimes} \Cat_\infty.
	\]
	The pullback functors $f_X^*\colon \SH(X)\to \SH(X\times_SS')$ are natural in $X\in\Span(\scr C,\all,\fet)$, i.e., they are the fibers of a map $f^*$ of cocartesian fibrations over $\Span(\scr C,\all,\fet)$ that preserves cocartesian edges. By Lemma~\ref{lemm:construct-relative-adjoint}(1), $f^*$ has a relative right adjoint $f_*$, which yields the section $X\mapsto f_*A(X\times_SS')$.
	The assumption on $f_*$ implies that this section is cocartesian over $\scr C^\op$.
	If $\scr C\subset\Sm_S$ (resp.\ if $f$ is proper), the assumption on $f_*$ holds by smooth base change (resp.\ by proper base change). 
\end{proof}

\begin{remark}
	The proof of each assertion of Proposition~\ref{prop:categorical-props} uses only a few properties of the functor $\SH^\otimes$ and applies much more generally. For example, if $\scr C\subset_\fet\Sch_S$ is small and $\scr A\colon\Span(\scr C,\all,\fet)\to\Cat_\infty$ is a normed $\infty$-category satisfying conditions (1), (3), and (5) of Definition~\ref{dfn:normed-category}, then $\NAlg_\scr C(\scr A)$ is presentable. Moreover, the forgetful functor $\NAlg_\scr C(\scr A)\to\CAlg(\scr A(S))$ preserves colimits and any type of limits that are preserved by the functors $f^*$ for $f\colon X\to S$ in $\scr C$.
\end{remark}

\begin{remark}\label{rmk:freeNAlg}
	If $\scr C\subset_\fet\Sm_S$, it follows from Proposition~\ref{prop:categorical-props}(2) that the forgetful functor $\NAlg_{\scr C}(\SH)\to \SH(S)$ has a left adjoint $\NSym_{\scr C}\colon \SH(S)\to\NAlg_{\scr C}(\SH)$. 
	When $\scr C=\Sm_S$ or $\scr C=\FEt_S$, it is given by the formula
	\[
	\NSym_{\scr C}(E) = \colim_{\substack{f\colon X\to S\\p\colon Y\to X}} f_\sharp p_\otimes(E_Y),
	\]
	where the indexing $\infty$-category is the source of the cartesian fibration classified by $\scr C^\op\to\scr S$, $X\mapsto \FEt_X^\simeq$. This can be proved using the formalism of Thom spectra, which we will develop in Section~\ref{sec:thomspectra} (see Remarks \ref{rmk:freeNAlg-2} and~\ref{rmk:freeNAlg-3}).
\end{remark}

\begin{remark} \label{rmk:coprod-nalg}
Finite coproducts in $\CAlg(\SH(S))$ are computed as smash products in $\SH(S)$ \cite[Proposition 3.2.4.7]{HA}.
By Proposition~\ref{prop:categorical-props}(1,3), the $\infty$-category $\NAlg_\mathcal{C}(\SH)$ has finite coproducts and the forgetful functor $\NAlg_\mathcal{C}(\SH) \to \CAlg(\SH(S))$ preserves them. Thus, finite coproducts in $\NAlg_\scr C(\SH)$ are also computed as smash products in $\SH(S)$. In other words, the cocartesian symmetric monoidal structure on $\NAlg_\scr C(\SH)$ lifts the smash product symmetric monoidal structure on $\SH(S)$.
We will see an elaboration of this in Theorem~\ref{thm:norm-pullback}.
\end{remark}

\begin{remark} \label{rmk:NAlg-ess-smooth-pullback}
Assertions (5) and (6) in Proposition~\ref{prop:categorical-props} can be used to strengthen (7) and (8).
For example, for $f\colon S'\to S$, we have a pullback functor
\[
f^*\colon \NAlg_{\Sm}(\SH(S)) \to \NAlg_{\Sm}(\SH(S'))
\]
in either of the following situations:
\begin{itemize}
	\item $S$ is the spectrum of a field with resolutions of singularities and $f$  is of finite type (use (5));
	\item $f$ is the limit of a cofiltered diagram of smooth $S$-schemes with affine transition maps (use (6)).
\end{itemize}
In fact, it is proved in \cite{bachmann-powerops} that the restriction functor $\NAlg_{\Sch}(\SH(S)) \to \NAlg_{\Sm}(\SH(S))$ is always an equivalence, so that the above pullback functor exists in complete generality.
\end{remark}

\begin{example}\label{ex:1}
	Since $\SH^\otimes$ takes values in symmetric monoidal $\infty$-categories, it has a \emph{unit section} sending a scheme $S$ to the motivic sphere spectrum $\1_S\in\SH(S)$, which is everywhere cocartesian. Thus, for every scheme $S$, the motivic sphere spectrum $\1_S$ is canonically a normed spectrum over $\Sch_S$; in fact, it is the initial object of $\NAlg_{\scr C}(\SH)$ for every $\scr C\subset_\fet\Sch_S$, by Proposition~\ref{prop:categorical-props}(3).
\end{example}

\begin{example}\label{ex:HZ}
	For every noetherian scheme $S$, Voevodsky's motivic cohomology spectrum $\HH\Z_S\in\SH(S)$ has a structure of normed spectrum over $\Sm_S$. We will prove this in Section~\ref{sec:PST} (see also \sect\ref{sub:HZ}).
\end{example}

\begin{example}\label{ex:KGL}
	For every scheme $S$, the homotopy $\K$-theory spectrum $\KGL_S\in\SH(S)$ has a structure of normed spectrum over $\Sch_S$. We will prove this in Section~\ref{sec:dgCat}.
\end{example}

\begin{example}\label{ex:MGL}
	For every scheme $S$, the algebraic cobordism spectrum $\MGL_S\in\SH(S)$ has a structure of normed spectrum over $\Sch_S$. We will prove this in Section~\ref{sec:thomspectra}.
\end{example}

\subsection{Cohomology theories represented by normed spectra} \label{subsec:coh-theories-with-norms}

We now investigate the residual structure on the cohomology theory associated with a normed spectrum (resp.\ an incoherent normed spectrum) $E\in\SH(S)$. 
Given a finite étale map $p\colon T\to S$
and $A\in\SH(T)$, we have a transfer map
\begin{equation*}\label{eqn:generalnorm}
\nu_p\colon \Map(A, E_T)\xrightarrow{p_\otimes} \Map(p_\otimes A,p_\otimes E_T) \xrightarrow{\mu_p} \Map(p_\otimes A, E).
\end{equation*}
In particular, if $A=\Sigma^\infty_+ X$ for some $X\in\SmQP_T$, then $p_\otimes A\simeq \Sigma^\infty_+\Weil_pX$ and we get
\[
\nu_p\colon \Map(\Sigma^\infty_+X,E) \to \Map(\Sigma^\infty_+\Weil_pX,E),
\]
which is readily seen to be a multiplicative $\E_\infty$-map in $\scr S$ (resp.\ in $\h\scr S$).
Note that the functor $p_\otimes$ does not usually send the bigraded spheres 
\[\S^{r,s}=\S^{r-s}\wedge \G_m^{\wedge s}\]
to bigraded spheres, and even when it does (for example when $p$ is free and $r=2s$), it only does so up to a noncanonical equivalence. So in general we do not obtain a multiplicative transfer on the usual bigrading of $E$-cohomology.
Instead, we should consider $E$-cohomology graded by the Picard $\infty$-groupoid of $\SH(T)$ \cite[\sect2.1]{GepnerLawson}; we then have a multiplicative family of transfers
\[
\nu_p\colon \Map(\Sigma^\infty_+X,\star\wedge E_T) \to \Map(\Sigma^\infty_+\Weil_pX,p_\otimes(\star)\wedge E),\quad \star \in \Pic(\SH(T)),
\]
which is an $\E_\infty$-map between commutative algebras for the Day convolution symmetric monoidal structure on $\Fun(\Pic(\SH(T)), \scr S)$ (resp.\ on $\Fun(\Pic(\SH(T)), \h\scr S)$).

We can say more if $E$ is \emph{oriented}. In that case, for every virtual vector bundle $\xi$ of rank $r$ on $Y\in\Sm_S$, we have the Thom isomorphism
\[
t\colon \Sigma^\xi E_Y \simeq \Sigma^{2r,r}E_Y.
\]
If $p\colon T\to S$ is finite étale of degree $d$, we therefore get a canonical family of maps
\begin{equation}\label{eqn:homogeneous}
t\mu_p\colon p_\otimes\Sigma^{2n,n}E_T \to \Sigma^{2nd,nd}E_S,\quad n\in\Z.
\end{equation}
We can even extend $t\mu_p$ to the sum $\bigvee_n \Sigma^{2n,n}E$ as follows. 
Let $I\subset\Z$ be a finite subset.
By Corollary~\ref{cor:Nthom}, $p_\otimes (\bigvee_{n\in I}\S^{2n,n})$ is the Thom spectrum of a virtual vector bundle on the finite étale $S$-scheme $\Weil_p(T^{\amalg I})$.
Using the Thom isomorphism, we can form the composite
\begin{equation}\label{eqn:inhomogeneous}
p_\otimes\left(\bigvee_{n\in I}\Sigma^{2n,n}E_T\right)
 \xrightarrow{\mu_p} p_\otimes\left(\bigvee_{n\in I}\S^{2n,n}\right) \wedge E_S
 \stackrel t\to \Sigma^\infty_+\Weil_p(T^{\amalg I}) \wedge \bigvee_{r\in\Z}\Sigma^{2r,r}E_S
\to \bigvee_{r\in\Z}\Sigma^{2r,r}E_S,
\end{equation}
where the last map is induced by the structure map $\Weil_p(T^{\amalg I})\to S$.
Note that this recovers~\eqref{eqn:homogeneous} when $I=\{n\}$.

\begin{proposition}\label{prop:oriented-normed}
	Let $E$ be an oriented incoherent normed spectrum over $\FEt_S$ and let $p\colon T\to S$ be finite étale. Then the maps~\eqref{eqn:inhomogeneous} fit together to induce
	\[
	\tilde\mu_p\colon p_\otimes\left(\bigvee_{n\in \Z}\Sigma^{2n,n}E_T\right) \to \bigvee_{r\in\Z}\Sigma^{2r,r}E_S.
	\]
	Moreover, $\tilde\mu_p$ is multiplicative up to homotopy. In particular, for every $X\in\SmQP_T$, we have a multiplicative transfer
	\[
	\tilde\nu_p\colon\bigoplus_{n\in \Z} \Map(\Sigma^\infty_+X,\Sigma^{2n,n}E) \to \bigoplus_{r\in\Z} \Map(\Sigma^\infty_+\Weil_pX,\Sigma^{2r,r}E).
	\]
\end{proposition}

\begin{proof}
	Given a finite subset $I\subset\Z$, $\theta_I$ will denote the virtual vector bundle on $T^{\amalg I}$ that is trivial of rank $i$ over the $i$th component, and $\xi_I$ will be the induced bundle over $\Weil_p(T^{\amalg I})$. 
	The map~\eqref{eqn:inhomogeneous} is then the composite
	\[
	p_\otimes (\Th(\theta_I)\wedge E_T)
	\xrightarrow{\mu_p} \Th(\xi_I) \wedge E_S
	\stackrel t\simeq \Th(\rk\xi_I) \wedge E_S
	\to \bigvee_{r\in \Z}\Sigma^{2r,r} E_S,
	\]
	where $t$ is the Thom isomorphism. It is clear that these maps fit together to induce $\tilde\mu_p$.
	To prove that $\tilde\mu_p$ is multiplicative, consider two finite subsets $I,J\subset \Z$. We must show that the boundary of the following diagram commutes, where unlabeled arrows are induced by morphisms between Thom spectra:
	\begin{tikzmath}
		\diagram{
		p_\otimes (\Th(\theta_I)\wedge E_T)\wedge p_\otimes(\Th(\theta_J)\wedge E_T) & p_\otimes(\Th(\theta_I)\wedge\Th(\theta_J)\wedge E_T) & p_\otimes(\Th(\theta_{I+J})\wedge E_T) \\
		 \Th(\xi_I) \wedge E_S\wedge\Th(\xi_J)\wedge E_S & \Th(\xi_I)\wedge\Th(\xi_J) \wedge E_S & \Th(\xi_{I+J})\wedge E_S \\
		 \Th(\rk\xi_I) \wedge E_S\wedge \Th(\rk\xi_J)\wedge E_S & \Th(\rk\xi_I)\wedge \Th(\rk\xi_J) \wedge E_S & \Th(\rk\xi_{I+J})\wedge E_S \\
		 \bigvee_{r}\Sigma^{2r,r} E_S\wedge\bigvee_s\Sigma^{2s,s} E_S & \bigvee_{r,s}\Sigma^{2r,r}\Sigma^{2s,s} E_S & \bigvee_{u}\Sigma^{2u,u} E_S\rlap. \\
		};
		\arrows 
		(11-) edge node[above]{$\mu_\nabla$} (-12) (12-) edge (-13)
		(21-) edge node[above]{$\mu_\nabla$} (-22) (22-) edge (-23)
		(31-) edge node[above]{$\mu_\nabla$} (-32) (32-) edge (-33)
		(41-) edge node[above]{$\mu_\nabla$} (-42) (42-) edge (-43)
		(11) edge node[left]{$\mu_p\wedge\mu_p$} (21) (12) edge node[left]{$\mu_p$} (22) (13) edge node[right]{$\mu_p$} (23)
		(21) edge node[left]{$t\wedge t$} (31) (22) edge node[left]{$t$} (32) (23) edge node[right]{$t$} (33)
		(31) edge (41) (32) edge (42) (33) edge (43)
		;
	\end{tikzmath}
	The middle squares commute by multiplicativity and naturality of the Thom isomorphisms, and every other square commutes for obvious reasons.
\end{proof}

\begin{remark}
	We do not know if the maps $\tilde\mu_p$ of Proposition~\ref{prop:oriented-normed} always make $\bigvee_{n\in\Z}\Sigma^{2n,n}E$ into an incoherent normed spectrum (this would follow from a certain compatibility between norm functors and Thom isomorphisms). 
	If $E$ is an oriented \emph{normed} spectrum, it is in any case unlikely that $\bigvee_{n\in\Z}\Sigma^{2n,n}E$ can be promoted to a normed spectrum without further assumptions. The analogous question for oriented $\E_\infty$-ring spectra is already subtle, a sufficient condition being that $E$ admits an $\E_\infty$-orientation $\MGL_S\to E$ (see \cite{SpitzweckP}). Similarly, since $\bigvee_{n\in\Z}\Sigma^{2n,n}\MGL_S$ has a canonical structure of normed $\MGL_S$-module (see Theorem~\ref{thm:normedMGL}), a morphism of normed spectra $\MGL_S\to E$ induces a structure of normed spectrum on $\bigvee_{n\in\Z}\Sigma^{2n,n}E$. We will show in Proposition~\ref{prop:periodization-norms} that the underlying incoherent structure is given by the maps $\tilde\mu_p$.
\end{remark}

\begin{proposition}\label{prop:tambara}
	Let $E$ be an incoherent normed spectrum over $\FEt_S$, let $p\colon T\to S$ be a finite étale map, and let $f\colon Y\to X$ be a smooth proper map in $\SmQP_T$.
	\begin{enumerate}
		\item For every $\star\in\Pic(\SH(T))$, the following square commutes:
		\begin{tikzmath}
			\def\colsep{1.5em}
			\diagram{ \Map(\Th_Y(-\rm T_f),\star\wedge E_T) & \Map(\Th_{\Weil_pY}(-\rm T_{\Weil_pf}), p_\otimes(\star)\wedge E) \\
			\Map(\Sigma^\infty_+X,\star\wedge E_T) & \Map(\Sigma^\infty_+\Weil_pX,p_\otimes(\star)\wedge E)\rlap. \\};
			\arrows (11-) edge node[above]{$\nu_p$} (-12) (21-) edge node[above]{$\nu_p$} (-22)
			(11) edge node[left]{$\tau_f$} (21) (12) edge node[right]{$\tau_{\Weil_pf}$} (22);
		\end{tikzmath}
		\item If $E$ is oriented, the following square commutes:
		\begin{tikzmath}
			\def\colsep{2em}
			\def\rowsep{2em}
			\diagram{\bigoplus_{n\in\Z} \Map(\Sigma^\infty_+Y,\Sigma^{2n,n}E) & \bigoplus_{r\in\Z} \Map(\Sigma^\infty_+\Weil_pY,\Sigma^{2r,r}E) \\
			\bigoplus_{n\in\Z} \Map(\Sigma^\infty_+X,\Sigma^{2n,n}E) & \bigoplus_{r\in\Z} \Map(\Sigma^\infty_+\Weil_pX,\Sigma^{2r,r}E)\rlap. \\};
			\arrows (11-) edge node[above]{$\tilde\nu_p$} (-12) (21-) edge node[above]{$\tilde\nu_p$} (-22)
			([yshift=10pt]11.south) edge node[right]{$\tau_f$} (21) ([yshift=10pt]12.south) edge node[right]{$\tau_{\Weil_pf}$} (22);
		\end{tikzmath}
	\end{enumerate}
\end{proposition}

\begin{proof}
	Both assertions follow easily from the definition of the multiplicative transfer and Corollary~\ref{cor:transfers}.
\end{proof}

\begin{lemma}\label{lem:extra-bc-square}
	Let $E$ be an incoherent normed spectrum over $\scr C\subset_\fet\Sch_S$. Consider a diagram
	\begin{tikzmath}
		\diagram{ W & \Weil_pW\times_UV & \Weil_pW \\ & V & U \\};
		\arrows (11-) edge[<-] node[above]{$e$} (-12) (12-) edge node[above]{$\pi_1$} (-13) (11) edge node[below left]{$f$} (22) (12) edge node[right]{$\pi_2$} (22) (22-) edge node[above]{$p$} (-23) (13) edge node[right]{$\Weil_p(f)$} (23);
	\end{tikzmath}
	in $\scr C$, where $p$ is finite étale and $f$ is smooth and quasi-projective.
	Then the following diagram commutes up to homotopy:
	\begin{tikzmath}
		\diagram{ \Map(\Sigma^\infty_+ W,E_V) & \Map(\1_W,E_W) & \Map(\1_{\Weil_pW\times_UV},E_{\Weil_pW\times_UV}) \\ & \Map(\Sigma^\infty_+ \Weil_pW,E_U) & \Map(\1_{\Weil_pW},E_{\Weil_pW})\rlap. \\};
		\arrows (12-) edge node[above]{$e^*$} (-13) (13) edge node[right]{$\nu_{\pi_1}$} (23) (11) edge node[below left]{$\nu_p$} (22) (11-) edge node[above]{$\simeq$} (-12) (22-) edge node[above]{$\simeq$} (-23);
	\end{tikzmath}
\end{lemma}

\begin{proof}
	Let $\delta\colon \Weil_pW\to \Weil_pW\times_U\Weil_pW$ be the diagonal map and $\gamma\colon \Weil_pW\times_UV \to \Weil_pW\times_UW$ the graph of $e$. We then have a diagram
	\begin{tikzmath}
		\def\colsep{2em}
		\diagram{
		\Map(\Sigma^\infty_+W,E_V) & \Map(\Sigma^\infty_+(\Weil_pW\times_UW),E_{\Weil_pW\times_UV}) & \Map(\1_{\Weil_pW\times_UV},E_{\Weil_pW\times_UV}) \\
		\Map(\Sigma^\infty_+\Weil_pW, E_U) & \Map(\Sigma^\infty_+(\Weil_pW\times_U\Weil_pW),E_{\Weil_pW}) & \Map(\1_{\Weil_pW},E_{\Weil_pW})\rlap, \\
		};
		\arrows (11-) edge node[above]{$\pi_2^*$} (-12) (12-) edge node[above]{$\gamma^*$} (-13)
		(21-) edge node[above]{$\Weil_p(f)^*$} (-22) (22-) edge node[above]{$\delta^*$} (-23)
		(11) edge node[left]{$\nu_p$} (21) (12) edge node[right]{$\nu_{\pi_1}$} (22) (13) edge node[right]{$\nu_{\pi_1}$} (23);
	\end{tikzmath}
	where the upper composite is $e^*$ and the lower composite is the identity.
	The first square commutes by~\eqref{eq:incoherent-compat-square} and the second square commutes because $\Weil_{\pi_1}(\gamma)=\delta$.
\end{proof}

Recall the notion of Tambara functor on finite étale schemes from \cite[Definition 8]{bachmann-gwtimes}.

\begin{corollary} \label{cor:normed-spectrum-tambara-functor}
	Let $E$ be an incoherent normed spectrum over $\FEt_S$. Then the presheaf $X\mapsto E^{0,0}(X)=[\Sigma^\infty_+X,E]$ is a Tambara functor on $\FEt_S$ with additive transfers $\tau_p$ and multiplicative transfers $\nu_p$.
\end{corollary}

\begin{proof}
	Both transfers are compatible with composition and base change, so we only need to show that for maps $q\colon W\to V$ and $p\colon V\to U$ in $\FEt_S$, the pentagon
	\begin{tikzmath}
		\diagram{E^{0,0}(W) & E^{0,0}(\Weil_pW\times_UV) & E^{0,0}(\Weil_pW) \\ & E^{0,0}(V) & E^{0,0}(U) \\};
		\arrows (11) edge node[below left]{$\tau_q$} (22) (11-) edge node[above]{$e^*$} (-12) (12-) edge node[above]{$\nu_{\pi_1}$} (-13) (22-) edge node[below]{$\nu_p$} (-23) (13) edge node[right]{$\tau_{\Weil_p(q)}$} (23);
	\end{tikzmath}
	commutes. By Lemma~\ref{lem:extra-bc-square}, the composition of the two upper horizontal maps is $\nu_p\colon E^{0,0}(W)\to E^{0,0}(\Weil_pW)$. Hence, the commutativity of the pentagon is a special case of Proposition~\ref{prop:tambara}(1). 
\end{proof}

\begin{example}\label{ex:GWnorms}
	If $k$ is a field, recall that $\pi_{0,0}(\1_k)$ is canonically isomorphic to the Grothendieck–Witt ring $\GW(k)$. As $\1_k$ is a normed spectrum, we obtain a structure of Tambara functor on $\GW\colon \FEt_k^\op \to \Set$. We will prove in Theorem~\ref{thm:GW-norms-comparison} that the norms coincide with Rost's multiplicative transfers, at least if $\Char(k)\neq 2$.
\end{example}

\begin{example}
	Let $S$ be essentially smooth over a field, let $p\colon T\to S$ be finite étale, and let $X\in\SmQP_T$.
	For every $n\in \Z$,
	\[
	\Map(\Sigma^\infty_+X,\Sigma^{2n,n}\HH\Z_T) \simeq z^n(X,*)
	\]
	is (the underlying space of) Bloch's cycle complex in weight $n$ \cite[Lecture 19]{Mazza:2006}.
	Since $\bigvee_{n\in\Z}\Sigma^{2n,n}\HH\Z_S$ is a normed spectrum (see Example~\ref{ex:periodicHZ}), we obtain an $\E_\infty$-multiplicative transfer
	\[
	\nu_p\colon \bigoplus_{n\in\Z}z^n(X,\ast) \to \bigoplus_{r\in\Z}z^{r}(\Weil_pX,\ast).
	\]
	We will prove in Theorem~\ref{thm:FultonMacPhersonComparison} that this refines the Fulton–\<MacPherson–\<Karpenko transfer on Chow groups.
\end{example}

\begin{example}
	Let $p\colon T\to S$ be finite étale and let $X\in\SmQP_T$.
	Then
	\[
	\Map(\Sigma^\infty_+X,\KGL_T) \simeq \Omega^\infty\KH(X)
	\]
	is Weibel's homotopy $\K$-theory space \cite[Proposition 2.14]{Cisinski}.
	 Since $\KGL_S$ is a normed spectrum (see Theorem~\ref{thm:KGL}), we obtain an $\E_\infty$-multiplicative transfer
	\[
	\nu_p\colon \Omega^\infty\KH(X) \to \Omega^\infty\KH(\Weil_pX).
	\]
	We will also construct such a transfer between ordinary $\K$-theory spaces, refining the Joukhovitski transfer on $\K_0$ (see Corollary~\ref{cor:normK}).
\end{example}

\begin{example}[Power operations]
	\label{ex:totalpower}
	Let $E$ be a normed spectrum over $\Sm_S$ and let $E_0=\Omega^\infty E\in \H(S)$ be the underlying space of $E$.
	Let $G$ be a finite étale group scheme over $S$ acting on a finite étale $S$-scheme $T$. For $i\geq 0$, let $U_i\subset \Hom_S(G,\A^i_S)$ be the open subset where the action of $G$ is free, so that the map $\colim_{i\to\infty} U_i/G\to \B_\et G$ classifying the principal $G$-bundles $U_i\to U_i/G$ is a motivic equivalence (see \cite[\sect4.2]{MV}).
	Then the \emph{total $T$-power operation}
	\[
	\rm P_{T}\colon E_0 \to \Hom(\B_\et G,E_{0})
	\]
	in $\H(S)$
	can be defined as the limit as $i\to\infty$ of the composition
	\[
	E_0 \xrightarrow{f_i^*} \Hom((U_i\times_S T)/G, E_0) \xrightarrow{\nu_{p_i}} \Hom(U_i/G,E_{0}),
	\]
	where $f_i\colon (U_i\times_ST)/G\to S$ is the structure map and $p_i\colon (U_i\times_ST)/G\to U_i/G$ is the projection.
	For example, if $S$ is smooth over a field and $E=\bigvee_{n\in\Z}\Sigma^{2n,n}\HH\Z_S$, we obtain a power operation
	\[
	\rm P_T\colon z^*(X,*) \to z^*(X\times \B_\et G,*),
	\]
	natural in $X\in\Sm_S$.
	On homogeneous elements, this recovers the total power operation in motivic cohomology constructed by Voevodsky \cite[\sect5]{Voevodsky:2003}: in fact, given the construction of the norms on $\HH\Z_S$ in \sect\ref{sub:PST}, this is just a repackaging of Voevodsky's construction. 
	
	If $E$ is an oriented normed spectrum and $n\in\Z$, we can use the Thom isomorphism
	\[
	\MGL_S \wedge \colim_{i\to\infty }\Th_{U_i/G}(-\Weil_{p_i}\A^n) \simeq \MGL_S\wedge \Sigma^{-\A^{nd}}\Sigma^\infty_+\B_\et G
	\]
	 to define more generally
	\[
	\rm P_T\colon E_n \to \Hom(\B_\et G,E_{nd}),
	\]
	where $E_n=\Omega^\infty\Sigma^{\A^n}E$ and $d$ is the degree of $T$ over $S$.
\end{example}

\section{The norm–pullback–pushforward adjunctions}
\label{sec:f-tens-*-adj}

In this section we shall prove the following two results. The first is an analog of the fact that in any $\infty$-category of commutative rings, the tensor product is the coproduct.

\begin{theorem}\label{thm:norm-pullback}
Let $f\colon S' \to S$ be finite étale, let $\mathcal{C} \subset_\fet \Sch_S$, and let $\scr C'=\scr C_{/S'}$. Then there is an adjunction
\[
f_\otimes: \NAlg_{\scr C'}(\SH) \adj \NAlg_\scr{C}(\SH) :f^*
\]
where $f^*$ is the functor from Proposition~\ref{prop:categorical-props}(7) and $f_\otimes$ lifts the norm $f_\otimes\colon \SH(S') \to \SH(S)$.
\end{theorem}

\begin{theorem}\label{thm:second-adjunction}
	Let $f\colon S' \to S$ be a morphism and let $\scr C\subset_\fet\Sch_S$ and $\scr C'\subset_\fet\Sch_{S'}$ be subcategories such that $f_\sharp(\scr C')\subset \scr C$ and $f^*(\scr C)\subset\scr C'$.
	Assume that the pushforward functor $f_*\colon \SH(S')\to\SH(S)$ is compatible with any base change $T\to S$ in $\scr C$ (e.g., $\scr C\subset\Sm_S$ or $f$ is proper).
	Then there is an adjunction
\[ f^*: \NAlg_{\scr C}(\SH) \adj \NAlg_{\scr C'}(\SH): f_* \]
where $f^*$ and $f_*$ are the functors from Proposition \ref{prop:categorical-props}(7,8).
\end{theorem}

Theorem~\ref{thm:norm-pullback} will be proved in \sect\ref{sub:norm-pullback} (as a consequence of a much more general result, Theorem~\ref{thm:norm-pullback2}), and Theorem~\ref{thm:second-adjunction} will be proved at the end of \sect\ref{sub:pull-push}.

\begin{remark}
	In the setting of Theorem~\ref{thm:norm-pullback}, the category $\scr C_{/S'}$ is the \emph{unique} category $\scr C'\subset_\fet\Sch_{S'}$ such that $f_\sharp(\scr C')\subset\scr C$ and $f^*(\scr C)\subset \scr C'$. Indeed, the unit map $X \to f^* f_\sharp X$ is finite étale, so $X \in \scr C'$ as soon as $f_\sharp X \in \scr C$, and consequently $\scr C_{/S'} \subset \scr C'$. The other inclusion is clear.
	Note also that if $\scr C$ is $\Sch_S$, $\Sm_S$, or $\FEt_S$, then $\scr C_{/S'}$ is $\Sch_{S'}$, $\Sm_{S'}$, or $\FEt_{S'}$, respectively.
\end{remark}

\begin{example}\label{ex:NAlg-ess-smooth-adjunction}
Let $f\colon S' \to S$ be a pro-smooth morphism, i.e., $S'$ is the limit of a cofiltered diagram of smooth $S$-schemes with affine transition maps.
Then the assumptions of Theorem~\ref{thm:second-adjunction} hold if $\scr C\subset\Sch_S$ and $\scr C'\subset\Sch_{S'}$ are the full subcategories of pro-smooth schemes, since $\SH(\ph)$ satisfies pro-smooth base change (this follows from smooth base change and the continuity of $\SH(\ph)$ \cite[Proposition C.12(4)]{HoyoisGLV}).
Together with Proposition~\ref{prop:categorical-props}(6), we obtain an adjunction
\[ f^*: \NAlg_\Sm(\SH(S)) \adj \NAlg_\Sm(\SH(S')): f_*. \]
\end{example}

Although these results are mostly formal, the formalism is somewhat involved. It is not difficult to write down a functor $f_\otimes\colon \NAlg_{\scr C'}(\SH) \to \NAlg_\scr{C}(\SH)$, but constructing the adjunction is tricky. To that end we shall employ a small amount of $(\infty,2)$-category theory \cite[Appendix A]{GRderalg}. Informally, an $(\infty,2)$-category $\mathbf{C}$ is an $\infty$-category enriched in $\infty$-categories. Thus for $X,Y \in \mathbf{C}$ there is not just a mapping \emph{space} $\Map(X,Y) \in \mathcal{S}$ but also a mapping \emph{$\infty$-category} $\MAP(X,Y) \in \Cat_\infty$ such that $\Map(X,Y)=\MAP(X,Y)^\simeq$. The prime example of an $(\infty,2)$-category is $\mathbf{Cat}_\infty$: the objects of $\mathbf{Cat}_\infty$ are the $\infty$-categories, and for $\mathcal{C}, \mathcal{D} \in \mathbf{Cat}_\infty$ we have $\MAP(\mathcal{C}, \mathcal{D}) = \Fun(\mathcal{C}, \mathcal{D})$. 

The way we shall use $(\infty,2)$-categories is as follows. If $\mathbf{C}$ is an $(\infty,2)$-category, then there is a notion of \emph{adjunction} between $1$-morphisms in $\mathbf{C}$. In the $(\infty,2)$-category $\mathbf{Cat}_\infty$ the notion of adjunction agrees with the standard notion of adjunction of $\infty$-categories. Now if $F\colon \mathbf{C} \to \mathbf{Cat}_\infty$ is an $(\infty,2)$-functor and $f \dashv g$ is an adjunction in $\mathbf{C}$, then there is an induced adjunction of $\infty$-categories $F(f) \dashv F(g)$.

\subsection{The norm–pullback adjunction}
\label{sub:norm-pullback}

Let $\scr C$ be an $\infty$-category and let ``$\lleft$'' and ``$\rright$'' be classes of morphisms in $\scr C$ that contain the equivalences and are closed under composition and pullback along one another. Suppose moreover that if $f$ and $f\circ g$ are right morphisms, then $g$ is a right morphism. 
The $\infty$-category $\Span(\scr C,\lleft,\rright)$ can then be promoted to an $(\infty,2)$-category $\mathbf{Span}(\scr C,\lleft,\rright)$ in which a $2$-morphism is a commutative diagram
\begin{tikzmath}
	\def\colsep{1em}
	\def\rowsep{1em}
	\diagram{ & Y & \\ X & & Z\rlap, \\ & Y' & \\};
	\arrows (12) edge (21) edge (23) (32) edge (21) edge (23) (12) edge node[right]{$r$} (32);
\end{tikzmath}
where $r\colon Y\to Y'$ is necessarily a right morphism. This $(\infty,2)$-category is denoted by $\mathrm{Corr}(\scr C)^{\rright}_{\lleft;\rright}$ in \cite[\sect V.1.1]{GRderalg}.

\begin{theorem}\label{thm:norm-pullback2}
	Let $\scr C$ be an $\infty$-category and let $\rright\subset\lleft$ be classes of morphisms in $\scr C$ that contain the equivalences and are closed under composition and pullback along one another. Suppose moreover that if $f$ and $f\circ g$ are right morphisms, then $g$ is a right morphism. For $X\in\scr C$, denote by $\scr C_X\subset\scr C_{/X}$ the full subcategory spanned by the left morphisms. Let
	\[
	\scr A\colon \Span(\scr C,\all,\rright)\to\what\Cat{}_\infty, 
	\quad (X\xleftarrow{f} Y\xrightarrow{p} Z)\mapsto  p_\otimes f^*,
	\]
	be a functor, and let $\scr A_X$ denote the restriction of $\scr A$ to $\Span(\scr C_X,\all,\rright)$.
	Then there exists an $(\infty,2)$-functor
	\[
	\mathbf{Span}(\scr C,\lleft,\rright) \to \what{\mathbf{Cat}}{}_\infty
	\]
	sending $X\xleftarrow{f} Y\xrightarrow{p} Z$ to 
	\[
	\Sect(\scr A_X) \xrightarrow{f^*} \Sect(\scr A_Y) \xrightarrow{p_\otimes} \Sect(\scr A_Z),
	\]
	where $f^*$ is the restriction functor and $p_\otimes$ lifts the functor $p_\otimes\colon\scr A(Y)\to\scr A(Z)$.
	Moreover, $f^*$ and $p_\otimes$ preserve the sections that are cocartesian over backward morphisms.
\end{theorem}

\begin{remark}
	When $\scr C=\Sch$, the assumption on the class of right morphisms in Theorem~\ref{thm:norm-pullback2} holds for finite étale morphisms, but not for finite locally free morphisms.
	In particular, the theorem does not apply to $\H_\pt^\otimes\colon \Span(\Sch,\all,\flf)\to\what\Cat_\infty$.
\end{remark}

If $f\colon Y\to X$ is a right morphism in $\scr C$, then the spans
\[
Y\xleftarrow{\id} Y\xrightarrow f X\quad\text{and}\quad X\xleftarrow{f}Y \xrightarrow{\id} Y
\]
are adjoint in $\mathbf{Span}(\scr C,\lleft,\rright)$. In particular, we deduce from Theorem~\ref{thm:norm-pullback2} an adjunction
	\[
	f_\otimes :\Sect(\scr A_Y)\rightleftarrows \Sect(\scr A_X): f^*,
	\]
as well as an induced adjunction between the full subcategories of sections that are cocartesian over backward morphisms.
Thus, Theorem~\ref{thm:norm-pullback} follows from Theorem~\ref{thm:norm-pullback2} applied to the functor $\SH^\otimes\colon \Span(\scr C,\all,\fet)\to\what\Cat_\infty$.

\begin{example}\label{ex:enhanced-monoidal-structure-on-NAlg}
	Let $\scr C=\Sch$ with smooth morphisms as left morphisms and finite étale morphisms as right morphisms, and let $\scr A=\SH^\otimes\colon \Span(\Sch,\all,\fet)\to\what\Cat{}_\infty$. Then we have an $(\infty,2)$-functor
	\[
	\NAlg_\Sm(\SH)^\otimes\colon \mathbf{Span}(\Sch,\mathrm{smooth},\fet) \to \what{\mathbf{Cat}}{}_\infty,\quad S\mapsto \NAlg_\Sm(\SH(S)),\quad (X\xleftarrow{f} Y\xrightarrow{p} Z)\mapsto  p_\otimes f^*.
	\]
	Similarly, there are $(\infty,2)$-functors $\NAlg_{\Sch}(\SH)^\otimes$ and $\NAlg_{\FEt}(\SH)^\otimes$ defined on $\mathbf{Span}(\Sch,\all,\fet)$ and $\mathbf{Span}(\Sch,\fet,\fet)$, respectively.
\end{example}

The construction of the $(\infty,2)$-functor of Theorem~\ref{thm:norm-pullback2} will be divided into two steps:
\begin{equation} \label{eq:plan-alpha-beta-gamma}
\mathbf{Span}(\scr C,\lleft,\rright) \to (\mathbf{Cat}_{\infty}^{\sslash \Span(\scr C,\all,\rright)})^{\oop} \to \what{\mathbf{Cat}}{}_\infty.
\end{equation}
Here, $\mathbf{Cat}_{\infty}^{\sslash \Span(\scr C,\all,\rright)}$ denotes the left-lax slice $(\infty,2)$-category (to be defined momentarily), and ``$\oop$'' inverts the direction of the $1$-morphisms. The first functor sends $X$ to $\Span(\scr C_X,\all,\rright)$ and will be constructed using the universal property of the $(\infty,2)$-category of spans. The second functor sends $t\colon \scr D\to\Span(\scr C,\all,\rright)$ to $\Sect(\scr A\circ t)$. It is this refined $(\infty,2)$-functoriality of $\Sect(\ph)$ that is ultimately responsible for the automatic adjunction.

Let $\mathbf C$ be an $(\infty,2)$-category and $X\in\mathbf C$. The \emph{left-lax slice $(\infty,2)$-category} $\mathbf C^{\sslash X}$ is 
\[
\mathbf C^{\sslash X}=\Fun(\Delta^1, \mathbf C)_\text{left-lax} \times_{\mathbf C} \{X\},\]
 where $\Fun(\Delta^1, \mathbf C)_\text{left-lax}$ is the $(\infty,2)$-category of strict functors and left-lax natural transformations, and $\Fun(\Delta^1, \mathbf C)_\text{left-lax} \to \mathbf C$ is evaluation at $1$ \cite[\sect A.2.5.1.1]{GRderalg}.
Thus, an object of $\mathbf C^{\sslash X}$ is a pair $(Y,t)$ where $Y\in\mathbf C$ and $t\colon Y\to X$, and the mapping $\infty$-categories are given by
 \begin{equation}\label{eqn:lax-slice-maps}
 \MAP_{\mathbf C^{\sslash X}}((Y,t),(Z,u)) \simeq \MAP(Y,Z) \times_{\MAP(Y,X)} \Fun(\Delta^1,\MAP(Y,X)) \times_{\MAP(Y,X)} \{t\},
 \end{equation}
 with the obvious composition law. Explicitly:
\begin{itemize}
	\item a $1$-morphism $(Y,t)\to (Z,u)$ is a pair $(f,\epsilon)$ where $f\colon Y\to Z$ and $\epsilon\colon u\circ f\to t$;
	\item a $2$-morphism $(f,\epsilon)\to (f',\epsilon')$ between $1$-morphisms $(Y,t)\to (Z,u)$ is a pair $(\phi,\alpha)$ where $\phi\colon f\to f'$ and $\alpha\colon\epsilon\simeq \epsilon'(u\phi)$.
\end{itemize}
The forgetful functor $(\mathbf C^{\sslash X})^\oop\to\mathbf C^\oop$ is the $1$-cocartesian fibration classified by $\MAP(-,X)$ \cite[\sect A.2.5.1.5]{GRderalg}.

\begin{lemma}\label{lem:swallowtail}
	Let $\mathbf C$ be an $(\infty,2)$-category and let $f\colon X\to Y$ be a morphism in $\mathbf C$ with right adjoint $g\colon Y\to X$ and counit $\epsilon\colon fg\to \id_Y$. Then the $1$-morphisms
	\[
	\begin{tikzpicture}[baseline=(N.south)]
		\def\colsep{1.5em}
		\diagram{
		X & \relax & Y \\ & Y & \\
		};
		\arrows (11-) edge node[above]{$f$} (-13) (11) edge node[below left,pos=.4]{$f$} (22) (13) edge node[below right,pos=.4]{$\id$} (22)
		(12) edge[draw=none] node(N)[pos=.4,rotate=45,font=\normalsize]{$\Leftarrow$} node[pos=.4,above left]{$\id$} (22);
	\end{tikzpicture}
	\quad\text{and}\quad
	\begin{tikzpicture}[baseline=(N.south)]
		\def\colsep{1.5em}
		\diagram{
		Y & \relax & X \\ & Y & \\
		};
		\arrows (11-) edge node[above]{$g$} (-13) (11) edge node[below left,pos=.4]{$\id$} (22) (13) edge node[below right,pos=.4]{$f$} (22)
		(12) edge[draw=none] node(N)[pos=.4,rotate=45,font=\normalsize]{$\Leftarrow$} node[pos=.4,above left]{$\epsilon$} (22);
	\end{tikzpicture}
	\]
	in $\mathbf C^{\sslash Y}$ are adjoint.
\end{lemma}

\begin{proof}
	Choose a unit $\eta\colon \id_X \to gf$ and equivalences
	\[
	\alpha\colon \id_f\simeq (\epsilon f)(f\eta) \quad\text{and}\quad \beta\colon \id_g\simeq (g\epsilon)(\eta g)
	\]
	witnessing the triangle identities.
	By the swallowtail coherence for adjunctions in $3$-categories, once $\epsilon$, $\eta$, and $\alpha$ are chosen, we can choose $\beta$ so that the composite
	\begin{equation}\label{eqn:swallowtail}
	\epsilon \stackrel\alpha\simeq \epsilon(\epsilon fg)(f\eta g) \simeq \epsilon(fg\epsilon)(f\eta g)\stackrel{\beta^{-1}}\simeq \epsilon
	\end{equation}
	is homotopic to $\id_\epsilon$ (see \cite[\sect 5]{Lack} or \cite[Remark 2.2]{Gurski})\footnote{For $\mathbf C=\mathbf{Cat}_\infty$, which suffices for our purposes, the complete higher coherence of adjunctions, including swallowtail coherence, is formulated in \cite[\sect1.1]{RiehlVerity}.}.
	
	We claim that the $2$-morphisms 
	\[
	(\eta,\alpha)\colon\id_{(X,f)}\to (g,\epsilon)(f,\id_f)\quad\text{and}\quad (\epsilon,\id_\epsilon)\colon (f,\id_f)(g,\epsilon)\to \id_{(Y,\id)}
	\]
	are the unit and counit for an adjunction between $(f,\id_f)$ and $(g,\epsilon)$ in $\mathbf C^{\sslash Y}$. We have
	\[
	((\epsilon,\id_\epsilon)(f,\id_f))((f,\id_f)(\eta,\alpha)) = ((\epsilon f)(f\eta), \alpha),
	\]
	which is homotopic to $(\id_f,\id_{\id_f})=\id_{(f,\id_f)}$ via $\alpha^{-1}$. For the other triangle identity, we have
	\[
	((g,\epsilon)(\epsilon,\id_\epsilon))((\eta,\alpha) (g,\epsilon)) = ((g\epsilon)(\eta g),\gamma)
	\]
	where $\gamma$ is the composite homotopy
	\[
	\epsilon \stackrel\alpha\simeq \epsilon(\epsilon fg)(f\eta g) \simeq \epsilon(fg\epsilon)(f\eta g).
	\]
	By~\eqref{eqn:swallowtail}, the $2$-morphism $((g\epsilon)(\eta g),\gamma)$ is homotopic to $(\id_g,\id_\epsilon)=\id_{(g,\epsilon)}$ via $\beta^{-1}$.
	This concludes the proof.
\end{proof}

The following proposition gives the first part of~\eqref{eq:plan-alpha-beta-gamma}.

\begin{proposition}\label{prop:construct-beta}
	Under the assumptions of Theorem~\ref{thm:norm-pullback2}, the functor
	\[
	\scr C_\lleft^\op \to (\Cat_{\infty/\Span(\scr C,\all,\rright)})^\op, \quad X\mapsto \Span(\scr C_X,\all,\rright),\quad (X\xleftarrow fY)\mapsto f_\sharp,
	\]
	extends uniquely to an $(\infty,2)$-functor
	\[
	\mathbf{Span}(\scr C,\lleft,\rright) \to (\mathbf{Cat}_\infty^{\sslash\Span(\scr C,\all,\rright)})^\oop,
	\]
	which sends a forward right morphism $p\colon Y\to Z$ to $p^*\colon \Span(\scr C_Z,\all,\rright)\to \Span(\scr C_Y,\all,\rright)$.
\end{proposition}

\begin{proof}
	For any right morphism $f\colon Y\to X$, the unit and counit of the adjunction $f_\sharp:\scr C_Y\adj\scr C_X: f^*$ are cartesian transformations, and their components are right morphisms (because the diagonal of a right morphism is a right morphism).
	It follows from Corollary~\ref{cor:span-adjunctions}(1) that there is an induced adjunction
	\[
	f_\sharp: \Span(\scr C_Y,\all,\rright) \adj \Span(\scr C_X,\all,\rright): f^*.
	\]
	By Lemma~\ref{lem:swallowtail}, we can lift this adjunction to $\mathbf{Cat}_\infty^{\sslash\Span(\scr C,\all,\rright)}$.
	Moreover, given a cartesian square
	\begin{tikzmath}
		\diagram{Y' & X' \\ Y & X \\};
		\arrows (11-) edge node[above]{$f'$} (-12) (11) edge node[left]{$q$} (21) (21-) edge node[below]{$f$} (-22) (12) edge node[right]{$p$} (22);
	\end{tikzmath}
	in $\scr C$ where $f$ is a left morphism and $p$ is a right morphism, the canonical transformation
	\[
	f'_\sharp q^* \to p^* f_\sharp\colon  \Span(\scr C_Y,\all,\rright) \to \Span(\scr C_{X'},\all,\rright)
	\]
	is clearly an equivalence.
	The proposition now follows from the universal property of the $(\infty,2)$-category of spans \cite[Theorem V.1.3.2.2]{GRderalg}.
\end{proof}

The following proposition gives the second part of~\eqref{eq:plan-alpha-beta-gamma}.

\begin{proposition} \label{prop:construct-gamma}
Let $p\colon \scr X\to\scr B$ be a cocartesian fibration of $\infty$-categories.
Then there is an $(\infty,2)$-functor
\[
(\mathbf{Cat}_\infty^{\sslash\scr B})^\oop \to \mathbf{Cat}_\infty, \quad \scr D\mapsto\Fun_{\scr B}(\scr D,\scr X),
\]
with the following properties:
\begin{enumerate}
	\item Its restriction to $(\mathbf{Cat}_{\infty/\scr B})^\oop$ is the $(\infty,2)$-functor represented by $p$.
	\item For any $\infty$-category $\scr D$, its restriction to $\Fun(\scr D,\scr B)$ classifies the cocartesian fibration \[p_*\colon\Fun(\scr D,\scr X)\to\Fun(\scr D,\scr B).\]
	\item The image of a $1$-morphism $(f,\epsilon)\colon (\scr D,t)\to(\scr E,u)$ is a functor $(f,\epsilon)^*\colon \Fun_{\scr B}(\scr E,\scr X)\to \Fun_{\scr B}(\scr D,\scr X)$ such that $(f,\epsilon)^*(s)$ sends $d\in\scr D$ to $\epsilon(d)_* s(f(d))$ and $e\colon d_1\to d_2$ to
\[
t(e)_*\epsilon(d_1)_* s(f(d_1))\simeq \epsilon(d_2)_* u(f(e))_* s(f(d_1))\xrightarrow{s(f(e))} \epsilon(d_2)_* s(f(d_2)).
\]
	\item The image of a $2$-morphism $(\phi,\alpha)\colon (f,\epsilon)\to (f',\epsilon')$ is a natural transformation $(f,\epsilon)^*\to (f',\epsilon')^*$ whose component at $s\in \Fun_{\scr B}(\scr E,\scr X)$ and $d\in\scr D$ is the map
\[
\epsilon(d)_* s(f(d)) \stackrel\alpha\simeq \epsilon'(d)_*u(\phi(d))_* s(f(d)) \xrightarrow{s(\phi(d))} \epsilon'(d)_* s(f'(d)).
\]
\end{enumerate}
\end{proposition}

\begin{proof}
Consider the $(\infty,2)$-functor
\begin{equation*}
p_*\colon (\mathbf{Cat}_\infty^{\sslash\scr X})^\oop\to (\mathbf{Cat}_\infty^{\sslash\scr B})^\oop
\end{equation*}
induced by $p$.
Note that the fiber of $p_*$ over $(\scr D,t)$ is the $\infty$-category of sections $\Fun_{\scr B}(\scr D,\scr X)$.
We claim that $p_*$ is a $1$-cocartesian fibration classified by the desired $(\infty,2)$-functor.
Properties (1) and (2) will follow from the cartesian squares
\begin{tikzmath}
	\diagram{
	 ((\mathbf{Cat}_{\infty/\scr B})^{\sslash p})^\oop & (\mathbf{Cat}_\infty^{\sslash\scr X})^\oop & \Fun(\scr D,\scr X) \\
	 (\mathbf{Cat}_{\infty/\scr B})^\oop & (\mathbf{Cat}_\infty^{\sslash\scr B})^\oop & \Fun(\scr D,\scr B) \\
	 & \mathbf{Cat}_\infty^\oop & \{\scr D\}\rlap. \\
	};
	\arrows (11) edge (21) (12) edge node[left]{$p_*$} (22) (13) edge node[right]{$p_*$} (23)
	(11-) edge[c->] (-12) (21-) edge[c->] (-22) (-13) edge (12-) (-23) edge (22-)
	(22) edge (32) (23) edge (33) (-33) edge (32-);
\end{tikzmath}
If $(\scr D,t),(\scr E,u)\in \mathbf{Cat}_\infty^{\sslash\scr X}$, it is easy to see using~\eqref{eqn:lax-slice-maps} that the functor
\[
p_*\colon \MAP((\scr D,t),(\scr E,u)) \to \MAP((\scr D,p\circ t),(\scr E,p\circ u))
\]
is the right fibration classified by
\begin{equation}\label{eqn:1-cocart}
\MAP((\scr D,p\circ t),(\scr E,p\circ u))^\op \to \scr S,\quad (f,\epsilon) \mapsto \Map(u\circ f,t)\times_{\Map(p\circ u\circ f,p\circ t)} \{\epsilon\}.
\end{equation}
It remains to show that $p_*$ restricts to a cocartesian fibration between the underlying $\infty$-categories:
\begin{equation}\label{eqn://p}
p_*\colon (\Cat_\infty^{\sslash\scr X})^\op\to (\Cat_\infty^{\sslash\scr B})^\op.
\end{equation}
This is the morphism of cocartesian fibrations over $\Cat_\infty^\op$ classified by the natural transformation
\[
p_*\colon \Fun(\ph,\scr X)\to \Fun(\ph,\scr B).
\]
We now use the assumption that $p$ is a cocartesian fibration: by \cite[Proposition 3.1.2.1]{HTT}, for every $\infty$-category $\scr D$, the functor $p_*\colon \Fun(\scr D,\scr X)\to\Fun(\scr D,\scr B)$ is a cocartesian fibration, and for every functor $f\colon\scr D\to\scr E$, the functor $f^*\colon\Fun(\scr E,\scr X)\to\Fun(\scr D,\scr X)$ preserves cocartesian edges. It follows from \cite[Lemma 1.4.14]{LurieGoodwillie} that~\eqref{eqn://p} is indeed a cocartesian fibration.
Moreover, an arrow $(f,\epsilon)\colon (\scr D,t)\to (\scr E,u)$ in $\Cat_\infty^{\sslash\scr X}$ is $p_*$-cocartesian if and only if, for every $d\in\scr D$, the morphism $\epsilon(d)\colon u(f(d))\to t(d)$ in $\scr X$ is $p$-cocartesian.
Combining this with~\eqref{eqn:1-cocart}, we easily obtain the desired description of the $(\infty,2)$-functor on $1$-morphisms and on $2$-morphisms.
\end{proof}

\begin{proof}[Proof of Theorem~\ref{thm:norm-pullback2}]
	We define the desired $(\infty,2)$-functor as the composition
	\[
	\mathbf{Span}(\scr C,\lleft,\rright) \to (\mathbf{Cat}_{\infty}^{\sslash \Span(\scr C,\all,\rright)})^{\oop} \to \what{\mathbf{Cat}}{}_\infty,
	\]
	where the first functor is given by Proposition~\ref{prop:construct-beta} and the second by Proposition~\ref{prop:construct-gamma} applied to the cocartesian fibration classified by $\scr A$.
	The additional claims about $f^*\colon\Sect(\scr A_X) \to \Sect(\scr A_Y)$ and ${p_\otimes}\colon\Sect(\scr A_Y)\to \Sect(\scr A_Z)$ follow immediately from the explicit description of the second functor on $1$-morphisms given in Proposition~\ref{prop:construct-gamma}(3).
\end{proof}

\subsection{The pullback–pushforward adjunction}
\label{sub:pull-push}

We now turn to the proof of Theorem~\ref{thm:second-adjunction}. It is again an instance of a more general result:

\begin{proposition}\label{prop:push-pull-adjunction}
	Let $p\colon\scr X\to\scr C$ be a cocartesian fibration and let $f:\scr C\rightleftarrows\scr D:g$ be an adjunction with unit $\eta\colon\id\to gf$. Suppose that, for every $c\in\scr C$, the functor $\eta(c)_*\colon\scr X_c\to\scr X_{gf(c)}$ has a right adjoint $\eta(c)^!$, giving rise to a relative adjunction $\eta_*\colon \scr X\rightleftarrows f^*g^*\scr X: \eta^!$ over $\scr C$. Then there is an adjunction
	\[
	g^*: \Fun_{\scr C}(\scr C,\scr X) \rightleftarrows \Fun_{\scr C}(\scr D,\scr X): \eta^!\circ f^*.
	\]
\end{proposition}

\begin{proof}
	We use the $(\infty,2)$-functor
	\[
	(\mathbf{Cat}_\infty^{\sslash \scr C})^\oop\to\what{\mathbf{Cat}}{}_\infty
	\]
	from Proposition~\ref{prop:construct-gamma}. By assertions (2) and (1) of that proposition, the functors $\eta_*$ and $f^*$ between $\infty$-categories of sections are the images of the $1$-morphisms
	\[
	\begin{tikzpicture}[baseline=(N.south)]
		\def\colsep{1.5em}
		\diagram{
		\scr C & \relax & \scr C \\ & \scr C & \\
		};
		\arrows (11-) edge node[above]{$\id$} (-13) (11) edge node[below left,pos=.4]{$gf$} (22) (13) edge node[below right,pos=.4]{$\id$} (22)
		(12) edge[draw=none] node(N)[pos=.4,rotate=45,font=\normalsize]{$\Leftarrow$} node[pos=.4,above left]{$\eta$} (22);
	\end{tikzpicture}
	\quad\text{and}\quad
	\begin{tikzpicture}[baseline=(N.south)]
		\def\colsep{1.5em}
		\diagram{
		\scr C & \relax & \scr D \\ & \scr C & \\
		};
		\arrows (11-) edge node[above]{$f$} (-13) (11) edge node[below left,pos=.4]{$gf$} (22) (13) edge node[below right,pos=.4]{$g$} (22)
		(12) edge[draw=none] node(N)[pos=.4,rotate=45,font=\normalsize]{$\Leftarrow$} node[pos=.4,above left]{$\id$} (22);
	\end{tikzpicture}
	\]
	in $\mathbf{Cat}_\infty^{\sslash \scr C}$.
	By Lemma~\ref{lem:swallowtail}, the second $1$-morphism has a right adjoint $(g,g\epsilon)$ in $\mathbf{Cat}_\infty^{\sslash \scr C}$, inducing a left adjoint $f_!$ of $f^*$.
	Using the triangle identity $(\eta g)(g\epsilon)\simeq \id_g$, we obtain an equivalence
	\[
	\begin{tikzpicture}[baseline=(N.south)]
		\def\colsep{1.5em}
		\diagram{
		\scr C & \relax & \scr C \\ & \scr C & \\
		};
		\arrows (11-) edge node[above]{$\id$} (-13) (11) edge node[below left,pos=.4]{$gf$} (22) (13) edge node[below right,pos=.4]{$\id$} (22)
		(12) edge[draw=none] node(N)[pos=.4,rotate=45,font=\normalsize]{$\Leftarrow$} node[pos=.4,above left]{$\eta$} (22);
	\end{tikzpicture}
	\;\circ\;
	\begin{tikzpicture}[baseline=(N.south)]
		\def\colsep{1.5em}
		\diagram{
		\scr D & \relax & \scr C \\ & \scr C & \\
		};
		\arrows (11-) edge node[above]{$g$} (-13) (11) edge node[below left,pos=.4]{$g$} (22) (13) edge node[below right,pos=.4]{$gf$} (22)
		(12) edge[draw=none] node(N)[pos=.4,rotate=45,font=\normalsize]{$\Leftarrow$} node[pos=.4,above left]{$g\epsilon$} (22);
	\end{tikzpicture}
	\;\simeq\;
	\begin{tikzpicture}[baseline=(N.south)]
		\def\colsep{1.5em}
		\diagram{
		\scr D & \relax & \scr C \\ & \scr C & \\
		};
		\arrows (11-) edge node[above]{$g$} (-13) (11) edge node[below left,pos=.4]{$g$} (22) (13) edge node[below right,pos=.4]{$\id$} (22)
		(12) edge[draw=none] node(N)[pos=.4,rotate=45,font=\normalsize]{$\Leftarrow$} node[pos=.4,above left]{$\id$} (22);
	\end{tikzpicture}
	\]
	of $1$-morphisms in $\mathbf{Cat}_\infty^{\sslash \scr C}$, whence an equivalence of functors $f_!\circ \eta_*\simeq g^*$. We deduce that $\eta^!\circ f^*$ is right adjoint to $g^*$, as desired.
\end{proof}

\begin{remark}
	One can give a different proof of Proposition~\ref{prop:push-pull-adjunction} using the theory of relative Kan extensions \cite[\sect4.3.2]{HTT}: the right adjoint functor $\Fun_{\scr C}(\scr D,\scr X) \to \Fun_{\scr C}(\scr C,\scr X)$ sends a section to its $p$-right Kan extension along $g$.
	The existence of these $p$-right Kan extensions amounts to the existence of $p$-cartesian lifts of $\eta(c)$, which is guaranteed by \cite[Corollary 5.2.2.4]{HTT}.
\end{remark}

\begin{proof}[Proof of Theorem~\ref{thm:second-adjunction}]
	By Corollary~\ref{cor:span-adjunctions}(2), the adjunction $f_\sharp:\scr C'\adj\scr C:f^*$ induces an adjunction
	\[
	f^*: \Span(\scr C, \all, \fet) \adj \Span(\scr C', \all, \fet): f_\sharp.
	\]
	Applying Proposition~\ref{prop:push-pull-adjunction} to the cocartesian fibration over $\Span(\scr C,\all,\fet)$ classified by $\SH^\otimes$, we obtain an adjunction
	\[
	f^*: \Sect(\SH^\otimes|\Span(\scr C,\all,\fet)) \adj \Sect(\SH^\otimes|\Span(\scr C',\all,\fet)) : f_*,
	\]
	where $f_*$ is precisely the functor described in Proposition~\ref{prop:categorical-props}(8). 
\end{proof}

\section{Spectra over profinite groupoids}
\label{sec:galois-equiv}

In this section, we will repeat the construction of \sect\ref{sub:coherence} in the setting of equivariant homotopy theory.
For the purpose of comparison with motivic homotopy theory (see Section~\ref{sec:ggt}), it is natural to consider equivariant homotopy theory parametrized by \emph{profinite groupoids}. By a profinite groupoid, we mean a pro-object in the $2$-category $\FinGpd$ of groupoids with finite $\pi_0$ and $\pi_1$. 
We denote by $\Pro(\FinGpd)$ the $2$-category of profinite groupoids. 

To every profinite groupoid $X$, we will associate a symmetric monoidal $\infty$-category $\SH(X)$ of \emph{$X$-spectra}. If $X= \B G$ for some finite group $G$, then $\SH(X)$ will be the usual symmetric monoidal $\infty$-category of (genuine) $G$-spectra. Moreover, for every ``finitely presented'' map $p\colon T\to S$ of profinite groupoids, we will define a norm functor $p_\otimes\colon \SH(T)\to\SH(S)$, and the coherence of these norms will be encoded by a functor
\[
\SH^\otimes\colon \Span(\Pro(\FinGpd), \all,\fp) \to \CAlg(\what\Cat{}_\infty^\mathrm{sift}), \quad X\mapsto \SH(X),\quad (U\stackrel f\from T\stackrel p\to S)\mapsto p_\otimes f^*.
\]
If $H\subset G$ is a subgroup and $p\colon \B H\to \B G$ is the induced finite covering map, then $p_\otimes$ will be the Hill–\<Hopkins–\<Ravenel norm functor $\Norm_H^G$. On the other hand, if $N\subset G$ is a normal subgroup and $p\colon \B G\to \B(G/N)$ is the induced map, then $p_\otimes$ will be the geometric fixed point functor $\Phi^N$.

\begin{remark}
	We do not aim to provide here a self-contained treatment of equivariant homotopy theory. In fact, we will feel free to use basic facts from equivariant homotopy theory in our proofs. We refer to \cite{lewis1986equivariant} for a classical treatment of the subject, and to the recent series \cite{BDGNSintro} for a different approach using spectral Mackey functors. These two points of view are connected by a theorem of Guillou and May \cite[Theorem 0.1]{GuillouMay} refined by Nardin \cite[Theorem A.4]{BDGNS4}, whose content we recall in Proposition~\ref{prop:spectralMackey} below. Both points of view will be useful in the sequel.
\end{remark}

\subsection{Profinite groupoids}

We start with some preliminaries on profinite groupoids. Profinite groupoids are in particular profinite $\infty$-groupoids, and we refer the reader to \cite[Appendix E]{SAG} for a comprehensive treatment of the latter; see also \cite[\sect A.8.1]{SAG} for a discussion of pro-objects in $\infty$-categories.
We shall say that a morphism in $\Pro(\scr S)$ is \emph{finitely presented} if it is the pullback of a morphism in $\scr S\subset\Pro(\scr S)$. 

\begin{lemma}\label{lem:fincov}
	The class of finitely presented maps in $\Pro(\scr S)$ is closed under composition and base change. Moreover, if $f\circ g$ and $f$ are finitely presented, then $g$ is finitely presented.
\end{lemma}

\begin{proof}
	The class of finitely presented maps is closed under base change by definition.
	Suppose $f\colon Y\to X$ is a finitely presented map and $g\colon Z\to Y$ an arbitrary map in $\Pro(\scr S)$.
	We must show that $g$ is finitely presented if and only if $f\circ g$ is.
	Write $X$ as a cofiltered limit of $\infty$-groupoids $X_\alpha$, $\alpha\in A$. Then there exists $0\in A$ and a finitely presented map $Y_0\to X_0$ such that $Y\simeq X\times_{X_0}Y_0$.
	If we set $Y_\alpha=X_\alpha\times_{X_0}Y_0$ for $\alpha\geq 0$, it follows that $Y\simeq\lim_\alpha Y_\alpha$. 
	
	Suppose that $g$ is finitely presented. Then there exists $1\geq 0$ and a map of $\infty$-groupoids $Z_1\to Y_1$ such that $Z\simeq Y\times_{Y_1}Z_1$. In the commutative diagram
	\begin{tikzmath}
		\diagram{
		Z & Z_1 & \\
		Y & Y_1 & Y_0 \\
		X & X_1 & X_0\rlap, \\
		};
		\arrows (11-) edge (-12) (12) edge (22) (11) edge node[left]{$g$} (21) (21) edge node[left]{$f$} (31) (21-) edge (-22) (22-) edge (-23) (22) edge (32) (23) edge (33) (31-) edge (-32) (32-) edge (-33);
	\end{tikzmath}
	the upper square, the right-hand square, and the horizontal rectange are cartesian, hence the vertical rectangle is also cartesian. It follows that $f\circ g$ is finitely presented.
	
Conversely, suppose that $f\circ g$ is finitely presented. We may then assume that $Z\simeq X\times_{X_0}Z_0$ for some map of $\infty$-groupoids $Z_0\to X_0$, and we let $Z_\alpha=X_\alpha\times_{X_0}Z_0$ for $\alpha\geq 0$. Then $Z\simeq\lim_\alpha Z_\alpha$ in $\Pro(\scr S_{/X_0})$, and so there exists $1\geq 0$ such that the map $Z\to Y\to Y_0$ factors as $Z\to Z_1\to Y_0$ over $X_0$. In the commutative diagram
\begin{tikzmath}
	\diagram{Z & Z_1 & Z_0 \\ Y & Y_1 & Y_0 \\ X & X_1 & X_0\rlap, \\};
	\arrows (11-) edge (-12) (12-) edge (-13) (11) edge node[left]{$g$} (21) (12) edge[dashed] (22)
	(21-) edge (-22) (22-) edge (-23)
	(21) edge node[left]{$f$} (31) (22) edge (32) (23) edge (33) (12) edge (23) (31-) edge (-32) (32-) edge (-33);
	\draw[->] (12) to[out=300,in=60] (32);
	\draw[->] (13) to[out=300,in=60] (33);
\end{tikzmath}
the bottom right square is cartesian and hence the map $Z_1\to Y_1$ exists as indicated. Moreover, since the boundary square, the right vertical rectangle, and the bottom horizontal rectangle are cartesian, the upper left square is also cartesian.
It follows that $g$ is finitely presented.
\end{proof}

If $\scr C$ is an $\infty$-category, the functor $\Fun(\ph,\scr C)\colon \scr S^\op\to\what\Cat{}_\infty$ extends uniquely to a functor on $\Pro(\scr S)^\op$ that preserves filtered colimits, and we will use the same notation for this extension. 
If $X\in\Pro(\scr S)$, we then have a functor
\[
{\textstyle\int}\colon\Fun(X,\scr S) \to \Pro(\scr S)_{/X},
\]
defined by pulling back the universal left fibration $\scr S_\pt\to\scr S$ \cite[Corollary 3.3.2.7]{HTT}.

\begin{lemma}\label{lem:proGrothendieck}
	Let $X\in\Pro(\scr S)$. Then $\int\colon\Fun(X,\scr S) \to \Pro(\scr S)_{/X}$ is fully faithful and its image consists of the finitely presented morphisms $Y\to X$.
\end{lemma}

\begin{proof}
	It is clear that the essential image of this functor is as described.
	Write $X$ as a cofiltered limit of $\infty$-groupoids $X_\alpha\in\scr S$, $\alpha\in A$.
	Let $Y,Z\in\Fun(X,\scr S)$, and choose $0\in A$ such that $Y$ and $Z$ come from functors $X_0\to\scr S$ classifying $Y_0\to X_0$ and $Z_0\to X_0$. For $\alpha\geq 0$, let $Y_\alpha= X_\alpha\times_{X_0}Y_0$ and $Z_\alpha= X_\alpha\times_{X_0}Z_0$. Then $\int Y =\lim_\alpha Y_\alpha$, $\int Z=\lim_\alpha Z_\alpha$, and the effect of $\int$ on mapping spaces is the obvious map
	\[
	\Map(Y,Z)\simeq \colim_{\alpha} \Map_{X_\alpha}(Y_\alpha,Z_\alpha) \to \lim_\alpha\colim_{\beta\geq \alpha} \Map_{X_\alpha}(Y_\beta,Z_\alpha)\simeq {\textstyle\Map(\int Y,\int Z)}.
	\]
	This map is an equivalence since $\Map_{X_\alpha}(Y_\beta,Z_\alpha)\simeq \Map_{X_0}(Y_\beta,Z_0)$ for all $\beta\geq\alpha$.
\end{proof}

It follows from Lemma~\ref{lem:proGrothendieck} that a finitely presented morphism $Y\to X$ in $\Pro(\scr S)$ has well-defined \emph{fibers} in $\scr S$, namely the values of the corresponding functor $X\to\scr S$. 
Hence, any full subcategory $\scr C\subset \scr S$ defines a pullback-stable class of finitely presented maps in $\Pro(\scr S)$, namely those whose fibers are in $\scr C$.
A morphism in $\Pro(\scr S)$ will be called a \emph{finite covering map} if it is finitely presented and its fibers are (possibly empty) finite sets.

The $\infty$-category $\Fun(X,\scr S)$ has finite limits and is extensive (see Definition~\ref{def:extensive}). If $\scr C$ is an $\infty$-category with finite limits, we may use the construction from \cite[\sect2]{BarwickGlasmanShah} to obtain a symmetric monoidal structure on the $\infty$-category $\Span(\scr C)$ such that the inclusion $\scr C\into\Span(\scr C)$ is symmetric monoidal (for the cartesian symmetric monoidal structure on $\scr C$). In fact, this construction defines a functor
\begin{equation}\label{eqn:Span}
\Span\colon \Cat_\infty^\mathrm{lex} \to \CAlg(\Cat_\infty),
\end{equation}
where $\Cat_\infty^\mathrm{lex}$ is the $\infty$-category of $\infty$-categories with finite limits and left exact functors.

Note that there is an equivalence
\[
\Fun(X,\scr S)_+ \simeq \Span(\Fun(X,\scr S),\mathrm{mono},\all),\quad Y_+\mapsto Y,\quad (f\colon Y_+\to Z_+)\mapsto (Y\hookleftarrow f^{-1}(Z) \to Z),
\]
where ``$\mathrm{mono}$'' is the class of monomorphisms.
In particular, we can identify $\Fun(X,\scr S)_+$ with a wide subcategory\footnote{If $\scr C$ is an $\infty$-category, a \emph{wide subcategory} of $\scr C$ is a functor $\scr D\to \scr C$ that induces an equivalence on spaces of objects and a monomorphism on spaces of arrows.} of $\Span(\Fun(X,\scr S))$. 
Moreover, since products preserve monomorphisms, $\Fun(X,\scr S)_+$ is closed under the tensor product in $\Span(\Fun(X,\scr S))$, and we obtain symmetric monoidal embeddings
\[
\Fun(X,\scr S) \xrightarrow{(\ph)_+} \Fun(X,\scr S)_+ \into \Span(\Fun(X,\scr S)),
\]
natural in $X\in \Pro(\scr S)^\op$.

\begin{lemma}\label{lem:FinX}
	Let $p\colon Y\to X$ be a finitely presented map of pro-$\infty$-groupoids.
	\begin{enumerate}
		\item The pullback functor $p^*\colon \Fun(X,\scr S)\to\Fun(Y,\scr S)$ admits a right adjoint $p_*$ and a left adjoint $p_\sharp$.
		\item For every cartesian square of pro-$\infty$-groupoids
		\begin{tikzmath}
			\diagram{Y' & Y \\ X' & X\rlap, \\};
			\arrows (11-) edge node[above]{$g$} (-12) (11) edge node[left]{$q$} (21) (21-) edge node[below]{$f$} (-22) (12) edge node[right]{$p$} (22);
		\end{tikzmath}
		the exchange transformations
		\begin{align*}
			\Ex_*^*&\colon f^*p_* \to q_*g^*\colon \Fun(Y,\scr S)\to \Fun(X',\scr S),\\
			\Ex^*_\sharp&\colon q_\sharp g^*\to f^*p_\sharp\colon \Fun(Y,\scr S)\to \Fin(X',\scr S)
		\end{align*}
		 are equivalences.
		\item For every $A\in\Fun(X,\scr S)$ and $B\in\Fun(Y,\scr S)$, the canonical map
		\[
		p_\sharp(p^*(A)\times B) \to A\times p_\sharp(B)
		\]
		is an equivalence.
		\item The symmetric monoidal functor $p_\otimes=\Span(p_*)\colon \Span(\Fun(Y,\scr S))\to \Span(\Fun(X,\scr S))$ restricts to a functor $p_\otimes\colon \Fun(Y,\scr S)_+\to \Fun(X,\scr S)_+$.
	\end{enumerate}
\end{lemma}

\begin{proof}
	For $\infty$-groupoids, (1) and (2) are clear: the functors $p_\sharp$ and $p_*$ are the left and right Kan extension functors, which are computed by taking the (co)limits over the fibers of $p$.
	In general, write $X=\lim_{\alpha\in A} X_\alpha$ where $A$ is cofiltered and each $X_\alpha$ is an $\infty$-groupoid, choose $0\in A$ such that $p\colon Y\to X$ is the pullback of a map of $\infty$-groupoids $Y_0\to X_0$, and let $Y_\alpha=X_\alpha\times_{X_0}Y_0$. The functor $p^*\colon \Fun(X,\scr S)\to\Fun(Y,\scr S)$ is then the filtered colimit of the functors $p_\alpha^*\colon\Fun(X_\alpha,\scr S)\to\Fun(Y_\alpha,\scr S)$. By the base change properties (2), we have in fact a filtered diagram of adjunctions $p_{\alpha\sharp}\dashv p_\alpha^*\dashv p_{\alpha *}$, which induce adjunctions in the colimit. This proves (1) in general. We can then prove (2) by writing the given square as a cofiltered limit of cartesian squares of $\infty$-groupoids. 
	Under the identification of Lemma~\ref{lem:proGrothendieck}, the functor $p^*$ is the base change functor, and assertion (3) becomes obvious.
	To prove (4), it suffices to observe that $p_*$ preserves monomorphisms, since it preserves pullbacks.
\end{proof}

We now specialize the discussion to profinite groupoids. If $X\in \Pro(\FinGpd)$, we denote by 
\[\Fin_X=\Fun(X,\Fin)\subset \Fun(X,\scr S)\] 
the $1$-category of \emph{finite $X$-sets}. By Lemma~\ref{lem:proGrothendieck}, $\Fin_X$ can be identified with the category of finite coverings of $X$. 
If $p\colon Y\to X$ is a finitely presented map between profinite groupoids, the fibers of $p$ have finitely many connected components. It follows that the functor $p_*\colon \Fun(Y,\scr S)\to\Fun(X,\scr S)$ sends $\Fin_Y$ to $\Fin_X$, and we have an induced adjunction
\[
p^*: \Fin_X\rightleftarrows \Fin_Y: p_*.
\]
If in addition $p$ is a finite covering map, it is clear that the functor $p_\sharp\colon \Fun(Y,\scr S)\to\Fun(X,\scr S)$ sends $\Fin_Y$ to $\Fin_X$, and we have an induced adjunction
\[
p_\sharp: \Fin_Y\rightleftarrows \Fin_X: p^*.
\]

\subsection{Norms in stable equivariant homotopy theory}
\label{sub:equivariant-norms}

We now proceed to define $\SH^\otimes$ on profinite groupoids.
We start with the functor
\[
\Pro(\FinGpd)^\op \to \Cat_1,\quad X\mapsto \Fin_X.
\]
By Lemma~\ref{lem:fincov} and Lemma~\ref{lem:FinX}(1,2), we may use either \cite[Proposition 11.6]{BarwickMackey} or \cite[Theorem V.1.3.2.2]{GRderalg} to obtain the functor
\[
\Fin^\otimes\colon\Span(\Pro(\FinGpd),\all,\fp) \to \Cat_1,\quad X\mapsto\Fin_X, \quad (U\stackrel f\from T\stackrel p\to S)\mapsto p_*f^*,
\]
where ``$\fp$'' is the class of finitely presented maps.
Since $f^*$ and $p_*$ are left exact, this functor lands in the subcategory $\Cat_1^\mathrm{lex}$. Composing with the functor~\eqref{eqn:Span}, we obtain the functor
\[
	\Span(\Fin)^\otimes\colon\Span(\Pro(\FinGpd),\all,\fp) \to \CAlg(\Cat_2),\quad X\mapsto\Span(\Fin_X), \quad (U\stackrel f\from T\stackrel p\to S)\mapsto p_\otimes f^*,
\]
which by Lemma~\ref{lem:FinX}(4) admits a subfunctor $\Fin_+^\otimes\colon X\mapsto \Fin_{X+}$.

We set
\[
\H(X) = \PSh_\Sigma(\Fin_X) \quad\text{and}\quad \H_\pt(X) =\PSh_\Sigma(\Fin_{X+}).
\]
By Lemma~\ref{lem:C+}, we have $\H_\pt(X)\simeq \H(X)_\pt$.
Composing the functors $\Fin^\otimes$ and $\Fin_+^\otimes$ with $\PSh_\Sigma$ (see \sect\ref{sub:coherence}), we obtain
\[
\H^\otimes,\; \H_\pt^\otimes\colon \Span(\Pro(\FinGpd),\all,\fp) \to \CAlg(\what\Cat{}_\infty^\mathrm{sift}).
\] 
If $G$ is a finite group, then $\H(\B G)$ and $\H_\pt(\B G)$ are the usual symmetric monoidal $\infty$-categories of $G$-spaces and of pointed $G$-spaces. 
Moreover, if $f\colon \B G\to \B H$ is induced by a group homomorphism $G\to H$ with kernel $N$, then $f_\otimes\colon \H_\pt(\B G) \to \H_\pt(\B H)$ is the usual functor $\Norm_{G/N}^H\circ \Phi^N$.
In particular, if $p\colon *\to \B G$ is the canonical map, then $p_\otimes(\S^1)\in\H_\pt(\B G)$ is the regular representation sphere.

Finally, we let $\SH(X)$ be the presentably symmetric monoidal $\infty$-category obtained from $\H_\pt(X)$ by inverting $p_\otimes(\S^1)$ for every finite covering map $p\colon Y\to X$. If $G$ is a finite group, $\SH(\B G)$ is the usual symmetric monoidal $\infty$-category of $G$-spectra: our construction can be compared with the symmetric monoidal model category of symmetric spectra in $G$-spaces (with respect to the regular representation sphere), using \cite[Theorem 2.26]{Robalo}.

\begin{lemma}\label{lem:continuity}
	The functors $\H,\H_\pt,\SH\colon \Pro(\FinGpd)^\op \to \CAlg(\Pr^\mathrm{L})$ preserve filtered colimits and finite products.
\end{lemma}

\begin{proof}
	Let $\Cat_\infty^\amalg$ be the $\infty$-category of small $\infty$-categories with finite coproducts and functors that preserve finite coproducts. The inclusion $\Cat_\infty^\amalg\subset\Cat_\infty$ clearly preserves filtered colimits, and hence the functor $\Pro(\Fin\Gpd)^\op\to\Cat_\infty^\amalg$, $X\mapsto\Fin_X$, preserves filtered colimits.
	The functor $\PSh_\Sigma\colon \Cat_\infty^\amalg\to \Pr^\mathrm{L}$, being a partial left adjoint, preserves colimits. This proves the first claim for $\H$ and hence $\H_\pt$. If $X\simeq\lim_{\alpha\in A} X_\alpha$ where $A$ is cofiltered, every finite covering map $Y\to X$ is the pullback of a finite covering map $Y_\alpha\to X_\alpha$ for some $\alpha\in A$. Hence, $\{p_\otimes(\S^1)\}_{p\in\Fin_X}$ is the union of the collections $\{f_\alpha^*p_\otimes(\S^1)\}_{p\in\Fin_{X_\alpha}}$, where $f_\alpha\colon X\to X_\alpha$. Since $\SH(X)$ is the image of $(\H_\pt(X),\{p_\otimes(\S^1)\}_{p\in\Fin_X})$ by a partial left adjoint, the first claim for $\SH$ follows.
	
	The functor $X\mapsto\Fin_X$ also transforms colimits into limits. In particular, $\Fin_{X\amalg Y}\simeq \Fin_X\times\Fin_Y$.
	But the cartesian product is also a \emph{coproduct} in $\Cat_\infty^\amalg$, and $\Pr^\mathrm{L}$ is semiadditive, which implies the second claim for $\H$ and hence $\H_\pt$. 
	To prove that $\SH(X\amalg Y)\simeq \SH(X)\times\SH(Y)$, since $\SH(\ph)$ preserves filtered colimits, we may assume that $X=\B G$ and $Y=\B H$ for some finite groups $G$ and $H$. Let $p\colon * \to \B G$ and $q\colon *\to \B H$. Then
	\[
	\SH(X\amalg Y) \simeq \H_\pt(X\amalg Y)[(p\amalg q)_\otimes(\S^1)^{-1}] \simeq (\H_\pt(X)\times \H_\pt(Y))[(p_\otimes(\S^1),q_\otimes(\S^1))^{-1}].
	\]
	The result follows by expressing the right-hand side as a limit, as in the proof of Lemma~\ref{lem:inversion}.
\end{proof}

\begin{lemma}\label{lem:SHinduction}
	Let $p\colon Y\to X$ be a finite covering map of profinite groupoids.
\begin{enumerate}
	\item The functor $\SH(X)\otimes_{\H_\pt(X)}\H_\pt(Y)\to \SH(Y)$ induced by $p^*\colon\SH(X)\to\SH(Y)$ and $\Sigma^\infty\colon\H_\pt(Y)\to\SH(Y)$ is an equivalence of symmetric monoidal $\infty$-categories. 
	\item The functor $p^*\colon\SH(X)\to\SH(Y)$ has a left adjoint $p_\sharp$ as an $\SH(X)$-module functor in $\Pr^\mathrm{L}$.
\end{enumerate}
\end{lemma}

\begin{proof}
	Since every finite covering of $Y$ is a summand of a finite covering pulled back from $X$, $\SH(Y)$ can be obtained from $\H_\pt(Y)$ by inverting $p^*q_\otimes(\S^1)$ for all $q\in\Fin_X$. Comparing universal properties, this proves (1).
	The analogs of (2) for $\H$ and $\H_\pt$ are immediate from Lemma~\ref{lem:FinX}(1,3). 
	The $\SH(X)$-module adjunction $p_\sharp\dashv p^*$ is then obtained from the $\H_\pt(X)$-module adjunction $p_\sharp\dashv p^*$ by extension of scalars along $\Sigma^\infty\colon \H_\pt(X)\to\SH(X)$. We refer to \cite[\sect6.1]{Hoyois} for a formal justification of the latter procedure. 
\end{proof}

\begin{lemma}\label{lem:fp-norm}
	Let $f\colon Y\to X$ be a finitely presented map of profinite groupoids with connected fibers.
	Then $f_\otimes\colon \H_\pt(Y)\to \H_\pt(X)$ preserves colimits.
	In particular, $f_\otimes(\S^1)\simeq \S^1\in \H_\pt(X)$.
\end{lemma}

\begin{proof}
	It suffices to show that the restriction of $f_\otimes$ to $\Fin_{Y+}$ preserves finite sums.
	We claim more generally that if $f\colon Y\to X$ is a finitely presented map in $\Pro(\scr S)$ with connected fibers, then $f_*\colon \Fun(Y,\scr S)\to\Fun(X,\scr S)$ preserves finite sums. By Lemma~\ref{lem:FinX}(2), we can assume that $X$ is an $\infty$-groupoid, and the claim then follows from the fact that connected limits commute with sums in $\scr S$.
\end{proof}

If $p\colon Y\to X$ is a finitely presented map of profinite groupoids, Lemma~\ref{lem:fp-norm} implies that $p_\otimes(\S^1)\simeq \tilde p_\otimes(\S^1)$, where $\tilde p$ is the $0$-truncation of $p$ in $\Pro(\FinGpd)_{/X}$. As $\tilde p$ is a finite covering map, we deduce that $\Sigma^\infty p_\otimes(\S^1)\in \SH(X)$ is invertible.

\begin{lemma}\label{lem:norm-compact}
	If $p\colon Y\to X$ is a finitely presented map of profinite groupoids, then $p_\otimes(\S^1)\in\H_\pt(X)$ is compact.
\end{lemma}

\begin{proof}
	If $f\colon X\to X'$ is an arbitrary map in $\Pro(\FinGpd)$, then $f^*\colon \H_\pt(X') \to \H_\pt(X)$ preserves compact objects, since $f_*\colon \H_\pt(X)\to \H_\pt(X')$ preserves sifted colimits.
	We may therefore assume that $X$ is a finite groupoid. Since $\H_\pt(\ph)$ transforms finite sums into finite products (by Lemma~\ref{lem:continuity}), and since an object in a product of $\infty$-categories is compact if and only if each of its components is compact, we may further assume that $X=\B G$ for some finite group $G$. Then $\H_\pt(X)$ is the $\infty$-category of pointed $G$-spaces and $p_\otimes(\S^1)$ is a representation sphere.
\end{proof}

By Lemma~\ref{lem:norm-compact} and Remark~\ref{rmk:inversion}, if $p\colon Y\to X$ is a finitely presented map, the functor $\Sigma^\infty p_\otimes\colon \H_\pt(Y)\to\SH(X)$ lifts uniquely to a symmetric monoidal functor $p_\otimes\colon \SH(Y)\to \SH(X)$ that preserves sifted colimits (and it preserves all colimits if $p$ has connected fibers, by Lemma~\ref{lem:fp-norm}).
As in \sect\ref{sub:coherence}, we can lift $\H_\pt^\otimes$ to a functor
\[
\Span(\Pro(\FinGpd),\all,\fp) \to \CAlg(\scr O\Cat{}_\infty^\mathrm{sift}),\quad X\mapsto (\H_\pt(X),\{p_\otimes(\S^1)\}_{p}),
\]
where $p$ ranges over all finitely presented maps, and we obtain
\[
\SH^\otimes\colon \Span(\Pro(\FinGpd),\all,\fp) \to \CAlg(\what\Cat{}_\infty^\mathrm{sift}),\quad X\mapsto \SH(X).
\]
Moreover, as explained in Remark~\ref{rmk:H->SH}, we have natural transformations
\[
\H^\otimes \xrightarrow{(\ph)_+} \H_\pt^\otimes \xrightarrow{\Sigma^\infty} \SH^\otimes\colon \Span(\Pro(\FinGpd),\all,\fp)\to\CAlg(\what\Cat{}_\infty^\mathrm{sift}).
\]

\begin{remark}
	Let $G$ be a finite group, $H\subset G$ a subgroup, and $p\colon \B H\to \B G$ the induced finite covering map. Then the functor $p_\otimes\colon \SH(\B H)\to\SH(\B G)$ coincides with the norm $\Norm_H^G$ introduced by Hill, Hopkins, and Ravenel in \cite[\sect2.3.3]{HHR}. Indeed, the latter is also a symmetric monoidal extension of the unstable norm functor that preserves filtered colimits, and such an extension is unique by Lemma~\ref{lem:inversion}.
	For the same reason, if $N\subset G$ is a normal subgroup and $p\colon \B G\to \B(G/N)$, then $p_\otimes\colon \SH(\B G)\to \SH(\B(G/N))$ is the geometric fixed point functor $\Phi^N$.
\end{remark}

\begin{proposition}\label{prop:spectralMackey}
	Let $X$ be a profinite groupoid. Then the inclusion $\Fin_{X+}\into\Span(\Fin_X)$ induces a symmetric monoidal equivalence
	\[
	\SH(X)\simeq \Sp(\PSh_\Sigma(\Span(\Fin_X))).
	\]
\end{proposition}

\begin{proof}
	We must show that the symmetric monoidal functor $\H_\pt(X) \to \Sp(\PSh_\Sigma(\Span(\Fin_X)))$ sends $p_\otimes(\S^1)$ to an invertible object, for every finite covering map $p\colon Y\to X$, and that the induced functor $\SH(X)\to \Sp(\PSh_\Sigma(\Span(\Fin_X)))$ is an equivalence. 
	Suppose first that $X=\B G$ for some finite group $G$. In this case, $\SH(\B G)\simeq\H_\pt(\B G)[p_\otimes(\S^1)^{-1}]$ for $p\colon *\to \B G$, and both claims follow from \cite[Theorem A.4]{BDGNS4}. 
	Using that $X\mapsto\SH(X)$ transforms finite coproducts into products (Lemma~\ref{lem:continuity}) and the analogous fact for $X\mapsto \Sp(\PSh_\Sigma(\Span(\Fin_X)))$, we deduce the claims for $X$ a finite groupoid.
	The first claim in general follows by base change, and the second claim follows from the fact that $X\mapsto\SH(X)$ transforms cofiltered limits into colimits in $\Pr^\mathrm{L}$ (Lemma~\ref{lem:continuity}), together with the analogous fact for $X\mapsto \Sp(\PSh_\Sigma(\Span(\Fin_X)))$.
\end{proof}

\begin{example}
	Let $G$ be a profinite group. Applying Proposition~\ref{prop:spectralMackey} with $X=\B G$, we deduce that $\SH(\B G)$ coincides with the $\infty$-category of $G$-spectra defined in \cite[Example B]{BarwickMackey}. Moreover, if we write $G$ as a cofiltered limit $\lim_{\alpha \in A} G_\alpha$, Lemma \ref{lem:continuity} implies that $\SH(\B G) \simeq \lim_{\alpha \in A} \SH(\B G_\alpha)$.
\end{example}

\begin{remark}\label{rmk:spectralMackey}
	We constructed the functor $\SH^\otimes$ starting with $\Fin_{+}^\otimes$ and using the steps
	\[
	\Fin_{X+} \rightsquigarrow \PSh_\Sigma(\Fin_{X+})=\H_\pt(X) \rightsquigarrow \SH(X).
	\]
	In light of Proposition~\ref{prop:spectralMackey}, we can also construct a functor $\SH_\mathrm{Mack}^\otimes$ starting with $\Span(\Fin)^\otimes$ and using the steps
	\[
	\Span(\Fin_X) \rightsquigarrow \PSh_\Sigma(\Span(\Fin_X)) \rightsquigarrow \Sp(\PSh_\Sigma(\Span(\Fin_X))),
	\]
	where the second step is the symmetric monoidal inversion of $p_\otimes(\S^1)$ for every finitely presented map $p$. Moreover, starting with the natural transformation $\Fin_{+}^\otimes\into \Span(\Fin)^\otimes$, we obtain a natural transformation $\SH^\otimes \to \SH_\mathrm{Mack}^\otimes$, which is then an equivalence.
\end{remark}

\begin{definition}
	Let $X$ be a profinite groupoid. A \emph{normed $X$-spectrum} is a section of $\SH^\otimes$ over $\Span(\Fin_X)$ that is cocartesian over backward morphisms. We denote by $\NAlg(\SH(X))$ the $\infty$-category of normed $X$-spectra.
\end{definition}

As in Proposition~\ref{prop:categorical-props}, $\NAlg(\SH(X))$ is a presentable $\infty$-category, monadic over $\SH(X)$, and monadic and comonadic over $\CAlg(\SH(X))$.
For $G$ a finite group, one can show by comparing monads that our notion of normed $\B G$-spectrum is equivalent to the classical notion of $G$-$\E_\infty$-ring spectrum.

\begin{example}
	If $X$ is a profinite set, then $\NAlg(\SH(X))\simeq \CAlg(\SH(X))$. This follows from Corollary~\ref{cor:fold-spans}, since every finite covering map between profinite sets is a sum of fold maps.
\end{example}

\begin{remark}
	One can consider higher versions of equivariant homotopy theory where $\Fin_X$ is replaced by the $n$-category $\Fun(X,\tau_{\leq n-1}\scr S_\pi)$, where $1\leq n\leq \infty$ and $\scr S_\pi\subset\scr S$ is the full subcategory of truncated $\infty$-groupoids with finite homotopy groups. If we let $\H^n(X)=\PSh_\Sigma(\Fun(X,\tau_{\leq n-1}\scr S_\pi))$, then $\H^n(\ph)$ and $\H^n_\pt(\ph)$ have norms for finitely presented maps of profinite $n$-groupoids.
	For example, $\H^2(*)$ is the unstable global homotopy theory of Gepner–Henriques \cite{GepnerHenriques} and Schwede \cite{Schwede} (restricted to finite groups). In general, $\H^n(\ph)$ is a setting for $n$-equivariant cohomology theories in the sense of \cite[Remark 5.9]{LurieElliptic}. 
	The norms on $\H^2_\pt(\ph)$ further extend to Schwede's \emph{stable} global homotopy theory, which over a profinite $2$-groupoid $X$ may be defined as the stabilization of $\PSh_\Sigma(\Span(\Fun(X,\Fin\Gpd),\fincov,\all))$.
\end{remark}

\section{Norms and Grothendieck's Galois theory}
\label{sec:ggt}

Let $S$ be a connected scheme, $x$ a geometric point of $S$, and $G=\hat\pi_1^\et(S,x)$ the profinite étale fundamental group of $S$ at $x$. Grothendieck's Galois theory constructs a canonical equivalence between the category of finite étale $S$-schemes and that of finite discrete $G$-sets. Our goal in this section is to show that, as a consequence, we obtain an adjunction $\SH(\B G)\rightleftarrows \SH(S)$, and moreover that this adjunction lifts to an adjunction $\NAlg(\SH(\B G))\rightleftarrows \NAlg_{\FEt}(\SH(S))$ between $\infty$-categories of normed spectra (see Proposition~\ref{prop:normedGspectra}). In the case where $S$ is the spectrum of a field $k$ and $p\colon G\to H$ is a finite quotient corresponding to a finite Galois extension $L/k$, the composition \[\SH(\B H)\xrightarrow{p^*}\SH(\B G) \to \SH(k)\] is the functor $c_{L/k}$ introduced by Heller and Ormsby in \cite{HellerOrmsby}, building on earlier work of Hu \cite[\sect 3]{HuBC}.

\subsection{The profinite étale fundamental groupoid}
\label{sub:ggt-general}

We will use a generalized form of Grothendieck's Galois theory that makes no connectedness assumption. 
Let $\scr S_\pi\subset\scr S$ denote the full subcategory of truncated (i.e., $n$-truncated for some $n$) $\infty$-groupoids with finite homotopy groups.
If $\scr X$ is an $\infty$-topos, we say that an object $X\in\scr X$ is \emph{finite locally constant} if there exists an effective epimorphism $\coprod_{i\in I}U_i\to *$ in $\scr X$ and objects $K_i\in\scr S_\pi$ such that $U_i\times X\simeq U_i\times K_i$ in $\scr X_{/U_i}$ for each $i\in I$; if moreover $I$ can be chosen finite, then $X$ is called \emph{locally constant constructible}. We denote by $\scr X^\mathrm{lcc}\subset\scr X^\mathrm{flc}\subset\scr X$ the corresponding full subcategories of $\scr X$. Shape theory provides a left adjoint functor $\Pi_\infty\colon \Top_\infty\to \Pro(\scr S)$ \cite[\sect E.2.2]{SAG}, and we denote by $\widehat\Pi_\infty\colon \Top_\infty\to\Pro(\scr S_\pi)$ its profinite completion.

\begin{proposition}\label{prop:finiteGalois}
	Let $\scr X$ be an $\infty$-topos. Then there are natural equivalences of $\infty$-categories
	\[
	 \scr X^\mathrm{flc}\simeq\Fun(\Pi_\infty(\scr X),\scr S_\pi)\quad\text{and}\quad \scr X^\mathrm{lcc}\simeq\Fun(\widehat\Pi_\infty(\scr X),\scr S_\pi).
	\]
\end{proposition}

\begin{proof}
	The first equivalence is \cite[Theorem 4.3]{HoyoisGalois}. This equivalence is such that the objects of $\scr S_\pi$ appearing in the local trivializations of a finite locally constant sheaf are precisely the values of the corresponding functor $\Pi_\infty(\scr X)\to\scr S_\pi$. In particular, we obtain an equivalence between $\scr X^\mathrm{lcc}$ and the full subcategory $\Fun_\fin(\Pi_\infty(\scr X),\scr S_\pi)\subset\Fun(\Pi_\infty(\scr X),\scr S_\pi)$ spanned by the functors with finitely many values. To conclude, we show that for every $X\in \Pro(\scr S)$ with profinite completion $\pi\colon X\to\widehat X$, the functor
	\[
	\pi^*\colon \Fun(\widehat X,\scr S_\pi) \to \Fun_\fin(X,\scr S_\pi)
	\]
	is an equivalence. Note that any $X\to\scr S_\pi$ with finitely many values lands in $\scr K^\simeq$ for some full subcategory $\scr K\subset\scr S_\pi$ with finitely many objects. Since $\scr K^\simeq$ belongs to $\scr S_\pi$, we deduce that $\pi^*$ is essentially surjective. 
	
	Let $A,B\colon X_0\to\scr S_\pi$ be functors for some $X_0\in(\scr S_\pi)_{\widehat X/}$, and let $H\to X_0$ be the internal mapping object from $\int A$ to $\int B$ in $\scr S_{/X_0}$. Then, for any $Y\in \Pro(\scr S)$ and $f\colon Y\to X_0$, there is a natural equivalence
	\[
	\Map_{\Fun(Y,\scr S_\pi)}(A\circ f,B\circ f) \simeq \Map_{X_0}(Y,H).
	\]
	Indeed, this is obvious if $Y\in\scr S$ and both sides are formally extended to $\Pro(\scr S)$.
	Since $X_0$ and $H$ belong to $\scr S_\pi$ (because all the fibers of $H\to X_0$ do), $\pi$ induces an equivalence $\Map_{X_0}(\widehat X,H)\simeq \Map_{X_0}(X,H)$. This shows that $\pi^*$ is fully faithful.
\end{proof}

If $S$ is a scheme, we will denote by $\Pi_1^\et(S)\in\Pro(\Gpd)$ the $1$-truncation of the shape of the étale $\infty$-topos of $S$, and by $\widehat\Pi_1^\et(S)\in\Pro(\FinGpd)$ its profinite completion.

\begin{corollary}\label{cor:finiteGalois}
	Let $S$ be an arbitrary scheme. Then there is a natural equivalence of categories 
	\[\FEt_S \simeq \Fun(\Pi_1^\et(S),\Fin).\]
	If $S$ is quasi-compact in the clopen topology (for example if $S$ is quasi-compact), there is a natural equivalence of categories
	\[\FEt_S \simeq \Fun(\widehat\Pi_1^\et(S),\Fin).\]
\end{corollary}

\begin{proof}
	An étale sheaf of sets on an arbitrary scheme $S$ is finite locally constant if and only if it is representable by a finite étale $S$-scheme. If $S$ is quasi-compact in the clopen topology, such sheaves are automatically constructible. The desired equivalences are thus the restrictions of the equivalences of Proposition~\ref{prop:finiteGalois} to $0$-truncated objects.
\end{proof}

\begin{remark}
	If $S$ is a connected scheme and $x$ is a geometric point of $S$, we have a canonical equivalence of profinite groupoids $\widehat\Pi_1^\et(S)\simeq \B\hat\pi_1^\et(S,x)$. 
\end{remark}

\begin{lemma}\label{lem:Pi1-spans}
	\leavevmode
	\begin{enumerate}
		\item For every scheme $S$, the square
		\begin{tikzmath}
			\diagram{\FEt_S & \Fun(\widehat\Pi_1^\et(S),\Fin) \\ \Sch & \Pro(\FinGpd) \\};
			\arrows (11-) edge node[above]{$\simeq$} (-12) (11) edge (21) (21-) edge node[above]{$\widehat\Pi_1^\et$} (-22) (12) edge node[right]{$\int$} (22);
		\end{tikzmath}
		commutes, where the equivalence is that of Corollary~\ref{cor:finiteGalois} (and the functor $\int$ was defined just before Lemma \ref{lem:proGrothendieck}).
		\item The functor $\widehat\Pi_1^\et$ sends finite étale morphisms of schemes to finite covering maps of profinite groupoids.
		\item The functor $\widehat\Pi_1^\et$ commutes with pullbacks along finite étale morphisms.
	\end{enumerate}
\end{lemma}

\begin{proof}
	Note that the given square commutes after composing with the functor $\Pro(\FinGpd)\to \Cat_1^\op$, $X\mapsto \Fin_X$: both composites then send $T\in\FEt_S$ to a category canonically equivalent to $\FEt_T$. But the latter functor has a retraction that sends a small extensive category $\scr C$ with finite limits to the profinite shape of the $\infty$-topos $\PSh_\Sigma(\scr C)$. This proves (1). Assertions (2) and (3) follow immediately from (1) and the naturality of the equivalences of Lemma~\ref{lem:proGrothendieck} and Corollary~\ref{cor:finiteGalois}.
\end{proof}

\subsection{Galois-equivariant spectra and motivic spectra}
\label{sub:galois}

In \sect\ref{sub:coherence} and \sect\ref{sub:equivariant-norms}, we constructed the functors
\begin{align*}
	\SH^\otimes &\colon \Span(\Sch,\all,\fet) \to \CAlg(\what\Cat{}_\infty^\mathrm{sift}),\\
	\SH^\otimes &\colon \Span(\Pro(\FinGpd),\all,\fp) \to \CAlg(\what\Cat{}_\infty^\mathrm{sift}).
\end{align*}
By Lemma~\ref{lem:Pi1-spans}(2,3), the profinite étale fundamental groupoid functor $\widehat\Pi_1^\et\colon\Sch\to \Pro(\FinGpd)$ extends to
\[
\widehat\Pi_1^\et\colon \Span(\Sch,\all,\fet) \to \Span(\Pro(\FinGpd),\all,\fincov),
\]
where $\fincov\subset\fp$ is the class of finite covering maps.

We will now construct a natural transformation
\begin{equation}\label{eqn:GaloisSH}
c\colon\SH^\otimes \circ \widehat\Pi_1^\et \to \SH^\otimes \colon \Span(\Sch,\all,\fet)\to\CAlg(\what\Cat{}_\infty^\mathrm{sift}).
\end{equation}
Our starting point is the equivalence $\Fun(\widehat\Pi_1^\et(S),\Fin)\simeq \FEt_S$ provided by Corollary~\ref{cor:finiteGalois}, which is natural in $S\in\Sch^\op$. 
By \cite[Theorem V.1.3.2.2]{GRderalg}, this natural equivalence lifts canonically to
\[
\Fin^\otimes \circ\widehat\Pi_1^\et \simeq \FEt^\otimes\colon \Span(\Sch,\all,\fet)\to \Cat_1^\mathrm{lex}.
\]
Applying $\Span\colon \Cat_1^\mathrm{lex}\to\CAlg(\Cat_2)$ and restricting to pointed objects, we obtain
\[
\Fin_+^\otimes \circ\widehat\Pi_1^\et \simeq \FEt_+^\otimes\colon \Span(\Sch,\all,\fet)\to \CAlg(\Cat_1).
\]
The inclusion $\FEt_{S+}\subset \SmQP_{S+}$ is natural in $S\in\Sch^\op$, and its source and target are sheaves in the finite étale topology (by Lemma~\ref{lem:corr-etale-descent}). By Corollary~\ref{cor:automatic-norms}, it lifts uniquely to a natural transformation
\[
\FEt_+^\otimes \into \SmQP_+^\otimes\colon \Span(\Sch,\all,\fet)\to \CAlg(\Cat_1),
\]
and so we obtain $\Fin_+^\otimes\circ\widehat\Pi_1^\et\into \SmQP_+^\otimes$. 
We can view this transformation as a functor $\Span(\Sch,\all,\fet)\times\Delta^1 \to \CAlg(\Cat_\infty)$.
After composing with $\PSh_\Sigma\colon \CAlg(\Cat_\infty)\to \CAlg(\what\Cat{}_\infty^\mathrm{sift})$, we can lift it to a functor
\[
\Span(\Sch,\all,\fet)\times\Delta^1 \to \CAlg(\what{\scr M\Cat}{}_\infty^\mathrm{sift})
\]
sending $(S,0\to 1)$ to
\[
(\H_\pt(\widehat\Pi_1^\et(S)),\text{equivalences}) \to (\PSh_\Sigma(\SmQP_{S+}),\text{motivic equivalences}).
\]
Composing with a partial left adjoint to the functor
\[
\CAlg(\what\Cat{}_\infty^\mathrm{sift}) \to \CAlg(\what{\scr M\Cat}{}_\infty^\mathrm{sift}), \quad \scr C\mapsto (\scr C,\text{equivalences}),
\]
we obtain a natural transformation
\[
c\colon\H_\pt^\otimes \circ \widehat\Pi_1^\et \to \H_\pt^\otimes\colon \Span(\Sch,\all,\fet)\to \CAlg(\what\Cat{}_\infty^\mathrm{sift}).
\]
We claim that $\Sigma^\infty c_S(p_\otimes(\S^1))\in\SH(S)$ is invertible, for every finite covering map $p\colon Y\to \widehat\Pi_1^\et(S)$. By Lemma~\ref{lem:Pi1-spans}(1), such a finite covering map is $\widehat\Pi_1^\et$ of a finite étale map $q\colon T\to S$, and $\Sigma^\infty c_S(p_\otimes(\S^1))\simeq \Sigma^\infty q_\otimes(\S^1)\simeq q_\otimes(\Sigma^\infty\S^1)\in \SH(S)$. This is indeed invertible because $\SH(T)$ is stable and $q_\otimes$ is symmetric monoidal. It follows that we can lift $\Sigma^\infty\circ c$ to a functor
\[
\Span(\Sch,\all,\fet)\times\Delta^1\to \CAlg(\what{\scr O\Cat}{}_\infty^\mathrm{sift})
\]
sending $(S,0\to 1)$ to
\[
(\H_\pt(\widehat\Pi_1^\et(S)), \{p_\otimes(\S^1)\}_{p}) \to (\SH(S), \pi_0\Pic(\SH(S))).
\]
Composing with a partial left adjoint to the functor
\[ \CAlg(\what{\Cat}{}_\infty^\mathrm{sift}) \to\CAlg(\what{\scr O\Cat}{}_\infty^\mathrm{sift}),
   \quad \scr C\mapsto (\scr C,\pi_0\Pic(\scr C)),\]
we obtain the desired transformation~\eqref{eqn:GaloisSH}. Let us set out explicitly what we have done.

\begin{proposition}\label{prop:ggt-compat}
Let $S$ be a scheme. There is an adjunction
\begin{equation}\label{eqn:GaloisAdj}
c_S:\SH(\widehat\Pi_1^\et(S))\rightleftarrows\SH(S):u_S,
\end{equation}
where $c_S$ is a symmetric monoidal functor natural in $S\in\Span(\Sch,\all,\fet)$.

In particular, if $S$ is connected and $S'/S$ is a finite Galois cover with group $G$, then there is a left adjoint symmetric monoidal functor $c_{S'/S}\colon \SH(\B G) \to \SH(S)$. If $H\subset G$ is a subgroup and $T=S'/H$ denotes the quotient in the category of schemes, then $c_{S'/S}$ sends the $G$-orbit $\Sigma^\infty_+(G/H)$ to $\Sigma^\infty_+ T$. Moreover, if $p\colon T\to S$ and $q\colon \B H\to \B G$ are the obvious maps, then the following diagram commutes:
\begin{tikzmath}
\diagram{\SH(\B H) & \SH(T) \\ \SH(\B G) & \SH(S)\rlap. \\};
\arrows (11-) edge node[above]{$c_{S'/T}$} (-12) (11) edge node[left]{$q_\otimes=\Norm_H^G$} (21) (21-) edge node[above]{$c_{S'/S}$} (-22) (12) edge node[right]{$p_\otimes$} (22);
\end{tikzmath}
\end{proposition}

\begin{proof}
By construction, $c_S$ preserves colimits and hence is a left adjoint. 
A finite Galois cover $S'/S$ with group $G$ is in particular a $G$-torsor in the étale topos of $S$, which is classified by a map $f\colon \widehat\Pi_1^\et(S)\to \B G$. We obtain $c_{S'/S}$ as the composition $c_S\circ f^*$.
The final diagram decomposes as follows, where $\tilde p=\widehat\Pi_1^\et(p)$:
\begin{tikzmath}
\diagram{\SH(\B H) & \SH(\widehat\Pi_1^\et(T)) & \SH(T) \\
\SH(\B G) & \SH(\widehat\Pi_1^\et(S)) & \SH(S)\rlap. \\};
\arrows (11-) edge (-12) (12-) edge node[above]{$c_{T}$} (-13) (11) edge node[left]{$q_\otimes$} (21) (12) edge node[left]{$\tilde p_\otimes$} (22) (13) edge node[left]{$p_\otimes$} (23) (21-) edge (-22) (22-) edge node[above]{$c_S$} (-23);
\end{tikzmath}
The right-hand square commutes by naturality of~\eqref{eqn:GaloisSH}. The left-hand square commutes because the square
\begin{tikzmath}
\diagram{\widehat\Pi_1^\et(T) & \B H \\ \widehat\Pi_1^\et(S) & \B G \\};
\arrows (11-) edge (-12) (21-) edge (-22) (11) edge node[left]{$\tilde p$} (21) (12) edge node[right]{$q$} (22); 
\end{tikzmath}
is cartesian in $\Pro(\FinGpd)$.
\end{proof}

\begin{proposition}\label{prop:normedGspectra}
	Let $S$ be a scheme. Then the adjunction~\eqref{eqn:GaloisAdj} lifts to an adjunction
	\[
	\NAlg(\SH(\widehat\Pi_1^\et(S))) \rightleftarrows \NAlg_{\FEt}(\SH(S)).
	\]
\end{proposition}

\begin{proof}
	By Lemmas \ref{lemm:construct-relative-adjoint}(1) and~\ref{lemm:adjoints-pass-to-sections}, we obtain from~\eqref{eqn:GaloisSH} an adjunction
	\[
	\Sect(\SH^\otimes|\Span(\Fin_{\widehat\Pi_1^\et(S)})) \rightleftarrows \Sect(\SH^\otimes|\Span(\FEt_S)),
	\]
	where the left adjoint preserves normed spectra. It remains to show that the right adjoint preserves normed spectra as well. It will suffice to show that, if $f\colon Y\to X$ is a finite étale map and $\tilde f=\widehat\Pi_1^\et(f)\colon \tilde Y\to\tilde X$, the exchange transformation
	\[
	\tilde f^*u_X\to u_Yf^*\colon \SH(X) \to \SH(\tilde Y)
	\]
	is an equivalence. Since $\tilde f$ is a finite covering map, $\tilde f^*$ has a left adjoint $\tilde f_\sharp$ which is an $\SH(\tilde X)$-module functor (Lemma~\ref{lem:SHinduction}(2)). By adjunction, it is equivalent to show that the exchange transformation
	\[
	f_\sharp c_Y \to c_X\tilde f_\sharp\colon \SH(\tilde Y)\to\SH(X)
	\]
	is an equivalence. Note that this is a transformation of $\SH(\tilde X)$-module functors in $\Pr^\mathrm{L}$. 
	By Lemma~\ref{lem:SHinduction}(1), $\SH(\tilde Y)$ is generated as an $\SH(\tilde X)$-module by finite $\tilde Y$-sets. It therefore suffices to show that the above transformation is an equivalence on finite $\tilde Y$-sets, which is clear.
\end{proof}

\begin{example}
	Let $S$ be essentially smooth over a field. Since $\bigvee_{n\in\Z}\Sigma^{2n,n}\HH\Z_S\in\SH(S)$ is a normed spectrum (see Example~\ref{ex:periodicHZ}), Bloch's cycle complex $z^*(S,*)$ can be promoted to a normed $\widehat\Pi_1^\et(S)$-spectrum.
\end{example}

\begin{example}
	Since $\KGL_S\in\SH(S)$ is a normed spectrum (see Theorem~\ref{thm:KGL}), the homotopy $\K$-theory spectrum $\KH(S)$ of any scheme $S$ can be promoted to a normed $\widehat\Pi_1^\et(S)$-spectrum. We will see in Remark~\ref{rmk:KBnormed} that this also holds for the ordinary (nonconnective) $\K$-theory spectrum of $S$.
\end{example}

\begin{example}
	Let $S$ be a smooth scheme over a field of characteristic zero. By \cite[Theorem 3.1]{Levine:2009}, there is an isomorphism $\Omega^n(S)\simeq \MGL^{2n,n}(S)$, where $\Omega^*$ is the algebraic cobordism of Levine–Morel. Since $\bigvee_{n\in\Z}\Sigma^{2n,n}\MGL_S\in\SH(S)$ is a normed spectrum (see Theorem~\ref{thm:normedMGL}), $\bigoplus_{n\in\Z}\Omega^n(S)$ is the $\pi_0$ of a normed $\widehat\Pi_1^\et(S)$-spectrum.
\end{example}

\subsection{The Rost norm on Grothendieck–Witt groups}

As another application of Galois theory, we can identify the norms on $\pi_{0,0}$ of the motivic sphere spectrum.
We first recall and slightly generalize Morel's computation of the latter:

\begin{theorem}\label{thm:GW}
	Let $S$ be a regular semilocal scheme over a field. Suppose that $\Char(S)\neq 2$ or that $S$ is the spectrum of a field. Then there is natural isomorphism of rings
	\[
	[\1_S,\1_S] \simeq \GW(S).
	\]
\end{theorem}

\begin{proof}
	Choose a map $S\to \Spec k$ where $k$ is a perfect field.
	By Popescu's desingularization theorem \cite[Tag 07GC]{Stacks}, $S$ can be written as the limit of a cofiltered diagram of smooth affine $k$-schemes $S_\alpha$, which we can replace by their semilocalizations without changing the limit.
	If $S$ is the spectrum of a field, we can take each $S_\alpha$ to be the spectrum of a field.
	By \cite[Lemma A.7(1)]{HoyoisMGL}, we have
		\[
		[\1_S,\1_S]\simeq \colim_\alpha [\1_{S_\alpha},\1_{S_\alpha}].
		\]
		On the other hand, by \cite[Lemma 49]{bachmann-gwtimes}, we have
		\[
		\GW(S) \simeq \colim_\alpha \GW(S_\alpha).
		\]
		We can therefore assume that $S$ is a semilocalization of a smooth affine $k$-scheme.

	Let $\pi^\pre_{0,0}(\1_k)$ denote the presheaf $U\mapsto [\Sigma^\infty_+U,\1_k]$ on $\Sm_k$.
	By \cite[Corollary 6.9(2) and Remark 6.42]{Morel}, we have
	\begin{equation}\label{eqn:Zar=Nis}
	\L_\Zar \pi^\pre_{0,0}(\1_k)\simeq \L_\Nis \pi^\pre_{0,0}(\1_k) \simeq \underline\GW,
	\end{equation}
	where $\underline\GW$ is the unramified Grothendieck–Witt sheaf on $\Sm_k$. In particular, we have natural transformations
	\begin{equation*}\label{eqn:GW-zigzag}
	\pi^\pre_{0,0}(\1_k) \to \underline{\GW} \from \GW
	\end{equation*}
	of ring-valued presheaves on $\Sm_k$, and it suffices to show that they are equivalences on $S$. This is obvious if $S$ is a field, so we assume $\Char(k)\neq 2$ from now on.
	
	The equivalence $\GW(S)\simeq\underline{\GW}(S)$ follows from the analogous one for Witt groups \cite[Theorem 100]{BalmerWitt} and the formula $\GW(S)\simeq \W(S)\times_{\Z/2}\Z$ \cite[\sect 1.4]{gille2015quadratic}.
	It remains to show that $\pi^\pre_{0,0}(\1_k)(S)\simeq \underline{\GW}(S)$.
	If $S$ is local, this follows directly from~\eqref{eqn:Zar=Nis}.
	In general, by left completeness of the $t$-structure on Nisnevich sheaves of spectra, it suffices to show that $\HH^n_{\Nis}(S,\spi_n(\1_k)_0) = 0$ for $n>0$. 
	If $k$ is infinite, this follows from \cite[Lemma 3.6]{BachmannFasel} since $\spi_n(\1_k)_0$ is a sheaf with Milnor–Witt transfers \cite[Theorem 8.12]{FramedMW}. 
	If $k$ is finite and $\alpha\in \HH^n_{\Nis}(S,\spi_n(\1_k)_0)$, then there exist finite extensions $k_1/k$ and $k_2/k$ of coprime degrees $d_1$ and $d_2$ such that $\alpha$ becomes zero after extending scalars to $k_1$ and $k_2$, and we can arrange that $d_i\equiv 1\pmod 4$. By \cite[Lemma II.2.1]{ConnerPerlis}, we then have $\Tr_{k_i/k}(1) = d_i$ in $\GW(k)$, whence $d_i\alpha=0$. By Bézout's identity, $\alpha=0$.
\end{proof}

Let $S$ be a regular semilocal scheme as in Theorem~\ref{thm:GW} and let $f\colon T\to S$ be a finite étale map. The functor $f_\otimes\colon \SH(T)\to\SH(S)$ induces a norm map
\[
\nu_f\colon\GW(T) \simeq [\1_T,\1_T] \xrightarrow{f_\otimes} [\1_S,\1_S] \simeq \GW(S).
\]
There is an a priori different norm map $\Norm_f\colon \GW(T) \to \GW(S)$ with similar properties, defined by Rost \cite{rost2003multiplicative}. The following theorem shows that these two constructions coincide, at least in characteristic $\neq 2$.

\begin{theorem} \label{thm:GW-norms-comparison}
Let $S$ be a regular semilocal scheme over a field of characteristic $\neq 2$ and let $f\colon T\to S$ be a finite étale map. Then 
\[\nu_f = \Norm_f\colon \GW(T) \to \GW(S).\]
\end{theorem}

\begin{proof}
	Under these assumptions, the map $\GW(S)\to \prod_{x}\GW(\kappa(x))$ is injective, where $x$ ranges over the generic points of $S$ \cite[Theorem 100]{BalmerWitt}. Since both $\nu_f$ and $\Norm_f$ are compatible with base change, we may assume that $S$ is the spectrum of a field. By Corollary~\ref{cor:normed-spectrum-tambara-functor}, the functor $\GW\colon\<\FEt_S^\op\to\Set$ becomes a Tambara functor on $\FEt_S$, with norms $\nu_p$ and traces $\tau_p$. The additive transfer $\tau_p\colon \GW(Y)\to\GW(X)$ was identified in \cite[Proposition 5.2]{HoyoisGLV} with the Scharlau transfer induced by the trace $\Tr_{p}\colon \scr O(Y)\to\scr O(X)$.
If we replace the norm maps $\nu_p$ by Rost's norm maps $\Norm_p$, then one again obtains a Tambara functor \cite[Corollary 13]{bachmann-gwtimes}, which we shall denote by $\GW^\rm R$.

If $S' \to S$ is a finite Galois extension with Galois group $G$, Galois theory identifies the category of finite $G$-sets with the category of finite étale $S$-schemes split by $S'$. Let us write $\GW_G$ and $\GW_G^\rm R$ for the $G$-Tambara functors obtained by restricting $\GW$ and $\GW^\rm R$ to this subcategory. We shall denote by $\rm A_G$ the Burnside $G$-Tambara functor. This is the initial $G$-Tambara functor \cite[Example 1.11]{Nakaoka2013237}, and hence there exist unique morphisms $a\colon\rm A_G \to \GW_G$ and $a^\rm R\colon\rm A_G \to \GW_G^\rm R$. Note that for $X$ a finite $G$-set, the ring $\rm A_G(X)$ is generated by $p_*(1)$ as $p$ ranges over all morphisms $p\colon Y \to X$. Since $\GW_G$ and $\GW_G^\rm R$ are defined using the same additive transfers, it follows that $a_X = a^\rm R_X\colon\rm A_G(X) \to \GW_G(X)$.

It thus suffices to show the following: for every $\omega \in \GW(T)$, there exist a finite Galois extension $S' \to S$ splitting $f$, with Galois group $G$, such that if $X$ is the finite $G$-set corresponding to $f$, then $\omega$ is in the image of $a_X\colon\rm A_G(X) \to \GW_G(X)$. Since $\Char(S)\neq 2$, $\GW(T)$ is generated as a ring by $\Tr_p\langle 1\rangle$ as $p$ ranges over finite étale morphisms $p\colon T' \to T$ (in fact it suffices to consider quadratic extensions \cite[p.\ 233]{bachmann-gwtimes}), so we can assume that $\omega=\Tr_p\langle 1\rangle$. In that case, the desired result clearly holds with any finite Galois extension $S'\to S$ splitting $f\circ p$.
\end{proof}

\begin{remark}
	The proof of Theorem~\ref{thm:GW-norms-comparison} shows that, if $k$ is a field of characteristic $\neq 2$, there exists a unique structure of Tambara functor on the Mackey functor $\GW\colon\Span(\FEt_k)\to\Ab$ that induces the usual ring structure on Grothendieck–Witt groups. The argument does not work if $\Char(k)=2$, because in this case trace forms of finite separable extensions only generate the subring $\Z\subset\GW(k)$.
\end{remark}

As a corollary, we obtain a formula for the Euler characteristic of a Weil restriction:

\begin{corollary}
	Let $S$ be a regular semilocal scheme over a field of characteristic $\neq 2$, let $f\colon T\to S$ be a finite étale map, and let $X$ be a smooth quasi-projective $T$-scheme such that $\Sigma^\infty_+X\in\SH(T)$ is dualizable (for example, $X$ is projective over $T$ or $S$ is a field of characteristic zero). Then
	\[
	\chi(\Sigma^\infty_+\Weil_{f}X)=\Norm_{f}(\chi(\Sigma^\infty_+X)) \in \GW(S).
	\]
\end{corollary}

\begin{proof}
	Since $f_\otimes$ is a symmetric monoidal functor, it preserves Euler characteristics. Hence, $\chi(\Sigma^\infty_+\Weil_{f}X)=\nu_{f}(\chi(\Sigma^\infty_+X))$, and we conclude by Theorem~\ref{thm:GW-norms-comparison}.
\end{proof}

\section{Norms and Betti realization}
\label{sec:betti}

In this section we show that the $\rm C_2$-equivariant Betti realization functor $\rB\colon \SH(\R) \to \SH(\B \rm C_2)$ \cite[\sect4.4]{HellerOrmsby} is compatible with norms and hence preserves normed spectra.

\subsection{A topological model for equivariant homotopy theory}

We first need to construct a topological model of the $\infty$-category $\H_\pt(X)$, where $X$ is a finite groupoid.
We will denote by $\Top^\cg$ the $1$-category of compactly generated topological spaces (so that $\Top^\cg_\pt$ is a symmetric monoidal category under the smash product).
An \emph{$X$-CW-complex} will mean a functor $F\colon X\to\Top^\cg$ such that, for every $x\in X$, $F(x)$ is an $\Aut_X(x)$-CW-complex.
We write $\CW(X)\subset\Fun(X,\Top^\cg)$ for the full subcategory of $X$-CW-complexes, and $\CW_\pt(X) = \CW(X)_{*/}$ for the category of pointed $X$-CW-complexes. Since finite (smash) products of (pointed) $G$-CW-complexes are $G$-CW-complexes, for any finite group $G$, the categories $\CW(X)$ and $\CW_\pt(X)$ acquire symmetric monoidal structures with tensor products computed pointwise.

Recall that a map in $\Fun(\B G,\Top^\cg)$ is a \emph{weak equivalence} if it induces a weak equivalence on $H$-fixed points for every subgroup $H\subset G$; we call a map in $\Fun(X,\Top^\cg)$ a \emph{weak equivalence} if it is so pointwise. These weak equivalences are part of a simplicial model structure on $\Fun(X,\Top^\cg)$, where all objects are fibrant and where $X$-CW-complexes are cofibrant. By Elmendorf's theorem \cite[Theorem 1]{Elmendorf}, if $\Orb_X\subset\Fin_X\subset\FinGpd_{/X}$ is the full subcategory spanned by the connected groupoids, then the obvious functor $\Orb_X\into\Fun(X,\Top^\cg)$ induces a simplicial Quillen equivalence between $\PSh(\Orb_X,\Set_\Delta)$ and $\Fun(X,\Top^\cg)$.

If $\scr C$ is a symmetric monoidal $\infty$-category, the functor 
\[
\FinGpd^\op\to\CAlg(\Cat_\infty),\quad X\mapsto\Fun(X,\scr C),
\] 
is clearly a sheaf for the effective epimorphism topology on $\FinGpd$. Since finite covering maps are fold maps locally in this topology, it follows from Corollary~\ref{cor:automatic-norms} that this sheaf extends uniquely to a functor
\[
\Span(\FinGpd,\all,\fincov)\to \CAlg(\Cat_\infty), \quad X\mapsto \Fun(X,\scr C),\quad (U\stackrel f\from T\stackrel p\to S)\mapsto p_\otimes f^*.
\]
The functor $p_\otimes\colon \Fun(T, \scr C) \to \Fun(S, \scr C)$ coincides with the \emph{indexed tensor product} defined in \cite[A.3.2]{HHR}, and the above functor encodes the compatibility of indexed tensor products with composition and base change (cf.\ \cite[Propositions A.29 and A.31]{HHR}). 
We now apply this with $\scr C=\Top^\cg_\pt$.

\begin{lemma}\label{lemm:G-norm-basics-top}
Let $f\colon X \to Y$ be any map of finite groupoids and $p\colon T \to S$ a finite covering map.
Consider the functors
\[
f^*\colon \Fun(Y,\Top^\cg_\pt) \to \Fun(X,\Top^\cg_\pt) \quad \text{and}\quad p_\otimes\colon \Fun(T,\Top^\cg_\pt) \to \Fun(S,\Top^\cg_\pt).
\]
\begin{enumerate}
\item $f^*$ sends $\CW_\pt(Y)$ to $\CW_\pt(X)$ and $p_\otimes$ sends $\CW_\pt(T)$ to $\CW_\pt(S)$.
\item $f^*$ preserves weak equivalences and the restriction of $p_\otimes$ to $\CW_\pt(T)$ preserves weak equivalences.
\end{enumerate}
\end{lemma}

\begin{proof}
(1) The claim about $f^*$ is straightforward.
For $p_\otimes$, it suffices to prove the following: for every finite covering map $h\colon T'\to T$ and every $n\geq 0$, $p_\otimes$ sends any pushout of $h_\sharp(\rm S^{n-1}_+\into \rm D^n_+)$ in $\CW_\pt(T)$ to a relative $S$-CW-complex. The analogous statement for orthogonal spectra is proved in \cite[Proposition B.89]{HHR}, and the same proof applies in our case; one is reduced to the standard fact that if $A$ is a finite $G$-set and $V=\mathbb R^A$ is the corresponding $G$-representation, then $\rm S(V)_+\into \rm D(V)_+$ is a relative $G$-CW-complex.

(2) If $x\in X$ and $H\subset \Aut_X(x)$ is a subgroup, then $(f^*F)(x)^H=F(f(x))^{f(H)}$ for any $F\in\Fun(Y,\Top^\cg_\pt)$. This implies the claim about $f^*$. For $p_\otimes$, we use that weak equivalences in $\CW_\pt(T)$ are homotopy equivalences. It is thus enough to show that $p_\otimes$ preserves homotopic maps. If $I$ is the unit interval and $h\colon I_+\wedge A\to B$ is a homotopy between $h_0$ and $h_1$ in $\Fun(T,\Top^\cg_\pt)$, then the composition
\[
I_+ \wedge p_\otimes(A) \to p_*p^*(I)_+ \wedge p_\otimes(A) \simeq p_\otimes (I_+) \wedge p_\otimes(A) \simeq p_\otimes (I_+\wedge A) \xrightarrow{p_\otimes(h)} p_\otimes(B)
\]
is a homotopy between $p_\otimes(h_0)$ and $p_\otimes(h_1)$.
\end{proof}

By Lemma~\ref{lemm:G-norm-basics-top}(1), we obtain a functor of $2$-categories
\[ \CW^\otimes_{\pt}\colon \Span(\FinGpd, \all, \fincov) \to \CAlg(\what\Cat_1), \quad X \mapsto \CW_\pt(X), \]
as a subfunctor of $X\mapsto \Fun(X,\Top^\cg_\pt)$. 
Using Lemma \ref{lemm:G-norm-basics-top}(2) we can improve this to
\[ (\CW^\otimes_{\pt}, w)\colon \Span(\FinGpd, \all, \fincov) \to \CAlg(\what{\scr M\Cat}_1), \quad X \mapsto (\CW_\pt(X), w_X), \]
where $w_X$ denotes the class of weak equivalences in $\CW_\pt(X)$. Composing with the inclusion $\CAlg(\what{\scr M\Cat}_1) \subset \CAlg(\what{\scr M\Cat}_\infty)$ and with a left adjoint to the functor $\CAlg(\what\Cat{}_\infty) \to \CAlg(\what{\scr M\Cat}_\infty)$, $\scr C \mapsto (\scr C,\text{equivalences})$, we obtain
\[ \CW_\pt[w^{-1}]^\otimes\colon \Span(\FinGpd, \all, \fincov) \to \CAlg(\what{\Cat}_\infty), \quad X\mapsto\CW_\pt(X)[w_X^{-1}]. \]

\begin{lemma}\label{lem:CW-inversion}
	Let $X$ be a finite groupoid. Then there is a unique equivalence of $\infty$-categories $\H_\pt(X)\simeq \CW_\pt(X)[w_X^{-1}]$ making the following square commutes:
	\begin{tikzmath}
		\diagram{\Fin_{X+} & \CW_\pt(X) \\ \H_\pt(X) & \CW_\pt(X)[w_X^{-1}]\rlap. \\};
		\arrows (11-) edge[c->] (-12) (11) edge[c->] (21) (21-) edge node[above]{$\simeq$} (-22) (12) edge (22);
	\end{tikzmath}
\end{lemma}

\begin{proof}
	Uniqueness is clear by the universal property of $\PSh_\Sigma$.
	Let $\scr C$ denote the standard simplicial enrichment of the category $\CW_\pt(X)$. Since the weak equivalences in $\CW_\pt(X)$ are precisely the simplicial homotopy equivalences in $\scr C$, it follows from \cite[Example 1.3.4.8]{HA} that $\CW_\pt(X)[w_X^{-1}]\simeq \mathrm{N}(\scr C)$, where $\mathrm{N}$ is the homotopy coherent nerve.
	Let $\scr C'$ be the simplicial category of fibrant-cofibrant objects in $\Fun(X,\Top^\cg_\pt)$. Then the inclusion $\scr C\subset\scr C'$ is essentially surjective on the homotopy categories, hence it is a weak equivalence of simplicial categories.
	Via the simplicial Quillen equivalence between $\Fun(X,\Top^\cg_\pt)$ and the injective model structure on $\PSh(\mathrm{Orb}_X,\Set_{\Delta\pt})$, $\scr C'$ is further weakly equivalent to the simplicial category $\scr C''$ of fibrant-cofibrant objects in the latter simplicial model category \cite[Proposition A.3.1.10]{HTT}.
	We therefore have equivalences $\mathrm{N}(\scr C)\simeq \mathrm{N}(\scr C')\simeq\mathrm{N}(\scr C'')$.
	Finally, by \cite[Proposition 4.2.4.4]{HTT}, $\mathrm{N}(\scr C'')$ is equivalent to $\PSh(\mathrm{Orb}_X)_\pt\simeq\H_\pt(X)$.
	The given square then commutes by construction.
\end{proof}

We can now compare $\CW_\pt[w^{-1}]^\otimes$ with the functor $\H_\pt^\otimes$ constructed in \sect\ref{sub:equivariant-norms}. Consider the obvious symmetric monoidal inclusion $\Fin_{X+}\into \CW_\pt(X)$, which is natural in $X\in\FinGpd^\op$. Since both $X\mapsto\Fin_{X+}$ and $X\mapsto \Fun(X,\Top^\cg_\pt)$ are sheaves for the effective epimorphism topology on $\FinGpd$, we can apply Corollary~\ref{cor:automatic-norms} with $\scr D=\Fun(\Delta^1,\what\Cat_1)$ to obtain a natural transformation
\[
\Fin_+^\otimes \to \CW_\pt^\otimes\colon \Span(\FinGpd,\all,\fincov)\to \CAlg(\what\Cat_1).
\]
It follows from Lemma~\ref{lem:CW-inversion} and the universal property of $\PSh_\Sigma$ that, for every map $f\colon T\to U$ of finite groupoids and every finite covering map $p\colon T\to S$, $p_\otimes f^*\colon \CW_\pt(U)[w_U^{-1}]\to \CW_\pt(S)[w_S^{-1}]$ is identified with $p_\otimes f^*\colon \H_\pt(U)\to \H_\pt(S)$. In particular, the functor $\CW_\pt[w^{-1}]^\otimes$ lands in $\CAlg(\what\Cat{}_\infty^\mathrm{sift})$, and the composite transformation $\Fin_{+}^\otimes \to \CW_\pt^\otimes \to \CW_\pt[w^{-1}]^\otimes$ factors through the objectwise sifted cocompletion. We therefore get a natural transformation
\[
\H_{\pt}^\otimes \to \CW_\pt[w^{-1}]^\otimes\colon \Span(\FinGpd,\all,\fincov)\to \CAlg(\what\Cat{}_\infty^\mathrm{sift}),
\]
which is an equivalence by Lemma~\ref{lem:CW-inversion}.

\subsection{The real Betti realization functor}

If $X$ is a scheme of finite type over $\R$, let $\rB(X) = X(\C)$ denote its set of complex points equipped with the analytic topology and with the action of $\rm C_2$ by conjugation. Note that if $X$ is finite étale over $\R$, then $\rB(X)$ is a finite discrete topological space. If we identify finite $\rm C_2$-sets with finite coverings of $\B \rm C_2$, $\rB$ restricts to a functor
\[
e\colon \FEt_\R \to \Fin_{\B \rm C_2}, \quad \Spec(\R)\mapsto \B \rm C_2, \quad \Spec(\C)\mapsto *,
\]
which is an equivalence of categories ($e$ can also be viewed as the restriction of $\widehat\Pi_1^\et$ to $\FEt_\R$, see \sect\ref{sub:ggt-general}).
If $S$ is finite étale over $\R$, we obtain a functor
\[
\rB\colon \Sch_S^\fp \to \Fun(\B \rm C_2,\Top^\cg)_{/\rB(S)}\simeq\Fun(e(S), \Top^\cg),
\]
natural in $S\in\FEt_\R^\op$, which is readily equipped with a symmetric monoidal structure.
If $X$ is smooth over $S$, then $\rB(X)$ is a smooth manifold and hence an $e(S)$-CW-complex \cite[Theorem 3.6]{Illman1978}. It follows that $\rB$ restricts to symmetric monoidal functors
\[
	\rB\colon \Sm_S\to \CW(e(S))\quad\text{and}\quad\rB\colon \Sm_{S+} \to \CW_\pt(e(S)).
\]
Since $e$ preserves colimits and $\Fun(\ph,\Top^\cg_\pt)$ transforms colimits into limits, $S\mapsto \Fun(e(S),\Top^\cg_\pt)$ is a finite étale sheaf on $\FEt_\R$. By Lemma~\ref{lem:corr-etale-descent} and Corollary~\ref{cor:automatic-norms} applied with $\scr D=\Fun(\Delta^1,\what\Cat_1)$, the symmetric monoidal functor $\rB\colon \SmQP_{S+}\to \Fun(e(S),\Top^\cg_\pt)$ is automatically natural in $S\in\Span(\FEt_\R)$, and we obtain a natural transformation
\[
\rB\colon \SmQP_+^\otimes \to \CW_\pt^\otimes\circ e\colon \Span(\FEt_\R) \to \CAlg(\what\Cat_1).
\]
The composite transformation $\SmQP_+^\otimes \to \CW_\pt^\otimes \to \CW_\pt[w^{-1}]^\otimes\simeq \H_\pt^\otimes$ then lifts to the objectwise sifted cocompletion, and we obtain
\begin{equation}\label{eqn:unstable-Betti-0}
\rB\colon \PSh_\Sigma(\SmQP)_\pt^\otimes \to \H_\pt^\otimes\circ e\colon \Span(\FEt_\R)\to \CAlg(\what\Cat{}_\infty^\mathrm{sift}).
\end{equation}

\begin{lemma}\label{lem:rB-mot-equiv}
	Let $S\in\FEt_\R$. The functor $\rB\colon\PSh_\Sigma(\SmQP_S)_\pt\to \H_\pt(e(S))$ inverts motivic equivalences.
\end{lemma}

\begin{proof}
This is well-known, see for example \cite[Theorem 5.5]{dugger2004topological} for the case of Nisnevich equivalences. The case of $\mathbb{A}^1$-homotopy equivalences is clear.
\end{proof}

By Lemma~\ref{lem:rB-mot-equiv}, we can lift the transformation~\eqref{eqn:unstable-Betti-0} to $\CAlg(\what{\scr M\Cat}{}_\infty^\mathrm{sift})$ and we obtain
\[
\rB\colon \H_\pt^\otimes \to \H_\pt^\otimes\circ e\colon \Span(\FEt_\R)\to \CAlg(\what\Cat{}_\infty^\mathrm{sift}).
\]
Since $\rB(\mathbb{P}^1_\R)$ is the Riemann sphere with its conjugation action, which is a $\rm C_2$-representation sphere, we can lift the previous transformation to $\CAlg(\what{\scr O\Cat}{}_\infty^\mathrm{sift})$ and we obtain
\[
\rB\colon \SH^\otimes \to \SH^\otimes\circ e\colon \Span(\FEt_\R)\to \CAlg(\what\Cat{}_\infty^\mathrm{sift}).
\]

Note that $\rB$ has a pointwise right adjoint $\SingB$, inducing a relative adjunction over $\Span(\FEt_\R)$ (by Lemma~\ref{lemm:construct-relative-adjoint}(1)).
If $f$ is a morphism in $\FEt_\R$, then $\rB$ commutes with $f_\sharp$; this can be checked directly on the generators. It follows that $\SingB$ commutes with $f^*$.
Taking sections via Lemma~\ref{lemm:adjoints-pass-to-sections}, we obtain an induced adjunction 
\[ \rB\colon \NAlg_\FEt(\SH(\R)) \adj \NAlg(\SH(\B \rm C_2)): \SingB.\]

\begin{example}[Hill–Hopkins {\cite[\sect8.2]{hill2016equivariant}}] 
	\label{ex:hill-hopkins}
	Let $\rho \in [\1, \Sigma^\infty\mathbb{G}_m]_{\SH(\Z)}$ denote the element often written as $[-1]$, induced by the morphism $\Spec(\Z) \to \mathbb{G}_m$ classifying $-1 \in \Z^\times$. Let \[\1[\rho^{-1}] = \colim\left(\1 \xrightarrow{\rho} \Sigma^\infty\mathbb{G}_m \xrightarrow{\rho} \Sigma^\infty\mathbb{G}_m^{\wedge 2} \xrightarrow{\rho} \dotsb\right).\] Then $\1_\R[\rho^{-1}]$ cannot be promoted to a normed spectrum over $\FEt_\R$. Indeed, $\rB$ preserves colimits (being a left adjoint), so $\rB(\1_\R[\rho^{-1}]) \simeq \1_{\rm C_2}[\rB(\rho)^{-1}]$, and the latter spectrum is well-known not to be a normed $\rm C_2$-spectrum (see for example \cite[Example 4.12]{hill2014equivariant}).
\end{example}

\begin{remark}
	The constructions of this section have obvious analogs in the unstable unpointed setting, which we omitted for brevity. In particular, we have a canonical equivalence of functors
	\[
	\H^\otimes\simeq \CW[w^{-1}]^\otimes\colon \Span(\FinGpd,\all,\fincov)\to \CAlg(\what\Cat_\infty),
	\] 
	and a Betti realization transformation 
	\[
	\rB\colon \H^\otimes \to \H^\otimes\circ e\colon \Span(\FEt_\R)\to \CAlg(\what\Cat{}_\infty^\mathrm{sift}).
	\]
\end{remark}

\begin{remark}\label{rmk:SH-top}
	For a topological space $X$ (or more generally an $\infty$-topos), write $\SH(X)$ for the $\infty$-category of sheaves of spectra over $X$. Then $X\mapsto \SH(X)$ is a sheaf of symmetric monoidal $\infty$-categories on $\Top$. Using Corollary~\ref{cor:automatic-norms}, it extends uniquely to a functor 
	\[
	\SH^\otimes\colon \Span(\Top,\all,\fincov)\to \CAlg(\what\Cat{}_\infty^\mathrm{sift}).
	\]
	Normed spectra in this setting are simply sheaves of $\E_\infty$-ring spectra (by Proposition~\ref{prop:fold-spans} and Corollary~\ref{cor:local-spans}).
	If $X$ is a scheme of finite type over $\C$, Ayoub constructed the Betti realization functor $\rB\colon\SH(X)\to \SH(X(\C))$ \cite{AyoubBetti}, natural in $X\in\Sch_\C^{\fp,\op}$. 
	Because the target is an étale sheaf in $X$, Corollary~\ref{cor:automatic-norms} implies that $\rB$ can be uniquely extended to a natural transformation on $\Span(\Sch_\C^\fp,\all,\fet)$.
\end{remark}

\begin{remark}
We expect that there is a common generalization of the $\rm C_2$-equivariant Betti realization over $\R$ and of the relative Betti realization of Remark~\ref{rmk:SH-top}.
Namely, if $X$ is a nice enough topological stack (e.g., an orbifold), there should be an $\infty$-category $\SH(X)$ such that:
\begin{itemize}
	\item if $X$ is a topological space, then $\SH(X)$ is the $\infty$-category of sheaves of spectra on $X$;
	\item if $X$ is the classifying stack of a compact Lie group $G$, then $\SH(X)$ is the $\infty$-category of genuine $G$-spectra.
\end{itemize}
These $\infty$-categories should assemble into a functor
\[
\SH^\otimes\colon \Span(\Top\scr S\mathrm{tk},\all,\fincov) \to\CAlg(\what\Cat{}_\infty^\mathrm{sift}),
\]
where ``$\fincov$'' is now the class of $0$-truncated finite covering maps.
On the other hand, we have the functor
\[
e\colon \Span(\Sch_\R^\fp,\all,\fet) \to \Span(\Top\scr S\mathrm{tk},\all,\fincov),\quad X\mapsto [X(\C)/\rm C_2].
\]
The general equivariant Betti realization should then take the form
\[
\rB\colon \SH^\otimes \to \SH^\otimes\circ e\colon \Span(\Sch_\R^\fp,\all,\fet) \to \CAlg(\what\Cat{}_\infty^\mathrm{sift}).
\]
\end{remark}

\section{Norms and localization}
\label{sec:localization}

Let $E$ be an $\E_\infty$-ring spectrum in $\SH(S)$, $L\in \Pic(\SH(S))$ an invertible spectrum, and $\alpha\colon L^{-1}\to E$ an element in the Picard-graded homotopy of $E$. It is then well-known that the spectrum $E[\alpha^{-1}]$, defined as the colimit of the sequence
\[
E \xrightarrow \alpha E\wedge L \xrightarrow{\alpha} E\wedge L^{\wedge 2} \to\dotsb,
\]
has a structure of $\E_\infty$-algebra under $E$ (see Example~\ref{ex:SH-localization}).
It is also well-known that the spectrum $E_\alpha^\comp$, defined as the limit of the sequence
\[
\dotsb \to E/\alpha^3 \to E/\alpha^2 \to E/\alpha,
\]
has a structure of $\E_\infty$-algebra under $E$. If $E$ is a normed spectrum, however, $E[\alpha^{-1}]$ and $E_\alpha^\comp$ need not be. In this section, we give necessary and sufficient conditions on $\alpha$ for $E[\alpha^{-1}]$ and $E_\alpha^\wedge$ to be normed spectra under $E$. We will then show that $E[1/n]$ and $E_n^\wedge$ are normed spectra for any integer $n$, and hence that the rationalization $E_\Q$ is a normed spectrum.

\subsection{Inverting Picard-graded elements}

Let $\scr C$ be a symmetric monoidal $\infty$-category, $L\in\Pic(\scr C)$ an invertible object, and $\alpha\colon \1\to L$ a morphism.
For every $E\in\scr C$, we will denote by $\alpha\colon E\to E\otimes L$ the map $\id_E\otimes\alpha$, and we say that $E$ is \emph{$\alpha$-periodic}\footnote{In \cite[Definition 7.2.3.15]{HA} or \cite[Definition 2.2]{hill2014equivariant}, the term ``$\alpha$-local'' is used instead, but this conflicts with the usual meaning of ``$p$-local''. See also \cite[\sect 3]{HoyoisCdh} for a discussion of $\alpha$-periodicity when $L$ is not necessarily invertible.} if this map is an equivalence. We denote by $\L_\alpha \scr C \subset \scr C$ the full subcategory of $\alpha$-periodic objects. A map $E\to E'$ is called an \emph{$\alpha$-periodization} of $E$ if it is an initial map to an $\alpha$-periodic object, in which case we write $E'=\L_\alpha E$.

Since $(\ph)\otimes L$ is an equivalence of $\infty$-categories, $\alpha$-periodic objects are closed under arbitrary limits and colimits. Note also that an object $E\in\scr C$ is $\alpha$-periodic if and only if it is local with respect to the maps $X\otimes L^{-1}\to X$ for all $X\in\scr C$. If $\scr C$ is presentable, this class of maps is generated under colimits by a set, and it follows from the adjoint functor theorem that the inclusion $\L_\alpha \scr C \subset\scr C$ admits left and right adjoints; in particular, every object admits an $\alpha$-periodization.
If moreover the tensor product in $\scr C$ preserves colimits in each variable, then by \cite[Proposition 2.2.1.9]{HA}, the $\infty$-category $\L_\alpha \scr C$ acquires a symmetric monoidal structure such that the functor $\L_\alpha\colon \scr C\to \L_\alpha\scr C$ is symmetric monoidal.

If the colimit of the sequence
\[
E\xrightarrow\alpha E\otimes L \xrightarrow\alpha E\otimes L^{\otimes 2} \to \dotsb
\]
exists, we denote it by $E[\alpha^{-1}]$ or $E[1/\alpha]$.
It is clear that every map from $E$ to an $\alpha$-periodic object factors uniquely through $E[\alpha^{-1}]$.
However, $E[\alpha^{-1}]$ is not $\alpha$-periodic in general (see Remark~\ref{rmk:not-local}).

\begin{lemma}\label{lem:Pic-localization}
	Let $\scr C$ be a symmetric monoidal $\infty$-category, $L\in\Pic(\scr C)$ an invertible object, $\alpha\colon \1\to L$ a morphism, and $E\in\scr C$ an object such that $E[\alpha^{-1}]$ exists.
	Suppose that there exists a $1$-category $\scr D$ and a conservative functor $\pi\colon \scr C\to \scr D$ that preserves sequential colimits (for example, $\scr C$ is stable and compactly generated). Then the canonical map $E\to E[\alpha^{-1}]$ exhibits $E[\alpha^{-1}]$ as the $\alpha$-periodization of $E$. In particular, $\L_\alpha E\simeq E[\alpha^{-1}]$.
\end{lemma}

\begin{proof}
	We only have to show that $E[\alpha^{-1}]$ is $\alpha$-periodic, i.e., that the map $\alpha\colon E[\alpha^{-1}] \to E[\alpha^{-1}]\otimes L$ is an equivalence. This map is the colimit of the sequence of maps
	\begin{tikzmath}
		\diagram{
		E & E\otimes L & E \otimes L^{\otimes 2} & \dotsb \\
		E\otimes L & E\otimes L\otimes L & E \otimes L^{\otimes 2}\otimes L & \dotsb\rlap, \\
		};
		\arrows (11-) edge node[above]{$\alpha$} (-12) (12-) edge node[above]{$\alpha$} (-13) (13-) edge node[above]{$\alpha$} (-14)
		(11) edge node[left]{$\alpha$} (21)
		(12) edge node[left]{$\alpha$} (22)
		(13) edge node[left]{$\alpha$} (23)
		(21-) edge node[above]{$\alpha\otimes L$} (-22) (22-) edge node[above]{$\alpha\otimes L$} (-23) (23-) edge node[above]{$\alpha\otimes L$} (-24);
	\end{tikzmath}
	where the squares commute by naturality. Since $\pi$ is conservative and preserves sequential colimits, it suffices to show that the induced diagram
	\begin{tikzmath}
		\diagram{
		\pi(E) & \pi(E\otimes L) & \pi(E \otimes L^{\otimes 2}) & \dotsb \\
		\pi(E\otimes L) & \pi(E\otimes L\otimes L) & \pi(E \otimes L^{\otimes 2}\otimes L) & \dotsb \\
		};
		\arrows (11-) edge node[above]{$\alpha$} (-12) (12-) edge node[above]{$\alpha$} (-13) (13-) edge node[above]{$\alpha$} (-14)
		(11) edge node[left]{$\alpha$} (21)
		(12) edge node[left]{$\alpha$} (22)
		(13) edge node[left]{$\alpha$} (23)
		(21-) edge node[above]{$\alpha\otimes L$} (-22) (22-) edge node[above]{$\alpha\otimes L$} (-23) (23-) edge node[above]{$\alpha\otimes L$} (-24);
	\end{tikzmath}
	is an isomorphism of ind-objects in $\scr D$. 
	We claim that the diagonal maps
	\begin{tikzmath}
		\diagram{
		\pi(E\otimes L^{\otimes n}) & \pi(E\otimes L^{\otimes n+2}) \\
		\pi(E\otimes L^{\otimes n}\otimes L) & \pi(E\otimes L^{\otimes n+2}\otimes L) \\
		};
		\arrows (11-) edge node[above]{$\alpha^2$} (-12) 
		(11) edge node[left]{$\alpha$} (21)
		(12) edge node[left]{$\alpha$} (22)
		(21-) edge node[above]{$\alpha^2\otimes L$} (-22)
		(21) edge[dashed] node[above left]{$\alpha$} (12);
	\end{tikzmath}
	provide the inverse. Since $\scr D$ is a $1$-category, it suffices to show that both triangles commute.
	The commutativity of the upper triangle is obvious.
	Since $L$ is invertible, cyclic permutations of $L^{\otimes 3}$ are homotopic to the identity \cite[Lemma 4.17]{dugger2014coherence}. It follows that the lower triangle commutes as well.
\end{proof}

\begin{remark}
	In Lemma~\ref{lem:Pic-localization}, if we only assume that $\pi$ preserves $\lambda$-sequential colimits for some infinite ordinal $\lambda$, then the $\alpha$-periodization of $E$ can be computed by a $\lambda$-transfinite version of $E[\alpha^{-1}]$ (the proof is exactly the same). In particular, this applies whenever $\scr C$ is stable and presentable.
\end{remark}

\begin{remark}\label{rmk:not-local}
	The existence of $\pi$ in Lemma~\ref{lem:Pic-localization} is crucial for the validity of the result. 
	The $\infty$-category of small stable $\infty$-categories is an example of a compactly generated symmetric monoidal semiadditive $\infty$-category where $\1[1/n]$ is not $n$-periodic, for any integer $n\geq 2$. In fact, even transfinite iterations of the construction $E\mapsto E[1/n]$ do not reach the $n$-periodization of $\1$, which is zero.
\end{remark}

\begin{example}\label{ex:SH-localization}
	Let $E\in\CAlg(\SH(S))$. Then the symmetric monoidal $\infty$-category $\Mod_E(\SH(S))$ of $E$-modules is stable and compactly generated, so Lemma~\ref{lem:Pic-localization} applies. If $\alpha\colon E\to L$ is an $E$-module map for some $L\in\Pic(\Mod_E)$, it follows that the symmetric monoidal $\alpha$-periodization functor is given by
	\[
	\L_\alpha\colon \Mod_E \to \L_\alpha \Mod_E, \quad M\mapsto M[\alpha^{-1}].
	\]
	In particular, $E[\alpha^{-1}]$ is an $\E_\infty$-ring spectrum under $E$.
\end{example}

\begin{remark}\label{rmk:SH-localization}
	In the setting of Example~\ref{ex:SH-localization}, let $f\colon S'\to S$ be an arbitrary morphism. Then the functors $f^*\colon \Mod_{E}\to \Mod_{f^*(E)}$ and $f_*\colon \Mod_{f^*(E)}\to \Mod_E$ commute with the $\alpha$-periodization functor, because they preserve colimits and because of the projection formula
	\[
	f_*(M)\wedge_{E}L \simeq f_*(M\wedge_{f^*(E)}f^*(L))
	\]
	(which holds by dualizability of $L$). Consequently, they preserve $\alpha$-periodic objects and $\L_\alpha$-equivalences. If $f$ is smooth, the same observations apply to $f_\sharp\colon \Mod_{f^*(E)}\to \Mod_{E}$.
\end{remark}

Note that Lemma~\ref{lem:Pic-localization} applies whenever $\scr C$ is a symmetric monoidal $1$-category. An important example is the following. Let $\Gamma$ be a Picard groupoid, let $R_\star$ be a $\Gamma$-graded commutative ring (i.e., a commutative algebra for the Day convolution in $\Ab^\Gamma$), and let $\alpha\in R_\gamma$ for some $\gamma\in\Gamma$. By the Yoneda lemma, we can regard $\alpha$ as a morphism of $R_\star$-modules $R_\star\otimes\gamma\to R_\star$. Since $\gamma$ is invertible, Lemma~\ref{lem:Pic-localization} implies that the symmetric monoidal $\alpha$-periodization functor is given by
\[
\L_\alpha\colon \Mod_{R_\star}(\Ab^\Gamma) \to \L_\alpha\Mod_{R_\star}(\Ab^\Gamma),\quad M_\star\mapsto M_\star[\alpha^{-1}].
\]
In particular, $R_\star[\alpha^{-1}]$ is a $\Gamma$-graded commutative ring under $R_\star$.
We will say that $\alpha$ is a \emph{unit} if the canonical map $R_\star\to R_\star[\alpha^{-1}]$ is an isomorphism. Equivalently, $\alpha$ is a unit if there exists $\beta\in R_{\gamma^{-1}}$ such that $\alpha\beta=1$.

Let $E\in\CAlg(\SH(S))$. If $M$ is an $E$-module, we will write $\pi_\star(M)$ for the Picard-graded homotopy groups of $M$, i.e., the functor
\[
\pi_\star(M)\colon \Pic(\Mod_E) \to \Ab, \quad L\mapsto \pi_L(M) = \pi_0\Map_E(L,M). 
\]
The functor $M\mapsto\pi_\star(M)$ is right-lax symmetric monoidal. In particular, if $M$ is an incoherent commutative $E$-algebra, then $\pi_\star(M)$ is a $\Pic(\Mod_E)$-graded commutative ring.

\subsection{Inverting elements in normed spectra}

We now turn to the question of periodizing normed spectra.
Let $\scr C\subset_\fet\Sch_S$, let $E\in\NAlg_\scr C(\SH)$, and let $\alpha\colon E\to L$ be an $E$-module map for some $L\in\Pic(\Mod_E)$. We say that a normed $E$-module is \emph{$\alpha$-periodic} if its underlying $E$-module is $\alpha$-periodic, and we denote by
\[
\L_\alpha^\otimes \NAlg_\scr C(\Mod_E) \subset \NAlg_\scr C(\Mod_E)
\]
the full subcategory of $\alpha$-periodic normed $E$-modules. If $R$ is a normed $E$-module, we will denote by $R\to \L_\alpha^\otimes R$ an initial map to an $\alpha$-periodic normed $E$-module, if it exists.

\begin{proposition} \label{prop:localization-general}
Let $\scr C \subset_\fet \Sch_S$, let $E \in \NAlg_{\scr C}(\SH)$, and let $\alpha \colon E\to L$ be an $E$-module map for some $L\in\Pic(\Mod_E)$. 
\begin{enumerate}
	\item If $\scr C\subset\Sm_S$, the inclusion $\L_\alpha^\otimes\NAlg_\scr C(\Mod_E)\subset \NAlg_\scr C(\Mod_E)$ has a left adjoint $\L_\alpha^\otimes$.
	\item The following conditions are equivalent:
	\begin{enumerate}
		\item $\L_\alpha^\otimes E$ exists and for every $X\in\scr C$ the canonical $E_X$-module map
\[
E_X[\alpha^{-1}_X] \to (\L_\alpha^\otimes E)_X
\]
is an equivalence.
		\item The inclusion $\L_\alpha^\otimes \NAlg_\scr C(\Mod_E) \subset \NAlg_\scr C(\Mod_E)$ has a left adjoint $\L_\alpha^\otimes$, and for every normed $E$-module $R$ and $X\in\scr C$, the canonical $R_X$-module map
\[
R_X[\alpha^{-1}_X] \to (\L_\alpha^\otimes R)_X
\]
is an equivalence.
\item For every finite étale morphism $f\colon X \to Y$ in $\scr C$, the element $\nu_f(\alpha_X)$ is a unit in the $\Pic(\Mod_{E_Y})$-graded commutative ring $\pi_\star(E_Y)[\alpha^{-1}_Y]$. (Here, $\nu_f$ is the multiplicative transfer defined in \sect\ref{subsec:coh-theories-with-norms}.)
	\end{enumerate}
\end{enumerate}
\end{proposition}

\begin{proof}
(1) This is an application of the adjoint functor theorem.
Since the subcategory $\L_\alpha \Mod_E \subset \Mod_E$ is reflective and stable under colimits, it is accessible by \cite[Proposition 5.5.1.2]{HTT}. Since $\NAlg_{\scr C}(\SH)$ is accessible and $\NAlg_{\scr C}(\Mod_E) \simeq \NAlg_{\scr C}(\SH)_{E/}$ by Proposition \ref{prop:categorical-props}(1,4), we find that $\NAlg_{\scr C}(\Mod_E)$ is accessible \cite[Corollary 5.4.5.16]{HTT}. By definition, we have a pullback square
\begin{tikzmath}
	\diagram{\L_\alpha^\otimes \NAlg_\scr C(\Mod_E) & \NAlg_\scr C(\Mod_E) \\ \L_\alpha \Mod_E & \Mod_E\rlap, \\};
	\arrows (11-) edge[c->] (-12) (21-) edge[c->] (-22) (11) edge (21) (12) edge (22);
\end{tikzmath}
where the bottom horizontal map preserves colimits and the right vertical map preserves sifted colimits (since the composition $\NAlg_{\scr C}(\SH)_{E/}\to \NAlg_{\scr C}(\SH)\to \SH(S)$ does by Proposition~\ref{prop:categorical-props}(2)).
It follows from \cite[Lemma 5.4.5.5 and Proposition 5.4.6.6]{HTT} that $\L_\alpha^\otimes \NAlg_\scr C(\Mod_E)$ is accessible and has sifted colimits, and that the inclusion $\L_\alpha^\otimes \NAlg_{\scr C}(\Mod_E) \into \NAlg_{\scr C}(\Mod_E)$ preserves sifted colimits. Clearly, a tensor product of $\alpha$-periodic $E$-modules is $\alpha$-periodic. It follows from Remark \ref{rmk:coprod-nalg} and Lemma~\ref{lem:sifted+coprod} that small colimits of $\alpha$-periodic normed $E$-modules are $\alpha$-periodic; in particular $\L_\alpha^\otimes \NAlg_{\scr C}(\Mod_E)$ is presentable. Since limits of normed $E$-modules are computed in $E$-modules by Proposition~\ref{prop:categorical-props}(2) (using the assumption $\scr C\subset\Sm_S$), the functor $\L_\alpha^\otimes \NAlg_{\scr C}(\Mod_E) \into \NAlg_{\scr C}(\Mod_E)$ preserves limits. Hence we may apply the adjoint functor theorem \cite[Corollary 5.5.2.9]{HTT} to obtain the left adjoint $\L_\alpha^\otimes$.

(2) We will prove (a) $\Rightarrow$ (c) $\Rightarrow$ (b); it is obvious that (b) $\Rightarrow$ (a).
 Suppose that $\L_\alpha^\otimes E$ exists.
For every finite étale morphism $f\colon X \to Y$ in $\scr C$, we have a commutative square of $\Pic(\Mod_{E_X})$-graded rings
\begin{tikzmath}
	\diagram{
	\pi_\star(E_X) & \pi_\star(\L_\alpha^\otimes E_X) \\
	\pi_{f_\otimes(\star)}(E_Y) & \pi_{f_\otimes(\star)}(\L_\alpha^\otimes E_Y)\rlap,\\
	};
	\arrows (11-) edge (-12) (11) edge node[left]{$\nu_f$} (21) (21-) edge (-22) (12) edge node[right]{$\nu_f$} (22);
\end{tikzmath}
where the horizontal maps are ring homomorphisms and the vertical maps are multiplicative.
The element $\alpha_X \in \pi_{L^{-1}}(E_X)$ is mapped to a unit in $\pi_\star(\L_\alpha^\otimes E_X)$, and since $\nu_f$ preserves units, it follows that $\nu_f(\alpha_X)$ is mapped to a unit in $\pi_\star(\L_\alpha^\otimes E_Y)$.
Under assumption (a), the canonical map of $\Pic(\Mod_{E_Y})$-graded rings
\[
\pi_\star(E_Y)[\alpha^{-1}_Y]\to \pi_\star(\L_\alpha^\otimes E_Y)
\]
is an isomorphism, since $\pi_\star$ preserves filtered colimits. Hence $\nu_f(\alpha_X)$ is a unit in $\pi_\star(E_Y)[\alpha^{-1}_Y]$. This proves (a) $\Rightarrow$ (c).

To prove (c) $\Rightarrow$ (b), we will apply Corollary~\ref{cor:localization-sections} to the cocartesian fibration classified by 
\[\Mod_{E}^\otimes\colon \Span(\scr C, \all, \fet)\to\what\Cat{}_\infty\]
 (see Proposition \ref{prop:categorical-props}(4)) with the localization functors $\Mod_{E_X}(\SH(X)) \to \Mod_{E_X}(\SH(X))$, $M \mapsto M[\alpha^{-1}_X]$ (see Example~\ref{ex:SH-localization}). Since these localization functors are given by filtered colimits, they commute with pullback functors (i.e., functors of the form $f^*$). In particular, pullback functors preserve $\L_\alpha$-equivalences. Since norm functors preserve filtered colimits, we similarly have $f_\otimes (M[\alpha^{-1}_X]) \simeq f_\otimes(M)[\nu_f(\alpha_Y)^{-1}]$ for every finite étale map $f\colon X\to Y$ in $\scr C$, and assumption (c) implies that the $E_Y$-module map $f_\otimes(M)\to f_\otimes(M)[\nu_f(\alpha_Y)^{-1}]$ is an $\L_\alpha$-equivalence. Hence, norm functors also preserve $\L_\alpha$-equivalences and the assumptions of Corollary~\ref{cor:localization-sections} are satisfied. We deduce that the full subcategory of $\Sect(\Mod_{E}^\otimes)$ consisting of the pointwise $\alpha$-periodic sections is reflective, with localization functor $F\colon \Sect(\Mod_{E}^\otimes) \to \Sect(\Mod_{E}^\otimes)$ given pointwise by $M\mapsto M[\alpha^{-1}_X]$. It is then clear that $F$ preserves the subcategory $\NAlg_{\scr C}(\Mod_{E})\subset\Sect(\Mod_{E}^\otimes)$, hence that $\L_\alpha^\otimes$ exists and has the claimed form.
\end{proof}

\begin{remark}
	If $\alpha \in [\1, E] = \pi_{0,0}(E)$, then condition (c) in Proposition~\ref{prop:localization-general}(2) reduces to $\nu_f(\alpha_X) \in (\pi_{0,0}(E_Y)[\alpha^{-1}_Y])^\times$, i.e., there is no need to consider Picard-graded homotopy groups.
\end{remark}

In the above discussion of $\alpha$-periodic objects and $\alpha$-periodization, one can replace $\alpha\colon \1\to L$ by an arbitrary family $\alpha=\{\alpha_i\colon \1\to L_i\}_{i\in I}$ of such elements. The construction $E[\alpha^{-1}]$ becomes a filtered colimit indexed by the poset of maps $I\to\bb N$ with finitely many nonzero values. An obvious modification of Proposition~\ref{prop:localization-general} holds in this generalized setting: one can either repeat the proof or reduce the general case to the single-element case.

\begin{proposition}\label{prop:E[1/n]}
Let $\scr C \subset_\fet \Sch_S$, let $E \in \NAlg_{\scr C}(\SH)$, and let $\alpha\subset \bb N$ be a set of nonnegative integers.
Then $\L_\alpha^\otimes E\in \NAlg_{\scr C}(\SH)$ exists and the canonical $E$-module map 
\[
E[\alpha^{-1}]\to \L_\alpha^\otimes E
\]
is an equivalence. 
In particular, the rationalization $E_\Q$ is a normed spectrum over $\scr C$.
\end{proposition}

\begin{proof}
By Proposition~\ref{prop:localization-general}(2), this follows from the first part of Lemma \ref{lemm:GW-units} below.
\end{proof}

\begin{lemma} \label{lemm:GW-units}
Let $f\colon X \to Y$ be a finite étale map. Then for any $n \in \bb N$ we have $f_\otimes(n) \in (\pi_{0,0}(\1_Y)[1/n])^\times$ and, provided that $f$ is surjective, also $n \in (\pi_{0,0}(\1_Y)[1/f_\otimes(n)])^\times$.
\end{lemma}

\begin{proof}
Using the compatibility with Grothendieck's Galois theory, i.e., the natural transformation \eqref{eqn:GaloisSH} (together with Corollary \ref{cor:finiteGalois}), we reduce to showing the analogous statement for $f\colon X\to Y$ a finite covering map of profinite groupoids. 
Since finite covering maps of profinite groupoids are defined at a finite stage (by definition), using Lemma \ref{lem:continuity} and the base change formula for norms, we can reduce to the case where $X$ and $Y$ are finite groupoids. 
We may then choose a finite covering map $Y\to \B G$ where $G$ is a finite group.

It is thus enough to prove that for any finite group $G$ and any map of finite $G$-sets $f\colon X \to Y$, the elements $f_\otimes(n) \in\rm A_G(Y)[1/n]$ and $n \in\rm A_G(Y)[1/f_\otimes(n)]$ are invertible (the latter assuming $f$ surjective), where $\rm A_G$ is the Burnside $G$-Tambara functor. The former is claimed to be true in \cite[second to last sentence]{hill2014equivariant}.
We give a proof of both claims in Lemma \ref{lemm:AG-units} below.
\end{proof}

\begin{lemma} \label{lemm:AG-units}
Let $G$ be a finite group and $f\colon X \to Y$ a map of finite $G$-sets. Then for any $n\in\bb N$ we have $f_\otimes(n) \in(\rm A_G(Y)[1/n])^\times$ and, provided that $f$ is surjective, also $n \in(\rm A_G(Y)[1/f_\otimes(n)])^\times$.
\end{lemma}

\begin{proof}
By decomposing into $G$-orbits, we may assume that $X = G/H_1$ and $Y = G/H_2$ for $H_1 \subset H_2$ subgroups. Then $\rm A_G(X) = \rm A(H_1)$ and $\rm A_G(Y) =\rm A(H_2)$, so we may just as well prove: if $H \subset G$ is a subgroup of a finite group, then $\Norm_H^G(n) \in (\rm A(G)[1/n])^\times$ and $n \in (\rm A(G)[1/\Norm_H^G(n)])^\times$. The last statement is equivalent to: every prime ideal of $\rm A(G)$ containing $\Norm_H^G(n)$ (respectively $n$) also contains $n$ (respectively $\Norm_H^G(n)$); indeed both statements are equivalent to $\rm A(G)/(\Norm_H^G(n))[1/n] = 0$ (respectively $\rm A(G)/(n)[1/\Norm_H^G(n))] = 0$). For a subgroup $K$ of $G$ define $\phi_{K}\colon\rm A(G) \to \Z$ as the additive extension of the map sending a finite $G$-set $X$ to $\lvert X^{K}\rvert$. Then $\phi_{K}$ is a ring homomorphism and every prime ideal of $\rm A(G)$ is of the form $\phi_{K}^{-1}((p))$ for some subgroup $K$ and some prime ideal $(p) \subset \Z$ \cite[Proposition V.3.1]{lewis1986equivariant}.

It thus suffices to prove the following: if $H\subset G$ and $K \subset G$ are subgroups and $p$ is a prime number, then $p$ divides $n$ if and only if $p$ divides $\phi_{K}(\Norm_H^G(n))$. Since 
\[ \phi_{K}(\Norm_H^G(n)) = \lvert\Map(G/H, n)^{K}\rvert = \lvert \Map(K\backslash G/H, n)\rvert = n^{\lvert K\backslash G/H\rvert}, \] 
this is clear.
\end{proof}

\begin{example}\label{ex:eta-localization}
As in equivariant homotopy theory, it can happen that $\L_\alpha^\otimes E\not\simeq E[\alpha^{-1}]$. For example, if $S$ has a point of characteristic $\neq 2$ and $\eta\in \pi_{1,1}(\1)$ is the Hopf element, then $\1[\eta^{-1}]$ cannot be promoted to a normed spectrum over $\Sm_S$. 
Indeed, by Proposition~\ref{prop:categorical-props}(6,7), we may assume that $S$ is the spectrum of a field $k$ of characteristic $\neq 2$. In this case, $\pi_{0,0}(\1)\to \pi_{0,0}(\1[\eta^{-1}])$ is identified with the quotient map $\GW(k)\to \W(k)$, and if $k'/k$ is a finite separable extension, the norm $\GW(k')\to \GW(k)$ is Rost's multiplicative transfer (Theorem~\ref{thm:GW-norms-comparison}). But if $k$ is not quadratically closed, as we may arrange, and $k'/k$ is a quadratic extension, then the norm $\GW(k') \to \GW(k)$ does not descend to $\W(k') \to \W(k)$ (i.e., the norm does not preserve hyperbolic forms) \cite[Bemerkung 2.14]{wittkop2006multiplikative}.
\end{example}

\begin{example}
We can strengthen Example~\ref{ex:eta-localization} as follows: if $S$ is pro-smooth over a field of characteristic $\neq 2$, then $\L_\eta^\otimes \1\simeq 0$ in $\NAlg_{\Sm}(\SH(S))$. As a first step, we note that if $f\colon S'\to S$ is pro-smooth, then $f^*$ commutes with the formation of $\L_\alpha^\otimes$ for any $\alpha\colon\1_S\to L$. Indeed, the functors $f^*$ and $f_*$ lift to an adjunction between normed spectra (Example~\ref{ex:NAlg-ess-smooth-adjunction}), and they both preserve $\alpha$-periodic spectra (Remark~\ref{rmk:SH-localization}).
 Thus we may assume that $S$ is the spectrum of a perfect field $k$. Consider the homotopy module $\ul{\pi}_0(\L_\eta^\otimes \1)_*$. Clearly this is a module over $\ul{\pi}_0(\1[\eta^{-1}])_* = \ul{\W}[\eta^{\pm 1}]$. Now let $K/k$ be a finitely generated field extension and $a \in K^\times$. The image of $\Norm_{K(\sqrt{a})/K}(\langle 1 \rangle + \langle -1 \rangle) \in \GW(K)$ must be zero in $\ul{\pi}_0(\L_\eta^\otimes \1)_0(K)$, by the previous remark. It follows from \cite[Lemma 2.13]{wittkop2006multiplikative} that the class of $\Norm_{K(\sqrt{a})/K}(\langle 1 \rangle + \langle -1 \rangle)$ in $\W(K)$ is $\langle 2 \rangle - \langle 2a \rangle$. Applied to $a = 2$, this shows that $\langle 2 \rangle = 1$ in $\ul{\pi}_0(\L_\eta^\otimes \1)_0(K)$. Applying this again to a general $a$ we find that $\langle a \rangle = 1$ in $\ul{\pi}_0(\L_\eta^\otimes \1)_0(K)$. We have 
\[
\ul{\pi}_0(\L_\eta^\otimes \1)_{-1}(k) = \ul{\pi}_0(\L_\eta^\otimes \1)_0(\mathbb{A}^1 \minus \{0\})/\ul{\pi}_0(\L_\eta^\otimes \1)_0(k) \into \ul{\pi}_0(\L_\eta^\otimes \1)_0(k(t))/\ul{\pi}_0(\L_\eta^\otimes \1)_0(k),
\]
and similarly for $\ul{\pi}_0(\1)_*$. We conclude that $\ul{\pi}_0(\1)_{-1}(k) \to \ul{\pi}_0(\L_\eta^\otimes \1)_{-1}(k)$ is the zero map, and so $\eta = 0$ in $\ul{\pi}_0(\L_\eta^\otimes \1)_{-1}$. This implies that $1= 0$ in $\ul{\pi}_0(\L_\eta^\otimes \1)_0$ and so $\L_\eta^\otimes \1 \simeq 0$.
\end{example}

\subsection{Completion of normed spectra}
Given a symmetric monoidal $\infty$-category $\scr C$, an invertible object $L\in\Pic(\scr C)$, and a morphism $\alpha\colon \1 \to L$, recall that an object $E\in\scr C$ is called $\alpha$-periodic if the morphism $\alpha\colon E \to E \otimes L$ is an equivalence.
We say that $E$ is \emph{$\alpha$-complete} if $\Map(A,E)\simeq *$ for every $\alpha$-periodic object $A\in\scr C$, and we denote by $\scr C_\alpha^\comp\subset \scr C$ the full subcategory of $\alpha$-complete objects. 
A map $E\to E'$ is called an \emph{$\alpha$-completion} of $E$ if it is an initial map to an $\alpha$-complete object, in which case we write $E'=E_\alpha^\comp$. If $\scr C$ is presentable, $E_\alpha^\comp$ always exists and 
$E \mapsto E_\alpha^\comp$ is left adjoint to the inclusion $\scr C_\alpha^\comp \subset \scr C$.
Moreover, by \cite[Proposition 2.2.1.9]{HA}, the $\infty$-category $\scr C_\alpha^\comp$ acquires a symmetric monoidal structure such that the functor $(\ph)_\alpha^\comp\colon \scr C\to \scr C_\alpha^\comp$ is symmetric monoidal.

If $\scr C$ is stable, $\alpha$-completion is given by a simple formula:

\begin{lemma}\label{lem:Pic-completion}
	Let $\scr C$ be a stable symmetric monoidal $\infty$-category, $L\in\Pic(\scr C)$ an invertible object, and $\alpha\colon \1\to L$ a morphism. For every $E\in\scr C$, the canonical map $E\to \lim_n E/\alpha^n$ exhibits $\lim_n E/\alpha^n$ as the $\alpha$-completion of $E$, whenever this limit exists.
\end{lemma}

\begin{proof}
	It is clear that $E/\alpha^n$ and hence $\lim_n E/\alpha^n$ are $\alpha$-complete, so it remains to show that the fiber of $E\to \lim_n E/\alpha^n$ is $\alpha$-periodic. This fiber is the limit of the tower
	\[
	E \xleftarrow{\alpha} E\otimes L^{-1} \xleftarrow{\alpha} E \otimes (L^{-1})^{\otimes 2} \leftarrow \dotsb.
	\]
	We must therefore show that the vertical maps in the diagram
	\begin{tikzmath}
		\def\colsep{4em}
		\diagram{
		E\otimes L^{-1} & E\otimes L^{-1}\otimes L^{-1} & E \otimes (L^{-1})^{\otimes 2}\otimes L^{-1} & \dotsb \\
		E & E\otimes L^{-1} & E \otimes (L^{-1})^{\otimes 2} & \dotsb \\
		};
		\arrows (11-) edge[<-] node[above]{$\alpha\otimes L^{-1}$} (-12) (12-) edge[<-] node[above]{$\alpha\otimes L^{-1}$} (-13) (13-) edge[<-] node[above]{$\alpha\otimes L^{-1}$} (-14)
		(11) edge node[left]{$\alpha$} (21)
		(12) edge node[left]{$\alpha$} (22)
		(13) edge node[left]{$\alpha$} (23)
		(21-) edge[<-] node[above]{$\alpha$} (-22) (22-) edge[<-] node[above]{$\alpha$} (-23) (23-) edge[<-] node[above]{$\alpha$} (-24);
	\end{tikzmath}
	induce an equivalence in the limit. For every $X\in\scr C$, applying $[X,\ph]$ to this diagram yields an isomorphism of pro-abelian groups (this uses the invertibility of $L$ as in the proof of Lemma~\ref{lem:Pic-localization}), and we conclude using the Milnor exact sequence.
\end{proof}

\begin{proposition} \label{prop:compn-general}
Let $\scr C \subset_\fet \Sm_S$, let $E \in \NAlg_{\scr C}(\SH)$, and let $\alpha \colon E\to L$ be an $E$-module map for some $L\in\<\Pic(\Mod_{E})$. Suppose that, for every surjective finite étale map $f\colon X \to Y$ in $\scr C$, $\alpha_Y$ is a unit in the $\Pic(\Mod_{E_Y})$-graded ring $\pi_\star(E_Y)[\nu_f(\alpha_X)^{-1}]$. Then there is an adjunction
\begin{tikzmath}
		\diagram{\NAlg_\scr C(\Mod_E(\SH)) & \NAlg_\scr C(\Mod_E(\SH)_\alpha^\comp)\\};
		\arrows (11-) edge[vshift=\dbl] node[above,vshift=\dbl]{$(\ph)_\alpha^\comp$} (-12) (-12) edge[vshift=\dbl,c->] (11-);
\end{tikzmath}
where the functor $(\ph)_\alpha^\comp$ is computed pointwise. In particular, $E_\alpha^\comp$ is a normed spectrum over $\scr C$.
\end{proposition}

\begin{proof}
If $f\colon X \to Y$ is a surjective finite étale map in $\scr C$ and $M \in \Mod_{E_X}(\SH(X))$ is $\alpha_X$-periodic, then $f_\otimes(M) \in \Mod_{E_Y}(\SH(Y))$ is clearly $\nu_f(\alpha_X)$-periodic, whence $\alpha_Y$-periodic by the assumption. This shows that the $\infty$-categories $\L_\alpha\Mod_{E_X}(\SH(X))$ form a normed ideal in $\Mod_E(\SH)^\otimes$ over $\scr C$ (see Definition~\ref{def:normed-ideal}). Applying Corollary~\ref{cor:normed-nullification} to this normed ideal, we deduce that the $\alpha$-completion functors assemble into a natural transformation 
\[
(\ph)_\alpha^\comp\colon \Mod_E(\SH)^{\otimes} \to \Mod_E(\SH)_\alpha^{\comp\otimes}\colon \Span(\scr C, \all, \fet) \to \CAlg(\what{\Cat}{}_\infty^\mathrm{sift}). 
\]
By Lemmas \ref{lemm:construct-relative-adjoint}(1) and~\ref{lemm:adjoints-pass-to-sections}, we obtain an adjunction between the $\infty$-categories of sections such that the left adjoint preserves normed $E$-modules. It remains to show that the right adjoint also preserves normed $E$-modules. As $\scr C\subset \Sm_S$, this follows from the fact that if $f$ is smooth, then $f_\sharp$ preserves $\alpha$-periodic objects (Remark~\ref{rmk:SH-localization}), and hence $f^*$ preserves $\alpha$-complete objects.
\end{proof}

\begin{proposition}
Let $\scr C\subset_\fet \Sm_S$, let $E\in\NAlg_\scr C(\SH)$, and let $n \in \bb N$. 
Then the $n$-completion $E_n^\comp$ has a canonical structure of normed spectrum over $\scr C$ such that the canonical map $E\to E_n^\comp$ is a morphism of normed spectra.
\end{proposition}

\begin{proof}
This follows from Proposition~\ref{prop:compn-general} using the second part of Lemma \ref{lemm:GW-units}.
\end{proof}

\begin{example}\label{ex:rho-completion}
Let $a\in\scr O(S)^\times$ and let $[a]$ be the corresponding morphism $\1_S\to \Sigma^\infty\G_{m}$ in $\SH(S)$.
Then $[a]$ satisfies the assumption of Proposition~\ref{prop:compn-general}, so that $[a]$-completion preserves normed spectra. Indeed, for any finite étale map $p\colon T\to S$, we have $\Weil_p (a_T)=u \circ a$ where $u\colon \G_{m,S}\to \Weil_p \G_{m,T}$ is the unit map. If $p$ is surjective, then the pointed map $u_+$ descends to a morphism $\G_m \to p_\otimes(\G_m)$ in $\PSh_\Sigma(\Sm_S)_\pt$, which shows that $[a]$ divides $p_\otimes p^*([a])$ in the $\Pic(\SH(S))$-graded ring $\pi_\star(\1_S)$.
\end{example}

\begin{example}\label{ex:h-completion}
	Let $h=1+\langle -1\rangle\in \GW(\R)\simeq \pi_{0,0}(\1_\R)$. Then $(\1_\R)_h^\comp \in \SH(\R)$ is an $\E_\infty$-ring spectrum that cannot be promoted to a normed spectrum over $\FEt_\R$. To see this, recall that
	\[
	\GW(\C)\simeq\Z\quad\text{and}\quad\GW(\R)\simeq\Z\times_{\Z/2}\W(\R)\simeq \Z\times_{\Z/2}\Z,
	\]
	with the element $h$ corresponding to $2$ and $(2,0)$, respectively. Observe that the $h^n$-torsion in these groups does not depend on $n$, and hence that $\pi_{0,0}((\1_\C)_h^\comp)\simeq\Z_2^\comp$ and $\pi_{0,0}((\1_\R)_h^\comp)\simeq \Z_2^\comp\times_{\Z/2}\Z$. Under the above identifications, the Rost norm $\GW(\C)\to\GW(\R)$ is the map $\Z\to \Z\times_{\Z/2}\Z$, $n\mapsto (n^2,n)$, which clearly does not extend to a multiplicative map $\Z_2^\comp \to \Z_2^\comp\times_{\Z/2}\Z$.
	
	For a closely related example, consider the spectrum $\1_{\bb R}[1/2]$, which is a normed spectrum over $\Sch_{\bb R}$ by Proposition~\ref{prop:E[1/n]}. Its $h$-completion is a nonzero spectrum whose pullback to $\C$ is zero, so it does not admit a structure of normed spectrum over $\FEt_{\bb R}$.
\end{example}

\section{Norms and the slice filtration}
\label{sec:slices}

In this section we show that the norm functors preserve effective and very effective spectra, and explore some of the consequences of this statement. Recall that the full subcategory of \emph{effective spectra} $\SH(S)^\eff \subset \SH(S)$ is generated under colimits by $\Sigma^{-n}\Sigma^\infty_+ X$ for $X \in \Sm_S$ and $n\geq 0$, and the full subcategory of \emph{very effective spectra} $\SH(S)^\veff\subset\SH(S)$ is generated under colimits and extensions by $\Sigma^\infty_+ X$ for $X\in\Sm_S$. Both $\SH(S)^\eff$ and $\SH(S)^\veff$ are symmetric monoidal subcategories of $\SH(S)$.

\subsection{The zeroth slice of a normed spectrum}
\label{sub:zero-slice}

We start by reviewing Voevodsky's slice filtration \cite{Voevodsky:2002} and its generalized variant defined in \cite{Bachmann-slices}.

For $n\in\Z$, we denote by $\SH(S)^\eff(n)\subset\SH(S)$ the subcategory of $n$-effective spectra, i.e., spectra of the form $E\wedge \S^{2n,n}$ with $E$ effective.
We have colimit-preserving inclusions
\[
\dotsb \subset\SH(S)^\eff(n+1) \subset \SH(S)^\eff(n) \subset \SH(S)^\eff(n-1) \subset\dotsb\subset \SH(S),
\]
giving rise to a functorial filtration
\[
\dotsb\to \f_{n+1}E \to \f_nE \to \f_{n-1}E \to \dotsb \to E
\]
of every spectrum $E\in\SH(S)$, where $\f_n$ is the right adjoint to the inclusion $\SH(S)^\eff(n)\subset\SH(S)$. The slice functors $\s_n\colon \SH(S)\to\SH(S)$ are then defined by the cofiber sequences 
\[
\f_{n+1}E\to \f_nE \to \s_nE.
\]

Similarly, we denote by $\SH(S)^\veff(n)\subset\SH(S)$ the subcategory of very $n$-effective spectra, i.e., spectra of the form $E\wedge \S^{2n,n}$ with $E$ very effective. We have colimit-preserving inclusions
\[
\dotsb \subset\SH(S)^\veff(n+1) \subset \SH(S)^\veff(n) \subset \SH(S)^\veff(n-1) \subset\dotsb\subset \SH(S),
\]
giving rise to a functorial filtration
\[
\dotsb\to \tilde\f_{n+1}E \to \tilde\f_nE \to \tilde\f_{n-1}E \to \dotsb \to E
\]
of every spectrum $E\in\SH(S)$, where $\tilde\f_n$ is the right adjoint to the inclusion $\SH(S)^\veff(n)\subset\SH(S)$. The generalized slice functors $\tilde\s_n\colon \SH(S)\to\SH(S)$ are then defined by the cofiber sequences 
\[
\tilde\f_{n+1}E\to \tilde\f_nE \to \tilde\s_nE.
\]

Finally, we note that $\SH(S)^\veff$ is the nonnegative part of a $t$-structure on $\SH(S)^\eff$ \cite[Proposition 1.4.4.11]{HA}. We will denote by $\SH(S)^\eff_{\geq n}$ and $\SH(S)^\eff_{\leq n}$ the subcategories of $n$-connective and $n$-truncated objects with respect to this $t$-structure, and by $\SH(S)^{\eff\heartsuit}$ its heart. Thus, $\SH(S)^\eff_{\geq 0}=\SH(S)^\veff$.

\begin{lemma}\label{lemm:effectivity-of-spheres}
	Let $S$ be a scheme, let $\xi\in\K(S)$ be of rank $n$, and let $\S^\xi\in\SH(S)$ be the Thom spectrum of $\xi$ (see for example \sect\ref{sub:Jhomomorphism}). Then $\S^\xi$ is very $n$-effective.
\end{lemma}

\begin{proof}
	Any $\K$-theory class is Zariski-locally trivial. Since $\tilde\f_n$ commutes with smooth base change, the canonical map $\tilde\f_n \S^\xi \to \S^\xi$ is an equivalence over any open $U \subset S$ trivializing $\xi$. Hence, it is an equivalence.
\end{proof}

\begin{lemma} \label{lemm:norms-circles}
Let $p\colon T\to S$ be a finite étale map.
\begin{enumerate}
\item $p_\otimes(\S^{-1}) \in \SH(S)^\eff$.
\item If $p$ is surjective, then $p_\otimes(\S^1) \in \SH(S)^\eff_{\geq 1}$.
\item If $p$ has degree $d$ and $\xi\in\K(T)$ has rank $n$, then $p_\otimes(\S^\xi) \in \SH(S)^\veff(dn)$.
\end{enumerate}
\end{lemma}

\begin{proof}
We will use the compatibility of Grothendieck's Galois theory with norms (Proposition~\ref{prop:ggt-compat}):
the morphism $\tilde p=\widehat\Pi_1^\et(p)$ is a finite covering map of profinite groupoids, and we have 
\[p_\otimes\circ c_T \simeq c_S\circ \tilde p_\otimes\colon \SH(\widehat\Pi_1^\et(T))\to \SH(S).\]
For any profinite groupoid $X$, Proposition~\ref{prop:spectralMackey} implies that $\SH(X)$ is generated under colimits and shifts by finite $X$-sets, and it follows that $c_S$ factors through effective spectra.
Assertion (1) follows immediately, since $p_\otimes(\S^{-1}) \simeq c_S(\tilde p_\otimes(\S^{-1}))$.

(2) If $p$ is surjective, then $\tilde p$ is also surjective. It therefore suffices to show the following: if $q\colon Y\to X$ is a surjective finite covering map of profinite groupoids, then $q_\otimes(\S^1)$ is a suspension in $\H_\pt(X)$. We can clearly assume that $X=\B G$ for some finite group $G$, so that $Y$ is the groupoid associated with a nonempty finite $G$-set $A$. Then the $G$-space $q_\otimes(\S^1)$ is the one-point compactification of the real $G$-representation $\R^A$. The claim follows because $\R^A$ has a trivial $1$-dimensional summand (consisting of the diagonal vectors).

(3) We have $p_\otimes(\S^\xi)\simeq \S^{p_*(\xi)}$ by Remark~\ref{rmk:Pic}, so the result follows from Lemma~\ref{lemm:effectivity-of-spheres}.
\end{proof}

We now apply the results of \sect\ref{sub:normed-categories} to construct various normed $\infty$-categories from $\SH^\otimes$. Since effective spectra in $\SH(S)$ are generated under colimits by $(\S^{-1})^{\wedge n}\wedge\Sigma^\infty_+ X$ for $X\in\SmQP_S$ and $n\geq 0$, it follows from Proposition~\ref{prop:normed-subcategory} and Lemma~\ref{lemm:norms-circles}(1) that effective spectra form a subfunctor
\[
\SH^{\eff\otimes}\subset \SH^{\otimes}\colon \Span(\Sch, \all, \fet) \to \CAlg(\what{\Cat}{}_\infty^\mathrm{sift}).
\]
Similarly, since very effective spectra are generated under colimits and extensions by $\Sigma^\infty_+ X$ for $X\in\SmQP_S$, it follows from Proposition~\ref{prop:normed-subcategory} that very effective spectra form a subfunctor
\[
\SH^{\veff\otimes}\subset \SH^{\eff\otimes}\colon \Span(\Sch, \all, \fet) \to \CAlg(\what{\Cat}{}_\infty^\mathrm{sift}).
\]
By Lemma~\ref{lemm:norms-circles}(3), $1$-effective spectra (resp.\ very $1$-effective spectra) form a normed ideal in $\SH^{\eff\otimes}$ (resp.\ in $\SH^{\veff\otimes}$). By Lemma~\ref{lemm:norms-circles}(2), $1$-connective effective spectra also form a normed ideal in $\SH^{\veff\otimes}$. Applying Corollary~\ref{cor:normed-nullification} to these normed ideals, we obtain natural transformations
\begin{gather*}
\s_0\colon \SH^{\eff\otimes} \to \s_0\SH^{\eff\otimes}\colon \Span(\Sch, \all, \fet) \to \CAlg(\what{\Cat}{}_\infty^\mathrm{sift}),\\
\tilde\s_0\colon \SH^{\veff\otimes} \to \tilde\s_0\SH^{\veff\otimes}\colon \Span(\Sch, \all, \fet) \to \CAlg(\what{\Cat}{}_\infty^\mathrm{sift}),\\
\spi_0^\eff\colon \SH^{\veff\otimes} \to \SH^{\eff\heartsuit\otimes}\colon \Span(\Sch, \all, \fet) \to \CAlg(\what{\Cat}{}_\infty^\mathrm{sift}).
\end{gather*}

\begin{proposition}\label{prop:NAlg-slices}
	Let $S$ be a scheme and $\scr C\subset_\fet\Sm_S$. Then there are adjunctions
	\begin{gather*}
		\begin{tikzpicture}[ampersand replacement=\&]
			\diagram{\NAlg_\scr C(\SH^\eff) \& \NAlg_\scr C(\SH)\rlap,\\};
			\arrows (11-) edge[vshift=\dbl,c->] (-12) (-12) edge[vshift=\dbl] node[below,vshift=\dbl]{$\f_0$} (11-);
		\end{tikzpicture}
		\\
		\begin{tikzpicture}[ampersand replacement=\&]
			\diagram{\NAlg_\scr C(\SH^\veff) \& \NAlg_\scr C(\SH)\rlap,\\};
			\arrows (11-) edge[vshift=\dbl,c->] (-12) (-12) edge[vshift=\dbl] node[below,vshift=\dbl]{$\tilde\f_0$} (11-);
		\end{tikzpicture}
		\\
		\begin{tikzpicture}[ampersand replacement=\&]
			\diagram{\NAlg_\scr C(\SH^\eff) \& \NAlg_\scr C(\s_0\SH^\eff)\rlap,\\};
			\arrows (11-) edge[vshift=\dbl] node[above,vshift=\dbl]{$\s_0$} (-12) (-12) edge[vshift=\dbl,c->] (11-);
		\end{tikzpicture}
		\\
		\begin{tikzpicture}[ampersand replacement=\&]
			\diagram{\NAlg_\scr C(\SH^\veff) \& \NAlg_\scr C(\tilde\s_0\SH^\veff)\rlap,\\};
			\arrows (11-) edge[vshift=\dbl] node[above,vshift=\dbl]{$\tilde\s_0$} (-12) (-12) edge[vshift=\dbl,c->] (11-);
		\end{tikzpicture}
		\\
		\begin{tikzpicture}[ampersand replacement=\&]
			\diagram{\NAlg_\scr C(\SH^\veff) \& \NAlg_\scr C(\SH^{\eff\heartsuit})\rlap,\\};
			\arrows (11-) edge[vshift=\dbl] node[above,vshift=\dbl]{$\spi_0^\eff$} (-12) (-12) edge[vshift=\dbl,c->] (11-);
		\end{tikzpicture}
	\end{gather*}
	where the functors $\f_0$, $\tilde\f_0$, $\s_0$, $\tilde\s_0$, and $\spi_0^\eff$ are computed pointwise.
\end{proposition}

\begin{proof}
	By Lemmas \ref{lemm:construct-relative-adjoint}(1) and~\ref{lemm:adjoints-pass-to-sections}, we obtain such adjunctions between $\infty$-categories of sections such that the left adjoints preserve normed spectra. It remains to show that the right adjoints also preserve normed spectra. This follows from the fact that if $f$ is smooth, then $f_\sharp$ preserves effective, very effective, $1$-effective, very $1$-effective, and $1$-connective effective spectra.
\end{proof}

\begin{example}\label{ex:slice-completion}
A motivic spectrum $E\in\SH(S)$ is called \emph{$\infty$-effective} if it is $n$-effective for all $n\in\Z$. Denote by $\SH(S)^\eff(\infty) = \bigcap_n \SH(S)^\eff(n)$ the $\infty$-category of $\infty$-effective spectra. We claim that $\infty$-effective spectra form a normed ideal in $\SH^\otimes$. 
If $p\colon T\to S$ is finite étale and surjective, it follows from Lemma \ref{lemm:norms-circles}(3) that $p_\otimes$ preserves $\infty$-effective spectra, so it remains to show that $\SH(S)^\eff(\infty)$ is a tensor ideal.
If $E$ is $n$-effective for some $n$, it is clear that $E\otimes(\ph)$ preserves $\infty$-effective spectra. Since for any $E \in \SH(S)$ we have $E \simeq \colim_{n\to-\infty} \f_{n} E$, it remains to observe that $\SH(S)^\eff(\infty)$ is closed under colimits, which is clear since the functors $\f_n$ preserve colimits.

The right orthogonal to $\SH(S)^\eff(\infty)$ in $\SH(S)$ is the subcategory of slice-complete spectra. Arguing as in Proposition~\ref{prop:NAlg-slices}, we deduce that slice completion preserves normed spectra over $\scr C$ for any $\scr C\subset_\fet\Sm_S$.
\end{example}

\begin{remark}
	The homotopy $t$-structure on $\SH(S)$ is not compatible with norms, in the sense that norms do not preserve connective spectra. For example, let $k\subset L$ be a quadratic separable field extension and let $p\colon \Spec L\to \Spec k$. Then $p_\otimes((\Sigma^\infty\G_m)^{-1})\in \SH(k)$ is not connective. If it were, then since $p_\otimes((\Sigma^\infty\G_m)^{-1})\simeq \Sigma^{-\A^2}p_\otimes(\S^1)$, $p_\otimes(\S^1)$ would be $2$-connective. By Example~\ref{ex:degree2}, $\Sigma^\infty\Sigma(\Spec L)$ would then be $1$-connective. From the cofiber sequence
	\[
	(\Spec L)_+ \to \S^0 \to \Sigma(\Spec L),
	\]
	we would deduce that the sphere spectrum splits off $\Sigma^\infty_+\Spec L$. But $\HH^{1,1}(\Spec k,\Z)\simeq k^\times$ is not a summand of $\HH^{1,1}(\Spec L,\Z)\simeq L^\times$.
\end{remark}

\subsection{Applications to motivic cohomology}
\label{sub:HZ}

Fix a base scheme $S$. 
It follows from Proposition~\ref{prop:NAlg-slices} that $\s_0(\1)$, the zeroth slice of the motivic sphere spectrum, is a normed spectrum over $\Sm_S$. The spectrum $\s_0(\1)$ is known to be a version of motivic cohomology in many cases:
\begin{enumerate}
	\item If $S$ is essentially smooth over a field, then $\s_0(\1)\simeq \HH\Z$ is Voevodsky's motivic cohomology spectrum \cite[Remark 4.20]{HoyoisMGL}.
	\item If $S$ is noetherian and finite-dimensional of characteristic $0$, then $\s_0(\1)\simeq \HH\Z^\cdh$ is the cdh variant of $\HH\Z$ defined by Cisinski and Déglise (combine \cite[Theorem 5.1]{CDintegral} and \cite[Corollary 3.8]{Pelaez:2013}), and $\HH\Z^\cdh\simeq \HH\Z$ if $S$ is regular \cite[Corollary 3.6]{CDintegral}.
	\item If $S$ is noetherian and finite-dimensional of characteristic $p>0$, then $\s_0(\1)[1/p]\simeq \HH\Z[1/p]^\cdh$ (combine \cite[Theorem 5.1]{CDintegral} and \cite[Theorem 3.1.40]{ShaneAsterisque}), and $\HH\Z[1/p]^\cdh\simeq \HH\Z[1/p]$ if $S$ is regular \cite[Corollary 3.6]{CDintegral}.
	\item If $S$ is essentially smooth over a Dedekind domain, then $s_0(\1)\simeq \HH\Z^\Spi$, where $\HH\Z^\Spi$ is Spitzweck's version of $\HH\Z$ that represents Bloch–Levine motivic cohomology (see Theorem~\ref{thm:s_0(1)}).
\end{enumerate}
We conclude that $\HH\Z$, $\HH\Z^\cdh$, $\HH\Z[1/p]^\cdh$, and $\HH\Z^\mathrm{Spi}$ are normed spectra over $\Sm_S$ in these cases (using Proposition~\ref{prop:E[1/n]} in the third case).

\begin{remark}\label{rmk:norm-uniqueness}
	Proposition~\ref{prop:NAlg-slices} implies that there is a unique normed spectrum over any $\scr C\subset_\fet\Sm_S$ whose underlying $\E_0$-algebra is equivalent to $\1\to \s_0(\1)$: any such normed spectrum is an initial object in $\NAlg_\scr C(\s_0\SH^\eff)\subset\NAlg_\scr C(\SH)$.
	Using Proposition~\ref{prop:E[1/n]}, a similar uniqueness statement holds for $\s_0(\1)_\Lambda$ for any $\Lambda\subset\Q$.
	In particular, in the cases discussed above, the $\E_0$-algebras $\HH\Z$, $\HH\Z^\cdh$, $\HH\Z[1/p]^\cdh$, and $\HH\Z^\mathrm{Spi}$ have unique normed enhancements.
\end{remark}

In \sect\ref{sub:PST}, we will show that in fact $\HH R$ and $\HH R^\cdh$ are normed spectra over any noetherian base and for any commutative ring $R$. Suppose now that $D$ is a Dedekind domain and let $m \ge 0$. Write $\HH\Z^\Spi/m$ for the reduction mod $m$ of Spitzweck's motivic cohomology spectrum over $D$ \cite{SpitzweckHZ}. 
Spitzweck constructs an $\E_\infty$-ring structure on $\HH\Z^\Spi/m$ for any $m$; we will show that it can be uniquely promoted to a structure of normed spectrum over $\Sm_D$. We begin with the following observation.

\begin{lemma}\label{lem:heart}
For every $m\geq 0$, $\HH\Z^\Spi/m \in \SH(D)^{\eff\heartsuit}$.
\end{lemma}
\begin{proof}
We first prove that $\HH\Z^\Spi/m$ is very effective. Since $\HH\Z^\Spi/m$ is stable under base change \cite[\sect 9]{SpitzweckHZ}, we may assume that $D = \Z$.  By Proposition~\ref{prop:veff-detect-on-points}, we may further assume that $D$ is the spectrum of a perfect field. In this case, $\HH\Z^\Spi/m \simeq \s_0(\1)/m$ by \cite[Theorem 10.5.1]{Levine:2008}, and in particular $\HH\Z^\Spi/m$ is effective. The fact that $\HH\Z^\Spi/m$ is very effective now follows from the characterization of $\SH(D)^\veff \subset \SH(D)^\eff$ in terms of homotopy sheaves \cite[Proposition 4]{Bachmann-slices}.

It remains to show that for $E \in \SH(D)^\eff_{\geq 1}$ we have $\Map(E, \HH\Z^\Spi/m) \simeq *$. The subcategory of such $E$ is closed under colimits and extensions, so we need only consider the case $E = \Sigma^\infty_+ X[1]$ for $X \in \Sm_D$. This follows from 
\[
\pi_i\Map(\Sigma^\infty_+ X[1], \HH\Z^\Spi/m) \simeq \HH^{-1-i}_\Zar(X, \Z/m) = 0.\qedhere
\]
\end{proof}

We refer to \sect\ref{sub:spans-descent} for the notion of descent for sections and objects of cocartesian fibrations.

\begin{lemma} \label{lemm:etale-descent-criterion}
Let $S$ be a scheme, $t$ a topology on $\Sm_S$, and $E \in \SH(S)^\veff$. Suppose that $\Omega^\infty \Sigma E$ is a $t$-separated presheaf, i.e., that for every $X\in\Sm_S$ and every $t$-covering sieve $\scr R\subset (\Sm_S)_{/X}$, the canonical map
\[
\Map(\Sigma^\infty_+ X, \Sigma E) \to \lim_{U\in\scr R} \Map(\Sigma^\infty_+ U, \Sigma E)
\]
is a monomorphism. Then the section $X\mapsto E_X$ of $\SH^\veff\colon \Sm_S^\op\to\Cat_\infty$ satisfies $t$-descent.
\end{lemma}
\begin{proof}
We need to prove that for every $X\in\Sm_S$, every $F \in \SH(X)^\veff$, and every $t$-covering sieve $\scr R$ on $X$, the morphism $\Map(F, E_X) \to \lim_{U\in\scr R} \Map(F_U, E_U)$ is an equivalence.
We shall prove the stronger statement that $\Map(F, \Sigma E_X) \to \lim_{U\in\scr R} \Map(F_U, \Sigma E_U)$ is a monomorphism; the desired result follows by taking loops. Since monomorphisms are stable under limits, the subcategory of all $F$ with the desired property is closed under colimits. By the 5-lemma, it is closed under extensions as well. Consequently, it suffices to treat the case $F = \Sigma^\infty_+ Y$ for some $Y \in \Sm_X$. In this case, the claim holds by assumption.
\end{proof}

\begin{theorem}\label{thm:HZ-Spi-norms}
Let $D$ be a Dedekind domain and $m\geq 0$. Then $\HH\Z^\Spi/m \in \SH(D)$ admits a unique structure of normed spectrum over $\Sm_D$ compatible with its homotopy commutative ring structure.
\end{theorem}

\begin{proof}
By Lemma~\ref{lem:heart}, $\HH\Z^\Spi/m$ belongs to the heart of the $t$-structure on $\SH(D)^\eff$, which is a $1$-category. Since this $t$-structure is compatible with the symmetric monoidal structure, there is a unique $\E_\infty$-ring structure on $\HH\Z^\Spi/m$ refining its homotopy commutative ring structure.

Note that the presheaf $\Omega^\infty (\HH\Z^\Spi/m)$ is an étale sheaf on $\Sm_D$ (namely, the constant sheaf $\Z/m$). Since $\HH^1_\Zar(X,\Z/m)=0$ for any $X\in\Sm_D$, we deduce that $\Omega^\infty(\Sigma\HH\Z^\Spi/m)$ is an étale-separated presheaf. Hence, by Lemma \ref{lemm:etale-descent-criterion}, $\HH\Z^\Spi/m\in \SH^\veff(D)$ satisfies étale descent on $\Sm_D$.
It now follows from Corollary \ref{cor:automatic-norms2} applied to $\SH^{\veff\otimes}$ that $\HH\Z^\Spi/m$ admits a unique structure of normed spectrum over $\Sm_D$ refining its $\E_\infty$-ring structure
\end{proof}

\begin{example}
	Let $D$ be a discrete valuation ring. For every $X\in\Sm_D$ and $n\geq 0$, we have
	\[
	\Map(\Sigma^\infty_+X, \Sigma^{2n,n}\HH\Z^\Spi) \simeq z^n(X,*),
	\]
	where $z^n(X,*)$ is Bloch's cycle complex in codimension $n$ \cite[Theorem 1.7]{LevineLocalization}. Since $\HH\Z^\Spi$ is oriented \cite[Proposition 11.1]{SpitzweckHZ}, for every finite étale map $p\colon Y\to X$ of degree $d$, we obtain a norm map
	\[
	\nu_p\colon z^n(Y,*) \to z^{nd}(X,*).
	\]
	We also obtain total power operations, as explained in Example~\ref{ex:totalpower}.
	If $D$ is a more general Dedekind domain, $\Map(\Sigma^\infty_+X, \Sigma^{2n,n}\HH\Z^\Spi)$ is the sheafification of Bloch's cycle complex with respect to the Zariski topology on $\Spec D$.
\end{example}

\begin{example}[Norms on Chow–Witt groups]
	\label{ex:Chow-Witt}
	Let $k$ be a field. The \emph{generalized motivic cohomology spectrum} $\HH\tilde\Z\in \SH(k)$ is defined in \cite{Bachmann-slices} as
	\[
	\HH\tilde\Z=\spi_0^\eff(\1).
	\]
	Assuming $k$ infinite perfect of characteristic $\neq 2$, it is proved in \cite[Theorem 5.2]{BachmannFasel} that $\HH\tilde\Z$ represents Milnor–Witt motivic cohomology as defined in \cite[Definition 4.1.1]{DegliseFasel2}.
	 In general, if $X$ is smooth over $k$ and $\xi\in\K(X)$ has rank $n$, we have
	\[
	[\1_X, \Sigma^{\xi}\HH\tilde\Z_X]\simeq \widetilde{\CH}{}^n(X,\det\xi),
	\]
	where $\widetilde{\CH}{}^n(X,\scr L)=\HH^n_\Nis(X,\underline{\K}{}_n^\MW\otimes_{\Z[\scr O^\times]}\Z[\scr L^\times])$ is the $n$th Chow–Witt group of $X$ twisted by a line bundle $\scr L$ \cite[\sect 3]{CalmesFasel}. Moreover, the forgetful map $\widetilde{\CH}{}^n(X,\scr L)\to \CH^n(X)$ is induced by the map $\spi_0^\eff(\1)\to \spi_0^\eff(\s_0(\1))\simeq \HH\Z$.
	
	By Proposition~\ref{prop:NAlg-slices}, $\HH\tilde\Z$ is a normed spectrum over $\Sm_k$, and $\HH\tilde\Z\to\HH\Z$ is a morphism of normed spectra. In particular, if $p\colon Y\to X$ is a finite étale map of degree $d$ between smooth $k$-schemes, we obtain a norm map $\nu_p\colon \widetilde\CH{}^n(Y,\scr L) \to \widetilde{\CH}{}^{nd}(X,\Norm_p(\scr L))$ and a commutative square
	\begin{tikzmath}
		\diagram{
		\widetilde{\CH}{}^n(Y,\scr L) & \CH^n(Y) \\
		\widetilde{\CH}{}^{nd}(X,\Norm_p(\scr L)) & \CH^{nd}(X)\rlap, \\
		};
		\arrows (11-) edge (-12) (21-) edge (-22) (11) edge node[left]{$\nu_p$} (21) (12) edge node[right]{$\nu_p$} (22);
	\end{tikzmath}
	where the right vertical map is the Fulton–MacPherson norm (see Theorem~\ref{thm:FultonMacPhersonComparison}).
\end{example}

\subsection{Graded normed spectra}

Our goal for the remainder of this section is to show that, if $E$ is a normed spectrum over $\scr C\subset_\fet\Sm_S$, the sum of all its slices $\bigvee_{n\in\Z}\s_n(E)$ is again a normed spectrum over $\scr C$. To do this, we will introduce a graded version of $\SH^\otimes$ as follows. If $\Gamma$ is a (discrete) commutative monoid, we will construct a functor
\[
\SH^{\Gamma\otimes}\colon \Span(\Sch,\all,\fet)\to \CAlg(\what{\Cat}{}_\infty^\mathrm{sift}), \quad S\mapsto \SH(S)^\Gamma,
\]
together with a natural transformation
\begin{equation}\label{eqn:gradedSH}
\bigvee_\Gamma\colon \SH^{\Gamma\otimes} \to \SH^\otimes.
\end{equation}

For $S$ a scheme, let $^\Gamma\SmQP_{S+}$ be the following category:
\begin{itemize}
	\item an object is a pair $(X,\gamma)$ where $X\in\SmQP_{S}$ and $\gamma\colon X\to\Gamma$ is a locally constant function;
	\item a morphism $f\colon (X,\gamma)\to (Y,\delta)$ is a morphism $f\colon X_+\to Y_+$ in $\SmQP_{S+}$ such that $\gamma$ and $\delta$ agree on $f^{-1}(Y)\subset X$.
\end{itemize}
The category $^\Gamma\SmQP_{S+}$ admits finite coproducts, and the obvious functor $\Gamma\times \SmQP_{S+} \to {}^\Gamma\SmQP_{S+}$ is the initial functor that preserves finite coproducts in its second variable (to see this, note that on a quasi-compact scheme a locally constant function takes only finitely many values, so $(X, \gamma) \in {}^\Gamma\SmQP_{S+}$ is a finite disjoint union of pairs $(Y, c)$ where $c$ is a constant function). In particular, we have
\[
\PSh_\Sigma({}^\Gamma\SmQP_{S+}) \simeq \PSh_\Sigma(\SmQP_{S})_\pt^\Gamma.
\]
Moreover, for every homomorphism $\phi\colon \Gamma\to\Gamma'$ the functor $\phi_!\colon \PSh_\Sigma(\SmQP_{S})_\pt^\Gamma\to\PSh_\Sigma(\SmQP_{S})_\pt^{\Gamma'}$ is the left Kan extension of the functor ${}^\Gamma\SmQP_{S+} \to {}^{\Gamma'}\SmQP_{S+}$, $(X,\gamma)\mapsto (X,\phi\circ\gamma)$. Using the commutative monoid structure on $\Gamma$, we can equip $^\Gamma\SmQP_{S+}$ with a symmetric monoidal structure inducing the Day convolution symmetric monoidal structure on $\PSh_\Sigma(\SmQP_{S})_\pt^\Gamma$.

Let $\Gamma^S$ denote the monoid of locally constant functions $S\to\Gamma$. If $f\colon Y\to X$ is any morphism, we let $f^*\colon \Gamma^X \to \Gamma^Y$ be precomposition with $f$. If $p\colon Y\to X$ is finite locally free and $\gamma\colon Y\to\Gamma$, we define $p_*(\gamma)\colon X\to\Gamma$ by
\[ 
p_*(\gamma)(x) = \sum_{y \in p^{-1}(x)} \deg_y(p)\gamma(y), 
\]
where $\deg_y(p)=\dim_{\kappa(x)}\scr O_{Y_x,y}$. Note that if $p$ is étale at $y$, then $\deg_y(p) = [\kappa(y):\kappa(x)]$, since then $\scr O_{Y_x,y} = \kappa(y)$.

\begin{lemma}\label{lem:normed-monoid}
With the notation as above, there is a functor
\[
\Span(\Sch,\all,\flf) \to \CAlg(\Set), \quad S\mapsto \Gamma^S,\quad (U\stackrel f\from T\stackrel p\to S)\mapsto p_* f^*.
\]
\end{lemma}

\begin{proof}
	First we need to check that $p_*(\gamma)$ is locally constant if $\gamma$ is.
	Let $Y=\coprod_i Y_i$ be a coproduct decomposition of $Y$ such that $\gamma$ is constant on $Y_i$ with value $\gamma_i$, and let $p_i\colon Y_i\to X$ be the restriction of $p$. Working locally on $X$, we can assume that each $p_i$ has constant degree $d_i$. Then $p_*(\gamma)$ is constant with value $\sum_i d_i\gamma_i$.
	
It is clear that $p_*f^*$ is a homomorphism and that it depends only on the isomorphism class of the span.
Compatibility with compositions of spans amounts to the following properties of the degree:
\begin{itemize}
	\item If $q\colon Z\to Y$ and $p\colon Y\to X$ are finite locally free and $q(z)=y$, then $\deg_z(p\circ q)=\deg_{y}(p)\deg_z(q)$.
	\item Given a cartesian square of schemes
	\begin{tikzmath}
		\diagram{Y' & Y \\ X' & X \\};
		\arrows (11-) edge node[above]{$g$} (-12) (11) edge node[left]{$p'$} (21) (21-) edge node[above]{$f$} (-22) (12) edge node[right]{$p$} (22);
	\end{tikzmath}
	with $p$ finite locally free and $p(y)=f(x')$, then $\deg_{y}(p)=\sum_{y'}\deg_{y'}(p')$, where $y'$ ranges over the points of $Y'$ such that $g(y')=y$ and $p'(y')=x'$.
\end{itemize}
These formulas follow from the isomorphisms $\scr O_{Z_y,z}\simeq \scr O_{Z_x,z}\otimes_{\scr O_{Y_x,y}}\kappa(y)$ and $\prod_{y'}\scr O_{Y'_{x'},y'}\simeq \kappa(x')\otimes_{\kappa(x)}\scr O_{Y_x,y}$, respectively.
\end{proof}

Using Lemma~\ref{lem:normed-monoid}, we can define
	\[
	^\Gamma\SmQP_+^\otimes\colon \Span(\Sch,\all,\flf)\to \CAlg(\Cat_1), \quad S\mapsto {}^\Gamma\SmQP_{S+}, \quad (U\stackrel f\from T\stackrel p\to S)\mapsto p_\otimes f^*,
	\]
by setting 
\[
f^*(X,\gamma)=(X\times_UT, f_X^*(\gamma)) \quad\text{and}\quad p_\otimes(X,\gamma)=(\Weil_pX,q_* e^*(\gamma)),\] 
where $X\xleftarrow{e} \Weil_pX\times_ST \xrightarrow{q}\Weil_pX$. The effects of $f^*$ and $p_\otimes$ on morphisms, as well as their symmetric monoidal structures, are uniquely determined by the requirement that the faithful functor $^\Gamma\SmQP_{S+} \to \SmQP_{S+}$ be natural in $S\in\Span(\Sch,\all,\flf)$. It is then easy to check that, for every homomorphism $\phi\colon \Gamma\to\Gamma'$, the functor $^\Gamma\SmQP_{S+} \to {}^{\Gamma'}\SmQP_{S+}$ is also natural in $S$. Hence, we obtain a functor
\[
^?\SmQP_+^\otimes\colon \CAlg(\Set) \times \Span(\Sch,\all,\flf) \to \CAlg(\Cat_1).
\]

It is now straightforward to repeat the steps of \sect\ref{sub:coherence} to arrive at the functor
\[
\SH^{?\otimes}\colon \CAlg(\Set) \times \Span(\Sch,\all,\fet) \to \CAlg(\what\Cat{}_\infty^\mathrm{sift}), \quad (\Gamma,S)\mapsto \SH(S)^\Gamma.
\]
The main points are that the $\infty$-category $\H_\pt(S)^\Gamma$ is obtained from $\PSh_\Sigma(\SmQP_{S})_\pt^\Gamma$ by inverting the families of motivic equivalences, which are generated by Nisnevich sieves and $\A^1$-homotopy equivalences in $^\Gamma\SmQP_{S+}$, and that the compactly generated symmetric monoidal $\infty$-category $\SH(S)^\Gamma$ is obtained from $\H_\pt(S)^\Gamma$ by inverting the motivic spheres placed in degree $0\in\Gamma$.
In particular, the homomorphism $\Gamma\to 0$ yields the natural transformation~\eqref{eqn:gradedSH}.

\begin{remark}\label{rmk:non-discrete-Gamma}
	If $\Gamma$ is more generally a symmetric monoidal $\infty$-groupoid, the above construction is still possible if ``locally constant'' is understood with respect to the finite étale topology. We can then use Corollary~\ref{cor:automatic-norms} to define $^?\SmQP_+^\otimes\colon \CAlg(\scr S)\times \Span(\Sch,\all,\fet)\to\CAlg(\Cat_\infty)$. However, it is not anymore the case that $\PSh_\Sigma({}^\Gamma\SmQP_{S+}) \simeq \PSh_\Sigma(\SmQP_{S})_\pt^\Gamma$. Essentially, the reason norms exist on $\Gamma$-graded motivic spectra when $\Gamma$ is discrete is that the $\pi_0$ of schemes in the Nisnevich and finite étale topologies agree.
\end{remark}

\begin{remark} \label{rmk:graded-f*}
If $E \in \SH(X)^\Gamma$ and $f\colon Y \to X$ is any morphism, then for $\gamma \in \Gamma$ we have $(f^* E)_\gamma = f^*(E_\gamma)$. In other words, the graded pullback functor $f^*\colon \SH(X)^\Gamma\to \SH(Y)^\Gamma$ is computed degreewise.
In particular, if $f$ is smooth, then $f^*\colon \SH(X)^\Gamma \to \SH(Y)^\Gamma$ has a left adjoint $f_\sharp$, which is also computed degreewise.
\end{remark}

\begin{remark} \label{rmk:graded-distributivity}
Given the diagram~\eqref{eqn:exponential} with $p$ finite étale and $h$ smooth and quasi-projective, the distributivity transformation
\[
\Dist_{\sharp\otimes}\colon f_\sharp q_\otimes e^* \to p_\otimes h_\sharp \colon \SH(U)^\Gamma \to \SH(S)^\Gamma
\]
is an equivalence. As in the proof of Proposition~\ref{prop:distributivity}(1), one reduces to checking that the distributivity transformation $\Dist_{\sharp\otimes}\colon f_\sharp q_\otimes e^* \to p_\otimes h_\sharp \colon {}^\Gamma\SmQP_{U+} \to {}^\Gamma\SmQP_{S+}$ is an equivalence, which is an immediate consequence of the definitions.
In particular, $\SH^{\Gamma\otimes}$ is a presentably normed $\infty$-category over the category of schemes, in the sense of Definition~\ref{dfn:normed-category}.
\end{remark}

Given $\gamma\in\Gamma$, the functor $\gamma^*\colon \SH(X)^\Gamma\to \SH(X)$, $E\mapsto E_\gamma$, has a left adjoint $\gamma_!\colon \SH(X)\to \SH(X)^\Gamma$ given by
\[
(\gamma_!E)_\delta =
\begin{cases}
	E & \text{if $\delta=\gamma$,} \\
	0 & \text{otherwise.}
\end{cases}
\]

\begin{lemma}\label{lem:special-graded-fo}
	Let $p\colon T\to S$ be a finite étale map of constant degree $d$ and let $\gamma\in\Gamma$.
	Then there is an equivalence of $\SH(T)$-module functors
	\[
	p_\otimes\gamma_! \simeq (d\gamma)_! p_\otimes\colon \SH(T) \to \SH(S)^\Gamma.
	\]
\end{lemma}

\begin{proof}
	If $X_+\in\SmQP_{T+}$, we have by definition $p_\otimes(X,\gamma) \simeq (\Weil_pX,d\gamma)$ in ${}^\Gamma\SmQP_{S+}$, naturally in $X_+$, and this is clearly an equivalence of $\SmQP_{T+}$-module functors. Since $\H_\pt(T)$ is a localization of $\PSh_\Sigma(\SmQP_{T+})$, we obtain an equivalence of $\H_\pt(T)$-module functors 
	\[p_\otimes\gamma_! \simeq (d\gamma)_! p_\otimes\colon  \H_\pt(T) \to \H_\pt(S)^\Gamma.\]
	We conclude using the universal property of $\Sigma^\infty\colon \H_\pt(T)\to\SH(T)$ as a filtered-colimit-preserving $\H_\pt(T)$-module functor (see Remark~\ref{rmk:C-module}).
\end{proof}

If $\scr C\subset_\fet\Sch_S$ and $\Gamma$ is a commutative monoid, a \emph{$\Gamma$-graded normed spectrum} over $\scr C$ is an object of $\NAlg_\scr C(\SH^\Gamma)$, i.e., a section of the functor $\SH^{\Gamma\otimes}\colon \Span(\scr C,\all,\fet)\to\what\Cat{}_\infty$ that is cocartesian over backward morphisms.

If $\phi\colon \Gamma \to \Gamma'$ is a morphism of commutative monoids, the functor $\SH^{?\otimes}$ provides a natural transformation $\phi_!\colon \SH^{\Gamma\otimes}\to \SH^{\Gamma'\otimes}$, which has an objectwise right adjoint given by the precomposition functor $\phi^*\colon \SH(S)^{\Gamma'} \to \SH(S)^\Gamma$. By Lemmas \ref{lemm:construct-relative-adjoint}(1) and~\ref{lemm:adjoints-pass-to-sections}, there is an induced adjunction 
\[
\phi_!:\Sect(\SH^{\Gamma\otimes}|\Span(\scr C,\all,\fet))\rightleftarrows \Sect(\SH^{\Gamma'\otimes}|\Span(\scr C,\all,\fet)):\phi^*,
\]
 where $\phi_!$ preserves $\Gamma$-graded normed spectra.
 Since the functors $\phi^*$ obviously commute with pullback functors, $\phi^*$ also preserves $\Gamma$-graded normed spectra
 and we obtain an adjunction
\[ \phi_!: \NAlg_\scr{C}(\SH^\Gamma) \rightleftarrows \NAlg_\scr{C}(\SH^{\Gamma'}): \phi^*. \]
In particular, the morphism $\Gamma \to 0$ gives rise to an adjunction
\begin{equation}\label{eqn:sum-gamma}
	\bigvee_\Gamma: \NAlg_\scr{C}(\SH^\Gamma) \rightleftarrows \NAlg_\scr{C}(\SH): \cst_\Gamma,
\end{equation}
where $\bigvee_\Gamma E = \bigvee_{\gamma\in\Gamma} E_\gamma$  and $\cst_\Gamma (E)_\gamma = E$ for all $\gamma\in\Gamma$.

\begin{example}\label{ex:free-graded}
Let $\scr C \subset_\fet \Sm_S$.
The element $1\in \bb N$ induces an adjunction 
\[1_!: \SH(S) \rightleftarrows \SH(S)^\bb N: 1^*,\]
 and the forgetful functor $U\colon \NAlg_\scr{C}(\SH) \to \SH(S)$ factors as 
\[ \NAlg_\scr{C}(\SH) \xrightarrow{\cst_\bb N} \NAlg_\scr{C}(\SH^\bb N) \xrightarrow{U^\bb N} \SH(S)^\bb N \xrightarrow{1^*} \SH(S). \] 
Consequently, the free normed spectrum functor $\NSym_\scr C\colon \SH(S) \to \NAlg_\scr C(\SH)$ left adjoint to $U$ (see Remark~\ref{rmk:freeNAlg}) can be written as
\[ \NSym_\scr C E \simeq \bigvee_{n\in\bb N} (\NSym_\scr C^\bb N 1_! E)_{n}, \]
where $\NSym_\scr C^\bb N$ is left adjoint to $U^\bb N$.
Thus, free normed spectra are canonically $\bb N$-graded. The spectrum
 $\NSym_\scr C^n(E)= (\NSym_\scr C^\bb N 1_! E)_{n}$ is the ``$n$th normed symmetric power'' of $E$. 
If $\scr C=\Sm_S$ or $\scr C=\FEt_S$, one can show that
	\[
	\NSym_\scr C^n(E) \simeq \colim_{\substack{f\colon X\to S\\p\colon Y\to X}} f_\sharp p_\otimes(E_Y),
	\]
	where $f$ ranges over $\scr C$ and $p$ ranges over the groupoid of finite étale covers of $X$ of degree $n$; see Remarks \ref{rmk:freeNAlg-2} and~\ref{rmk:freeNAlg-3}.
\end{example}

\subsection{The graded slices of a normed spectrum}

In this subsection, we show that if $E$ is a normed spectrum over $\scr C\subset_\fet\Sm_S$, then its slices $\s_*(E)$ and generalized slices $\tilde\s_*(E)$ are $\Z$-graded normed spectra over $\scr C$.

Let $\SH(S)^{\Z\eff}\subset\SH(S)^\Z$ be the full subcategory consisting of the $\Z$-graded spectra $(E_n)_{n\in \Z}$ with $E_n \in \SH(S)^\eff(n)$.
Note that $\SH(S)^{\Z\eff}$ is a symmetric monoidal subcategory of $\SH(S)^\Z$, since
\[
\SH(S)^\eff(n) \wedge \SH(S)^\eff(m) \subset \SH(S)^\eff(n+m).
\]
This subcategory is generated under colimits by $n_!\Sigma^{2n,n}E$ for $E\in\SH(S)^\eff$ and $n\in \Z$.
If $p\colon S\to S'$ is a finite étale map of constant degree $d$, it follows from Lemma~\ref{lem:special-graded-fo} that
\[
p_\otimes(n_!\Sigma^{2n,n} E) \simeq (dn)_!(p_\otimes(\S^{2n,n})\wedge p_\otimes(E)),
\]
which belongs to $\SH(S')^{\Z\eff}$ since $p_\otimes$ preserves effective spectra and $p_\otimes(\S^{2n,n})$ is $dn$-effective (Lemma~\ref{lemm:norms-circles}(3)). Applying Proposition~\ref{prop:normed-subcategory}, we obtain a subfunctor
\[
\SH^{\Z\eff\otimes}\subset\SH^{\Z\otimes} \colon \Span(\Sch,\all,\fet) \to \CAlg(\what\Cat{}_\infty^\mathrm{sift}).
\]

Note that the inclusion $\SH(S)^{\Z\eff} \into \SH(S)^\Z$ has a right adjoint $\f_{*}\colon \SH(S)^\Z \to \SH(S)^{\Z\eff}$ given by $\f_{*}(E)_{n} = \f_n(E_{n})$ for $n \in \Z$. Furthermore, the localization functors $\s_n\colon \SH(S)^\eff(n)\to \SH(S)^\eff(n)$ assemble into a localization functor $\s_*\colon \SH(S)^{\Z\eff} \to \SH(S)^{\Z\eff}$, given by $\s_* (E)_{n} = \s_n E_{n}$. This localization is obtained by killing the tensor ideal in $\SH(S)^{\Z\eff}$ generated by $0_!\S^{\A^1}$. By Lemma~\ref{lemm:norms-circles}(3), these tensor ideals form a normed ideal in $\SH^{\Z\eff\otimes}$. Applying Corollary~\ref{cor:normed-nullification} to this normed ideal, we obtain a natural transformation
\[
\s_*\colon \SH^{\Z\eff\otimes}\to \s_*\SH^{\Z\eff\otimes}\colon \Span(\Sch,\all,\fet) \to \CAlg(\what\Cat{}_\infty^\mathrm{sift}).
\]

Similarly, if we define $\SH(S)^{\Z\veff}\subset \SH(S)^{\Z}$ to be the full subcategory consisting of the $\Z$-graded spectra $(E_n)_{n\in \Z}$ with $E_n \in \SH(S)^\veff(n)$, we have a subfunctor
\[
\SH^{\Z\veff\otimes}\subset \SH^{\Z\eff\otimes}\colon \Span(\Sch,\all,\fet) \to \CAlg(\what\Cat{}_\infty^\mathrm{sift})
\]
and a natural transformation
\[
\tilde\s_*\colon \SH^{\Z\veff\otimes}\to \tilde\s_*\SH^{\Z\veff\otimes}\colon \Span(\Sch,\all,\fet) \to \CAlg(\what\Cat{}_\infty^\mathrm{sift}).
\]

Finally, $\SH(S)^{\Z\veff}$ is the nonnegative part of a $t$-structure on $\SH(S)^{\Z\eff}$. The $1$-connective part is the tensor ideal in $\SH(S)^{\Z\veff}$ generated by $0_!\S^1$. By Lemma~\ref{lemm:norms-circles}(2), these tensor ideals form a normed ideal in $\SH^{\Z\veff\otimes}$. Applying Corollary~\ref{cor:normed-nullification}, we obtain a natural transformation
\[
\spi_0^{\Z\eff}\colon \SH^{\Z\veff\otimes}\to \SH^{\Z\eff\heartsuit\otimes}\colon \Span(\Sch,\all,\fet) \to \CAlg(\what\Cat{}_\infty^\mathrm{sift}).
\]

\begin{proposition}\label{prop:NAlg-graded-slices}
	Let $S$ be a scheme and $\scr C\subset_\fet\Sm_S$. Then there are adjunctions
	\begin{gather*}
		\begin{tikzpicture}[ampersand replacement=\&]
			\diagram{\NAlg_\scr C(\SH^{\Z\eff}) \& \NAlg_\scr C(\SH^\Z)\rlap,\\};
			\arrows (11-) edge[vshift=\dbl,c->] (-12) (-12) edge[vshift=\dbl] node[below,vshift=\dbl]{$\f_*$} (11-);
		\end{tikzpicture}
		\\
		\begin{tikzpicture}[ampersand replacement=\&]
			\diagram{\NAlg_\scr C(\SH^{\Z\veff}) \& \NAlg_\scr C(\SH^\Z)\rlap,\\};
			\arrows (11-) edge[vshift=\dbl,c->] (-12) (-12) edge[vshift=\dbl] node[below,vshift=\dbl]{$\tilde\f_*$} (11-);
		\end{tikzpicture}
		\\
		\begin{tikzpicture}[ampersand replacement=\&]
			\diagram{\NAlg_\scr C(\SH^{\Z\eff}) \& \NAlg_\scr C(\s_*\SH^{\Z\eff})\rlap,\\};
			\arrows (11-) edge[vshift=\dbl] node[above,vshift=\dbl]{$\s_*$} (-12) (-12) edge[vshift=\dbl,c->] (11-);
		\end{tikzpicture}
		\\
		\begin{tikzpicture}[ampersand replacement=\&]
			\diagram{\NAlg_\scr C(\SH^{\Z\veff}) \& \NAlg_\scr C(\tilde\s_*\SH^{\Z\veff})\rlap,\\};
			\arrows (11-) edge[vshift=\dbl] node[above,vshift=\dbl]{$\tilde\s_*$} (-12) (-12) edge[vshift=\dbl,c->] (11-);
		\end{tikzpicture}
		\\
		\begin{tikzpicture}[ampersand replacement=\&]
			\diagram{\NAlg_\scr C(\SH^{\Z\veff}) \& \NAlg_\scr C(\SH^{\Z\eff\heartsuit})\rlap,\\};
			\arrows (11-) edge[vshift=\dbl] node[above,vshift=\dbl]{$\spi_0^{\Z\eff}$} (-12) (-12) edge[vshift=\dbl,c->] (11-);
		\end{tikzpicture}
	\end{gather*}
	where the functors $\f_*$, $\tilde\f_*$, $\s_*$, $\tilde\s_*$, and $\spi_0^{\Z\eff}$ are computed pointwise.
\end{proposition}

\begin{proof}
	Same as Proposition~\ref{prop:NAlg-slices}
\end{proof}

\begin{example}\label{ex:graded-slice}
	Let $\scr C\subset_\fet\Sm_S$.
If $E \in \NAlg_\scr C(\SH)$ then $\s_{*} \f_{*}\cst_\Z (E) \in \NAlg_\scr C(\SH^{\Z})$ is a $\Z$-graded normed spectrum whose underlying $\Z$-graded spectrum is $\s_*E$. In particular, if $p\colon X \to Y$ is a finite étale map in $\scr C$ of constant degree $d$, we get a norm map $p_\otimes (\s_n E_X) \to \s_{nd} E_Y$. The same statements hold with $\tilde\f_*$ and $\tilde\s_*$ instead of $\f_*$ and $\s_*$.
\end{example}

\begin{example}\label{ex:sum-of-slices}
	Combining Example~\ref{ex:graded-slice} with the adjunction~\eqref{eqn:sum-gamma}, we deduce that if $E$ is a normed spectrum over $\scr C\subset_\fet\Sm_S$, then $\bigvee_{n\in\Z}\s_n E$ and $\bigvee_{n\in\Z}\tilde\s_n E$ are normed spectra over $\scr C$.
\end{example}

\section{Norms of cycles}
\label{sec:PST}

In this section, schemes are assumed to be noetherian. We denote by $\Sch^\noe$ the category of noetherian schemes.
Our first goal is to construct a functor
\[
\DM^\otimes\colon \Span(\Sch^\noe,\all,\fet) \to \CAlg(\what\Cat{}_\infty^\mathrm{sift}),\quad S\mapsto \DM(S),
\]
where $\DM(S)$ is Voevodsky's $\infty$-category of motives over $S$, together with a natural transformation
\[
\Z_\tr\colon \SH^\otimes \to \DM^\otimes.
\]
As a formal consequence, we will deduce that Voevodsky's motivic cohomology spectrum $\HH\Z_S\in\SH(S)$ is a normed spectrum, for every noetherian scheme $S$. 
In the case where $S$ is smooth and quasi-projective over a field, we will then compare the residual structure on Chow groups with the multiplicative transfers constructed by Fulton and MacPherson \cite{FultonMacPherson}.

\subsection{Norms of presheaves with transfers}
\label{sub:PST}

Let us first recall the definition of $\DM(S)$. Let $\SmQPCor_S$ denote the additive category of finite correspondences between smooth quasi-projective $S$-schemes \cite[\sect9.1]{CD}. The set of morphisms from $X$ to $Y$ in $\SmQPCor_S$ will be denoted by $c_S(X,Y)$; it is the group of relative cycles on $X\times_SY/X$ that are finite and universally integral over $X$. The category $\SmQPCor_S$ admits a symmetric monoidal structure such that the functor 
\begin{equation}\label{eqn:graph}
\SmQP_{S+}\to \SmQPCor_S, \quad X_+\mapsto X,\quad (f\colon X_+\to Y_+)\mapsto \Gamma_f\cap(X\times Y),
\end{equation}
is symmetric monoidal \cite[\sect9.2]{CD}.
 The symmetric monoidal $\infty$-category of \emph{presheaves with transfers} on $\SmQP_S$ is the $\infty$-category $\PSh_\Sigma(\SmQPCor_S)$ equipped with the Day convolution symmetric monoidal structure.
We let
\[
\H_\tr(S) \subset \PSh_\Sigma (\SmQPCor_S)
\]
be the reflective subcategory spanned by the presheaves with transfers whose underlying presheaves are $\A^1$-homotopy invariant Nisnevich sheaves. The localization functor $\PSh_\Sigma (\SmQPCor_S)\to\H_\tr(S)$ is then compatible with the Day convolution symmetric monoidal structure and can be promoted to a symmetric monoidal functor. 
We will say that a morphism in $\PSh_\Sigma (\SmQPCor_S)$ is a \emph{motivic equivalence} if its reflection in $\H_\tr(S)$ is an equivalence. 

\begin{remark}\label{rmk:compat-transfers}
	Because the Nisnevich topology is compatible with transfers \cite[Proposition 10.3.3]{CD}, it follows from Lemma~\ref{lem:voevodsky} that the forgetful functor $\PSh_\Sigma(\SmQPCor_S) \to \PSh_\Sigma(\SmQP_S)$ reflects motivic equivalences (c.f.\ \cite[Theorem 1.7]{MEMS}). 
\end{remark}

The functor~\eqref{eqn:graph} induces by left Kan extension a symmetric monoidal functor 
\[
\Z_\tr\colon \PSh_\Sigma(\SmQP_S)_\pt\to \PSh_\Sigma (\SmQPCor_S),
\]
which preserves motivic equivalences (by definition) and hence induces a symmetric monoidal functor 
\begin{equation}\label{eqn:ZtrH}
\Z_\tr\colon \H_\pt(S)\to\H_\tr(S).
\end{equation}
Finally, $\DM(S)$ is the presentably symmetric monoidal $\infty$-category obtained from $\H_\tr(S)$ by inverting $\Z_\tr\S^{\A^1}$, and $\Z_\tr\colon \SH(S)\to\DM(S)$ is the unique colimit-preserving symmetric monoidal extension of~\eqref{eqn:ZtrH}. The underlying $\infty$-category of $\DM(S)$ is equivalently the limit of the tower
\[
\dotsb \xrightarrow{\Hom(\Z_\tr\S^{\A^1},\ph)} \H_\tr(S) \xrightarrow{\Hom(\Z_\tr\S^{\A^1},\ph)} \H_\tr(S) \xrightarrow{\Hom(\Z_\tr\S^{\A^1},\ph)} \H_\tr(S),
\]
by \cite[Corollary 2.22]{Robalo}.
 The canonical symmetric monoidal functor $\H_\tr(S)\to \DM(S)$ also satisfies a stronger universal property by Lemma~\ref{lem:inversion}.

\begin{lemma}\label{lem:corr-etale-descent}
	The functors
	\begin{align*}
	\Sch^\op \to \CAlg(\Cat_1),& \quad S\mapsto \SmQP_{S+},\\
	\Sch^{\noe,\op} \to \CAlg(\Cat_1),& \quad S\mapsto \SmQPCor_S,
	\end{align*}
	are sheaves for the finite locally free and finite étale topologies, respectively.
\end{lemma}

\begin{proof}
	It is clear that both functors transform finite sums into finite products.
	Surjective finite locally free morphisms are of effective descent for quasi-projective schemes \cite[Exposé VIII, Corollaire 7.7]{SGA1}, and smoothness is fpqc-local \cite[Proposition 17.7.1(ii)]{EGA4-4}. It remains to check the sheaf condition for morphisms: given $X,Y\in\SmQP_S$, we must show that the functors $T\mapsto \Map_T(X_{T+},Y_{T+})$ and $T\mapsto c_T(X_T,Y_T)$ are sheaves for the given topologies on $\Sch_{/S}$. The former is obviously a sheaf for the canonical topology, and the finite locally free topology is subcanonical. For the latter, since $c_S(\ph,Y)$ is an étale sheaf on $\SmQP_S$ \cite[Proposition 10.2.4(1)]{CD}, it suffices to note that there is an isomorphism $c_T(X_T,Y_T)\simeq c_S(X_T,Y)$ natural in $T\in \SmQP_S$.
\end{proof}

By Lemma~\ref{lem:corr-etale-descent} and Corollary~\ref{cor:automatic-norms}, applied with $\scr C=\Sch^\noe$, $t$ the finite étale topology, $m=\fet$, and $\scr D=\Fun(\Delta^1,\Cat_1)$, we obtain a natural transformation
\[
\SmQP_+^\otimes \to \SmQPCor^\otimes\colon \Span(\Sch^\noe,\all,\fet) \to \CAlg(\Cat_1)
\]
whose components are the functors~\eqref{eqn:graph}.
We can view this transformation as a functor
\[
\Span(\Sch^\noe,\all,\fet)\times\Delta^1 \to \CAlg(\Cat_1).
\]
Repeating the steps of \sect\ref{sub:coherence}, we obtain a functor
\[
\Span(\Sch^\noe,\all,\fet)\times\Delta^1 \to \CAlg(\what\Cat{}_\infty^\mathrm{sift}),\quad (S,0\to1)\mapsto (\SH(S) \to \DM(S)),
\]
or equivalently a natural transformation
\[
\Z_\tr\colon \SH^\otimes\to\DM^\otimes\colon \Span(\Sch^\noe,\all,\fet) \to \CAlg(\what\Cat{}_\infty^\mathrm{sift}).
\]

\begin{theorem}\label{thm:HZnormed}
	The assignment $S\mapsto \HH\Z_S\in\SH(S)$ on noetherian schemes can be promoted to a section of $\SH^\otimes$ over $\Span(\Sch^\noe,\all,\fet)$ that is cocartesian over backward essentially smooth morphisms. 
	In particular, for every noetherian scheme $S$, Voevodsky's motivic cohomology spectrum $\HH\Z_S$ has a structure of normed spectrum over $\Sm_S$.
\end{theorem}

\begin{proof}
	Consider $\Z_\tr$ as a map of cocartesian fibrations over $\Span(\Sch^\noe,\all,\fet)$.
	By Lemma~\ref{lemm:construct-relative-adjoint}(1), it admits a relative right adjoint $u_\tr$, given fiberwise by the forgetful functor $\DM(S)\to\SH(S)$.
	Hence, composing the unit section of $\DM^\otimes$ with $u_\tr$, we obtain the desired section of $\SH^\otimes$.
	It is cocartesian over essentially smooth morphisms by \cite[Theorem 4.18]{HoyoisMGL}.
\end{proof}

\begin{remark}
	We can further enhance $\DM^\otimes$ to a functor
	\[
	\DM^\otimes\colon \Span(\Sch^\noe,\all,\fet) \times \CAlg(\Ab) \to \CAlg(\what{\Cat}{}_\infty^\mathrm{sift}),\quad
	(S,R)\mapsto \DM(S,R),
	\]
	where $\DM(S,R)$ is the $\infty$-category of motives over $S$ with coefficients in $R$. Let $\SmQPCor_{S,R}$ be the symmetric monoidal category whose sets of morphisms are the $R$-modules $c_S(X,Y)_\Lambda\otimes R$, where $\Lambda\subset\Q$ is the maximal subring such that $R$ is a $\Lambda$-algebra and $c_S(X,Y)_\Lambda$ is defined using $\Lambda$-universal cycles instead of universally integral cycles. We then have a functor
	 \[
	 \Sch^{\noe,\op}\times\CAlg(\Ab)\to \CAlg(\Cat_1),\quad (S,R)\mapsto \SmQPCor_{S,R}.
	 \]
	 Moreover, $c_S(\ph,Y)_\Lambda\otimes R$ is still an étale sheaf \cite[Proposition 2.1.4]{CDetale}, so that $S\mapsto \SmQPCor_{S,R}$ is a finite étale sheaf, and the rest of the construction is the same.
	 It follows that, for any noetherian scheme $S$, we have a functor
	 \[
	 \CAlg(\Ab) \to \NAlg_\Sm(\SH(S)),\quad R\mapsto \HH R_S.
	 \]
\end{remark}

\begin{remark}\label{rmk:cdh-transfers}
	The same construction produces norms for the cdh version $\HH R_S^\cdh$ constructed in \cite{CDintegral}.
	We must replace $\SmQPCor_{S,R}$ by $\QPCor_{S,R}$, and we need to know that if $p\colon T\to S$ is finite étale then $p_*\colon \PSh_\Sigma(\QP_T)\to \PSh_\Sigma(\QP_S)$ preserves cdh equivalences. This is proved as in Proposition~\ref{prop:integralsifted}, using that the points of the cdh topology are henselian valuation rings \cite[Table 1]{GabberKelly} and that a finite étale extension of a henselian valuation ring is a finite product of henselian valuation rings \cite[Tag 0ASJ]{Stacks}.
	It follows from \cite[Theorem 5.1]{CDintegral} that the section $S\mapsto \HH R_S^\cdh$ of $\SH^\otimes$ is cocartesian over finite-dimensional noetherian $\Q$-schemes (or $\mathbb F_p$-schemes if $p$ is invertible in $R$).
\end{remark}

\begin{remark}
	The structures of normed spectra on $\HH\Z$, $\HH\Z^\cdh$, and $\HH\Z[1/p]^\cdh$ constructed above coincide with those obtained in \sect\ref{sub:HZ}. This follows from Remark~\ref{rmk:norm-uniqueness}.
\end{remark}

\begin{remark}
	Let $\Chow(S)$ denote the full subcategory of $\DM(S)$ spanned by retracts of motives of the form $\Z_\tr\Sigma^\xi\Sigma^\infty_+ X$, where $X$ is smooth and projective over $S$ and $\xi\in \K(S)$.
	If $p\colon T\to S$ is finite étale, the norm functor $p_\otimes\colon \DM(T)\to\DM(S)$ then sends $\Chow(T)$ to $\Chow(S)$, since $p_\otimes(\S^\xi)\simeq \S^{p_*(\xi)}$ (see Remark~\ref{rmk:Pic}) and since Weil restriction along finite étale maps preserves properness \cite[\sect7.6, Proposition 5(f)]{NeronModels}.
	It follows that $\DM^\otimes$ admits a full subfunctor
	\[
	\Span(\Sch^\noe,\all,\fet) \to \CAlg(\Cat_\infty), \quad S\mapsto \Chow(S).
	\]
	If we only include twists with $\rk\xi\geq 0$, we get a further full subfunctor $S\mapsto \Chow(S)^\eff$.
	If $k$ is a field, the homotopy category $\h\Chow(k)$ (resp.\ $\h\Chow(k)^\eff$) is the opposite of Grothendieck's category of Chow motives (resp.\ of effective Chow motives) over $k$, by Poincaré duality and the representability of Chow groups in $\DM(k)$. In particular, $\DM^\otimes$ contains $\infty$-categorical versions of the norm functors between categories of (effective) Chow motives constructed by Karpenko \cite[\sect5]{Karpenko}; that the former are indeed refinements of the latter follows from Theorem~\ref{thm:FultonMacPhersonComparison} below.
\end{remark}

Recall that $\HH\Z_S$ is an oriented motivic spectrum.
By Theorem~\ref{thm:HZnormed} and Proposition~\ref{prop:oriented-normed}, we obtain for every finite étale map $p\colon T\to S$ a multiplicative norm map
\[
\nu_p\colon \bigoplus_{n\in\Z}\Map(\1_T,\Sigma^{2n,n}\HH\Z_T) \to \bigoplus_{r\in\Z}\Map(\1_S,\Sigma^{2r,r}\HH\Z_S).
\]
Moreover, when $S$ is smooth over a field, $\Map(\1_S,\Sigma^{2r,r}\HH\Z_S)$ is (the underlying space of) Bloch's cycle complex $z^r(S,*)$. Our next goal is to compare the norm maps so obtained on Bloch's cycle complexes with the norm maps constructed by Fulton and MacPherson on Chow groups \cite{FultonMacPherson}.

\subsection{The Fulton–MacPherson norm on Chow groups}

If $X$ is a noetherian scheme, we denote by $z^*(X)$ the group of cycles on $X$, i.e., the free abelian group on the points of $X$. Flat pullback and proper pushforward of cycles define a functor
\[
z^*\colon \Span(\Sch^\noe,\mathrm{flat},\mathrm{proper})\to \Ab
\]
\cite[Proposition 1.7]{Fulton}, which is moreover an étale sheaf \cite[Theorem 1.1]{Anschutz}. If $f\colon Y\to X$ is a map between smooth $S$-schemes where $S$ is regular, there is a subgroup $z^*_f(X)\subset z^*(X)$ consisting of cycles \emph{in good position} with respect to $f$, such that the pullback $f^*\colon z^*_f(X) \to z^*(Y)$ is defined (see for instance \cite[Appendix B]{SpitzweckHZ}).

Let $p\colon T\to S$ be a finite étale map between regular schemes and let $\alpha\in z^*(T)$ a cycle on $T$. We shall say that $\alpha$ is \emph{$p$-normable} if, locally on $S$ in the finite étale topology, it has the form $\coprod_i\alpha_i\in z^*(\coprod_i S)$ where the cycles $\alpha_i$ intersect properly. If (and only if) $\alpha$ is $p$-normable, we may therefore define $\Norm_p(\alpha)\in z^*(S)$ using finite étale descent: it is the unique cycle on $S$ such that, if $S'\to S$ is a finite étale morphism such that $\alpha_{S'\times_ST}$ has the form $\coprod_i\alpha_i$ as above, then $\Norm_p(\alpha)_{S'}$ is the intersection product of the cycles $\alpha_i$. For example, if $X,Y\in\SmQP_T$, the map
\[
c_T(X,Y) \to c_S(\Weil_pX,\Weil_pY)
\]
induced by the functor $p_\otimes\colon \SmQPCor_T\to \SmQPCor_S$ is given by the composition $\Norm_q\circ e^*$, where \[X\times_TY \stackrel e\from \Weil_p(X\times_TY)\times_ST \stackrel q\to \Weil_p(X\times_TY).\]

The Fulton–MacPherson construction in a nutshell is the following observation: given a field $k$ and a finite étale map $p\colon T\to S$ in $\SmQP_k$, there exists a cartesian square
\begin{tikzmath}
	\diagram{ T & T'\\ S & S' \\};
	\arrows (11-) edge[c->] (-12) (11) edge node[left]{$p$} (21) (21-) edge[c->] (-22) (12) edge node[right]{$q$} (22);
\end{tikzmath}
in $\SmQP_k$, where $q$ is finite étale, such that every cycle on $T$ is the pullback of a $q$-normable cycle on $T'$. To construct such a square, we may assume that $p\colon T\to S$ has constant degree $d$. Associated with $p$ is a principal $\Sigma_d$-bundle $P=\mathrm{Isom}_S(d\times S,T)\to S$, and we have a $\Sigma_d$-equivariant immersion
\[
\Hom_S(d\times S,T) \simeq \underbrace{T\times_S\dotsb\times_ST}_{d\text{ times}} \into T^d,
\]
where the product $T^d$ is formed over $k$. Let $P'$ be the largest open subset of $T^d$ where $\Sigma_d$ acts freely.
Since $T$ is quasi-projective over $k$, we can form the quotients $S'=P'/\Sigma_d$ and $T'=(d\times P')/\Sigma_d$ in $\SmQP_k$. We then obtain a cube
\begin{tikzequation}[cross line/.style={preaction={draw=white, -,line width=6pt}}]]
	\label{eqn:PBnormable}
	\def\colsep{1em}
	\def\rowsep{1em}
	\diagram{
	& d\times P & & & d \times P' &&\\
	T & & & T' & & & T \\
	& P & & & P'&& \\
	S & & & S' & && \\
	};
	\arrows (12-) edge[c->] node[above]{$d\times u$} (-15) (32-) edge[c->] node[above,xshift=-10pt]{$u$} (-35) (41-) edge[c->] node[above]{$s$} (-44)
	(21) edge node[left]{$p$} (41)
	(12) edge (32) (15) edge (35)
	(12) edge (21) (15) edge (24) (32) edge (41) (35) edge (44)
	(21-) edge[c->,cross line] node[above,xshift=10pt]{$t$} (-24) (24-) edge[cross line] node[above,xshift=10pt]{$r$} (-27)
	(24) edge[cross line] node[right,inner sep=2pt,outer sep=1pt,fill=white]{$q$} (44)
	(15) edge node[above right]{$\sum_i \pi_i$} (27);
\end{tikzequation}
with cartesian faces, where $q$ is finite étale of degree $d$, $s$ is an immersion, the maps from the back face to the front face are principal $\Sigma_d$-bundles, and $r$ is a smooth retraction of $t$.

\begin{lemma}\label{lem:FMP}
	Let $k$ be a field, $S$ a smooth quasi-projective $k$-scheme, and $p\colon T\to S$ a finite étale map of constant degree $d$. Form the diagram~\eqref{eqn:PBnormable}.
	Then, for all $\alpha\in z^*(T)$, $r^*(\alpha)$ is $q$-normable. 
\end{lemma}

\begin{proof}
Note that $r^*(\alpha)$ corresponds to the $\Sigma_d$-invariant cycle $\coprod_{i}\pi_i^*(\alpha)\in z^*(d\times P')$. We must therefore show that the cycles $\pi_i^*(\alpha)$ on $P'$ intersect properly, which is clear ($k$ being a field). In fact, $\Norm_q r^*(\alpha)$ is the cycle on $S'$ corresponding to the $\Sigma_d$-invariant cycle $\alpha^{\times d}$ on $P'$.
\end{proof}

\begin{proposition}\label{prop:Chownorms}
	Let $k$ be a field. The functor $\CH^*\colon \SmQP_k^\op\to \Set$ extends uniquely to a functor
	\[
	\CH^*\colon \Span(\SmQP_k,\all,\fet) \to \Set, \quad (U\stackrel f\from T\stackrel p\to S)\mapsto \nu_p^\FM f^*,
	\]
	such that, if $p\colon T\to S$ is finite étale and $\alpha\in z^*(T)$ is $p$-normable, then $\nu_p^\FM[\alpha]=[\Norm_p(\alpha)]$.
\end{proposition}

\begin{proof}
	By Lemma~\ref{lem:FMP}, there is a unique family of maps $\bar \Norm_p\colon z^*(T)\to \CH^*(S)$ compatible with (partially defined) pullbacks. It remains to show that $\bar \Norm_p$ passes to rational equivalence classes. By definition, we have a coequalizer diagram of abelian groups
	\[
	z^*_{T\times\{0,1\}}(T\times\A^1) \rightrightarrows z^*(T)\to \CH^*(T).
	\]
	Since this is a reflexive coequalizer, it is also a coequalizer in the category of sets, so it suffices to show that for any $\alpha\in z^*_{T\times\{0,1\}}(T\times\A^1)$, $\bar \Norm_pi_0^*(\alpha)=\bar \Norm_pi_1^*(\alpha)$. By $\A^1$-invariance of Chow groups, we have
	\[
	\bar \Norm_pi_0^*(\alpha) = i_0^*\bar \Norm_{p\times\id}(\alpha) = i_1^*\bar \Norm_{p\times\id}(\alpha) = \bar \Norm_pi_1^*(\alpha).\qedhere
	\]
\end{proof}

The map $\nu_p^\FM\colon \CH^*(T)\to\CH^*(S)$ from Proposition~\ref{prop:Chownorms} will be called the \emph{Fulton–MacPherson norm}.

\begin{example}
	Let $k$ be a field, $p\colon T\to\Spec k$ a finite étale map, $X\in\SmQP_T$, and $X\stackrel e\from (\Weil_{p}X)_T \stackrel q\to \Weil_{p}X$. Then $e$ is smooth and every cycle of the form $e^*(\alpha)$ is $q$-normable (this can be checked when $p$ is a fold map). Hence, we have a norm map
	\[
	\Norm_qe^*\colon z^*(X) \to z^*(\Weil_{p}X),
	\]
	which descends to Chow groups by Proposition~\ref{prop:Chownorms}. These norm maps were studied by Karpenko in \cite{Karpenko}.
\end{example}

\subsection{Comparison of norms}

\begin{theorem}\label{thm:FultonMacPhersonComparison}
	Let $k$ be a field, $S$ a smooth quasi-projective $k$-scheme, and $p\colon T\to S$ a finite étale map. 
	Then the norm map $\nu_p\colon z^*(T,*) \to z^*(S,*)$ induced by the normed structure and the orientation of $\HH\Z_S$ induces the Fulton–MacPherson norm $\nu_p^\FM\colon\CH^*(T)\to \CH^*(S)$ on $\pi_0$.
\end{theorem}

Before we can prove this theorem, we need to recall the comparison theorem between Voevodsky's motivic cohomology and higher Chow groups from \cite[Part 5]{Mazza:2006}. In fact, we will formulate a slightly more general form of this comparison theorem that also incorporates the Thom isomorphism (see Proposition~\ref{prop:ChowComparison}).

We denote by $\SmSep_S\subset\Sm_S$ the full subcategory of smooth separated $S$-schemes.

\begin{definition}\label{def:Thom}
	Let $S$ be a scheme and $V\to S$ a vector bundle. A \emph{Thom compactification} of $V$ is an open immersion $V\into P$ in $\SmSep_S$ together with a presheaf $P_\infty\in\PSh(\SmSep_S)_{/P}$ such that:
	\begin{enumerate}
		\item $P_\infty\times_PV$ is the empty scheme;
		\item $P/P_\infty\to P/(P\minus S)$ is a motivic equivalence.
	\end{enumerate}
\end{definition}

If $V\into P\from P_\infty$ is a Thom compactification of $V$, then $P/P_\infty$ is a model for the motivic sphere $\S^{V}$, since we have a Zariski equivalence $V/(V\minus S)\to P/(P\minus S)$.

\begin{example}
	The prototypical example of a Thom compactification is $V\into \P(V\times\A^1) \hookleftarrow \P(V)$. 
	More generally, if $p\colon T\to S$ is a finite étale map and $V\to T$ is a vector bundle, then
	\[
	\Weil_pV\into \Weil_p(\P(V\times \A^1))\hookleftarrow p_*(\P(V\times\A^1)\Vert \P(V))
	\]
	is a Thom compactification of $\Weil_pV$: (1) is clear and (2) follows from Propositions~\ref{prop:quotients} and~\ref{prop:pairs}.
\end{example}
 
If $S$ is regular and $X$ is of finite type over $S$, we write $z^\equi_0(X/S)$ for the presheaf on $\Sm_S$ sending $U$ to the abelian group of relative cycles on $X\times_SU/U$ that are equidimensional of relative dimension $0$ over $U$.

\begin{lemma}\label{lem:finite-to-equi}
	Let $S$ be a regular noetherian scheme, $V\to S$ a vector bundle, and $V\stackrel{j}\longinto P \from P_\infty$ a Thom compactification of $V$. Then the restriction $j^*\colon \Z_{\tr,S}(P)\to z^\equi_0(V/S)$ induces a motivic equivalence
	\[
	j^*\colon \Z_{\tr,S}(P/P_\infty) \to z^\equi_0(V/S)
	\]
	in $\PSh(\SmSep_S)$.
\end{lemma}

\begin{proof}
	By Definition~\ref{def:Thom}(1), the composition $\Z_\tr(P_\infty) \to \Z_\tr(P) \to z^\equi_0(V/S)$ is the zero map, and we have an induced map as claimed. 
	Let $F\subset z^\equi_0(V/S)$ be the subpresheaf consisting of cycles that do not meet the zero section, and consider the commutative square
	\begin{tikzmath}
		\diagram{\Z_{\tr}(P)/\Z_\tr(P_\infty) & \Z_\tr(P)/\Z_\tr(P\minus S) \\
		z^\equi_0(V/S) & z^\equi_0(V/S)/F \\};
		\arrows (11-) edge (-12) (11) edge (21) (21-) edge (-22) (12) edge (22);
	\end{tikzmath}
	of presheaves of connective $\HH\Z$-modules on $\SmSep_S$.
	By Definition~\ref{def:Thom}(2), the top horizontal map is a motivic equivalence of presheaves with transfers, whence also a motivic equivalence of underlying presheaves (Remark~\ref{rmk:compat-transfers}).
	The proof of \cite[Lemma 16.10]{Mazza:2006} shows that $F$ is $\A^1$-contractible, so that the bottom horizontal map is an $\A^1$-equivalence. Note that $\Z_\tr(P\minus S)$ is exactly the preimage of $F$ by the restriction map $j^*\colon \Z_\tr(P) \to z^\equi_0(V/S)$, so that the right vertical map is a monomorphism. 
	The proof of \cite[Lemma 16.11]{Mazza:2006} shows that it is also surjective on henselian local schemes. It is therefore a Nisnevich equivalence, and we are done. 
\end{proof}

If $X^\bullet$ is a cosimplicial noetherian scheme with smooth degeneracy maps, we will denote by $z^*(X^\bullet)$ the simplicial abelian group whose $n$-simplices are those cycles in $z^*(X^n)$ intersecting all the faces properly; this is a meaningful simplicial set because the face maps in $X^\bullet$ are regular immersions (cf.\ \cite[Definition 17.1]{Mazza:2006}). 
For example, $\lvert z^*(\Delta^\bullet_X)\rvert$ is Bloch's cycle complex $z^*(X,*)$, where $\Delta^n_X\simeq \A^n_X$ is the standard algebraic $n$-simplex over $X$.

\begin{lemma}
	\label{lem:equi-to-Bloch}
	Let $k$ be a field, $X$ a smooth $k$-scheme, and $V\to X$ a vector bundle of rank $r$. Then the inclusion
	\[
	z^\equi_0(V/X)(\Delta^\bullet\times\ph) 
	\subset 
	z^r(\Delta^\bullet\times V\times_X \ph)
	\]
	of simplicial presheaves on the étale site of $X$ becomes a Zariski equivalence after geometric realization.
\end{lemma}

\begin{proof}
	 By standard limit arguments, we can assume that $k$ is perfect. Since moreover the question is Zariski-local on $X$, we can assume that $V$ is free. In this case the result is proved in \cite[Theorem 19.8]{Mazza:2006}.
\end{proof}

We denote by $\Sm_S^\mathrm{flat}\subset\Sm_S$ and $\SmSep_S^\mathrm{flat}\subset \SmSep_S$ the wide subcategories whose morphisms are the flat morphisms.

\begin{proposition}\label{prop:ChowComparison}
	Let $k$ be a field, $X$ a smooth $k$-scheme, $V\to X$ a vector bundle of rank $r$, and $(j\colon V\into P, P_\infty\to P)$ a Thom compactification of $V$. Then the natural transformation
	\[
	j^*\colon \Z_{\tr,X}(P/P_\infty)(\Delta^\bullet\times\ph) \to z^r(\Delta^\bullet\times V\times_X\ph)
	\]
	on $\SmSep_X^\mathrm{flat}$ induces an equivalence
	\[
	\L_\mot \Z_{\tr,X}(P/P_\infty)|\Sm_X^\mathrm{flat} \simeq z^r(V\times_X\ph,*)
	\]
	in $\PSh(\Sm_X^\mathrm{flat})$.
\end{proposition}

\begin{proof}
	Since $z^r(-,*)$ is an $\A^1$-invariant Nisnevich sheaf on $\Sm_k^\mathrm{flat}$, it follows from Lemma~\ref{lem:equi-to-Bloch} that
	\[
	\L_\mot z^\equi_0(V/X)|\Sm_X^\mathrm{flat} \simeq z^r(V\times_X\ph,*).
	\]
	We conclude using Lemma~\ref{lem:finite-to-equi}.
\end{proof}

\begin{remark}\label{rmk:movingLemma}
	The presheaf $z^r(V\times_X\ph,*)$ on $\Sm_X^\mathrm{flat}$ can be promoted to a presheaf on $\Sm_X$ using Levine's moving lemma, and Proposition~\ref{prop:ChowComparison} can be improved to an equivalence in $\PSh(\Sm_X)$. Recall that, for any morphism $f\colon Y\to X$ in $\Sm_k$, there is a largest subsimplicial abelian group $z^*_f(\Delta^\bullet_X)\subset z^*(\Delta^\bullet_X)$ such that the pullback $f^*\colon z^*_f(\Delta^\bullet_X) \to z^*(\Delta^\bullet_Y)$ is defined. By Levine's moving lemma \cite[Theorem 2.6.2 and Lemma 7.4.4]{LevineChow}, if $X$ is affine, the inclusion $z^*_f(\Delta^\bullet_X)\subset z^*(\Delta^\bullet_X)$ induces an equivalence on geometric realization.
	 The simplicial sets $z^*_f(\Delta^\bullet_X)$ can be arranged into a presheaf $\scr L(\Sm_k)^\op \to \Fun(\Delta^\op,\Set)$, where $\pi\colon\scr L(\Sm_k)\to\Sm_k$ is a certain locally cartesian fibration with weakly contractible fibers (see \cite[\sect7.4]{LevineChow}). Its geometric realization $\scr L(\Sm_k)^\op \to\scr S$ is then constant along the fibers of $\pi$ over $\SmAff_k$, and hence it induces a Zariski sheaf $\Sm_k^\op\to\scr S$, $X\mapsto z^*(X,*)$.
	 Using that the inclusion $z^\equi_0(V/X)(\Delta^\bullet_X)\subset z^r(\Delta^\bullet_V)$ lands in $z^r_f(\Delta^\bullet_V)$ for any $f\colon Y\to X$, the transformation $j^*\colon \Z_{\tr,X}(P/P_\infty)(\Delta^\bullet\times\ph) \to z^r(\Delta^\bullet\times V\times_X\ph)$ on $\SmSep_X^\mathrm{flat}$ can be promoted to a transformation on $\scr L(\SmSep_X)$. Its geometric realization then induces a transformation $\L_\mot \Z_{\tr,X}(P/P_\infty)\to z^r(V\times_X\ph,*)$ on $\Sm_X$, which is an equivalence by Proposition~\ref{prop:ChowComparison}.
\end{remark}

\begin{remark}\label{rmk:motivicThom}
	By homotopy invariance of higher Chow groups, the right-hand side of the equivalence of Proposition~\ref{prop:ChowComparison} is equivalent to $z^r(\ph,*)$. Since $\L_\mot\Z_{\tr,X}(\P(V\times\A^1)/\P(V))\simeq \Omega^\infty \Sigma^{V} \HH\Z_X$, we obtain in particular an isomorphism
	\[
	[\1_X, \Sigma^V\HH\Z_X] \simeq \CH^r(X),
	\]
	which is an instance of the Thom isomorphism in motivic cohomology.
\end{remark}

If $E$ is an $X$-scheme equipped with a morphism of $X$-schemes $\sigma\colon\Delta^1_X\to E$, we will denote by $E^\bullet$ the right Kan extension to $\Delta$ of the diagram
 \begin{tikzmath}
 	\diagram{ X & E\rlap. \\};
	\arrows (11-) edge[vshift=2*\dbl] node[above=2*\dbl]{$\sigma_0$} (-12) (11-) edge[vshift=-2*\dbl] node[below=2*\dbl]{$\sigma_1$} (-12)
	(-12) edge (11-);
 \end{tikzmath}
The map $\sigma$ then extends uniquely to a cosimplicial map $\sigma\colon\Delta^\bullet_X\to E^\bullet$. 
More concretely, we can regard $E$ as an interval object via $\sigma$, and $E^\bullet$ is the associated cosimplicial $X$-scheme described in \cite[\sect2.3.2]{MV}.

If $X^{\bullet,\bullet}$ is a bicosimplicial scheme with smooth degeneracy maps, the bisimplicial abelian group $z^*(X^{\bullet,\bullet})$ is defined in an obvious way: its simplices are the cycles intersecting all the bifaces properly.

\begin{lemma}\label{lem:vbBloch}
	Let $X$ be a smooth affine scheme over a field, $E$ a vector bundle on $X$, and $\sigma\colon\Delta^1_X\into E$ a linear immersion with a linear retraction $\rho\colon E\to \Delta^1_X$.
	Then there is a commutative square of simplicial abelian groups
	\begin{tikzmath}
		\diagram{z^*(E^\bullet) & \delta^*z^*(\Delta^\bullet\times E^\bullet) \\
		z^*(\Delta^\bullet_X) & \delta^*z^*(\Delta^\bullet\times \Delta_X^\bullet)\rlap, \\};
		\arrows (11-) edge node[above]{$\pi_2^*$} (-12) (21) edge node[left]{$\rho^*$} (11) (21-) edge node[above]{$\pi_2^*$} node[below]{$\sim$} (-22) (22) edge node[right]{$\rho^*$} node[left]{$\sim$} (12);
	\end{tikzmath}
	where $\delta^*$ is restriction along the diagonal $\delta\colon \Delta^\op \to \Delta^\op\times \Delta^\op$ and ``$\sim$'' indicates a simplicial homotopy equivalence.
\end{lemma}

\begin{proof}
	First, note that there is a well-defined simplicial map $\rho^*\colon z^*(\Delta^\bullet_X) \to z^*(E^\bullet)$. Indeed, we can write $E=\Delta^1\times F$ so that $\rho$ is the projection onto the first summand, and if $\alpha\in z^*(\Delta^n_X)$ intersects all faces properly, it is then clear that $\alpha\times F^n\in z^*(\Delta^n\times F^n)$ does as well. Similarly, we have a well-defined bisimplicial map $\rho^*\colon z^*(\Delta^\bullet\times\Delta^\bullet_X) \to z^*(\Delta^\bullet\times E^\bullet)$.
	
	Let $\scr A_n$ be the collection of faces of $\Delta^n_X$ and $\scr B_n$ that of $E^n$.
	In the commutative triangle of simplicial sets
	\begin{tikzmath}
		\diagram{z^*(\Delta^\bullet_X) & z^*_{\scr B_n}(\Delta^\bullet_X\times_X E^n)\rlap, \\
		z^*_{\scr A_n}(\Delta^\bullet_X\times_X \Delta^n_X) & \\};
		\arrows (11-) edge node[above]{$\pi_1^*$} (-12) (11) edge node[left]{$\pi_1^*$} (21) (21) edge node[below right]{$(\rho^*)_n$} (12);
	\end{tikzmath}
	the maps $\pi_1^*$ are simplicial homotopy equivalences by Levine's moving lemma \cite[Theorem 2.6.2]{LevineChow}
	 and the homotopy invariance of higher Chow groups. Hence, $(\rho^*)_n$ is also a simplicial homotopy equivalence. The maps $(\rho^*)_n$ are the components of the bisimplicial map $\rho^*$, which therefore induces a simplicial homotopy equivalence on the diagonal. Similarly, the fact that the vertical map $\pi_1^*$ above is a simplicial homotopy equivalence implies that $\pi_1^*\colon z^*(\Delta^\bullet_X)\to \delta^*z^*(\Delta^\bullet_X\times_X\Delta^\bullet_X)$ is a simplicial homotopy equivalence, whence also $\pi_2^*$ by symmetry.
\end{proof}

\begin{proof}[Proof of Theorem~\ref{thm:FultonMacPhersonComparison}]
We can assume that $p$ has constant degree $d\geq 1$, and we will work with the cube~\eqref{eqn:PBnormable}.
The norms $\nu_p,\nu_p^\FM\colon \CH^*(T)\rightrightarrows \CH^*(S)$ can then be factored as
\begin{tikzmath}
	\diagram{\CH^*(T) & \CH^*(S') & \CH^*(S),\\};
	\arrows (11-) edge[vshift=\dbl,-top] node[above=\dbl]{$\nu_q r^*$} (-12) edge[vshift=-\dbl,-bot] node[below=\dbl]{$\nu_q^\FM r^*$} (-12) (12-) edge node[above]{$s^*$} (-13);
\end{tikzmath}
and we will prove that the two parallel arrows are equal.

For $X\in\SmQP_T$, let $X'\in\SmQP_{T'}$ be the pullback of $X$ along $r$. We then have a diagram
\begin{tikzmath}
	\def\colsep{1.5em}
	\diagram{X & X' & X \\ \Weil_pX\times_ST & \Weil_qX'\times_{S'}T' & \\ \Weil_pX & \Weil_qX' & \\};
	\arrows (11-) edge[c->] (-12) (12-) edge (-13) (21) edge (11) (22) edge (12) (21-) edge[c->] node[above]{$t$} (-22) (21) edge node[left]{$p$} (31) (22) edge node[right]{$q$} (32) (31-) edge[c->] node[above]{$s$} (-32) (22) edge node[below right]{$r$} (13);
\end{tikzmath}
natural in $X$, where the top row is the identity, the bottom square is cartesian, and the maps $p$, $q$, $r$, $s$, and $t$ specialize to the ones in~\eqref{eqn:PBnormable} when $X=T$. As in Lemma~\ref{lem:FMP}, for every cycle $\alpha\in z^*(X)$, $r^*(\alpha)$ is $q$-normable: it corresponds to the $\Sigma_d$-invariant cycle $\coprod_i\pi_i^*(\alpha)$ on $d\times X^d\times_{T^d}P'\subset d\times X^d$.
We therefore obtain a map
\[
\Norm_qr^*\colon z^*(X) \to z^*(\Weil_qX'),
\]
which is natural in $X\in\SmQP_T^\mathrm{flat}$ (since $\Weil_q$ preserves flat morphisms \cite[\sect7.6, Proposition 5(g)]{NeronModels}).
If $X^\bullet$ is a cosimplicial object in $\SmQP_T$ and if $\alpha\in z^*(X^n)$ intersects all faces properly, then $\Norm_qr^*(\alpha)\in z^*(\Weil_q(X^{n\prime}))$ does as well. In particular, we obtain a map of simplicial sets
\begin{equation*}
\Norm_{q}r^*\colon z^*(\Delta^\bullet_X) \to z^*(\Weil_q(\Delta^{\bullet}_{X'})).
\end{equation*}
By construction, the composition
\[
z^*(X) \xrightarrow{\Norm_qr^*} z^*(\Weil_q{X'}) \to \CH^*(\Weil_qX') \xrightarrow{s^*} \CH^*(\Weil_pX)
\]
induces the Fulton–MacPherson norm $\nu_p^\FM$ when $X=T$.

Under the equivalence of Proposition~\ref{prop:ChowComparison}, our norm map $\nu_p\colon z^r(X,*) \to z^{rd}(\Weil_pX,*)$ is the unique transformation making the following square commute in $\PSh(\SmQP_T)$:
\begin{tikzmath}
	\diagram{
	\Z_{\tr,T}(\P^r_T/\P^{r-1}_T) & \Z_{\tr,S}(\Weil_p(\P^r_T)/p_*(\P^r_T\Vert \P^{r-1}_T))(\Weil_p(\ph)) \\
	\L_\mot\Z_{\tr,T}(\P^r_T/\P^{r-1}_T) & \L_\mot\Z_{\tr,S}(\Weil_p(\P^r_T)/p_*(\P^r_T\Vert \P^{r-1}_T))(\Weil_p(\ph))\rlap; \\
	};
	\arrows (11-) edge node[above]{$s^* \Norm_q r^*$} (-12) (11) edge (21) (12) edge (22) (21-) edge[dashed] node[above]{$\nu_p$} (-22);
\end{tikzmath}
it exists because the functor $\Weil_p\colon \SmQP_T\to \SmQP_{S}$ preserves Nisnevich sieves, by Proposition~\ref{prop:integralsifted}.

We have a commutative diagram of simplicial sets
\begin{tikzmath}
	\def\colsep{1em}
	\diagram{
	\Z_{\tr,T}(\P^r_T/\P^{r-1}_T)(\Delta^\bullet\times X) & z^r(\Delta^\bullet\times\A^r_X) & \CH^r(\A^r_X)\\
	\Z_{\tr,S'}(\Weil_q(\P^r_{T'})/q_*(\P^r_{T'}\Vert \P^{r-1}_{T'}))(\Weil_q(\Delta^\bullet\times{X'})) & z^{rd}(\Weil_q(\Delta^\bullet\times\A^r_{X'})) & \CH^{rd}(\Weil_q(\A^r_{X'}))\rlap,\\
	};
	\arrows (11-) edge (-12) (11) edge node[left]{$\Norm_qr^*$} (21) (12) edge node[left]{$\Norm_qr^*$} (22) (21-) edge (-22) 
	(12-) edge (-13) (22-) edge (-23) 
	(13) edge node[right]{$\nu_q^\FM r^*$} (23) 
	;
\end{tikzmath}
natural in $X\in\SmQP_T^\mathrm{flat}$. 
Since the canonical immersion $\Delta^1_{S'}\into \Weil_q(\Delta^1_{T'})$ is a universal monomorphism of vector bundles, its cokernel computed in the category of quasi-coherent sheaves is locally free, and hence its base change to any affine scheme admits a linear retraction.
As $S'$ is regular and has affine diagonal, it admits an ample family of line bundles. By the Jouanolou–Thomason trick, there exists a vector bundle torsor $\tilde S'\to S'$ with $\tilde S'$ affine.

We explain the next step of the argument in a generic context.
Let $Y$ be a smooth affine $k$-scheme, $V$ a vector bundle of rank $n$ over $Y$, and $V\into P\from P_\infty$ a Thom compactification of $V$.
Let $E$ be a vector bundle over $Y$ and $\sigma\colon \Delta^1_Y\into E$ a linear immersion with a linear retraction $\rho\colon E\to \Delta^1_Y$. Then we have a commutative diagram
\begin{tikzmath}
	\diagram{
	\Z_{\tr,Y}(P/P_\infty)(E^\bullet) & z^n(E^\bullet\times_YV) & & \CH^n(V) \\
	&  & \delta^*z^*(\Delta^\bullet\times E^\bullet\times_YV) & \CH^n(V) \\
	\Z_{\tr,Y}(P/P_\infty)(\Delta^\bullet_Y) & z^n(\Delta^\bullet\times V) & \delta^*z^*(\Delta^\bullet\times\Delta^\bullet\times V) & \CH^n(V)\rlap,\\
	};
	\arrows (11-) edge (-12) (12-) edge (-14)
	(23-) edge (-24)
	(31-) edge (-32) (32-) edge node[below]{$\sim$} node[above]{$\pi_2^*$} (-33) (33-) edge (-34)
	(31) edge node[left]{$\rho^*$} node[right]{$\sim$} (11)
	(32) edge node[left]{$\rho^*$} (12)
	(12) edge node[above right]{$\pi_2^*$} (23) (33) edge node[left]{$\rho^*$} node[right]{$\sim$} (23)
	(14) edge[-,vshift=1pt] (24) edge[-,vshift=-1pt] (24)
	(24) edge[-,vshift=1pt] (34) edge[-,vshift=-1pt] (34)
	;
\end{tikzmath}
where ``$\sim$'' indicates a simplicial homotopy equivalence. The middle part of this diagram is an instance of Lemma~\ref{lem:vbBloch}. That the leftmost $\rho^*$ is a simplicial homotopy equivalence (with homotopy inverse $\sigma^*$) follows from the observation that $\A^1$-homotopy equivalences and $E$-homotopy equivalences coincide.

Combining the geometric realizations of the two previous diagrams, we obtain a commutative diagram
\begin{tikzmath}
	\diagram{\L_{\A^1}\Z_{\tr,T}(\P^r_T/\dots)(X) & z^r(\A^r_X,*) & \CH^r(\A^r_X) \\
	\L_{\A^1}\Z_{\tr,S'}(\Weil_q(\P^r_{T'})/\dots)(\Weil_qX'\times_{S'}\tilde S') & z^{rd}(\Weil_q(\A^r_{X'})\times_{S'}\tilde S',*) & \CH^{rd}(\Weil_q(\A^r_{X'})\times_{S'}\tilde S')\rlap, \\};
	\arrows (11-) edge (-12) (12-) edge (-13) (11) edge (21) (21-) edge (-22) (12) edge (22) (22-) edge (-23) (13) edge (23);
\end{tikzmath}
natural in $X\in\SmAff_T$. As $T$ has affine diagonal, the subcategory $\SmAff_T\subset \SmQP_T$ is dense for the Nisnevich topology. The simplicial presheaves in the middle column are both $\A^1$-invariant Nisnevich sheaves in $X\in \SmQP_T^\mathrm{flat}$. Using Remark~\ref{rmk:movingLemma}, we deduce that the diagram remains commutative if we replace $\L_{\A^1}$ by $\L_\mot$, in which case the leftmost horizontal maps become equivalences by Proposition~\ref{prop:ChowComparison}. By the affine bundle invariance of Chow groups, we conclude in particular that $\nu_qr^*\colon z^r(T,*) \to z^{rd}(S',*)$ induces $\nu_q^\FM r^*$ on $\pi_0$.

Note that $r$ can be a locally constant integer on $T$ in the above argument. The fact that $\nu_p$ and $\nu_p^\FM$ agree on inhomogeneous elements follows by applying the homogeneous case to the finite étale map $\Weil_p(\coprod_{I}T)\times_ST\to \Weil_p(\coprod_IT)$, for any finite subset $I\subset\Z$.
\end{proof}

\begin{remark}
	The assumption that $S$ is quasi-projective in Theorem~\ref{thm:FultonMacPhersonComparison} can be weakened slightly.
	Indeed, the definition of the Fulton–MacPherson norm $\CH^*(T)\to\CH^*(S)$ only requires $S$ to be an FA-scheme in the sense of \cite[\sect2.2]{GLL}, and the proof of the theorem works in this generality since FA-schemes have affine diagonal.
	By standard limit arguments, the theorem holds more generally if $S$ is the limit of a cofiltered diagram of smooth FA-schemes over $k$ with affine flat transition maps.
\end{remark}

\section{Norms of linear \texorpdfstring{$\infty$}{∞}-categories}
\label{sec:dgCat}

Our goal in this section is to construct a functor
\[
\SH_\nc^\otimes\colon \Span(\Sch,\all,\fet) \to \CAlg(\what\Cat{}_\infty^\mathrm{sift}),\quad S\mapsto \SH_\nc(S),
\]
where $\SH_\nc(S)$ is Robalo's $\infty$-category of noncommutative motives over $S$, together with a natural transformation
\[
\scr L\colon \SH^\otimes \to \SH_\nc^\otimes.
\]
As a formal consequence, we will deduce that the homotopy $\K$-theory spectrum $\KGL_S\in\SH(S)$ is a normed spectrum, for every scheme $S$. The main difficulty will be to prove a noncommutative analog of Theorem~\ref{thm:norm}(4), which is the content of Proposition~\ref{prop:ncequiv}.

\subsection{Linear \texorpdfstring{$\infty$}{∞}-categories}

We start with some recollections on linear $\infty$-categories.
We write $\Pr^\mathrm{L}_\St$ for the $\infty$-category of stable presentable $\infty$-categories and left adjoint functors, and $\Pr^\mathrm{L,\omega}_\St\subset\Pr^\mathrm{L}_\St$ for the subcategory whose objects are the compactly generated stable $\infty$-categories and whose morphisms are the left adjoint functors that preserve compact objects (which is equivalent to the $\infty$-category of small stable idempotent complete $\infty$-categories \cite[Proposition 5.5.7.8]{HTT}).
For $R$ a commutative ring, denote by
\[
\Cat_R^\St = \Mod_{\Mod_R(\Sp)}(\Pr^\mathrm{L}_\St)
\]
the $\infty$-category of stable presentable $R$-linear $\infty$-categories, and let
\[
\Cat_R^\cg = \Mod_{\Mod_R(\Sp)}(\Pr^\mathrm{L,\omega}_\St).
\]
If $X$ is an arbitrary $R$-scheme, we will write $\QCoh(X)\in \Cat_R^\St$ for the stable $R$-linear $\infty$-category of quasi-coherent sheaves on $X$ \cite[Definition 2.2.2.1]{SAG}. If $X$ is quasi-compact and quasi-separated, then $\QCoh(X)$ belongs to $\Cat_R^\cg$ \cite[Theorem 10.3.2.1(b)]{SAG}.

We recall that the inclusion $\Pr^\mathrm{R} \subset \Cat_\infty$ preserves limits \cite[Proposition 5.5.3.18]{HTT}, which allows us to compute colimits in $\Pr^\mathrm{L}\simeq \Pr^{\mathrm{R},\op}$. Moreover, the inclusions $\Pr^\mathrm{L,\omega}_\St\subset\Pr^\mathrm{L}_\St\subset\Pr^\mathrm{L}$ preserve colimits by \cite[Theorem 1.1.4.4]{HA} and \cite[Proposition 5.5.7.6]{HTT}, hence the inclusion $\Cat_R^\cg\subset \Cat_R^\St$ also preserves colimits. A sequence $\scr A\to \scr B\to\scr C$ in $\Cat_R^\St$ is called \emph{exact} if it is a cofiber sequence and $\scr A\to\scr B$ is fully faithful \cite[Definition 5.8]{BGT}. Equivalently, by \cite[Proposition 5.6]{BGT}, this sequence is exact if it is a fiber sequence and the right adjoint of $\scr B\to\scr C$ is fully faithful.

By \cite[Proposition 4.7]{HSS}, the $\infty$-category $\Cat_R^\cg$ is compactly generated. We write
\[
\Cat_R^\fp = (\Cat_R^\cg)^\omega
\]
for the full subcategory of compact objects in $\Cat_R^\cg$. We also consider the full subcategory $\Cat_R^\ft\subset\Cat_R^\cg$ of $R$-linear $\infty$-categories possessing a compact generator. We then have fully faithful inclusions\footnote{The $R$-linear $\infty$-categories in $\Cat_R^\ft$ and $\Cat_R^\fp$ could reasonably be called \emph{of finite type} and \emph{of finite presentation}, respectively, which explains our notation. However, ``finite type'' is commonly used to refer to the $\infty$-categories in $\Cat_R^\fp$ (see for example \cite[Definition 2.4]{Toen:2007}).}
\[
\Cat_R^\fp\subset \Cat_R^\ft\subset\Cat_R^\cg
\]
(see the proof of \cite[Proposition 6.1.27]{RobaloThesis}).

We refer to \cite[\sect4.4]{HSS} for the construction of the functors
\[
\Aff^\op \to \CAlg(\Cat_{(\infty,2)}), \quad R\mapsto\Cat_R^\St\text{ and }R\mapsto\Cat_R^\cg.
\]
If $f\colon R\to R'$ is a ring homomorphism, then $f^*\colon \Cat_R^\cg \to \Cat_{R'}^\cg$ preserves the subcategories $\Cat^\ft$ and $\Cat^\fp$. Its right adjoint $f_*\colon \Cat_{R'}^\cg \to \Cat_R^\cg$ also preserves $\Cat^\ft$, but it only preserves $\Cat^\fp$ when $f$ is smooth:

\begin{lemma}\label{lem:smoothaffine}
	Let $f\colon R\to R'$ be a smooth morphism. Then the functor $f_*\colon \Cat_{R'}^\cg \to \Cat_{R}^\cg$ sends $\Cat_{R'}^\fp$ to $\Cat_{R}^\fp$.
\end{lemma}

\begin{proof}
	This is \cite[Proposition 9.2.10]{RobaloThesis}.
\end{proof}

We call a morphism $f\colon \scr A\to\scr B$ in $\Cat_R^\cg$ a \emph{monogenic localization} if its right adjoint is fully faithful and $f^{-1}(0)$ has a generator that is compact in $\scr A$. For example, if $j\colon U\into X$ is an open immersion between quasi-compact quasi-separated $R$-schemes, then $j^*\colon \QCoh(X)\to\QCoh(U)$ is a monogenic localization in $\Cat_R^\cg$ \cite[Theorem 6.8]{Rouquier}.

\begin{lemma}\label{lem:openimmersion}
	Let $f\colon\scr A\to\scr B$ be a monogenic localization in $\Cat_R^\cg$.
	If $\scr A\in \Cat_R^\fp$, then $\scr B\in\Cat_R^\fp$.
\end{lemma}

\begin{proof}
	Since the right adjoint of $f$ is fully faithful, we have an exact sequence
	\[
	f^{-1}(0) \into \scr A \xrightarrow{f} \scr B
	\]
	in $\Cat_R^\St$. As $f^{-1}(0)$ is generated by compact objects of $\scr A$, this sequence belongs to $\Cat_R^\cg$, hence it is a cofiber sequence in $\Cat_R^\cg$.
	For every $\scr C\in\Cat_R^\cg$, we may therefore identify $\Fun_R(\scr B^\omega,\scr C^\omega)$ with the full subcategory of $\Fun_R(\scr A^\omega,\scr C^\omega)$ spanned by the functors that send $f^{-1}(0)^\omega$ to zero.
	By \cite[Proposition 1.1.4.6 and Lemma 7.3.5.10]{HA}, the forgetful functor $\Cat_R^\cg\to\Cat_\infty$, $\scr A\mapsto\scr A^\omega$, preserves filtered colimits. Since $f^{-1}(0)$ has a compact generator and $\scr A$ is compact, we immediately deduce that $\scr B$ is compact.
\end{proof}

 A commutative square
\begin{tikzequation}\label{eqn:excision}
	\diagram{
	\scr A & \scr B \\ \scr C & \scr D \\
	};
	\arrows (11-) edge node[above]{$f$} (-12) (11) edge (21) (12) edge (22) (21-) edge node[above]{$g$} (-22);
\end{tikzequation}
in $\Cat_R^\cg$ is called an \emph{excision square} if it is cartesian in $\Cat_R^\St$ and $g$ is a monogenic localization, in which case $f$ is also a monogenic localization.
For example, $\QCoh$ sends Nisnevich squares of quasi-compact quasi-separated $R$-schemes to excision squares.

Our next goal is to construct norms of linear $\infty$-categories.
Recall that a \emph{Frobenius algebra} in a symmetric monoidal $\infty$-category $\scr C$ is an algebra $A$ together with a map $\lambda\colon A\to\1$ such that the composition
\[
A\tens A\to A\xrightarrow{\lambda} \1
\]
exhibits $A$ as dual to itself.

\begin{lemma}\label{lem:frobenius}
	Let $\scr C$ be a closed symmetric monoidal $\infty$-category and let $A$ be a Frobenius algebra in $\scr C$.
	Then the forgetful functor $\mathrm{LMod}_A(\scr C)\to\scr C$ has equivalent left and right adjoints.
\end{lemma}

\begin{proof}
	The left and right adjoint functors are given by $M\mapsto A\tens M$ and $M\mapsto\Hom(A,M)$, respectively.
	Since $A$ is a Frobenius algebra, it is dualizable and its dual $A^\vee$ is equivalent to $A$ as a left $A$-module \cite[Proposition 4.6.5.2]{HA}. Thus, $\Hom(A,M) \simeq A^\vee \tens M \simeq A \tens M$, which completes the proof.
\end{proof}

\begin{lemma}\label{lem:local-generators}
	Let $p\colon R\to R'$ be a faithfully flat finite étale morphism and let $\scr A\in\Cat_R^\cg$.
	\begin{enumerate}
		\item If $p^*(\scr A)\in\Cat_{R'}^\ft$, then $\scr A\in\Cat_R^\ft$.
		\item If $p^*(\scr A)\in\Cat_{R'}^\fp$, then $\scr A\in\Cat_R^\fp$.
	\end{enumerate}
\end{lemma}

\begin{proof}
	If $f\colon R\to R'$ is any finite étale morphism, both functors
	\[
		f^*\colon \Mod_R(\Sp) \to \Mod_{R'}(\Sp)\quad\text{and}\quad f^*\colon \Cat_R^\cg \to \Cat_{R'}^\cg
	\]
	have equivalent left and right adjoints, by Lemma~\ref{lem:frobenius}.
	Indeed, the trace map $R'\to R$ exhibits $R'$ as a Frobenius algebra over $R$, and the pushforward $f_*\colon \Mod_{R'}(\Sp)\to \Mod_R(\Sp)$ exhibits $\Mod_{R'}(\Sp)$ as a Frobenius algebra in $\Cat_R^\cg$ \cite[Remark 11.3.5.4]{SAG}.
	
	Let $R\to R'_\bullet$ be the Čech nerve of $p$. Then, since $\QCoh$ is an fpqc sheaf of $\infty$-categories \cite[Proposition 6.2.3.1]{SAG},
	\[
	\Mod_R(\Sp) \simeq \lim_{n\in\Delta} \Mod_{R_n'}(\Sp)
	\]
	in $\Cat_\infty$.
	By \cite[Proposition 5.5.7.6]{HTT}, this is also a limit diagram in $\Pr^\mathrm{R,\omega}$. Taking adjoints, we get
	\[
	\Mod_R(\Sp) \simeq \colim_{n\in\Delta^\op} \Mod_{R'_n}(\Sp)
	\]
	in $\Pr^\mathrm{L,\omega}$, whence in $\Cat_R^\cg$.
	 As $\Mod_R(\Sp)$ is compact in $\Cat_R^\cg$, it is a retract of $\colim_{n\in\Delta_{\leq k}^\op}\Mod_{R'_n}(\Sp)$ for some $k$.
	Let $\scr C\in\Cat_R^\cg$ be such that $\scr C\tens_{R}R'$ is compact in $\Cat_{R'}^\cg$. Then each $\scr C\tens_{R}R'_n$ is compact in $\Cat_R^\cg$, since the pushforward functors $\Cat_{R'_n}^\cg\to\Cat_R^\cg$ have colimit-preserving right adjoints. As $\scr C$ is a retract of a finite colimit of such objects, it is also compact.
	Similarly, if $\scr C\tens_RR'$ has a compact generator, then $\scr C$ is a retract of a finite colimit of $R$-linear $\infty$-categories having a compact generator, and hence it has a compact generator.
\end{proof}

\begin{lemma}\label{lem:qcat-etale-descent}
	The functors
	\begin{align*}
	\Aff^\op \to \CAlg(\Cat_{(\infty,2)}),&\quad R\mapsto \Cat_R^\St,\\
	\Aff^\op \to \CAlg(\Cat_{(\infty,2)}),&\quad R\mapsto \Cat_R^\cg,\\
	\Aff^\op \to \CAlg(\Cat_{(\infty,2)}),&\quad R\mapsto \Cat_R^\ft\text{ and }R\mapsto \Cat_R^\fp,
	\end{align*}
	are sheaves for the fppf, étale, and finite étale topologies, respectively.
\end{lemma}

\begin{proof}
	By \cite[Theorem D.3.6.2 and Proposition D.3.3.1]{SAG}, the functor $R\mapsto \Cat_R^\St$ is an fppf sheaf of $(\infty,1)$-categories. It remains to show that, given $\scr C,\scr D\in \Cat_R^\St$, the functor
	\[
	(\Aff_{/\Spec R})^\op \to \Cat_R^\St,\quad R'\mapsto \Fun_{R}^\mathrm{L}(\scr C,\scr D\tens_RR')
	\]
	is an fppf sheaf. Since $\Fun_R^\mathrm{L}(\scr C,\ph)$ preserves limits, this follows from the fact that $R'\mapsto\scr D\tens_RR'$ is an fppf sheaf \cite[Theorem D.3.5.2]{SAG}.
	The second statement follows from the first and \cite[Theorem D.5.3.1 and Proposition D.5.2.2]{SAG}.
	For the final statement, we must show that, if $f\colon R\to R'$ is a faithfully flat finite étale morphism, then $f^*\colon \Cat_R^\cg\to \Cat_{R'}^\cg$ reflects compactness and the property of having a compact generator. This follows from Lemma~\ref{lem:local-generators}.
\end{proof}

\begin{remark}
	The presheaf $R\mapsto \Cat_R^\ft$ is in fact a sheaf for the étale topology: see \cite[Theorem 4.7]{ToenAzumaya}. We do not know if $R\mapsto\Cat_R^\fp$ is an étale sheaf.
\end{remark}

By Lemma~\ref{lem:qcat-etale-descent} and Corollary~\ref{cor:automatic-norms},
applied with $\scr C=\Aff$, $t$ the finite étale topology, $m=\fet$, and $\scr D=\CAlg(\Cat_{(\infty,2)})$,
we obtain a functor
\[
\Span(\Aff,\all,\fet)\to\CAlg(\Cat_{(\infty,2)}),\quad R\mapsto \Cat_R^\St,
\]
and subfunctors $R\mapsto \Cat_R^\cg$, $R\mapsto\Cat_R^\ft$, and $R\mapsto \Cat_R^\fp$. In particular, for every finite étale map $p\colon R\to R'$, we have a norm functor
\[
p_\otimes\colon \Cat_{R'}^\St \to \Cat_R^\St
\]
that preserves the subcategories $\Cat^\cg$, $\Cat^\ft$, and $\Cat^\fp$.
Being a symmetric monoidal $(\infty,2)$-functor, $p_\otimes$ preserves adjunctions, dualizability, smoothness, and properness.
Note however that $p_\otimes$ does not preserve exact sequences of linear $\infty$-categories (unless $p$ is the identity).

\begin{remark}\label{rmk:TabuadaWeil}
	If $L/k$ is a finite separable extension of fields, the norm functor $\Cat_L^\cg\to \Cat_k^\cg$ extends the Weil transfer of dg-algebras constructed by Tabuada \cite[\sect6.1]{TabuadaWeil}.
\end{remark}

\subsection{Noncommutative motivic spectra and homotopy \texorpdfstring{$\K$}{K}-theory}
\label{sub:nc-motives}

We now recall the construction of $\SH_\nc(S)$ from \cite[\sect6.4.2]{RobaloThesis}. As in \emph{op.\ cit.}, we will first give the definition for $S$ affine and only extend it to more general schemes at the end.
We set 
 \[
 \SmNC_R = \Cat_R^{\fp,\op} \in \CAlg(\Cat_\infty).
 \]
Then $\SmNC_R$ is a small idempotent complete semiadditive $\infty$-category with finite limits, and presheaves on $\SmNC_R$ are functors $\Cat_R^\cg\to\scr S$ that preserve filtered colimits. We will sometimes denote by
\[
\r\colon \Cat_R^{\fp,\op} \into \PSh(\SmNC_R), \quad \r(\scr A)(\scr B)=\Fun_R(\scr A^\omega,\scr B^\omega)^{\simeq},
\]
the Yoneda embedding, viewed as a contravariant functor on $\Cat_R^\fp$.

Consider the following conditions on a presheaf $F\in\PSh(\SmNC_R)$:
\begin{itemize}
	\item \emph{excision}: $F(0)$ is contractible and $F$ takes excision squares in $\Cat_R^\fp$ to pullback squares;
	\item \emph{weak excision}: $F$ preserves finite products, and for every excision square~\eqref{eqn:excision} in $\Cat_R^\fp$, the induced map $\fib(F(f))\to \fib(F(g))$ is an equivalence; 
	\item \emph{$\A^1$-invariance}: for every $\scr A\in\Cat_R^\fp$, the map $F(\scr A) \to F(\scr A \tens_{R}\QCoh(\A^1_R))$ is an equivalence.
\end{itemize}
We denote by $\PSh_\exc(\SmNC_R)$ (resp.\ by $\PSh_\wexc(\SmNC_R)$; by $\H_\nc(\Spec R)$; by $\H_\wnc(\Spec R)$)
the reflective subcategory of $\PSh(\SmNC_R)$ spanned by the presheaves satisfying excision (resp.\ weak excision; excision and $\A^1$-invariance; weak excision and $\A^1$-invariance). Note that representable presheaves satisfy excision and that there are inclusions
\[
\PSh_\exc(\SmNC_R)\subset \PSh_\wexc(\SmNC_R)\subset \PSh_\Sigma(\SmNC_R).
\]
Note also that for a presheaf on $\SmNC_R$ with values in a \emph{stable} $\infty$-category, excision and weak excision are equivalent conditions.
A morphism in $\PSh(\SmNC_R)$ will be called an \emph{excisive equivalence} (resp.\ a \emph{weakly excisive equivalence}; a \emph{motivic equivalence}; a \emph{weakly motivic equivalence}) if its reflection in $\PSh_\exc(\SmNC_R)$ (resp.\ in $\PSh_\wexc(\SmNC_R)$; in $\H_\nc(\Spec R)$; in $\H_\wnc(\Spec R)$) is an equivalence.
By \cite[Proposition 3.19]{Robalo}, each of these localizations of $\PSh(\SmNC_R)$ is compatible with the Day convolution symmetric monoidal structure.

Let $\SmAff_R\subset\Sm_R$ denote the subcategory of smooth affine $R$-schemes.
By Lemma~\ref{lem:smoothaffine}, for every $U\in\SmAff_R$, $\QCoh(U)$ belongs to $\Cat_R^\fp$.
By \cite[Corollary 9.4.2.3]{SAG}, the functor
\[
\QCoh\colon \Sch_R^\op \to \CAlg(\Cat_R^\St) 
\]
preserves finite colimits. Since the coproduct in $\CAlg(\Cat_R^\St)$ is the tensor product in $\Cat_R^\St$, we obtain a symmetric monoidal functor $\QCoh\colon \SmAff_R \to \SmNC_R$. The induced symmetric monoidal functor on $\PSh_\Sigma$ lifts uniquely to pointed objects since the target is pointed \cite[Proposition 4.8.2.11]{HA}, and we obtain a colimit-preserving symmetric monoidal functor
\begin{equation}\label{eqn:QCoh0}
\scr L\colon \PSh_\Sigma(\SmAff_R)_\pt \to \PSh_\Sigma(\SmNC_R).
\end{equation}
Since Nisnevich squares in $\SmAff_R$ generate the Nisnevich topology (see Proposition~\ref{prop:nisnevich-definition}), it follows from Lemma~\ref{lem:voevodsky} that the functor $\scr L$ preserves motivic equivalences and induces
\[
\scr L\colon \H_\pt(\Spec R) \to \H_\nc(\Spec R).
\]
Let $\SH_\nc(\Spec R)$ be the presentably symmetric monoidal $\infty$-category obtained from $\H_\nc(\Spec R)$ by inverting $\scr L(\S^{\A^1})$. We then obtain a colimit-preserving symmetric monoidal functor
\[
\scr L\colon \SH(\Spec R) \to \SH_\nc(\Spec R).
\]
Both $\SH$ and $\SH_\nc$ are sheaves for the Zariski topology on affine schemes \cite[Proposition 9.2.1]{RobaloThesis}.
We can therefore extend $\SH_\nc$ and $\scr L$ to all schemes by right Kan extension.

\begin{lemma}\label{lem:cofib-section}
	Let $\scr C$ be a pointed $\infty$-category and let $f\colon A\to B$ be a morphism in $\scr C$ with a section $s\colon B\to A$.
	If the cofibers of $f$ and $s$ exist, then $\cofib(f)\simeq \Sigma\cofib(s)$.
\end{lemma}

\begin{proof}
	Form the pushout squares
\begin{tikzmath}
	\diagram[column sep={4.5em,between origins}]{
	B & 0 & \\
	A & \cofib(s) & 0 \\
	B & C & \cofib(f)\rlap. \\
	};
	\arrows (11-) edge (-12) (11) edge node[left]{$s$} (21) (21-) edge (-22) (12) edge (22) (21) edge node[left]{$f$} (31) (31-) edge (-32) (22) edge (32) (22-) edge (-23) (23) edge (33) (32-) edge (-33);
\end{tikzmath}	
Since $f\circ s$ is an equivalence, $C$ is contractible.
\end{proof}

\begin{lemma}\label{lem:nc-stability}
	\leavevmode
	\begin{enumerate}
		\item The functors $\H_\wnc(\Spec R)\to \H_\nc(\Spec R)\to \SH_\nc(\Spec R)$ induce symmetric monoidal equivalences
	\[
	\Sp(\H_\wnc(\Spec R))\simeq \Sp(\H_\nc(\Spec R)) \simeq \SH_\nc(\Spec R).
	\]
	\item The presentably symmetric monoidal $\infty$-category $\SH_\nc(\Spec R)$ is obtained from $\H_\wnc(\Spec R)$ by inverting $\scr L(\A^1/\G_m)$.
	\end{enumerate}
\end{lemma}

\begin{proof}
	(1) The first equivalence holds because excision and weak excision are equivalent for presheaves of spectra. It remains to show that $\scr L(\S^{\A^1})$ is already invertible in $\Sp(\H_\nc(\Spec R))$.
	Let $\pi\colon \P^1_R\to\Spec R$ be the structure map. The exact sequence
	\[
	\Mod_R \into \QCoh(\P^1_R) \xrightarrow{\pi_*} \Mod_R
	\]
	in $\Cat_R^\fp$ becomes a cofiber sequence
	\[
	\r\Mod_R \xrightarrow{\r(\pi_*)} \r\QCoh(\P^1_R) \to \r\Mod_R
	\]
	 in $\PSh_\wexc(\SmNC_R)$. Let $\infty\colon\Spec R\into\P^1_R$ be the section of $\pi$ at infinity. Then $\r(\pi_*)$ and $\r(\infty^*)$ are both sections of $\r(\pi^*)$. By Lemma~\ref{lem:cofib-section}, two sections of the same map in a stable $\infty$-category have equivalent cofibers, so we obtain a cofiber sequence
	\[
	\Sigma^\infty \r\Mod_R \xrightarrow{\r(\infty^*)} \Sigma^\infty \r\QCoh(\P^1_R) \to \Sigma^\infty \r\Mod_R
	\]
	in $\Sp(\PSh_\exc(\SmNC_R))$.
	On the other hand, the standard affine cover of $\P^1_R$ gives an equivalence $\r\QCoh(\P^1_R) \simeq \scr L(\P^1_{R+})$ in $\PSh_\exc(\SmNC_R)$.
	From the equivalence $(\P^1_R,\infty)\simeq \S^{\A^1}$ in $\H_\pt(\Spec R)$, we then deduce that $\scr L(\S^{\A^1})$ is equivalent to the unit in $\Sp(\H_\nc(\Spec R))$. In particular, it is invertible.
	
	(2) Since $\scr L(\A^1/\G_m)\simeq \S^1\wedge \scr L(\G_m,1)$ in $\H_\wnc(\Spec R)$, the formal inversion of $\scr L(\A^1/\G_m)$ is stable and we conclude by (1).
\end{proof}

\begin{remark}
	It is mistakenly claimed in \cite[Proposition 3.24]{Robalo} that $\H_\nc(\Spec R)$ is already stable. The issue is in the proof of \cite[Lemma 3.25]{Robalo}, where it is claimed that $\scr L(\P^1_R,\infty)\in\PSh_\exc(\SmNC_R)$ is equivalent to the unit; this only holds after one suspension.
\end{remark}

Consider the symmetric monoidal functor
\begin{equation*}
\SmAff_{R+} \to \SmNC_R,\quad U_+\mapsto \QCoh(U),
\end{equation*}
which is the restriction of~\eqref{eqn:QCoh0}. By Lemma~\ref{lem:qcat-etale-descent} and Corollary~\ref{cor:automatic-norms}, applied with $\scr C=\Aff$, $t$ the finite étale topology, $m=\fet$, and $\scr D=\Fun(\Delta^1,\Cat_\infty)$, we can promote it to a natural transformation
\[
\SmAff_+^\otimes \to \SmNC^\otimes\colon \Span(\Aff,\all,\fet) \to \CAlg(\Cat_\infty).
\]
We can view this transformation as a functor
\[
\Span(\Aff,\all,\fet)\times\Delta^1 \to \CAlg(\Cat_\infty).
\]
Composing with $\PSh_\Sigma$, we get
\begin{equation}\label{eqn:LPsigma}
\Span(\Aff,\all,\fet)\times\Delta^1 \to \CAlg(\what\Cat{}_\infty^\mathrm{sift}),\quad (S,0\to1)\mapsto (\PSh_\Sigma(\SmAff_S)_\pt \to \PSh_\Sigma(\SmNC_S)).
\end{equation}
To go further, we need to investigate the effect of norms on excisive equivalences. This will require noncommutative versions of some of the results from Section~\ref{sec:pointednorms}.

Let $p\colon R\to R'$ be a finite étale morphism and consider an exact sequence $\scr K\into\scr A\to\scr B$ in $\Cat_{R'}^\St$. Then we have two induced exact sequences
\begin{gather*}
	p_\otimes\scr K \into p_\otimes\scr A \to \scr B',\\
	\scr K'\into p_\otimes\scr A \to p_\otimes\scr B
\end{gather*}
in $\Cat_R^\St$. More generally, if $n$ is a locally constant integer on $\Spec(R)$, we define \[p_\otimes(\scr A|n\scr B) \in \Cat_R^\St\] to be the quotient of $p_\otimes\scr A$ by the subcategory generated,
locally in the finite étale topology, by products $\bigotimes_i a_i$ with fewer than $n$ factors not in $\scr K$. It is a monogenic localization if $\scr A\to\scr B$ is monogenic and $\scr A\in\Cat_{R'}^\ft$ (this can be checked when $p$ is a fold map, by Lemma~\ref{lem:local-generators}(1)). For example, $\scr B'=p_\otimes(\scr A|\scr B)$ and $p_\otimes\scr B=p_\otimes(\scr A|d\scr B)$ where $d=\deg(p)$. Even more generally, we define
\[
p_\otimes(\scr A|n_1\scr B_1,\dotsc,n_k\scr B_k)=p_\otimes(\scr A|n_1\scr B_1)\cap\dotsb\cap p_\otimes(\scr A|n_k\scr B_k)
\]
as a reflective subcategory of $p_\otimes\scr A$. 

If $\scr A=\QCoh(X)$ and $\scr B_i=\QCoh(Y_i)$ for some quasi-projective $R'$-scheme $X$ and quasi-compact open subschemes $Y_i\subset X$, then $p_\otimes(\scr A)\simeq \QCoh(\Weil_pX)$ and $p_\otimes(\scr A|n_1\scr B_1,\dotsc,n_k\scr B_k)$ is the monogenic localization corresponding to the open subscheme of $\Weil_pX$ classifying $R'$-morphisms $s\colon U_{R'}\to X$ such that the fibers of $s^{-1}(Y_i)\to U$ have degree at least $n_i$. The following lemma is thus a noncommutative analog of Lemma~\ref{lem:crossterms}.

\begin{lemma}\label{lem:crossterms-cat}
	Let $p\colon R\to R'$ be a finite étale map, let $\scr A\in\Cat_{R'}^\St$, and for each $1\leq i\leq k$ let $\scr A\to\scr B_i$ be a localization functor and $n_i$ a locally constant integer on $\Spec(R)$. For every coproduct decomposition $\scr A=\scr A'\times\scr A''$, there is a coproduct decomposition
	\[
	p_\otimes(\scr A|n_1\scr B_1,\dotsc,n_k\scr B_k) = p_\otimes(\scr A'|n_1\scr B_1',\dotsc,n_k\scr B_k')\times p_\otimes(\scr A|\scr A'',n_1\scr B_1,\dotsc,n_k\scr B_k),
	\]
	where $\scr B_i'=\scr B_i\cap \scr A'$.
\end{lemma}

\begin{proof}
	By finite étale descent, this may be checked when $p$ is a fold map, and in this case the result is obvious.
\end{proof}

With $\scr A=\QCoh(X)$ and $\scr B_i=\QCoh(Y_i)$ as above, suppose moreover that $p$ is a fold map of degree $d$. Then $X=X_1\amalg\dotsb\amalg X_d$ and $\Weil_pX=X_1\times_R\dotsb\times_R X_d$. In this case the open subscheme of $\Weil_pX$ corresponding to $p_\otimes(\scr A|n_1\scr B_1,\dotsc,n_k\scr B_k)$ is the union of all the rectangles $U_1\times_R\dotsb\times_R U_d$, where $U_j=\bigcap_{i\in I_j} Y_i\times_XX_j\subset X_j$ for some subsets $I_j\subset\{1,\dotsc,k\}$ such that $\{j\suchthat i\in I_j\}$ has at least $n_i$ elements. The following lemma is the noncommutative analog of this observation.

\begin{lemma}\label{lem:crossterms-as-limit}
	Under the assumptions of Lemma~\ref{lem:crossterms-cat}, suppose that $p$ is a fold map of degree $d$. Accordingly, write 
	$\scr A=\scr A_1\times\dotsb\times\scr A_d$ and $\scr B_i=\scr B_{i1}\times\dotsb\times \scr B_{id}$.
	Let $\chi\colon C\to \Cat_R^\St$ be the cube of dimension $kd$ in $\Cat_R^\St$ whose vertices are the tensor products $\scr C_1\otimes_R\dotsb\otimes_R\scr C_d$, where $\scr C_j$ is a reflective subcategory of $\scr A_j$ of the form $\bigcap_{i\in I_j}\scr B_{ij}$ for some subset $I_j\subset\{1,\dotsc,k\}$. Let $C'\subset C$ be the subposet where $\{j\suchthat i\in I_j\}$ has at least $n_i$ elements for all $i$. Then 
	\[
	p_\otimes(\scr A|n_1\scr B_1,\dotsc,n_k\scr B_k)\simeq \lim_{C'}\chi.
	\]
\end{lemma}

\begin{proof}
	Consider the restriction morphism $f\colon p_\otimes(\scr A)=\lim_C\chi \to \lim_{C'}\chi$ in $\Cat_R^\St$. Its right adjoint $g$ has the following explicit description: its value on a section $s$ of $\chi$ over $C'$ is the limit of $s$ viewed as a $C'$-indexed diagram in $p_\otimes(\scr A)$. Since $f$ preserves finite limits, we see by inspection that the counit transformation $f\circ g\to\id$ is an equivalence. Hence $f$ is a localization functor, and again by inspection its kernel is as desired.
\end{proof}

\begin{proposition}\label{prop:ncequiv}
	Let $p\colon R\to R'$ be a finite étale morphism. Then the functor
	\[
	p_\otimes\colon \PSh_\Sigma(\SmNC_{R'}) \to \PSh_\Sigma(\SmNC_R)
	\]
	sends weakly excisive (resp.\ weakly motivic) equivalences to excisive (resp.\ motivic) equivalences.
\end{proposition}

\begin{proof}
	It is clear that $p_\otimes$ preserves $\A^1$-homotopy equivalences. 
	By Lemma~\ref{lem:voevodsky}, it remains to show that $p_\otimes$ sends a generating set of weakly excisive equivalences to excisive equivalences.
	Consider an excision square
	\begin{tikzmath}
		\diagram{\scr A & \scr B \\ \scr C & \scr D \\};
		\arrows (11-) edge (-12) (11) edge (21) (21-) edge (-22) (12) edge (22);
	\end{tikzmath}
	in $\Cat_{R'}^\fp$, and the corresponding weakly excisive equivalence $\r\scr C/\r\scr D\to \r\scr A/\r\scr B$ in $\PSh_\Sigma(\SmNC_{R'})$.
	We must show that $p_\otimes(\r\scr C/\r\scr D)\to p_\otimes(\r\scr A/\r\scr B)$ is an excisive equivalence in $\PSh_\Sigma(\SmNC_R)$. 
	
	We claim that the square
	\begin{tikzequation}\label{eqn:normed-square}
		\diagram{p_\otimes\scr A & p_\otimes(\scr A|\scr B) \\ p_\otimes\scr C & p_\otimes(\scr C|\scr D) \\};
		\arrows (11-) edge (-12) (11) edge (21) (21-) edge (-22) (12) edge (22);
	\end{tikzequation}
	is an excision square. Since the horizontal maps are monogenic localizations, it suffices to check that it is a pullback square in $\Cat_R^\St$, and by Lemma~\ref{lem:qcat-etale-descent} we may assume that $p$ is a fold map. In this case, we can use Lemma~\ref{lem:crossterms-as-limit} to write $p_\otimes(\scr A|\scr B)$ and $p_\otimes(\scr C|\scr D)$ as explicit limits of tensor products. The claim is then a formal computation, using the fact that if $\scr E\in\Cat_R^\cg$ then $\scr E\otimes (\ph)$ preserves limits in $\Cat_R^\St$ (since $\scr E$ is dualizable).
	 Note also that the square~\eqref{eqn:normed-square} belongs to $\Cat_R^\fp$, by Lemma~\ref{lem:openimmersion}. It therefore induces a weakly excisive equivalence $\r p_\otimes(\scr C)/\r p_\otimes(\scr C|\scr D)\to \r p_\otimes(\scr A)/\r p_\otimes(\scr A|\scr B)$. To conclude the proof, we will produce a natural map
	\begin{equation}\label{eqn:ncequiv}
	p_\otimes(\r\scr A/\r\scr B) \to \r p_\otimes(\scr A)/\r p_\otimes(\scr A|\scr B)
	\end{equation}
	in $\PSh_\Sigma(\SmNC_R)$ and prove that it is an excisive equivalence.
	
	By Lemma~\ref{lem:bar-construction}, the quotient $\r\scr A/\r\scr B$ may be computed as the colimit of a simplicial object $X_\bullet$ in $\SmNC_{R'}$ with $X_n=\scr A\times\scr B^{\times n}$. Similarly, the quotient $\r p_\otimes(\scr A)/\r p_\otimes(\scr A|\scr B)$ may be computed as the colimit of a simplicial object $X_\bullet'$ in $\SmNC_R$ with $X_n'=p_\otimes\scr A\times p_\otimes(\scr A|\scr B)^{\times n}$. Using Lemma~\ref{lem:crossterms-cat} repeatedly, we get a coproduct decomposition
	\[
	p_\otimes X_n=p_\otimes(\scr A\times \scr B^{\times n}) = p_\otimes\scr A \times p_\otimes(\scr A\times\scr B|\scr B)\times\dotsb\times p_\otimes(\scr A\times\scr B^{\times n}|\scr B).
	\]
	By sending the term $p_\otimes(\scr A\times\scr B^{\times i}|\scr B)$ to the $i$th copy of $p_\otimes(\scr A|\scr B)$ in $X_n'$, we obtain a simplicial morphism $p_\otimes X_\bullet \to X_\bullet'$ which induces~\eqref{eqn:ncequiv} in the colimit. 
	
	Consider the simplicial square
	\begin{tikzmath}
		\def\colsep{1.5em}
		\diagram{\r p_\otimes(\scr A\times\scr B^{\times \bullet}|\scr B^{\times \bullet+1}) & \r p_\otimes(\scr A\times\scr B^{\times \bullet}) \\ \r(p_\otimes(\scr A|\scr B) \times p_\otimes(\scr A|\scr B)^{\times\bullet}) & \r(p_\otimes\scr A\times p_\otimes(\scr A|\scr B)^{\times\bullet})\rlap, \\};
		\arrows (11-) edge[c->] (-12) (11) edge (21) (21-)edge[c->] (-22) (12) edge (22);
	\end{tikzmath}
	where the right vertical map is $p_\otimes X_\bullet \to X_\bullet'$. This square is degreewise a pushout in $\PSh_\Sigma(\SmNC_R)$, by Lemma~\ref{lem:crossterms-cat}, and the lower left corner has an evident extra degeneracy. It therefore remains to show that the colimit of the upper left corner in $\PSh_\exc(\SmNC_R)$ is contractible. We will prove more generally that the simplicial object
	\[
	Y_\bullet(k)=\r p_\otimes(\scr A\times\scr B^{\times \bullet}|k\scr B^{\times \bullet+1})
	\]
	has contractible colimit in $\PSh_\exc(\SmNC_R)$ for every integer $k\geq 1$, by descending induction on $k$. If $k$ is greater than the degree of $p$, then $Y_n(k)=\r0$ for all $n$ and the claim holds. For $k\geq 1$, consider the simplicial square
	\begin{tikzequation}\label{eqn:excisive-induction}
		\def\colsep{1.5em}
		\diagram{\r p_\otimes(\scr A\times\scr B\times\scr B^{\times \bullet}|k\scr B^{\times \bullet+1},(k+1)\scr B^{\times\bullet+2}) & \r p_\otimes(\scr A\times\scr B\times\scr B^{\times \bullet}|k\scr B^{\times \bullet+1}) \\ \r p_\otimes(\scr A\times\scr B^{\times \bullet}|(k+1)\scr B^{\times \bullet+1}) & \r p_\otimes(\scr A\times\scr B^{\times \bullet}|k\scr B^{\times \bullet+1})\rlap. \\};
		\arrows (11-) edge[c->] (-12) (11) edge (21) (21-)edge[c->] (-22) (12) edge (22);
	\end{tikzequation}
	Here the simplicial objects in the first row are obtained from the simplicial diagram of $R'$-linear $\infty$-categories and right adjoint functors
	\begin{tikzmath}
		\def\colsep{.9em}
		\diagram{
		\dotsb & \scr A\times\scr B\times\scr B\times\scr B & \scr A\times\scr B\times\scr B & \scr A\times\scr B\rlap, \\
		};
		\arrows
		(11-) edge[-top,vshift=3*\dbl] (-12) edge[-mid,vshift=\dbl] (-12) edge[-mid,vshift=-\dbl] (-12) edge[-bot,vshift=-3*\dbl] (-12)
		(12-) edge[-top,vshift=2*\dbl] (-13) edge[-mid] (-13) edge[-bot,vshift=-2*\dbl] (-13)
		(13-) edge[-top,vshift=\dbl] (-14) edge[-bot,vshift=-\dbl] (-14);
	\end{tikzmath}
	which has an obvious augmentation to $\scr A$ with an extra degeneracy.
	The upper right corner of~\eqref{eqn:excisive-induction} has an induced augmentation to $\r p_\otimes(\scr A|k\scr B^{\times 0})$ with an extra degeneracy; since $k\geq 1$, $p_\otimes(\scr A|k\scr B^{\times 0})=0$, hence the colimit of this simplicial diagram is contractible in $\PSh_\exc(\SmNC_R)$. For the same reason, the colimit of the upper left corner of~\eqref{eqn:excisive-induction} is contractible in $\PSh_\exc(\SmNC_R)$. By the induction hypothesis, the colimit of the lower left corner of~\eqref{eqn:excisive-induction} is contractible in $\PSh_\exc(\SmNC_R)$.
	To conclude, it suffices to show that~\eqref{eqn:excisive-induction} is degreewise an excision square.
	We only have to check that the square
	\begin{tikzmath}
		\def\colsep{.9em}
		\diagram{p_\otimes(\scr A\times\scr B^{\times n}|k\scr B^{\times n+1}) & p_\otimes(\scr A\times\scr B^{\times n}|(k+1)\scr B^{\times n+1})\\
		p_\otimes(\scr A\times\scr B\times\scr B^{\times n}|k\scr B^{\times n+1}) & p_\otimes(\scr A\times\scr B\times\scr B^{\times n}|k\scr B^{\times n+1},(k+1)\scr B^{\times n+2}) \\};
		\arrows (11-) edge (-12) (11) edge (21) (21-)edge (-22) (12) edge (22);
	\end{tikzmath}
	is cartesian in $\Cat_R^\St$, and we may assume that $p$ is a fold map. This is again formal once we express each object as a limit of tensor products using Lemma~\ref{lem:crossterms-as-limit}.
\end{proof}

\begin{remark}
	We do not know if $p_\otimes\colon \PSh_\Sigma(\SmNC_{R'})\to\PSh_\Sigma(\SmNC_R)$ preserves excisive equivalences.
\end{remark}

It follows from Proposition~\ref{prop:ncequiv} that the functor $p_\otimes\colon \PSh_\Sigma(\SmNC_{R'}) \to \PSh_\Sigma(\SmNC_R)$ induces
\[
p_\otimes\colon \H_\wnc(\Spec R') \to \H_\nc(\Spec R).
\]
Note that the inclusion $\H_\wnc(\Spec R')\subset \PSh_\Sigma(\SmNC_{R'})$ preserves filtered colimits, so that $\H_\wnc(\Spec R')$ is compactly generated and $\scr L(\A^1/\G_m)\in \H_\wnc(\Spec R')$ is compact.
Since $p_\otimes\scr L(\A^1/\G_m)\simeq \scr Lp_\otimes(\A^1/\G_m)$ is invertible in $\SH_\nc(\Spec R)$, the functor $p_\otimes$ lifts uniquely to a symmetric monoidal functor
\begin{equation}\label{eqn:normSHnc}
p_\otimes\colon \SH_\nc(\Spec R') \to \SH_\nc(\Spec R),
\end{equation}
by Lemma~\ref{lem:inversion} and Lemma~\ref{lem:nc-stability}(2).

Because we do not have norms on $\H_\nc$, we have to proceed slightly differently than in \sect\ref{sub:coherence} to construct $\SH_\nc^\otimes$. Recall that $\SH(\Spec R)$ can be obtained from $\PSh_\Sigma(\SmAff_{R})_\pt$ by performing two universal constructions within the $\infty$-category $\CAlg(\what\Cat{}_\infty^\mathrm{sift})$: first invert motivic equivalences and then invert the motivic spheres $\S^{V}$. The same steps yield $\SH_\nc(\Spec R)$ from $\PSh_\Sigma(\SmNC_R)$. Note that it is formally equivalent to first invert $p_\otimes(\A^1/\G_m)$ for all finite étale maps $p\colon R\to R'$ and then invert all maps that become invertible in $\SH(\Spec R)$, both steps being understood as universal constructions in $\CAlg(\what\Cat{}_\infty^\mathrm{sift})$. On the noncommutative side, the second step is now compatible with norms by the existence of the functors~\eqref{eqn:normSHnc}. From~\eqref{eqn:LPsigma}, we thus obtain a functor
\[
\Span(\Aff,\all,\fet)\times\Delta^1 \to \CAlg(\what\Cat{}_\infty^\mathrm{sift}),\quad (S,0\to1)\mapsto (\SH(S) \to \SH_\nc(S)),
\]
or equivalently a natural transformation
\[
\scr L\colon \SH^\otimes\to\SH_\nc^\otimes\colon \Span(\Aff,\all,\fet) \to \CAlg(\what\Cat{}_\infty^\mathrm{sift}).
\]
Finally, using Proposition~\ref{prop:span-RKE}, we extend this transformation to
\[
\scr L\colon \SH^\otimes\to\SH_\nc^\otimes\colon \Span(\Sch,\all,\fet) \to \CAlg(\what\Cat{}_\infty^\mathrm{sift}).
\]

\begin{theorem}\label{thm:KGL}
	The assignment $S\mapsto\KGL_S\in\SH(S)$ can be promoted to a section of $\SH^\otimes$ over $\Span(\Sch,\all,\fet)$ that is cocartesian over $\Sch^\op$. In particular, for every scheme $S$, the homotopy $\K$-theory spectrum $\KGL_S$ is a normed spectrum over $\Sch_S$.
\end{theorem}

\begin{proof}
	Consider $\scr L$ as a map of cocartesian fibrations over $\Span(\Sch,\all,\fet)$.
	By Lemma~\ref{lemm:construct-relative-adjoint}(1), it admits a relative right adjoint $u_\nc$, given fiberwise by the forgetful functor $\SH_\nc(S)\to\SH(S)$, that sends the unit to $\KGL_S$ \cite[Theorem 4.7]{Robalo}.
	Hence, composing the unit section of $\SH_\nc^\otimes$ with $u_\nc$, we obtain the desired section of $\SH^\otimes$.
\end{proof}

\begin{corollary}\label{cor:normKH}
	The presheaf $\Omega^\infty\KH\colon \Sch^\op\to\CAlg(\scr S)$ of homotopy $\K$-theory spaces with their multiplicative $\E_\infty$-structures extends to a functor
		\[
		\Omega^\infty\KH\colon \Span(\Sch,\all,\fet) \to \CAlg(\scr S).
		\]
\end{corollary}

\begin{remark}
	Let us mention another approach to the construction of the normed spectrum $\KGL_S$, suggested to us by Akhil Mathew.
	The Picard groupoid functor $\Pic\colon \Sch^\op\to\scr S$ can be extended to $\Span(\Sch,\all,\flf)$ using norms of invertible sheaves.
	If $\Pic_S$ denotes the restriction of $\Pic$ to $\SmQP_S$, it follows that $S\mapsto \Pic_S$ is a section of $\PSh_\Sigma(\SmQP)^\otimes$ over $\Span(\Sch,\all,\flf)$. By Proposition~\ref{prop:integralsifted}, $S\mapsto \L_\mot\Pic_S$ is thus a section of $\H^\otimes$ over $\Span(\Sch,\all,\flf)$, which is cocartesian over $\Sch^\op$ since the map $\P^\infty_S\to \Pic_S$ classifying $\scr O(1)$ is a motivic equivalence.
	In particular, $\Sigma^\infty_+\L_\mot\Pic_S$ is a normed spectrum over $\Sch_S$.
	By \cite[Theorem 4.17]{GepnerSnaith} or \cite[Theorem 1.1]{SpitzweckOstvaer}, the canonical map $\Sigma^\infty_+\L_\mot\Pic_S\to \KGL_S$ induces an equivalence
	\[
	\Sigma^\infty_+\L_\mot\Pic_S[1/\beta] \simeq \KGL_S
	\]
	in $\SH(S)$, where $\beta\colon \S^{\A^1} \to \Sigma^\infty_+\L_\mot\Pic_S$ is the difference of the maps $\P^1_S\to\Pic_S$ classifying $\scr O(-1)$ and $\scr O$. Using Proposition~\ref{prop:localization-general}(2), $\Sigma^\infty_+\L_\mot\Pic_S[1/\beta]$ will be a normed spectrum over $\Sch_S$ provided that, for every finite étale map $f\colon X\to Y$, $\nu_f(\beta)$ becomes invertible in the localized $\Pic(\SH(Y))$-graded commutative ring $\pi_\star(\Sigma^\infty_+\L_\mot\Pic_Y)[1/\beta]\simeq \pi_\star(\KGL_Y)$. 
	Our construction of norms on $\KGL_S$ shows that this must be the case, as $\Sigma^\infty_+\L_\mot\Pic_S\to \KGL_S$ is easily seen to be a morphism of normed spectra, but it might be possible to check this more directly.
\end{remark}

\subsection{Nonconnective \texorpdfstring{$\K$}{K}-theory}

We conclude this section by showing how to obtain norms on ordinary $\K$-theory using a variant of the previous construction.
Such norms are also constructed in \cite{BDGNS9} using different methods, although our proof of Lemma~\ref{lem:stable-norm-nc} below ultimately relies on the existence of norms for spectral Mackey functors. 
Let $\SH_\exc\colon\< \Sch^\op\to \CAlg(\what\Cat{}_\infty)$ denote the right Kan extension of the functor
\[
\Aff^\op \to\CAlg(\what\Cat{}_\infty),\quad R\mapsto \PSh_\exc(\SmNC_R,\Sp).
\]
The spectrum of endomorphisms of the unit in $\SH_\exc(S)$ is then the nonconnective $\K$-theory spectrum of $S$ \cite[Theorem 4.4]{Robalo}. 
The key difference with the previous situation is that we now have to consider presheaves of \emph{spectra}, so that
Proposition~\ref{prop:ncequiv} does not suffice to extend $\SH_\exc$ to $\Span(\Sch,\all,\fet)$.
For this we need the following result.

\begin{lemma}\label{lem:stable-norm-nc}
	Let $p\colon R\to R'$ be a finite étale morphism. Then the symmetric monoidal functor
	\[
	\Sigma^\infty p_\otimes\colon \PSh_\Sigma(\SmNC_{R'}) \to \PSh_\Sigma(\SmNC_R,\Sp)
	\]
	sends $\S^1$ to an invertible object and preserves excisive equivalences.
\end{lemma}

\begin{proof}
	The second statement follows from the first and Proposition~\ref{prop:ncequiv}, since weak excision implies excision for presheaves of spectra.
	The functor 
	\[\FEt_R^\op\to\Cat_R^\fp,\quad R'\mapsto \Mod_{R'}(\Sp)\]
	is a finite étale sheaf. Since $\Cat_R^\fp$ is semiadditive, Corollary~\ref{cor:automatic-norms} implies that it lifts uniquely to a functor $\Span(\FEt_R)\to \SmNC_R$. Hence, we get a symmetric monoidal functor
	\[
	\PSh_\Sigma(\Span(\FEt_R),\Sp) \to \PSh_\Sigma(\SmNC_R,\Sp).
	\]
	The commutative square
	\begin{tikzmath}
		\diagram{\FEt_{R'} & \SmNC_{R'} \\ \FEt_R & \SmNC_R \\};
		\arrows (11-) edge (-12) (11) edge node[left]{$p_*$} (21) (21-) edge (-22) (12) edge node[right]{$p_\otimes$} (22);
	\end{tikzmath}
	similarly induces a commutative square of symmetric monoidal functors
	\begin{tikzmath}
		\diagram{\PSh_\Sigma(\Span(\FEt_{R'})) & \PSh_\Sigma(\SmNC_{R'}) \\ \PSh_\Sigma(\Span(\FEt_R),\Sp) & \PSh_\Sigma(\SmNC_R,\Sp)\rlap. \\};
		\arrows (11-) edge (-12) (11) edge (21) (21-) edge (-22) (12) edge node[right]{$\Sigma^\infty p_\otimes$} (22);
	\end{tikzmath}
	It therefore suffices to show that the left vertical functor sends $\S^1$ to an invertible object, but this follows from Proposition~\ref{prop:spectralMackey} (more precisely, from the fact that the equivalence of Proposition~\ref{prop:spectralMackey} is compatible with norms, as explained in Remark~\ref{rmk:spectralMackey}).
\end{proof}

Using Lemma~\ref{lem:stable-norm-nc}, we can construct as before the functor
\[
\SH_\exc^\otimes\colon \Span(\Sch,\all,\fet)\to \CAlg(\what\Cat{}_\infty^\mathrm{sift}).
\]
Taking endomorphisms of the unit, we obtain in particular the following result:

\begin{corollary}\label{cor:normK}
	The presheaf $\K\colon \Sch^\op\to\CAlg(\scr S)$ of $\K$-theory spaces with their multiplicative $\E_\infty$-structures extends to a functor
		\[
		\K\colon \Span(\Sch,\all,\fet) \to \CAlg(\scr S).
		\]
\end{corollary}

For quasi-projective schemes over a field, this construction recovers Joukhovitski's norms on $\K_0$ \cite{Joukhovitski}, which are thereby extended to higher $\K$-theory. Indeed, if $p\colon R\to R'$ is finite étale, the effect of $p_\otimes\colon\Cat_{R'}^\fp\to\Cat_R^\fp$ on the endomorphisms of the unit is the usual norm $\Perf(R')\to\Perf(R)$, as can be checked when $p$ is a fold map.

\begin{remark}\label{rmk:KBnormed}
	The theory of noncommutative motives $\SH_\exc$ can be compared with equivariant homotopy theory using Grothendieck's Galois theory (see Section~\ref{sec:ggt}). In the proof of Lemma~\ref{lem:stable-norm-nc}, we observed that there is a canonical symmetric monoidal functor 
	\[
	\Span(\Fin_{\widehat\Pi_1^\et(\Spec R)}) \simeq \Span(\FEt_R)\to\SmNC_R, 
	\]
	natural in $R$. As both sides are finite étale sheaves in $R$, we obtain by Corollary~\ref{cor:automatic-norms} a natural transformation
	\[
	\Span(\Fin)^\otimes\circ \widehat\Pi_1^\et \to \SmNC^\otimes\colon \Span(\Aff,\all,\fet) \to \CAlg(\Cat_\infty).
	\]
	Using Remark~\ref{rmk:spectralMackey}, this can be extended in the usual steps to a natural transformation
	\[
	\SH^\otimes \circ \widehat\Pi_1^\et \to \SH_\exc^\otimes \colon \Span(\Sch,\all,\fet)\to \CAlg(\what\Cat{}_\infty^\mathrm{sift}).
	\]
	As in Proposition~\ref{prop:normedGspectra}, we then obtain for every scheme $S$ an adjunction
	\[
	\NAlg(\SH(\widehat\Pi_1^\et(S))) \rightleftarrows \NAlg_{\FEt}(\SH_\exc(S)).
	\]
	The right adjoint sends the initial object to a normed $\widehat\Pi_1^\et(S)$-spectrum, whose value on a finite $\widehat\Pi_1^\et(S)$-set is the nonconnective $\K$-theory spectrum of the corresponding finite étale $S$-scheme.
\end{remark}

\begin{remark}
	Let $\Chow_\nc(R)$ denote the full subcategory of $\SH_\exc(R)$ spanned by retracts of motives of smooth and proper $R$-linear $\infty$-categories. If $\scr A$ and $\scr B$ are such $\infty$-categories, the mapping space $\Map(\Sigma^\infty\r\scr A,\Sigma^\infty\r\scr B)$ in $\Chow_\nc(R)$ is the $\K$-theory space $\K(\scr A^\omega\otimes_R\scr B^{\omega,\op})$. Thus, the homotopy category $\h\Chow_\nc(R)$ is the opposite of Kontsevich's category of noncommutative Chow motives over $R$ \cite{TabuadaNChow}.
	 If $p\colon R\to R'$ is finite étale, the norm $p_\otimes\colon \Cat_{R'}^\cg \to \Cat_R^\cg$ preserves smooth and proper $\infty$-categories, since they are precisely the dualizable objects. It follows that $\SH_\exc^\otimes$ admits a full subfunctor
	\[
	\Span(\Sch,\all,\fet) \to \CAlg(\Cat_\infty), \quad S\mapsto \Chow_\nc(S).
	\]
	If $L/k$ is a finite separable extension of fields, the norm functor $\Chow_\nc(L)\to \Chow_\nc(k)$ is an $\infty$-categorical enhancement of Tabuada's Weil transfer \cite[Theorem 2.3]{TabuadaWeil} (see Remark~\ref{rmk:TabuadaWeil}).
\end{remark}

\section{Motivic Thom spectra}
\label{sec:thomspectra}
In this section, we prove that Voevodsky's algebraic cobordism spectrum $\MGL_S$ and related spectra are normed spectra. Let us first recall the definition of the spectrum $\MGL_S$ over a scheme $S$. 
Let
\[\Gr_n=\colim_k\Gr_n(\A^k_S)\in\PSh(\Sm_S)\]
 be the Grassmannian of $n$-planes and let $\gamma_n$ be the tautological vector bundle on $\Gr_n$.
The algebraic cobordism spectrum $\MGL_S\in\SH(S)$ is then defined by
\[
\MGL_S = \colim_n \Sigma^{-\A^n}\Sigma^\infty\Th(\gamma_n),
\]
where the transition maps are given by
\[
\Sigma^{\A^1}\Th(\gamma_n) \simeq \Th(\gamma_{n+1}|\Gr_n) \to \Th(\gamma_{n+1}).
\]

Although it is known that $\MGL_S$ admits an $\E_\infty$-ring structure, the existing constructions rely on specific models for the symmetric monoidal structure on $\SH(S)$ (see for example \cite[Theorem 14.2]{HuSmodules} and \cite[\sect2.1]{Panin:2008}), and they do not obviously generalize to a construction of $\MGL_S$ as a normed spectrum. Our first goal is to give a new description of $\MGL_S$ that makes its $\E_\infty$-ring structure apparent; its normed structure will then be apparent as well. We will show that
\[
\MGL_S\simeq \colim_{(X,\xi)} \Th_X(\xi),
\]
where $X$ ranges over $\Sm_S$ and $\xi\in \K(X)$ has rank $0$ (see Theorem~\ref{thm:MGL}). The fact that $\MGL_S$ is a normed spectrum will then follow formally from the fact that the motivic $\rm J$-homomorphism is a morphism of ``normed spaces''.

\subsection{The motivic Thom spectrum functor}

We will denote by
\[
\Sph(S) = \Pic(\SH(S))
\]
the Picard $\infty$-groupoid of $\SH(S)$, i.e., the subgroupoid of $\SH(S)^\simeq$ spanned by the invertible objects, which is a grouplike $\E_\infty$-space. Since pullbacks and norms preserve invertible objects, the assignment $S\mapsto\Sph(S)$ is a functor on $\Span(\Sch,\all,\fet)$, and it is a cdh sheaf since $\SH(\ph)$ is.

Let $(\Sm_S)_{\sslash\SH}\to \Sm_S$ denote the cartesian fibration classified by $\SH\colon \Sm_S^\op\to\what\Cat{}_\infty$. As the notation suggests, the $\infty$-category $(\Sm_S)_{\sslash\SH}$ can be interpreted as a right-lax slice in the sense of \cite[\sect A.2.5.1]{GRderalg}. 
Let $\PSh((\Sm_S)_{\sslash\SH})$ denote the $\infty$-category of presheaves on $(\Sm_S)_{\sslash\SH}$ that are small colimits of representables.

\begin{definition} \label{def:M_S}
	The \emph{motivic Thom spectrum functor}
	\[
	\M_S\colon \PSh((\Sm_S)_{\sslash\SH}) \to \SH(S)
	\]
	is the colimit-preserving extension of
	\[
	(\Sm_S)_{\sslash\SH} \to \SH(S),\quad (f\colon X\to S, P\in\SH(X))\mapsto f_\sharp P.
	\]
\end{definition}

We will give a formal construction of $\M_S$ in \sect\ref{sub:construct-M}.
The $\infty$-category $\PSh((\Sm_S)_{\sslash\SH})$ contains several subcategories of interest:
\begin{equation}\label{eqn:sourceM}
\PSh(\Sm_S)_{/\Sph} \stackrel{\text{full}}\subset \PSh(\Sm_S)_{/\SH} \stackrel{\text{wide}}\subset \PSh(\Sm_S)_{\sslash\SH} \stackrel{\text{full}}\subset \PSh((\Sm_S)_{\sslash\SH}).
\end{equation}
The functors $\PSh(\Sm_S)_{/\Sph}\into \PSh(\Sm_S)_{/\SH}\into \PSh((\Sm_S)_{\sslash\SH})$ in~\eqref{eqn:sourceM} are the colimit-preserving extensions of the embeddings $(\Sm_S)_{/\Sph} \into (\Sm_S)_{/\SH} \into (\Sm_S)_{\sslash\SH}$, and the fact that the last functor in~\eqref{eqn:sourceM} is fully faithful is easily checked by direct comparison of the mapping spaces. Note that $\PSh(\Sm_S)_{/\SH}\simeq\PSh(\Sm_S)_{/\SH^\simeq}$.

\begin{remark}
	For our applications to $\MGL$ and related spectra, we will only need the restriction of $\M_S$ to $\PSh(\Sm_S)_{/\Sph}$, and we suggest that the reader ignore the technicalities associated with the more general form of the Thom spectrum functor.
	The latter will only be used to identify the free normed spectrum functor on noninvertible morphisms (see Proposition~\ref{prop:freeNAlg}).
\end{remark}

\begin{remark}
	The definition of $\M_S$ can be rephrased in terms of relative colimits \cite[\sect4.3.1]{HTT}. If $K$ is a small simplicial set and
	\begin{tikzmath}
		\diagram{K & (\Sm_S)_{\sslash \SH} \\ K^{\triangleright} & \Sm_S \\};
		\arrows (11-) edge node[above]{$q$} (-12) (11) edge[c->] (21) (21-) edge node[below]{$f$} (-22) (12) edge node[right]{$p$} (22) (21) edge[dashed] node[above left]{$\bar q$} (12);
	\end{tikzmath}
	is a commutative square with $f(\infty)=S$, one can show that there exists a $p$-colimit diagram $\bar q$ as indicated. The object $\bar q(\infty)$ is obtained by pushing forward the diagram $q$ to the fiber over $S$, using the functors $(\ph)_\sharp$, and taking an ordinary colimit in $\SH(S)$. The functor $\M_S$ is a special case of this construction: if $q$ is the right fibration classified by a presheaf $F\in\PSh((\Sm_S)_{\sslash\SH})$, then $\M_S(F)\simeq \bar q(\infty)$ (in this case $K$ is not necessarily small, but it has a small cofinal subcategory since $F$ is a small colimit of representables). 
	Conversely, an arbitrary diagram $q\colon K\to (\Sm_S)_{\sslash\SH}$ can be factored as a cofinal map $K\to K'$ followed by a right fibration $q'\colon K'\to (\Sm_S)_{\sslash\SH}$, so that $\bar q(\infty)$ is the value of $\M_S$ on the straightening of $q'$.
\end{remark}

\begin{remark} \label{rmk:computing-M}
The restriction of $\M_S$ to $\PSh(\Sm_S)_{/\SH}$ can be described explicitly as follows.
Let $X\in\PSh(\Sm_S)$ and let $\phi\colon X\to\SH$ be a natural transformation. Then $\phi\simeq\colim_{\alpha\colon U\to X} \phi\circ\alpha$ in $\PSh(\Sm_S)_{/\SH}$, where the indexing $\infty$-category is the source of the right fibration classified by $X$ (i.e., the $\infty$-category of elements of $X$).
Since $\M_S$ preserves colimits, we obtain the formula
\[
\M_S(\phi\colon X\to\SH) \simeq \colim_{\substack{f\colon U\to S \\ \alpha\in X(U)}} f_\sharp \phi_U(\alpha).
\]
The right-hand side is also the colimit of $\phi$ in $\Sm_S^\op$-parametrized $\infty$-category theory, in the sense of \cite[\sect5]{ShahThesis}.
\end{remark}

\begin{example}
	The functor $\Sigma^\infty_+ \L_\mot\colon\PSh(\Sm_S)\to\SH(S)$ is equivalent to the composition
	\[
	\PSh(\Sm_S) \to \PSh(\Sm_S)_{/\Sph} \xrightarrow{\M_S} \SH(S),
	\]
	where the first functor is the extension of $\Sm_S\to (\Sm_S)_{/\Sph}$, $X\mapsto (X,\1_X)$.
\end{example}

\begin{lemma}\label{lem:ThomBC}
	Let $f\colon T\to S$ be a morphism of schemes. Then the square
	\begin{tikzmath}
		\diagram{ \PSh((\Sm_S)_{\sslash\SH}) & \SH(S) \\ \PSh((\Sm_T)_{\sslash\SH}) & \SH(T) \\};
		\arrows (11-) edge node[above]{$\M_S$} (-12) (21-) edge node[above]{$\M_T$} (-22) (11) edge node[left]{$f^*$} (21) (12) edge node[right]{$f^*$} (22);
	\end{tikzmath}
	commutes.
\end{lemma}

\begin{proof}
	This follows immediately from smooth base change.
\end{proof}

\begin{lemma}\label{lem:local-slice}
	Let $u\colon \scr E\to\scr C$ be a cartesian fibration classified by a functor $F\colon \scr C^\op\to\what\Cat{}_\infty$. Let $W$ be a collection of morphisms in $\scr C$, $\scr C'\subset\scr C$ the full subcategory of $W$-local objects, and $\scr E'=u^{-1}(\scr C')\subset\scr E$. 
	Suppose that $F$ inverts the morphisms in $W$. Then $\scr E'\subset\scr E$ coincides with the subcategory of objects that are local with respect to the $u$-cartesian morphisms in $u^{-1}(W)$. 
\end{lemma}

\begin{proof}
	Let us write objects of $\scr E$ as pairs $(X,x)$ with $X\in\scr C$ and $x\in F(X)$. Suppose that $(X,x)\in \scr E'$, i.e., that $X\in\scr C'$, and let $(f,\alpha)\colon (A,a)\to (B,b)$ be a cartesian morphism in $u^{-1}(W)$. The mapping space $\Map((A,a),(X,x))$ is given by a pullback square
	\begin{tikzmath}
		\diagram{\Map((A,a),(X,x)) & F(A)_{a/} \\ \Map(A,X) & F(A)\rlap, \\};
		\arrows (11-) edge (-12) (21-) edge (-22) (11) edge (21) (12) edge (22);
	\end{tikzmath}
	and similarly for $\Map((B,b),(X,x))$.
	Since $F$ inverts $W$, $f^*\colon F(B)\to F(A)$ is an equivalence, and since $(f,\alpha)$ is cartesian, the induced functor $F(B)_{b/}\to F(A)_{a/}$ is an equivalence. It follows from the above cartesian squares that $(X,x)$ is local with respect to $(f,\alpha)$. 
	
	Conversely, suppose that $(X,x)\in\scr E$ is local with respect to cartesian morphisms in $u^{-1}(W)$, and let $f\colon A\to B$ be a morphism in $W$. We must show that the map $f^*\colon\Map(B,X)\to \Map(A,X)$ is an equivalence. 
	Let $\scr E^\mathrm{cart}\subset\scr E$ be the wide subcategory on the $u$-cartesian morphisms. Then $f^*$ may be identified with
	\[
	f^*\colon \colim_{b\in F(B)}\Map_{\scr E^\mathrm{cart}}((B,b),(X,x)) \to \colim_{a\in F(A)}\Map_{\scr E^\mathrm{cart}}((A,a),(X,x))
	\]
	Since $f^*\colon F(A)\to F(B)$ is an equivalence, it suffices to show that for every $b\in F(B)$, the map
	\[
	f^*\colon \Map_{\scr E^\mathrm{cart}}((B,b),(X,x)) \to \Map_{\scr E^\mathrm{cart}}((A,f^*(b)),(X,x))
	\]
	is an equivalence. This also implies that a map $(B,b)\to (X,x)$ in $\scr E$ is $u$-cartesian if and only if the composite $(A,f^*(b))\to (B,b)\to (X,x)$ is $u$-cartesian. Hence, the previous map is a pullback of
	\[
	f^*\colon \Map_{\scr E}((B,b),(X,x)) \to \Map_{\scr E}((A,f^*(b)),(X,x)),
	\]
	which is an equivalence by the assumption on $(X,x)$.
\end{proof}

\begin{proposition}\label{prop:thominvariance}
	Let $S$ be a scheme.
	\begin{enumerate}
		\item The functor $\M_S\colon \PSh(\Sm_S)_{/\SH}\to \SH(S)$ inverts Nisnevich equivalences.
		\item Let $A\in\H(S)_{/\SH}$. Then the restriction of $\M_S$ to $\PSh(\Sm_S)_{/A}$ inverts motivic equivalences.
	\end{enumerate}
\end{proposition}

\begin{proof}
Let $W$ be a set of morphisms in $\PSh(\Sm_S)$, $\bar W$ its strong saturation, $A$ a (possibly large) presheaf on $\Sm_S$ with a transformation $A\to\SH$, and $u\colon \PSh(\Sm_S)_{/A}\to\PSh(\Sm_S)$ the forgetful functor.
Then $u$ is a right fibration, and in particular every morphism in $\PSh(\Sm_S)_{/A}$ is $u$-cartesian.
By Lemma~\ref{lem:local-slice}, if $A$ is $W$-local, then $u^{-1}(\bar W)$ and $u^{-1}(W)$ have the same strong saturation, since they determine the same class of local objects.
 Since the restriction of $\M_S$ to $\PSh(\Sm_S)_{/A}$ preserves colimits, it inverts $u^{-1}(\bar W)$ if and only if it inverts $u^{-1}(W)$. To prove (1) (resp.\ (2)), it therefore remains to check that $\M_S$ inverts $u^{-1}(W)$ for $W$ a generating set of Nisnevich equivalences (resp.\ of motivic equivalences).
	
Let $X\in\Sm_S$ and $\phi\colon X\to\SH$. 
If $\iota\colon U\into X$ is a Nisnevich sieve, then the restriction map
\[
\M_S(\phi\circ\iota) \to \M_S(\phi)
\]
is an equivalence, since $\colim_{f\in U} f_\sharp f^* \simeq \id_{\SH(X)}$.
Similarly, if $\pi\colon X\times\A^1\to X$ is the projection, then the restriction map
\[
\M_S(\phi\circ \pi) \to \M_S(\phi)
\]
is an equivalence, since the counit $\pi_\sharp\pi^*\to\id_{\SH(X)}$ is an equivalence. 
\end{proof}

\begin{remark}
	For every scheme $X$, the pullback functor $\SH(X)\to \SH(X\times\A^1)$ is fully faithful, but it is not an equivalence (unless $X=\emptyset$), so we cannot deduce from Proposition~\ref{prop:thominvariance} that $\M_S\colon \PSh(\Sm_S)_{/\SH}\to\SH(S)$ inverts all motivic equivalences; in fact, we will see in Remark~\ref{rmk:freeNAlg-3} that it does not.
	On the other hand, it seems plausible that $\Sph(X)\to \Sph(X\times\A^1)$ is an equivalence (i.e., is surjective), which would imply that $\M_S\colon\PSh(\Sm_S)_{/\Sph}\to \SH(S)$ inverts motivic equivalences. 
	We offer two pieces of evidence that $\Sph$ might be $\A^1$-invariant:
	\begin{itemize}
		\item There are no counterexamples coming from $\K$-theory: for every $\xi\in \K(X\times\A^1)$, the motivic sphere $\S^\xi$ is pulled back from $X$. This follows from Remark~\ref{rmk:KST} below. 
		\item The analogous assertion for the $\infty$-category of $\ell$-adic sheaves is true, because dualizable $\ell$-adic sheaves are locally constant. From this perspective, one might even expect the full subcategory of $\SH(-)$ generated under colimits by the dualizable objects to be $\A^1$-invariant.
	\end{itemize}
\end{remark}

\subsection{Algebraic cobordism and the motivic \texorpdfstring{$\rm J$}{J}-homomorphism}
\label{sub:Jhomomorphism}

For a scheme $X$, we write $\K(X)$ for the Thomason–Trobaugh $\K$-theory space of $X$, i.e., the $\K$-theory of perfect complexes on $X$ \cite{TT}; this is a functor on $\Span(\Sch,\all,\mathrm{pffp})$, where ``pffp'' is the class of proper flat morphisms of finite presentation \cite[D.19]{BarwickMackey}. We will construct a natural transformation of grouplike $\E_\infty$-spaces
\[
j\colon \K \to \Sph\colon \Span(\Sch,\all,\fet)\to\CAlg^\mathrm{gp}(\scr S),
\]
which we call the \emph{motivic $\rm J$-homomorphism}; it is such that, if $\xi$ is a virtual vector bundle on $X\in\Sm_S$, then
\[
\M_S(j\circ \xi) \simeq \Th_X(\xi).
\]
Let $\Vect(X)$ be the \emph{groupoid} of vector bundles over $X$, with the symmetric monoidal structure given by direct sum. The assignment $X\mapsto\Vect(X)$ is a functor on $\Span(\Sch,\all,\flf)$ via the pushforward of vector bundles.
We start with the symmetric monoidal functor
\[
\Vect(X) \to \Shv_{\Nis}(\Sm_X,\Set)_\pt\subset \Shv_{\Nis}(\Sm_X)_\pt, \quad \xi \mapsto \xi/\xi^\times.
\]
This is readily made natural in $X\in \Span(\Sch,\all,\fet)$ using Lemma~\ref{lem:0truncated} and Proposition~\ref{prop:pairs}, so that we obtain a functor of $\infty$-categories
\[
\Span(\Sch,\all,\fet) \times \Delta^1 \to \CAlg(\Cat_\infty),\quad (X,0\to 1)\mapsto (\Vect(X) \to \Shv_\Nis(\Sm_X)_\pt).
\]
Composing with motivic localization and stabilization (c.f.\ Remark~\ref{rmk:H->SH}), we obtain a symmetric monoidal functor
\[
\Vect(X) \to \SH(X), \quad \xi \mapsto \S^\xi,
\]
natural in $X\in \Span(\Sch,\all,\fet)$.
As this functor lands in $\Sph(X)$, we get a natural transformation
\[
\Vect \to \Sph\colon \Span(\Sch,\all,\fet) \to \CAlg(\scr S),
\]
whence
\[
\K^\oplus \to \Sph\colon \Span(\Sch,\all,\fet) \to \CAlg^\mathrm{gp}(\scr S),
\]
where $\K^\oplus=\Vect^\gp$ is the group completion of $\Vect$. 
Finally, $\Sph$ is the right Kan extension of its restriction to affine schemes, since it is a Zariski sheaf, while the right Kan extension of $\K^\oplus$ is the Thomason–Trobaugh $\K$-theory. Hence, by Proposition~\ref{prop:span-RKE}, the above transformation factors through $\K$.

\begin{remark}\label{rmk:KST}
	Since $\Sph$ is a cdh sheaf, the motivic $\rm J$-homomorphism $j\colon \K\to\Sph$ factors through the cdh sheafification $\L_\mathrm{cdh}\K$. We claim that $\L_\mathrm{cdh}\K\simeq \Omega^\infty\KH$. On the category of finite-dimensional noetherian schemes, this is proved in \cite[Theorem 6.3]{KST}. Both $\K$ and $\Omega^\infty\KH$ transform limits of cofiltered diagrams of qcqs schemes with affine transition maps into colimits, and it is easy to show that $\L_\mathrm{cdh}$ preserves this property, which implies the claim in general.
	Using the compatibility of the cdh topology with finite étale transfers (Remark~\ref{rmk:cdh-transfers}), we deduce that $j$ factors through a natural transformation
	\[
	\Omega^\infty\KH \to \Sph\colon \Span(\Sch,\all,\fet)\to \CAlg^\gp(\scr S).
	\]
	In particular, by Proposition~\ref{prop:thominvariance}(2), the Thom spectrum functor $\M_S\colon \PSh(\Sm_S)_{/\K}\to \SH(S)$ inverts motivic equivalences for any scheme $S$.
\end{remark}

\begin{lemma}\label{lem:Kpullback}
	Let $f\colon T\to S$ be a morphism of schemes. Then the canonical map
	\[
	f^*(\K|\Sm_S) \to \K|\Sm_T
	\]
	in $\PSh(\Sm_T)$ is a Zariski equivalence.
\end{lemma}

\begin{proof}
	Since vector bundles are Zariski-locally trivial, $\Vect$ is the Zariski sheafification of $\coprod_{n\geq 0}\B\GL_n$. 
	Both sheafification $\L_\Zar\colon \PSh(\Sm_S)\to\Shv_\Zar(\Sm_S)$ and pullback $f^*\colon \PSh(\Sm_S)\to\PSh(\Sm_T)$ commute with finite products. It follows that they preserve commutative monoids and commute with group completion \cite[Lemma 5.5]{HoyoisCdh}. On the one hand, this implies that
	\[
	\Biggl(\coprod_{n\geq 0}\B\GL_n\Biggr)^\gp \to \K
	\]
	is a Zariski equivalence on $\Sm_S$. On the other hand, it implies that the left-hand side is stable under base change (since $\GL_n$ is). Since $f^*$ preserves Zariski sieves and hence Zariski equivalences, we conclude by 2-out-of-3 that $f^*(\K|\Sm_S) \to \K|\Sm_T$ is a Zariski equivalence.
\end{proof}

We denote by $e\colon \K^\circ\into \K$ the inclusion of the rank $0$ part of $\K$-theory.
Let $\Gr_\infty = \colim_n\Gr_n$ be the infinite Grassmannian, and let
\[
\gamma\colon \Gr_\infty \to \K^\circ
\]
be the map whose restriction to $\Gr_n$ classifies the tautological bundle minus the trivial bundle.
The definition of $\MGL_S\in\SH(S)$ can then be recast as
\[
\MGL_S = \M_S(j\circ e\circ \gamma).
\]

We recall that $\gamma\colon \Gr_\infty\to \K^\circ$ is a motivic equivalence. To see this, consider the diagram
\begin{tikzmath}
	\diagram{ & \B\GL & \K\langle 0\rangle \\
	\Gr_\infty & \B_{\et}\GL & \K^\circ\rlap, \\};
	\arrows  (12-) edge (-13) (12) edge (22) (13) edge (23)
	(21-) edge (-22) (22-) edge (-23);
\end{tikzmath}
where 
$\K\langle 0\rangle\subset \K^\circ$ is the connected component of $0$. The vertical maps are Zariski equivalences (because $\GL_n$-torsors are Zariski-locally trivial). The map $\Gr_\infty \to \B_\et\GL$ is a motivic equivalence by \cite[\sect4 Proposition 2.6]{MV}. Finally, the map $\L_{\A^1}\B\GL \to \L_{\A^1}\K\langle 0\rangle$ is an equivalence on affines, since it is a homology equivalence between connected $\rm H$-spaces (that $\L_{\A^1}\B\GL$ is an $\rm H$-space follows from the fact that even permutation matrices are $\A^1$-homotopic to the identity matrix).

\begin{theorem}\label{thm:MGL}
	Let $S$ be a scheme. Then $\gamma\colon \Gr_\infty\to \K^\circ$ induces an equivalence
	\[
	\MGL_S \simeq \M_S(j\circ e).
	\]
\end{theorem}

\begin{proof}
	Suppose first that $S$ is regular. Then $\K$ is an $\A^1$-homotopy invariant Nisnevich sheaf on $\Sm_S$. By Proposition~\ref{prop:thominvariance}(2), $\M_S(j\circ\ph)\colon \PSh(\Sm_S)_{/\K}\to\SH(S)$ inverts motivic equivalences.
	Since $\gamma\colon \Gr_\infty\to \K^\circ$ is a motivic equivalence, the theorem holds in that case.
	In particular, the given equivalence holds over $\Spec\Z$. Hence, for $f\colon S\to\Spec\Z$ arbitrary, we have
	\[
	\MGL_S \simeq f^*(\MGL_{\Z}) \simeq \M_S(f^*(\K^\circ|\Sm_\Z) \to \K^\circ \to \Sph)
	\]
	by Lemma~\ref{lem:ThomBC}.
	But $f^*(\K^\circ|\Sm_\Z) \to \K^\circ|\Sm_S$ is a Zariski equivalence by Lemma~\ref{lem:Kpullback}, so we conclude by Proposition~\ref{prop:thominvariance}(1).
\end{proof}

\begin{remark}
	The motivic Thom spectrum associated with $j$ itself is the $(2,1)$-periodization of $\MGL$:
	 \[\M_S(j)\simeq \bigvee_{n\in\Z}\Sigma^{2n,n}\MGL_S.\]
	 Indeed, with respect to the decomposition $\K\simeq \L_\Sigma(\coprod_{n\in\Z}\K^\circ)$ induced by the rank map $\K\to\Z$, the motivic $\rm J$-homomorphism $j\colon \K\to \Sph$ has components $\Sigma^{2n,n}\circ j\circ e$.
\end{remark}

\subsection{Multiplicative properties}
\label{sub:construct-M}

We now make the assignment 
\[S\mapsto \M_S\colon \PSh_\Sigma((\SmQP_S)_{\sslash\SH}) \to \SH(S)\]
 functorial in $S\in\Span(\Sch,\all,\fet)$.
   This will in particular equip each $\M_S$ with a symmetric monoidal structure. A simpler version of the same construction produces the functor $\M_S\colon \PSh((\Sm_S)_{\sslash\SH})\to\SH(S)$ of Definition~\ref{def:Thom}, functorial in $S\in\Sch^\op$.

For later applications, we consider a slightly more general situation. Let $S$ be a scheme, let $\scr C\subset_\fet\Sch_S$, and let $\scr L$ be a class of smooth morphisms in $\scr C$ that is closed under composition, base change, and Weil restriction along finite étale maps; for example, $\scr C=\Sch$ and $\scr L$ is the class of smooth quasi-projective morphisms. For $X\in\scr C$, we denote by $\scr L_X\subset\Sm_X$ the full subcategory spanned by the morphisms in $\scr L$. For simplicity, we will write
   \[
   \Span = \Span(\scr C,\all,\fet)
   \]
   in what follows. Let
   \[
   \Fun_{\scr L}(\Delta^1,\Span) \subset \Fun(\Delta^1,\Span)
   \]
	be the full subcategory on the spans $X\stackrel f\from Y\stackrel{\id}\to Y$ with $f\in\scr L$, and let
	\[
	s,t\colon \Fun_{\scr L}(\Delta^1,\Span)\to\Span
	\]
	be the source and target functors.
  	The composition
	\[
	\Fun_{\scr L}(\Delta^1,\Span)\times\Delta^1 \xrightarrow{\mathrm{ev}} \Span \xrightarrow{\SH^\otimes} \what\Cat{}_\infty
	\]
	encodes a natural transformation
   \[
   \phi\colon \SH^\otimes\circ s \to \SH^\otimes \circ t\colon \Fun_{\scr L}(\Delta^1,\Span) \to \what\Cat{}_\infty.
   \]
	If $\scr E\to\Span^\op$ is the \emph{cartesian} fibration classified by $\SH^\otimes$, we can regard $\phi$ as a map of cartesian fibrations $\phi\colon s^*\scr E\to t^*\scr E$ over $\Fun_{\scr L}(\Delta^1,\Span)^\op$.
	Its fiber over a map $f\colon Y\to X$ in $\scr L$ is the functor $f^*\colon \SH(X) \to \SH(Y)$. Since $f$ is smooth, this functor admits a left adjoint $f_\sharp$, i.e., $\phi$ has a fiberwise left adjoint. By the dual of Lemma~\ref{lemm:construct-relative-adjoint}(1), this implies that there is a relative adjunction
\[
\psi : t^*\scr E \rightleftarrows s^*\scr E : \phi
\]
over $\Fun_{\scr L}(\Delta^1,\Span)^\op$. The left adjoint $\psi$ encodes a \emph{right-lax} natural transformation $\SH^\otimes\circ t \to \SH^\otimes\circ s$ with components $f_\sharp\colon \SH(Y)\to\SH(X)$, the right-lax naturality being witnessed by the exchange transformations $\Ex_\sharp^*$ and $\Ex_{\sharp\otimes}$. We now consider the diagram
\begin{tikzmath}
	\diagram{t^*\scr E & s^*\scr E & \scr E \\ & \Fun_{\scr L}(\Delta^1,\Span)^\op & \Span^\op\rlap, \\};
	\arrows (11-) edge node[above]{$\psi$} (-12) (12-) edge node[above]{$\chi$} (-13) (11) edge (22) (12) edge (22) (13) edge (23) (22-) edge node[above]{$s$} (-23);
\end{tikzmath}
where the square is cartesian. The existence of Weil restrictions implies that the source map
  \[
  s\colon  \Fun_{\scr L}(\Delta^1,\Span) \to \Span
  \]
  is a cocartesian fibration: a cocartesian edge starting at $U\in\scr L_X$ over the span $X\from Y \to Z$ is the span
  \[
  U\from U_Y\from \Weil_{Y/Z}(U_Y)\times_ZY \to \Weil_{Y/Z}(U_Y).
  \]
  Hence, $\chi\circ\psi \colon t^*\scr E\to\scr E$ is a morphism of cartesian fibrations over $\Span^\op$. Its fiber over a scheme $X\in\scr C$ is a functor
  \[
  (\scr L_X)_{\sslash\SH} \to \SH(X),\quad (f\colon Y\to X, P\in\SH(Y))\mapsto f_\sharp P.
  \]
  We claim that $\chi\circ\psi$ preserves cartesian edges. Indeed, given the above description of $s^\op$-cartesian edges, this amounts to the following two facts: the transformation $\Ex_\sharp^*$ associated with a \emph{cartesian} square is an equivalence, and the distributivity transformation $\Dist_{\sharp\otimes}$ is an equivalence (Proposition~\ref{prop:distributivity}(1)). Hence, $\chi\circ\psi$ encodes a \emph{strict} natural transformation
  \[
  \scr L_{\sslash\SH}^\otimes \to \SH^\otimes \colon \Span \to \what\Cat{}_\infty.
  \]
  Finally, since $\SH^\otimes$ is valued in $\what\Cat{}_\infty^\mathrm{sift}$, this transformation lifts to the objectwise sifted cocompletion, and we obtain
  \begin{equation}\label{eqn:M}
  \M\colon \PSh_\Sigma(\scr L_{\sslash\SH})^\otimes \to \SH^\otimes\colon \Span(\scr C,\all,\fet)\to\what\Cat{}_\infty^\mathrm{sift}.
  \end{equation}
  Note that each of the subcategories considered in~\eqref{eqn:sourceM} is preserved by base change and by norms. Hence, the source of $\M$ admits subfunctors $\PSh_\Sigma(\scr L)_{\sslash\SH}^\otimes$, $\PSh_\Sigma(\scr L)_{/\SH}^\otimes$, and $\PSh_\Sigma(\scr L)_{/\Sph}^\otimes$.

\begin{lemma}\label{lem:SectPSigma}
	Let $\scr C$ be a small $\infty$-category and $F\colon \scr C\to\Cat_\infty$ a functor classifying a cocartesian fibration $\int F\to\scr C$. For every $\scr E\in\Pr^\mathrm{L}$, there is a canonical equivalence of $\infty$-categories
	\[
	\Fun\left(\int F,\scr E\right) \simeq \Sect(\Fun(\ph,\scr E)_!\circ F),
	\]
	where $\Fun(\ph,\scr E)_!\colon \Cat_\infty \to\Pr^\mathrm{L}$ is the opposite of $\Fun(\ph,\scr E)\colon \Cat_\infty^\op\to\Pr^\mathrm{R}\simeq \Pr^{\mathrm L,\op}$.
\end{lemma}

\begin{proof}
	This is an instance of the fact that $\Fun(\ph,\scr E)$ transforms left-lax colimits into right-lax limits.
	In more details, let $\Tw(\scr C)\to\scr C\times\scr C^\op$ be the right fibration classified by $\Map\colon\scr C^\op\times\scr C\to\scr S$.
	If $X\colon\scr C\to\Cat_\infty$ classifies the cocartesian fibration $\scr X\to\scr C$, then
	\[
	\scr X \simeq \colim_{(a\to b)\in\Tw(\scr C)} \scr C_{b/}\times X(a),
	\]
	and if $Y\colon\scr C^\op\to\Cat_\infty$ classifies the cartesian fibration $\scr Y\to\scr C$, then
	\[
	\Fun_{\scr C}(\scr C,\scr Y) \simeq \lim_{(a\to b)\in\Tw(\scr C)^\op} \Fun(\scr C_{b/},Y(a));
	\]
	see \cite[Theorem 7.4 and Proposition 7.1]{GHN}. The cocartesian fibration classified by $\Fun(\ph,\scr E)_!\circ F\colon\scr C\to\Pr^\mathrm{L}$ is also the cartesian fibration classified by $\Fun(\ph,\scr E)\circ F^\op\colon\scr C^{\op}\to\Pr^\mathrm{R}$, so
	\[
	\Fun\left(\int F,\scr E\right)\simeq \lim_{(a\to b)}\Fun(\scr C_{b/}\times F(a),\scr E) \simeq 
	\lim_{(a\to b)}\Fun(\scr C_{b/},\Fun(F(a),\scr E)) \simeq \Sect(\Fun(\ph,\scr E)_!\circ F),
	\]
	as desired.
\end{proof}

\begin{proposition}\label{prop:normedThomSpectra}
	Let $S$ be a scheme, $\scr C\subset_\fet\Sch_S$, and $\scr L$ a class of smooth morphisms in $\scr C$ closed under composition, base change, and finite étale Weil restriction. Then there is a functor
	\[
	\M_{|\scr L}\colon \Fun^\times(\Span(\scr C,\all,\fet),\scr S)_{\sslash\SH} \to \Sect(\SH^\otimes|\Span(\scr C,\all,\fet))
	\]
	sending $\psi\colon A\to \SH$ to the section $X\mapsto \M_X(\psi|\scr L_X^\op)$. Moreover, this section is cocartesian over morphisms $f\colon Y\to X$ in $\scr C$ such that $f^*(A|\scr L_X^\op)\to A|\scr L_Y^\op$ is an $\M_Y$-equivalence, in particular over $\scr L$-morphisms.
\end{proposition}

\begin{proof}
	The source map $s\colon \Fun_{\scr L}(\Delta^1,\Span(\scr C,\all,\fet))\to\Span(\scr C,\all,\fet)$ is the cocartesian fibration classified by $\Span(\scr C,\all,\fet)\to\Cat_\infty$, $S\mapsto \scr L_S^\op$. By Lemma~\ref{lem:SectPSigma}, we have an equivalence
	\[
	\Fun(\Fun_{\scr L}(\Delta^1,\Span(\scr C,\all,\fet)),\scr S) \simeq \Sect(\PSh(\scr L)^\otimes|\Span(\scr C,\all,\fet)),
	\]
	which restricts to an equivalence
	\[
	\alpha\colon \Fun'(\Fun_{\scr L}(\Delta^1,\Span(\scr C,\all,\fet)),\scr S) \simeq \Sect(\PSh_\Sigma(\scr L)^\otimes|\Span(\scr C,\all,\fet)),
	\]
	where $\Fun'$ denotes the full subcategory of functors that preserve finite products on each fiber of $s$.
	On the other hand, the target map $t\colon \Fun_{\scr L}(\Delta^1,\Span(\scr C,\all,\fet))\to\Span(\scr C,\all,\fet)$ induces
	\[
	t^*\colon \Fun^\times(\Span(\scr C,\all,\fet),\scr S) \to \Fun'(\Fun_{\scr L}(\Delta^1,\Span(\scr C,\all,\fet)),\scr S).
	\]
	The functor $\M_{|\scr L}$ is then the composition
	\begin{align*}
	\Fun^\times(\Span(\scr C,\all,\fet),\scr S)_{\sslash\SH} &\xrightarrow{t^*} \Fun'(\Fun_{\scr L}(\Delta^1,\Span(\scr C,\all,\fet)),\scr S)_{\sslash\SH\circ t}\\
	&\stackrel\alpha\simeq  \Sect(\PSh_\Sigma(\scr L)^\otimes|\Span(\scr C,\all,\fet))_{\sslash\SH} \\
	&\subset \Sect(\PSh_\Sigma(\scr L)_{\sslash\SH}^\otimes|\Span(\scr C,\all,\fet)) \\
	&\xrightarrow{\M} \Sect(\SH^\otimes|\Span(\scr C,\all,\fet)).
	\end{align*}
	The final assertion follows from Lemma~\ref{lem:ThomBC}.
\end{proof}

\begin{remark}\label{rmk:C=L}
	In the setting of Proposition~\ref{prop:normedThomSpectra}, the target functor
	\[
	t\colon \Fun_{\scr L}(\Delta^1,\Span(\scr C,\all,\fet)) \to \Span(\scr C,\all,\fet)
	\]
	is left adjoint to the section $X\mapsto\id_X$. It follows that the functor
	\[
	\alpha\circ t^*\colon \Fun^\times(\Span(\scr C,\all,\fet),\scr S) \to \Sect(\PSh_\Sigma(\scr L)^\otimes|\Span(\scr C,\all,\fet))
	\]
	is fully faithful, and its essential image is the subcategory of sections that are cocartesian over $\scr L$-morphisms. In particular, if $\scr C=\scr L_S$, we have an equivalence
	\[
	\Fun^\times(\Span(\scr C,\all,\fet),\scr S) \simeq\NAlg_{\scr C}(\PSh_\Sigma(\scr L)).
	\]
\end{remark}

\begin{theorem}\label{thm:normedMGL}
	The assignments $S\mapsto\MGL_S\in\SH(S)$ and $S\mapsto\bigvee_{n\in\Z}\Sigma^{2n,n}\MGL_S\in\SH(S)$ can be promoted to sections of $\SH^\otimes$ over $\Span(\Sch,\all,\fet)$ that are cocartesian over $\Sch^\op$. In particular, for every scheme $S$, the algebraic cobordism spectrum $\MGL_S$ and its $(2,1)$-periodization $\bigvee_{n\in\Z}\Sigma^{2n,n}\MGL_S$ are normed spectra over $\Sch_S$. Moreover, the canonical map $\MGL_S\to\bigvee_{n\in\Z}\Sigma^{2n,n}\MGL_S$ is a morphism of normed spectra over $\Sch_S$.
\end{theorem}

\begin{proof}
   Recall that the $\rm J$-homomorphism $j\colon \K\to \Sph$ is a natural transformation on $\Span(\Sch,\all,\fet)$. In light of Theorem~\ref{thm:MGL}, the result follows by applying Proposition~\ref{prop:normedThomSpectra} to the transformations $j\circ e$ and $j$, and to the morphism $e\colon \K^\circ \into \K$ over $\Sph$.
\end{proof}

\begin{remark}
	There is another $\E_\infty$-ring structure on $\bigvee_{n\in\Z}\Sigma^{2n,n}\MGL_S$ constructed by Gepner and Snaith \cite{GepnerSnaith}, using the equivalence
	\[
	\bigvee_{n\in\Z}\Sigma^{2n,n}\MGL_S \simeq \Sigma^\infty_+ \L_{\mot} \K^\circ [1/\beta].
	\]
	At least if $S$ has a complex point, this $\E_\infty$-ring structure does \emph{not} coincide with that of Theorem~\ref{thm:normedMGL}, because they do not coincide after Betti realization \cite{HahnYuan}. We do not know if the Gepner–Snaith $\E_\infty$-ring structure can be extended to a normed structure: this is the case if and only if the Bott element $\beta$ satisfies condition (c) of Proposition~\ref{prop:localization-general}(2).
\end{remark}

\begin{example}
	Since $\bigvee_{n\in\Z}\Sigma^{2n,n}\MGL_S$ is a normed spectrum, we obtain by Corollary~\ref{cor:normed-spectrum-tambara-functor} a structure of Tambara functor on the presheaf $\FEt_S^\op \to \Set$, $X\mapsto \bigoplus_{n\in\mathbb Z}\MGL^{2n,n}(X)$.
	If $k$ is a field of characteristic zero, there is a canonical isomorphism
	\[
	\bigoplus_{n\in\Z}\MGL^{2n,n}(\Spec k) \simeq \bb L
	\]
	natural in $k$, where $\bb L$ is the Lazard ring \cite[Proposition 8.2]{HoyoisMGL}. In particular, $\bigoplus_{n\in\mathbb Z}\MGL^{2n,n}(\ph)$ is a finite étale sheaf on $\FEt_k$, and hence its Tambara structure is uniquely determined by its ring structure (see Corollary~\ref{cor:automatic-norms}): for $k'/k$ a finite extension of degree $d$, the induced additive (resp.\ multiplicative) transfer $\bb L \to \bb L$ is $x\mapsto dx$ (resp.\ $x\mapsto x^d$).
\end{example}

\begin{example}\label{ex:MSL}
	Using the same method as in the proof of Theorem~\ref{thm:normedMGL}, one can show that the spectra $\MSL$, $\MSp$, and their $(4,2)$-periodizations are normed spectra over $\Sch$, and that the maps $\MSp\to\MSL\to\MGL$ and their $(4,2)$-periodic versions are morphisms of normed spectra (see \cite{PaninWalter} for the definitions of $\MSL$ and $\MSp$). These normed spectra and the maps between them can be obtained by applying Proposition~\ref{prop:normedThomSpectra} to the natural transformations
	\[
	\K^\Sp \to \K^\SL \to \K^\mathrm{ev} \xrightarrow{j} \Sph \colon \Span(\Sch,\all,\fet) \to \scr S.
	\]
	Here, $\K^\SL$ (resp.\ $\K^\Sp$) can be defined as the right Kan extension of the functor $\Span(\Aff,\all,\fet)\<\to\scr S$ sending $X$ to the group completion of the symmetric monoidal groupoid of even-dimensional oriented (resp.\ symplectic) vector bundles\footnote{Here, an oriented vector bundle means a vector bundle with trivialized determinant. The direct sum monoidal structure on oriented bundles does not admit a braiding, but it acquires a symmetric braiding on even-dimensional bundles.} over $X$, and $\K^\mathrm{ev}=\rk^{-1}(2\Z)\subset \K$.
	The key ingredients for the oriented and symplectic versions of Theorem~\ref{thm:MGL} are the Zariski-local triviality of oriented and symplectic bundles and the fact that $\L_{\A^1}\B\SL$ and $\L_{\A^1}\B\Sp$ are $\rm H$-spaces. 
\end{example}

\begin{example}\label{ex:MG}
	Fix an integer $k\geq 1$.
	Using the symmetric monoidal functor $\M$, we can define motivic Thom spectra for structured vector bundles as $\rm A_\infty$-ring spectra.
	Let $G=(G_n)_{n\in\bb N}$ be a family of flat finitely presented $S$-group schemes equipped with a morphism of associative algebras $G\to (\GL_{nk,S})_{n\in\bb N}$ for the Day convolution on $\Fun(\bb N,\mathrm{Grp}(\Sch_S))$.
	For every $S$-scheme $X$, we then have a monoidal groupoid 
	\[
	\Vect^G(X)\simeq \L_{\mathrm{fppf}}\Biggl(\coprod_{n\geq 0}\B G_n\Biggr)(X)
	\] 
	of vector bundles with structure group in the family $G$, and a monoidal forgetful functor $\Vect^G(X)\to\Vect (X)$. For $X$ affine, let $\K^G(X)=\Vect^G(X)^\gp$, and extend $\K^G$ to all schemes by right Kan extension. We thus have a natural transformation
	\[
	u_G\colon \K^G \to \K\colon \Sch_S^\op\to \mathrm{Grp}(\scr S).
	\]
	For $T\in\Sch_S$, let $\M G_T= \M_T(u_G^{-1}(\K^\circ)\to \K^\circ \to \Sph)\in\SH(T)$.
	Then $T\mapsto \M G_T$ is a section of $\mathrm{Alg}(\SH(\ph))$ over $\Sch_S^{\op}$ that is cocartesian over smooth morphisms, and in fact fully cocartesian provided that each $G_n$ admits a faithful linear representation Nisnevich-locally on $S$ (cf.\ \cite[Corollary 2.9]{HoyoisCdh}). For example:
	\begin{itemize}
		\item If $*$ is the family of trivial groups, $\M*\simeq\1$.
		\item Any morphism of $S$-group schemes $G\to\GL_{k,S}$ can be extended to a family $(G^n)_{n\in\bb N}\to (\GL_{nk,S})_{n\in\bb N}$ and hence gives rise to an $\rm A_\infty$-ring spectrum $\M G^\infty_S$.
		\item For $\mathrm{Br}=(\mathrm{Br}_n)_{n\in \bb N}$ the family of braid groups with $\mathrm{Br}_n\onto\Sigma_n\into\GL_n$, the monoidal category $\Vect^\mathrm{Br}(X)$ has a canonical braiding, and we obtain an absolute $\E_2$-ring spectrum $\mathrm{MBr}$. 
	\end{itemize}
\end{example}

\begin{example}\label{ex:MG2}
Suppose given a family of $S$-group schemes $G=(G_n)_{n\in\bb N}$ and a morphism of associative algebras $G\to (\GL_{nk,S})_{n\in\bb N}$, as in Example~\ref{ex:MG}. Suppose moreover given a factorization of $(\Sigma_n)_{n\in\bb N}\into (\GL_{nk})_{n\in\bb N}$ through $G$.
Then, for every $S$-scheme $X$, the forgetful functor $\Vect^G(X)\to\Vect (X)$ acquires a symmetric monoidal structure. By Corollary~\ref{cor:automatic-norms}, the assignment $X\mapsto \Vect^G(X)$ extends uniquely to $\Span(\Sch_S,\all,\fet)$. Using Proposition~\ref{prop:span-RKE}, we obtain a natural transformation
	\[
	u_G\colon \K^G \to \K\colon \Span(\Sch_S,\all,\fet)\to \CAlg^\gp(\scr S).
	\]
	Hence, by Proposition~\ref{prop:normedThomSpectra}, $T\mapsto \M G_T$ can be promoted to a section of $\SH^\otimes$ over $\Span(\Sch_S,\all,\fet)$, and similarly for $T\mapsto \M_T(j\circ u_G)\simeq\bigvee_{n\in\Z}\Sigma^{2nk,nk}\M G_T$. In particular, $\M G_S$ is a $(2k,k)$-periodizable normed spectrum.
	For example:
	\begin{itemize}
		\item Any morphism of $S$-group schemes $G\to\GL_{k,S}$ can be extended to a family $(G\wr\Sigma_n)_{n\in\bb N}\to (\GL_{nk,S})_{n\in\bb N}$ and hence gives rise to a $(2k,k)$-periodizable normed spectrum $\M(G\wr\Sigma)_S$.
		\item The families of $S$-group schemes $\rm O(q_{2n})$ and $\mathrm{SO}(q_{2n})$, where $q_{2n}$ is the standard split quadratic form of rank $2n$, give rise to $(4,2)$-periodizable normed spectra $\mathrm{MO}_S$ and $\mathrm{MSO}_S$ over $\Sch_S$.
		Using a theorem of Schlichting and Tripathy \cite[Theorem 5.2]{SchlichtingTripathy}, one can show as in Theorem~\ref{thm:MGL} that the canonical map $\Gr\rm O_\infty \to u_{\rm O}^{-1}(\K^\circ)$ induces an equivalence on Thom spectra when $2$ is invertible on $S$.
	\end{itemize}
\end{example}

\subsection{Free normed spectra}

As another application of the machinery of Thom spectra, we give a formula for the free normed spectrum functor $\NSym_\scr C$ in some cases (see Remark~\ref{rmk:freeNAlg}).
Let $\scr C\subset_\fet\Sch_S$ and let $A\colon\scr C^\op\to\scr E$ be a functor where $\scr E$ is cocomplete.
The left Kan extension of $A$ to $\Span(\scr C,\all,\fet)$ is then given by
	\[
	\NSym(A)\colon \Span(\scr C,\all,\fet)\to\scr E,\quad X\mapsto \colim_{Y\in\FEt_X^\simeq} A(Y).
	\]
	Note that if $\scr E$ has finite products that preserve colimits in each variable and if $A$ preserves finite products, then $\NSym(A)$ preserves finite products as well.
	Given $\psi\colon A\to\SH$ in $\PSh_\Sigma(\scr C)_{\sslash\SH}$, we get an induced transformation
	\[
	\NSym(A) \xrightarrow{\NSym(\psi)} \NSym(\SH) \xrightarrow{\mu} \SH\colon \Span(\scr C,\all,\fet)\to\what\Cat{}_\infty,
	\]
	where $\mu$ is the counit transformation. This defines a functor
	\[
	\mu\circ\NSym(\ph)\colon \PSh_\Sigma(\scr C)_{\sslash\SH} \to \Fun^\times(\Span(\scr C,\all,\fet),\scr S)_{\sslash\SH}.
	\]
	Moreover, the unit transformation $A\to\NSym(A)$ induces a natural transformation $\id\to \mu\circ\NSym(\ph)$ of endofunctors of $\PSh_\Sigma(\scr C)_{\sslash\SH}$.

\begin{proposition}\label{prop:freeNAlg}
	Let $S$ be a scheme, $\scr C\subset_\fet\Sm_S$, and $\scr L$ a class of smooth morphisms in $\scr C$ that is closed under composition, base change, and finite étale Weil restriction, such that $\scr C=\scr L_S$. Then there is a commutative square
	\begin{tikzmath}
		\def\colsep{5em}
		\diagram{\PSh_\Sigma(\scr C)_{\sslash\SH} & \Fun^\times(\Span(\scr C,\all,\fet),\scr S)_{\sslash\SH} \\
		\SH(S) & \NAlg_\scr C(\SH)\rlap, \\};
		\arrows (11-) edge node[above]{$\mu\circ\NSym(\ph)$} (-12) (11) edge node[left]{$\M_S$} (21) (12) edge node[right]{$\M_{|\scr L}$} (22) (21-) edge node[below]{$\NSym_\scr C$} (-22);
	\end{tikzmath}
	where the right vertical functor is that of Proposition~\ref{prop:normedThomSpectra}.
	 In particular, the free normed spectrum functor $\NSym_\scr C\colon \SH(S)\to \NAlg_\scr C(\SH)$ is the composition
	\[
	\SH(S) \into \PSh_\Sigma(\scr C)_{\sslash\SH} \xrightarrow{\mu\circ\NSym(\ph)} \Fun^\times(\Span(\scr C,\all,\fet),\scr S)_{\sslash\SH} \xrightarrow{\M_{|\scr L}} \NAlg_\scr C(\SH).
	\]
\end{proposition}

\begin{proof}
	By Remark~\ref{rmk:C=L}, the functor $\M_{|\scr L}$ is the composition
	\[
	\Fun^\times(\Span(\scr C,\all,\fet),\scr S)_{\sslash\SH} \simeq \NAlg_{\scr C}(\PSh_\Sigma(\scr L))_{\sslash\SH} \subset\NAlg_{\scr C}(\PSh_\Sigma(\scr L)_{\sslash\SH}) \xrightarrow{\M} \NAlg_\scr C(\SH).
	\]
	In particular, we have a commutative square
	\begin{tikzmath}
		\diagram{\Fun^\times(\Span(\scr C,\all,\fet),\scr S)_{\sslash\SH} & \PSh_\Sigma(\scr C)_{\sslash\SH} \\
		\NAlg_\scr C(\SH) & \SH(S)\rlap, \\};
		\arrows (11-) edge (-12) (11) edge node[left]{$\M_{|\scr L}$} (21) (12) edge node[right]{$\M$} (22) (21-) edge (-22);
	\end{tikzmath}
	where the horizontal maps are the forgetful functors. 
	By adjunction, the natural transformation $\id \to \mu\circ \NSym(\ph)$ then induces a natural transformation $\NSym_\scr C\circ \M \to \M_{|\scr L}\circ (\mu\circ\NSym(\ph))$.
	
	Let $\scr A^\otimes$ be a \emph{small} full subfunctor of $\SH^\otimes\colon \Span(\scr C,\all,\fet)\to\what\Cat{}_\infty$ (i.e., valued in small $\infty$-categories).
	Then the transformation~\eqref{eqn:M} restricts to
	\[
	\M\colon \PSh_\Sigma(\scr L_{/\scr A})^\otimes \to \SH^\otimes\colon \Span(\scr C,\all,\fet)\to\what\Cat{}_\infty^\mathrm{sift}.
	\]
	Since $\scr A$ is small, each component $\M_X\colon \PSh_\Sigma((\scr L_X)_{/\scr A})\to\SH(X)$ admits a right adjoint $R_X$ given by
	\[
	R_X(E)(\psi) = \Map(\M_X(\psi), E).
	\]
	In the induced relative adjunction over $\Span(\scr C,\all,\fet)$ (see Lemma~\ref{lemm:construct-relative-adjoint}(1)), the relative right adjoint preserves cocartesian edges over backward $\scr L$-morphisms. It then follows from Lemma~\ref{lemm:adjoints-pass-to-sections} that $\M_S$ and its right adjoint lift to an adjunction
	\[
	\Fun^\times(\Span(\scr C,\all,\fet),\scr S)_{/\scr A} \simeq \NAlg_\scr C(\PSh_\Sigma(\scr L_{/\scr A})) \rightleftarrows \NAlg_{\scr C}(\SH),
	\]
	so that the left adjoint commutes with formation of free normed objects. This implies that the natural transformation $\NSym_\scr C\circ \M \to \M_{|\scr L}\circ (\mu\circ\NSym(\ph))$ is an equivalence on $\PSh_\Sigma(\scr C)_{/\scr A}$. Since $\SH^\otimes$ is the union of its small subfunctors, this concludes the proof.
\end{proof}

\begin{remark}\label{rmk:freeNAlg-2}
	Let $\scr C\subset_\fet\Sm_S$ be as in Proposition~\ref{prop:freeNAlg}.
	Using the formula for $\M_S$ from Remark~\ref{rmk:computing-M}, we see that the underlying motivic spectrum of $\NSym_\scr C(E)$ is the colimit
	\[
	\colim_{\substack{f\colon X\to S\\p\colon Y\to X}} f_\sharp p_\otimes(E_Y),
	\]
	whose indexing $\infty$-category is the source of the cartesian fibration classified by $\scr C^\op\to\scr S$, $X\mapsto\FEt_X^\simeq$. 
	This applies when $\scr C=\FEt_S$ (with $\scr L$ the class of finite étale morphisms) and when $\scr C=\SmQP_S$ (with $\scr L$ the class of smooth quasi-projective morphisms).
	Note that when $\scr C=\SmQP_S$, one may replace it by $\Sm_S$ without changing either the functor $\NSym_\scr C$ (Proposition~\ref{prop:categorical-props}(5)) or the above colimit (Proposition~\ref{prop:thominvariance}(1)).
\end{remark}

\begin{remark}\label{rmk:freeNAlg-3}
	Let $\scr C\subset_\fet\Sm_S$ be as in Proposition~\ref{prop:freeNAlg}, let $A\in\PSh_\Sigma(\scr C)$, and let $\psi\colon A\to\SH$.
	The degree maps $\FEt_X^\simeq \to \bb N^X$ induce a coproduct decomposition of $\NSym(A)$ in $\PSh_\Sigma(\scr C)$, whence a coproduct decomposition of $\NSym_\scr C(\M_S(\psi))$ in $\SH(S)$. For example, when $A=*$ and $\psi(*)=E\in\SH(S)$, we have $\NSym(*)\simeq \L_\Sigma(\coprod_{d\geq 0} \B_\et\Sigma_d)$ and hence
\[
\NSym_\scr C(E) = \bigvee_{d\geq 0} \NSym_\scr C^d(E),
\]
where $\NSym_\scr C^d(E)$ is the Thom spectrum of a map $\B_\et\Sigma_d\to\SH|\scr C^\op$ refining $E^{\wedge d}\in\SH(S)^{h\Sigma_d}$. 
When $\scr C=\SmQP_S$, this map does not always factor through the motivic localization of $\B_\et\Sigma_d$. For example, if $S$ has characteristic $2$, then $*\to \B_\et\Sigma_2$ is a motivic equivalence, but the induced map $E^{\wedge 2}\to\NSym_\scr C^2(E)$ has no retraction if $E$ is the sum of two nonzero spectra.
\end{remark}

\subsection{Thom isomorphisms}

We now discuss the Thom isomorphism for motivic Thom spectra.
Let $\scr C$ be an $\infty$-category with finite products and let $A$ be a monoid in $\scr C$.
To any morphism $\phi\colon B\to A$ in $\scr C$, we can associate a \emph{shearing map}
\[
\sigma\colon A\times B\to A\times B,\quad (a,b)\mapsto (a\phi(b),b).
\]
It is a map of $A$-modules, and it is an automorphism if $A$ is grouplike. If $\phi$ is an $\E_n$-map for some $1\leq n\leq\infty$, then $\sigma$ is an $\E_{n-1}$-map under $A$.

The following proposition is a motivic analog of \cite[Theorem 1.2]{Mahowald}, and the proof is essentially the same.

\begin{proposition}[Thom isomorphism]
	\label{prop:Thom-iso}
	Let $S$ be a scheme, $A$ a grouplike $\rm A_\infty$-space in $\PSh_\Sigma(\Sm_S)$, $\psi\colon A\to\SH^\simeq$ an $\rm A_\infty$-map, and $\phi\colon B\to A$ an arbitrary map in $\PSh_\Sigma(\Sm_S)$. Then the shearing automorphism $\sigma$ induces an equivalence of $\M_S(\psi)$-modules
	\[
	\M_S(\psi)\wedge \M_S(\psi\circ \phi) \simeq \M_S(\psi) \wedge \Sigma^\infty_+\L_\mot B,
	\]
	natural in $\phi\in\PSh_\Sigma(\Sm_S)_{/A}$.
	If $\psi$ and $\phi$ are $\E_n$-maps for some $1\leq n\leq \infty$ (resp.\ are natural transformations on $\Span(\Sm_S,\all,\fet)$), this is an equivalence of $\E_{n-1}$-ring spectra (resp.\ of normed spectra) under $\M_S(\psi)$.
\end{proposition}

\begin{proof}
	Since $\M_S$ is symmetric monoidal, we have
	\[
	\M_S(\psi)\wedge \M_S(\psi\circ\phi)\simeq \M_S(\mu\circ(\psi\times(\psi\circ\phi)))\quad\text{and}\quad \M_S(\psi) \wedge \Sigma^\infty_+\L_\mot B\simeq \M_S(\mu\circ(\psi\times\1)),
	\]
	where $\mu\colon \SH\times\SH\to\SH$ is the smash product. To conclude, note that 
	\begin{equation}\label{eqn:shearing}
		\mu\circ (\psi\times \1) \circ \sigma\simeq \mu\circ (\psi\times(\psi\circ\phi)).
	\end{equation}
	If $\psi$ and $\phi$ are $\E_n$-maps, then~\eqref{eqn:shearing} is an equivalence of $\E_{n-1}$-maps and the additional claim follows from the fact that $\M_S$ is symmetric monoidal. If $\psi$ and $\phi$ are natural transformations on $\Span(\Sm_S,\all,\fet)$, then~\eqref{eqn:shearing} is an equivalence of such transformations and the claim follows from Proposition~\ref{prop:normedThomSpectra}.
\end{proof}

\begin{example}\label{ex:MGL-thom-iso}
	Let $S$ be a scheme and let $\xi\in \K^\circ(S)$.
	Let $\psi=j\circ e\colon \K^\circ \to \Sph$ and let $\phi=\xi\colon S \to \K^\circ$. 
	In this case the shearing automorphism is $\sigma\colon \K^\circ\to \K^\circ$, $\eta\mapsto \eta+\xi$, and the equivalence of Proposition~\ref{prop:Thom-iso} is the usual Thom isomorphism $\Sigma^\xi\MGL_S\simeq\MGL_S$. More generally, if $G=(G_n)_{n\in\bb N}$ is a family of $S$-group schemes as in Example~\ref{ex:MG} and $\xi\in \K^G(S)$ has rank $n$, we obtain a Thom isomorphism $\Sigma^\xi\M G_S\simeq \Sigma^{2n,n}\M G_S$.
\end{example}

\begin{example}
	Applying Proposition~\ref{prop:Thom-iso} with $\psi=j\circ e\colon \K^\circ \to \Sph$ and with $\phi=\id_{\K^\circ}$, we obtain an equivalence of normed spectra over $\Sm_S$
	\[
	\MGL_S\wedge \MGL_S \simeq \MGL_S \wedge \Sigma^\infty_+\L_\mot\K^\circ.
	\]
	More generally, if $G=(G_n)_{n\in\bb N}$ is a family of $S$-group schemes as in Example \ref{ex:MG} (resp.\ as in Example \ref{ex:MG2}), we obtain an equivalence of $\M G_S$-modules (resp.\ of normed spectra over $\Sm_S$)
	\[
	\M G_S\wedge \M G_S \simeq \M G_S\wedge \Sigma^\infty_+\L_\mot u_G^{-1}(\K^\circ),
	\]
	where $u_G\colon \K^G\to \K$ is the forgetful map.
\end{example}

	Let $\scr C\subset_\fet\Sch_S$ and let $\MGL_S\to E$ be a morphism of normed spectra over $\scr C$. 
	By Proposition~\ref{prop:categorical-props}(3), the pushout of $\E_\infty$-ring spectra
	\[
	\bigvee_{n\in\Z} \Sigma^{2n,n}E \simeq E \wedge_{\MGL_S} \Biggl(\bigvee_{n\in\Z}\Sigma^{2n,n}\MGL_S\Biggr)
	\]
	has a structure of normed spectrum over $\scr C$. In particular, normed $\MGL_S$-modules are $(2,1)$-periodizable. The underlying incoherent normed structure can be made explicit using the Thom isomorphism:

\begin{proposition}\label{prop:periodization-norms}
	Let $\MGL_S\to E$ be a morphism of normed spectra over $\FEt_S$. Then the induced incoherent normed structure on the periodization $\bigvee_{n\in\Z}\Sigma^{2n,n}E$ is given by the maps $\tilde\mu_p$ of Proposition~\ref{prop:oriented-normed}.
\end{proposition}

\begin{proof}
	It clearly suffices to prove this for $E=\MGL_S$ itself.
	Let us write $\K_S$ for the restriction of $\K$ to $\Sm_S$. 
	If $p\colon T\to S$ is finite étale, $p_\otimes(\M_T(j))$ is the Thom spectrum of the map $j\circ p_*\colon p_*(\K_T)\to \Sph$, and the map $\mu_p\colon p_\otimes \M_T(j)\to \M_S(j)$ is induced by $p_*\colon p_*(\K_T)\to \K_S$.
	The normed spectrum structure on $\bigvee_{n\in\Z}\Sigma^{2n,n}\MGL_S$ is obtained via the equivalence $\M_S(j)\simeq\bigvee_{n\in\Z}\Sigma^{2n,n}\MGL_S$, which is induced by the decomposition $\K\simeq \L_\Sigma(\K^\circ \times\Z)$, $\xi\mapsto (\xi-\rk \xi,\rk\xi)$.
	 Using this decomposition, the map $p_*\colon p_*(\K_T)\to \K_S$ in $\PSh_\Sigma(\Sm_S)$ factorizes as
	\[
	p_*(\K_T)\simeq p_*(\K^\circ_T\times\Z_T)\simeq p_*(\K^\circ_T) \times p_*(\Z_T) \xrightarrow{p_*\times\id} \K^\circ_S\times p_*(\Z_T)\xrightarrow\sigma \K^\circ_S\times p_*(\Z_T) \to \K^\circ_S\times \Z_S\simeq \K_S,
	\]
	where $\sigma$ is the shearing automorphism associated with the composition 
	\[p_*(\Z_T)\to p_*(\K_T)\xrightarrow{p_*} \K_S\xrightarrow{\id-\rk} \K^\circ_S\]
	and $p_*(\Z_T)\to \Z_S$ comes from the additive structure on $\Z$ (see Lemma~\ref{lem:normed-monoid}). Applying $\M_S$ to this composition, we find a corresponding factorization of $\mu_p\colon p_\otimes(\bigvee_{n\in\Z}\Sigma^{2n,n}\MGL_T)\to \bigvee_{n\in\Z}\Sigma^{2n,n}\MGL_S$. The shearing automorphism $\sigma$ becomes the Thom isomorphism (see Example~\ref{ex:MGL-thom-iso}), and it is then easy to see that this factorization of $\mu_p$ is exactly the definition of $\tilde\mu_p$.
\end{proof}

\begin{remark}\label{rmk:MSL-modules}
	Similarly, normed $\MSL$-, $\MSp$-, $\mathrm{MO}$-, and $\mathrm{MSO}$-modules are $(4,2)$-periodizable (see Examples \ref{ex:MSL} and~\ref{ex:MG2}), and the norm maps on their $(4,2)$-periodizations can be described using the respective Thom isomorphisms.
\end{remark}

\begin{example}\label{ex:periodicHZ}
	The unit map $\1_S\to\MGL_S$ induces an equivalence $\s_0(\1_S)\simeq \s_0(\MGL_S)$ \cite[Corollary 3.3]{Spitzweck:2010}. By Proposition~\ref{prop:NAlg-slices}, the induced map $\MGL_S\to \s_0(\1_S)$ is a morphism of normed spectra over $\Sm_S$, and we deduce that $\bigvee_{n\in\Z} \Sigma^{2n,n}\s_0(\1_S)$ has a structure of normed spectrum over $\Sm_S$. When $S$ is essentially smooth over a field, we have $\s_0(\1_S)\simeq \HH\Z_S$ as normed spectra by Remark~\ref{rmk:norm-uniqueness}, so that $\bigvee_{n\in\Z}\Sigma^{2n,n}\HH\Z_S$ is a normed spectrum over $\Sm_S$.
	Hence, if $p\colon T\to S$ is finite étale, we obtain an $\E_\infty$-multiplicative transfer
		\[
		\nu_p\colon \bigoplus_{n\in\Z}z^n(T,*) \to \bigoplus_{r\in\Z} z^r(S,*).
		\]
		If $S$ is moreover quasi-projective, it follows from Theorem~\ref{thm:FultonMacPhersonComparison} and Proposition~\ref{prop:periodization-norms} that $\nu_p$ induces the Fulton–\<MacPherson norm on $\pi_0$.
\end{example}

\begin{example}\label{ex:graded-Chow-Witt}
	The geometric arguments in \cite[\sect 3]{HoyoisMGL} show that the cofiber of the unit map $\1_S\to\MSL_S$ belongs to $\SH(S)^\veff(2)$. In particular, we have equivalences $\tilde\s_0(\1_S)\simeq\tilde\s_0(\MSL_S)$ and $\spi_0^\eff(\1_S)\simeq \spi_0^\eff(\MSL_S)$. By Proposition~\ref{prop:NAlg-slices}, we deduce that $\tilde\s_0(\1_S)$ and $\spi_0^\eff(\1_S)$ are normed $\MSL_S$-module over $\Sm_S$, and hence by Remark~\ref{rmk:MSL-modules} that
	\[
	\bigvee_{n\in\Z}\Sigma^{4n,2n}\tilde\s_0(\1_S)\quad\text{and}\quad \bigvee_{n\in\Z}\Sigma^{4n,2n}\spi_0^\eff(\1_S)
	\]
	 are normed spectra over $\Sm_S$.
	 
	Suppose now that $k$ is a field. By Example~\ref{ex:Chow-Witt},
	the spectrum $\bigvee_{n\in\Z}\Sigma^{4n,2n}\spi_0^\eff(\1_k)$ represents the sum of the even Chow–Witt groups.
	For every finite étale map $p\colon Y\to X$ in $\Sm_k$ and every line bundle $\scr L$ on $Y$, we thus obtain a norm map $\nu_p\colon \widetilde\CH{}^\ev(Y,\scr L) \to \widetilde{\CH}{}^\ev(X,\Norm_p(\scr L))$ and a commutative square
		\begin{tikzmath}
			\diagram{
			\widetilde{\CH}{}^\ev(Y,\scr L) & \CH^\ev(Y) \\
			\widetilde{\CH}{}^\ev(X,\Norm_p(\scr L)) & \CH^\ev(X)\rlap, \\
			};
			\arrows (11-) edge (-12) (21-) edge (-22) (11) edge node[left]{$\nu_p$} (21) (12) edge node[right]{$\nu_p$} (22);
		\end{tikzmath}
		where the right vertical map is the Fulton–MacPherson norm (see Example~\ref{ex:periodicHZ}). Note that these norms cannot be extended to odd Chow–Witt groups, since the ring structure on the sum of all Chow–Witt groups is not commutative.
\end{example}

\appendix
\section{The Nisnevich topology}
\label{app:nisnevich}

In this appendix, we observe that all existing definitions of the Nisnevich topology are equivalent. 
The Nisnevich topology was introduced by Nisnevich in \cite{Nisnevich}. A different definition with \emph{a priori} better properties was given in \cite[Appendix C]{HoyoisGLV} and \cite[\sect3.7]{SAG}. The two definitions are known to be equivalent for noetherian schemes, by a result of Morel and Voevdosky \cite[\sect3 Lemma 1.5]{MV}. We start by generalizing their result to nonnoetherian schemes:

\begin{lemma}\label{lem:MV}
	Let $X$ be a quasi-compact quasi-separated scheme and let $Y\to X$ be a smooth morphism which is surjective on $k$-points for every field $k$. Then there exists a sequence $\emptyset=Z_n\subset Z_{n-1}\subset\dotsb\subset Z_0=X$ of finitely presented closed subschemes such that $Y\to X$ admits a section over $Z_{i-1}\minus Z_{i}$ for all $i$.
\end{lemma}

\begin{proof}
	Consider the set $\Phi$ of all closed subschemes $Z\subset X$ for which the map $Y\times_X Z\to Z$ does not admit such a sequence. If $Z$ is a cofiltered intersection $\bigcap_\alpha Z_\alpha$ and $Z\notin\Phi$, then there exists $\alpha$ such that $Z_\alpha\notin\Phi$, by \cite[Proposition 8.6.3 and Théorème 8.8.2(i)]{EGA4-3}.
	 In particular, $\Phi$ is inductively ordered. Next, we note that if $x$ is a maximal point then the local ring $\scr O_x$ is henselian, since its reduction is a field. Thus, any smooth morphism that splits over $x$ splits over $\Spec(\scr O_x)$ \cite[Théorème 18.5.17]{EGA4-4} and hence over some open neighborhood of $x$.
	If $Z\in\Phi$, then $Z$ is nonempty and hence has a maximal point (by Zorn's lemma). It follows that $Y\times_XZ\to Z$ splits over some nonempty open subscheme of $Z$, which may be chosen to have a finitely presented closed complement $W\subset Z$, by \cite[Lemma 2.6.1(c)]{TT}. Clearly, $W\in\Phi$. This shows that $\Phi$ does not have a minimal element.
	By Zorn's lemma, therefore, $\Phi$ is empty.
\end{proof}

\begin{proposition}\label{prop:nisnevich-definition}
	Consider the following collections of étale covering families:
	\begin{enumerate}
		\renewcommand{\theenumi}{\alph{enumi}}
		\item Families of étale maps that are jointly surjective on $k$-points for every field $k$.
		\item Open covers and singleton families $\{\Spec(A_f\times B) \to\Spec(A)\}$, where $f\in A$, $A\to B$ is étale, and the induced map $A/fA\to B/fB$ is an isomorphism.
		\item Finite families of étale maps $\{U_i\to X\}_i$ such that $\coprod_iU_i\to X$ admits a finitely presented splitting sequence.
		\item The empty covering family of $\emptyset$ and families $\{U\into X,V\to X\}$, where $U\into X$ is an open immersion, $V\to X$ is étale, and the projection $(X\minus U)_\mathrm{red}\times_XV \to (X\minus U)_\mathrm{red}$ is an isomorphism.
		\item The empty covering family of $\emptyset$ and families $\{U\into X,V\to X\}$, where $U\into X$ is an open immersion, $V\to X$ is affine étale, and the projection $(X\minus U)_\mathrm{red}\times_XV \to (X\minus U)_\mathrm{red}$ is an isomorphism.
		\item The empty covering family of $\emptyset$ and families $\{\Spec(A_f) \into \Spec(A),\Spec(B)\to\Spec(A)\}$, where $f\in A$, $A\to B$ is étale, and the induced map $A/fA\to B/fB$ is an isomorphism.
	\end{enumerate}
	Then (a) and (b) generate the same topology on the category of schemes, (a)--(d) generate the same topology on the category of quasi-compact quasi-separated schemes, (a)--(e) generate the same topology on the category of quasi-compact separated schemes, and (a)--(f) generate the same topology on the category of affine schemes.
\end{proposition}

\begin{proof}
	The equivalence of (a) and (c) follows from Lemma~\ref{lem:MV}, the equivalence of (c) and (d) follows from the proof of \cite[\sect3 Proposition 1.4]{MV}, the equivalence of (c) and (e) follows from \cite[Proposition 2.1.4]{AHW}, and the equivalence of (c) and (f) follows from \cite[Proposition 2.3.2]{AHW}. The equivalence of (a) and (b) follows.
\end{proof}

Each topology appearing in Proposition~\ref{prop:nisnevich-definition} will be called the \emph{Nisnevich topology}. 
This presents no risk of confusion since every inclusion among these categories of schemes is clearly continuous and cocontinuous for the respective topologies.
Note that the Nisnevich topology on qcqs, qcs, or affine schemes is generated by a cd-structure such that Voevodsky's descent criterion applies \cite[Theorem 3.2.5]{AHW}.

For $X$ a scheme, the small Nisnevich $\infty$-topos $X_\Nis$ of $X$ is the $\infty$-topos of Nisnevich sheaves on the category $\Et_X$ of étale $X$-schemes. Even though $\Et_X$ is not small, this is sensible by the comparison lemma \cite[Lemma C.3]{HoyoisGLV}. If $X$ is qcqs, we can replace $\Et_X$ by the category of finitely presented étale $X$-schemes, where the Nisnevich topology admits the simpler description (d).

\begin{proposition}\label{prop:nisnevich-properties}
	Let $X$ be an arbitrary scheme.
	\begin{enumerate}
		\item The functors
		\[
		F\mapsto F(\Spec\scr O_{Y,y}^h),\quad Y\in\Et_X,\quad y\in Y,
		\]
		form a conservative family of points of the hypercompletion of $X_\Nis$.
		\item If $X$ is qcqs, then $X_\Nis$ is coherent and is compactly generated by finitely presented étale $X$-schemes.
		\item If $X$ is qcqs of finite Krull dimension $d$, then $X_\Nis$ is locally of homotopy dimension $\leq d$; in particular, it is Postnikov complete and hypercomplete.
	\end{enumerate}
\end{proposition}

\begin{proof}
	By considering homotopy sheaves, it suffices to prove (1) for Nisnevich sheaves of sets on $\Et_X$. The proof is the same as in the étale case \cite[Exposé VIII, Théorème 3.5(b)]{SGA4-2}, using characterization (a) of the Nisnevich topology.
	If $X$ is qcqs, the compact generation of $X_\Nis$ follows from characterization (d) and Voevodsky's descent criterion \cite[Theorem 3.2.5]{AHW}, and its coherence follows from \cite[Proposition A.3.1.3(3)]{SAG}.
	 Assertion (3) is \cite[Theorem 3.17]{ClausenMathew}.
\end{proof}

\section{Detecting effectivity}

We fix a subset $I \subset \Z \times \Z$ and we denote by $\SH(S)_{\geq I}$ the full subcategory of $\SH(S)$ generated under colimits and extensions by 
\[\{\Sigma^\infty_+ X \wedge \S^a \wedge \G_m^{\wedge b} \suchthat X \in \Sm_S,\; (a,b) \in I\}.\]
Note that $\SH(S)_{\geq I}$ depends only on the upward closure of $I$. The main examples are the following:
\begin{center}
	\begin{tabular}{cc}\addlinespace[5pt]
		 $I$ & $\SH(S)_{\geq I}$  \\ \midrule 
		$\Z \times \Z$ & $\SH(S)$ \\
		$\Z \times \{n\}$ & $\SH(S)^\eff(n)$ \\
		$\{n\} \times \Z$ & $\SH(S)_{\ge n}$ \\
		$\{(n,n)\}$ & $\SH(S)^\veff(n)$ \\
		$\{(n,0)\}$ & $\SH(S)^\eff_{\geq n}$ \\
		$\emptyset$ & $\{0\}$ \\
		\midrule\addlinespace[5pt]
	\end{tabular}
\end{center}
The subcategory $\SH(S)_{\geq I} \subset \SH(S)$ is the nonnegative part of a $t$-structure \cite[Proposition 1.4.4.11]{HA}. We denote by $\SH(S)_{<I}$ its negative part, which consists of all $E \in \SH(S)$ such that $\Map(F, E) \simeq *$ for every $F \in \SH(S)_{\geq I}$. 
 We denote by $\tau_{\geq I}$ the right adjoint to $\SH(S)_{\geq I} \into \SH(S)$ and by $\tau_{<I}$ the left adjoint to $\SH(S)_{< I} \into \SH(S)$.
 
For every morphism $f\colon S'\to S$, it is clear that the functor $f^*\colon \SH(S) \to \SH(S')$ is right $t$-exact, i.e., it preserves $\SH(\ph)_{\geq I}$. Consequently, its right adjoint $f_*$ is left $t$-exact, i.e., it preserves $\SH(\ph)_{< I}$. If $f$ is smooth, it is equally clear that $f_\sharp\colon \SH(S')\to \SH(S)$ is right $t$-exact, and hence that $f^*$ is $t$-exact. The following two lemmas improve on these observations.

\begin{lemma} \label{lemm:f*-trunc-commute}
If $f\colon S' \to S$ is a pro-smooth morphism, then $f^*\colon \SH(S) \to \SH(S')$ is $t$-exact. 
\end{lemma}

\begin{proof}
Let $E\in\SH(S)_{<I}$. We must show that $\Map(F,f^*E)\simeq *$ for every $F\in\SH(S')_{\geq I}$, and we may assume that $F=\Sigma^\infty_+X' \wedge\S^a\wedge\G_m^{\wedge b}$ for some $X'\in\Sm_{S'}$ and $(a,b)\in I$. Let $(S_\alpha)_{\alpha\in A}$ be a cofiltered system of smooth $S$-schemes with limit $S'$. By continuity of $S \mapsto \Sm_S$, there exist $0\in A$ and $X_0\in\Sm_{S_0}$ with $X'\simeq X_0\times_{S_0}S'$. For $\alpha\geq 0$, let $X_\alpha=X_0\times_{S_0}S_\alpha$.
By continuity of $\SH(\ph)$ \cite[Proposition 4.3.4]{CD}, we then have
\[ 
\Map_{\SH(S')}(\Sigma^\infty_+ X' \wedge \S^a \wedge \G_m^{\wedge b}, f^* E) 
\simeq \colim_{\alpha \geq 0} \Map_{\SH(S)}(\Sigma^\infty_+ X_\alpha \wedge \S^a \wedge \G_m^{\wedge b}, E),
\]
which is contractible since $E\in\SH(S)_{<I}$.
\end{proof}

\begin{lemma} \label{lemm:i*-preserves}
If $i\colon Z \into S$ is a closed immersion, then $i_*\colon \SH(Z)\to \SH(S)$ is $t$-exact. 
\end{lemma}

\begin{proof}
Since $i_*$ preserves colimits, the subcategory of all $E \in \SH(Z)$ such that $i_*(E) \in \SH(S)_{\geq I}$ is closed under colimits and extensions. It thus suffices to prove that $i_* (\Sigma^\infty_+ Y \wedge \S^a \wedge \G_m^{\wedge b}) \in \SH(S)_{\geq I}$ for every $Y \in \Sm_{Z}$ and every $(a,b) \in I$. By \cite[Proposition 18.1.1]{EGA4-4}, we can further assume that $Y\simeq X\times_SZ$ for some $X\in\Sm_S$, since this is true locally on $Y$.
It is therefore enough to show that $i_*i^*$ preserves $\SH(S)_{\geq I}$. Consider the localization cofiber sequence
\[
j_\sharp j^* E \to E \to i_* i^* E,
\]
where $j\colon S\minus Z\into S$. If $E\in \SH(S)_{\geq I}$, then also $j_\sharp j^*E\in \SH(S)_{\geq I}$, and hence $i_*i^*E\in\SH(S)_{\geq I}$.
\end{proof}

For a scheme $S$ and $s \in S$, we abusively write $s$ for the scheme $\Spec \kappa(s)$ and we denote by $E_s$ the pullback of $E$ to $s$.

\begin{proposition}\label{prop:veff-detect-on-points}
Let $S$ be a scheme locally of finite Krull dimension and let $E \in \SH(S)$. Then $E \in \SH(S)_{\geq I}$ if and only if $E_s \in \SH(s)_{\geq I}$ for every point $s \in S$.
In particular, if each $E_s$ is effective, connective, very effective, or zero, then so is $E$.
\end{proposition}

\begin{proof}
Necessity is clear. We prove sufficiency by induction on the dimension of $S$. Suppose $E_s \in \SH(s)_{\geq I}$ for all $s \in S$. We need to show that $E \in \SH(S)_{\geq I}$, or equivalently that $\tau_{<I}(E) \simeq 0$. By Proposition~\ref{prop:nisnevich-properties}(1,3), it suffices to show that $\tau_{<I}(E)_{S_s} \simeq 0$ for all $s \in S$, where $S_s$ is the localization of $S$ at $s$. By Lemma \ref{lemm:f*-trunc-commute}, we have $\tau_{<I}(E)_{S_s} \simeq \tau_{<I}(E_{S_s})$, so we may assume that $S$ is local with closed point $i\colon \{s\} \into S$ and open complement $j\colon U \into S$ of strictly smaller dimension. Consider the localization cofiber sequence
 \[ j_\sharp j^* E \to E \to i_*i^* E. \] 
We have $j^* E \in \SH(U)_{\geq I}$ by the induction hypothesis, whence $j_\sharp j^* E\in\SH(S)_{\geq I}$. We also have $i^* E \in \SH(s)_{\geq I}$ by assumption, whence $i_* i^* E \in \SH(S)_{\geq I}$ by Lemma~\ref{lemm:i*-preserves}. Since $\SH(S)_{\geq I}$ is closed under extensions, we deduce that $E\in\SH(S)_{\geq I}$.
\end{proof}

As an application of Proposition~\ref{prop:veff-detect-on-points}, we compute the zeroth slice of the motivic sphere spectrum over a Dedekind domain. This was previously done by Spitzweck after inverting the residual characteristics \cite[Theorem 3.1]{SpitzweckMGL}. We denote by $\HH\Z^\Spi_S$ the motivic spectrum representing Bloch–Levine motivic cohomology, as constructed by Spitzweck \cite{SpitzweckHZ}.

\begin{theorem}\label{thm:s_0(1)}
	Suppose that $S$ is essentially smooth over a Dedekind domain. Then there is an equivalence $\s_0(\1_S)\simeq \HH\Z^\Spi_S$.
\end{theorem}

\begin{proof}
	When $S$ is the spectrum of a field, $\HH\Z^\Spi_S$ is equivalent to Voevodsky's motivic cohomology spectrum $\HH\Z_S$, and the result is a theorem of Levine \cite[Theorem 10.5.1]{Levine:2008} (see also \cite[Remark 4.20]{HoyoisMGL} for the case of imperfect fields). Since Spitzweck's spectrum is stable under base change \cite[\sect 9]{SpitzweckHZ}, it follows from Proposition~\ref{prop:veff-detect-on-points} that $\HH\Z^\Spi_S$ is effective in general. Moreover, we have $\f_1(\HH\Z^\Spi_S)=0$ because motivic cohomology vanishes in negative weights:
	\[
	\Map(\Sigma^\infty_+ X \wedge \S^a\wedge\G_m^{\wedge b}, \HH\Z^\Spi_S)\simeq *
	\]
	for any $X\in\Sm_S$, $a\in\Z$, and $b\geq 1$.
	Hence, $\HH\Z^\Spi_S \simeq \s_0(\HH\Z^\Spi_S)$. It remains to show that the cofiber of the unit map $\1_S\to \HH\Z^\Spi_S$ is $1$-effective. This again follows from the case of fields by Proposition~\ref{prop:veff-detect-on-points}.
\end{proof}

\section{Categories of spans}
\label{app:spans}

Let $\scr C$ be an $\infty$-category equipped with classes of ``$\mathrm{left}$'' and ``$\mathrm{right}$'' morphisms that contain the equivalences and are closed under composition and pullback along one another. In that situation, we can form an $\infty$-category
\[
\Span(\scr C,\mathrm{left},\mathrm{right})
\]
with the same objects as $\scr C$ and whose morphisms are spans $X\from Y\to Z$ with $X\from Y$ a left morphism and $Y\to Z$ a right morphism;
composition of spans is given by pullbacks. We refer to \cite[\sect5]{BarwickMackey} for the precise definition of $\Span(\scr C,\mathrm{left},\mathrm{right})$ as a complete Segal space.
We use the abbreviation
\[\Span(\scr C)=\Span(\scr C,\all,\all).\]
We often refer to spans of the form $X\from Y = Y$ (resp.\ of the form $Y = Y \to Z$) as \emph{backward morphisms} (resp.\ \emph{forward morphisms}) in $\Span(\scr C,\mathrm{left},\mathrm{right})$. Every span is thus the composition of a backward morphism followed by a forward morphism.
Note that if $\scr C$ is an $n$-category, then $\Span(\scr C,\lleft,\rright)$ is an $(n+1)$-category. The results of this appendix are only used in the rest of the paper when $\scr C$ is a $1$-category, but the proofs are no simpler with this assumption.

Let $\scr C_\mathrm{left}$ be the wide subcategory of $\scr C$ spanned by the left morphisms. Then, for every $X\in\scr C$, the inclusion
\[
(\scr C_\mathrm{left}^\op)_{X/} \into \Span(\scr C,\mathrm{left},\mathrm{right})_{X/}
\]
is fully faithful and has a right adjoint sending a span $X\from Y\to Z$ to $X\from Y$. In particular, it is a coinitial functor \cite[Theorem 4.1.3.1]{HTT}. Variants of this observation will be used repeatedly in the cofinality arguments below.

\subsection{Spans in extensive \texorpdfstring{$\infty$}{∞}-categories}

Let $\scr D$ be an $\infty$-category with finite products.
 Recall that a commutative monoid in $\scr D$ is a functor
\[
M\colon\Fin_\pt \to \scr D
\]
such that, for every $n\geq 0$,
the $n$ collapse maps $\{1,\dotsc,n\}_+ \to \{i\}_+$ induce an equivalence
\[
M(\{1,\dotsc,n\}_+)\simeq \prod_{i=1}^n M(\{i\}_+).
\]
In particular, a symmetric monoidal $\infty$-category is such a functor $\Fin_\pt\to\Cat_\infty$.
There is an obvious equivalence of categories
\[
\Fin_\pt \simeq \Span(\Fin,\inj,\all), \quad X_+\mapsto X, \quad (f\colon X_+ \to Y_+)\mapsto (X \hookleftarrow f^{-1}(Y) \stackrel f\to Y),
\]
where ``inj'' denotes the class of injective maps. In particular, we can identify $\Fin_\pt$ with a wide subcategory of $\Span(\Fin)$.

The following proposition and its corollary were previously proved by Cranch \cite[\sect5]{Cranch}.

\begin{proposition}\label{prop:comm-monoids}
	Let $\scr D$ be an $\infty$-category with finite products. Then the restriction functor
	\[
	\Fun(\Span(\Fin),\scr D) \to \Fun(\Fin_\pt,\scr D)
	\]
	induces an equivalence of $\infty$-categories between:
	\begin{itemize}
		\item functors $M\colon \Span(\Fin) \to \scr D$ such that $M|\Fin^\op$ (or equivalently $M$ itself) preserves finite products;
		\item commutative monoids in $\scr D$.
	\end{itemize}
	The inverse is given by right Kan extension.
\end{proposition}

\begin{proof}
	Let $i\colon \Fin_\pt \into \Span(\Fin)$ be the inclusion functor. It is clear that $i^*$ restricts to
	\[
	i^*\colon \Fun^\times(\Span(\Fin),\scr D) \to \CAlg(\scr D).
	\]
	 Since $i^*$ is obviously conservative, it suffices to show that, for every commutative monoid $M$, the right Kan extension $i_*(M)$ exists and the counit map $i^*i_*(M) \to M$ is an equivalence (which implies that $i_*(M)$ preserves finite products).
	By the formula for right Kan extension, we must show that, for every finite set $X$, the restriction of $M$ to
	\[
	\Span(\Fin,\inj,\all) \times_{ \Span(\Fin)} \Span(\Fin)_{X/}
	\]
	has limit $M(X)$. The embedding
	\[
	\Fin_{\mathrm{inj}}^\op\times_{\Fin^\op}(\Fin^\op)_{X/} \into \Span(\Fin,\inj,\all) \times_{ \Span(\Fin)} \Span(\Fin)_{X/},
	\quad (X\from S) \mapsto (X\from S\stackrel\id\to S),
	\]
	is coinitial since it has a right adjoint, so it suffices to show that the restriction of $M$ to the left-hand side has limit $M(X)$. Since $M|\Fin_\inj^\op$ is given by $S\mapsto M(*)^S$, it suffices to show that $X$ is the colimit of the forgetful functor $\Fin_\inj\times_{\Fin}\Fin_{/X} \to \Fin\subset\scr S$. This follows from the observation that this functor is left Kan extended from the full subcategory of points of $X$.
\end{proof}

Note that $\Fin^\op$ is freely generated by $*$ under finite products, so that
\[
\Fun^\times(\Fin^\op,\scr D) \to\scr D, \quad M\mapsto M(*),
\]
is an equivalence of $\infty$-categories. Thus, Proposition~\ref{prop:comm-monoids} says that a commutative monoid structure on an object $X\in\scr D$ is equivalent to a lift of the functor $X^{(\ph)}\colon\Fin^\op\to\scr D$ to $\Span(\Fin)$. As another corollary, we can also describe commutative algebras in noncartesian symmetric monoidal $\infty$-categories in terms of spans:

\begin{corollary}\label{cor:comm-algebras}
	Let $\scr A$ be a symmetric monoidal $\infty$-category and let $\hat{\scr A}\colon \Span(\Fin)\to\Cat_\infty$ be the corresponding functor as in Proposition~\ref{prop:comm-monoids}. 
	Then the inclusion $\Fin_\pt\into \Span(\Fin)$ induces an equivalence of $\infty$-categories between $\CAlg(\scr A)$ and sections of $\hat{\scr A}$ that are cocartesian over $\Fin^\op\subset \Span(\Fin)$.
\end{corollary}

\begin{proof}
	Let $\scr D$ be the full subcategory of $(\Cat_\infty)_{/\Delta^1}$ spanned by the cocartesian fibrations. The $\infty$-category of sections of $\hat{\scr A}$ can be expressed as the limit of the diagram
	\begin{tikzmath}
		\def\rowsep{1em}
		\def\colsep{0em}
		\diagram{
		\{*\} && \Fun^\times(\Span(\Fin),\scr D) && \{\hat{\scr A}\} \\
		& \Fun^\times(\Span(\Fin),\Cat_\infty) & & \Fun^\times(\Span(\Fin),\Cat_\infty) & \\
		};
		\arrows (11) edge (22) (13) edge (22) (13) edge (24) (15) edge (24);
	\end{tikzmath}
	By Proposition~\ref{prop:comm-monoids}, this limit is identified with the $\infty$-category of sections of $\scr A\colon \Fin_\pt\to\Cat_\infty$. It remains to observe that a section of $\hat{\scr A}$ is cocartesian over $\Fin^\op\subset\Span(\Fin)$ if and only if it is cocartesian over $\Fin_\inj^\op$, and this happens if and only if the corresponding section of $\scr A$ is cocartesian over inert morphisms.
\end{proof}

If $\scr C$ is an extensive $\infty$-category (see Definition~\ref{def:extensive}), we will denote by ``fold'' the class of maps that are finite sums of fold maps $S^{\amalg n}\to S$; it is easy to check that this class is closed under composition and base change.
There is an obvious functor
\[
\Theta\colon\scr C^\op\times\Span(\Fin) \to \Span(\scr C,\all,\fold), \quad (X,S) \mapsto \coprod_SX.
\]
Our next goal is to show that $\Theta$ is the universal functor that preserves finite products in each variable (Proposition~\ref{prop:fold-spans}).

Recall that an $\infty$-category is \emph{semiadditive} if it has finite products and coproducts, if $\emptyset\to *$ is an equivalence, and if, for every $X,Y\in\scr C$, the canonical map $X\amalg Y\to X \times Y$ is an equivalence.

\begin{lemma}\label{lem:Span-semiadditive}
	Let $\scr C$ be an extensive $\infty$-category and let $m$ be a class of morphisms in $\scr C$ that contains the equivalences and is closed under composition and base change.
	If $m$ is closed under binary coproducts, then the inclusion $\scr C\into \Span(\scr C,m,\all)$ preserves finite coproducts. 
	If moreover $\mathrm{fold}\subset m$, then $\Span(\scr C,m,\all)$ is semiadditive.
\end{lemma}

\begin{proof}
	Since every map $X\to\emptyset$ in $\scr C$ is an equivalence, $\emptyset$ is an initial object of $\Span(\scr C,m,\all)$.
	If $m$ is closed under binary coproducts, the spans $X\from X\into X\amalg Y$ and $Y\from Y\into X\amalg Y$ exhibit $X\amalg Y$ as the coproduct of $X$ and $Y$: the induced map $\Map(X\amalg Y,Z)\to \Map(X,Z)\times\Map(Y,Z)$ has inverse
	\[
	(X\from W_X\to Z, Y\from W_Y\to Z)\mapsto (X\amalg Y\from W_X\amalg W_Y\to Z).
	\]
	If $\mathrm{fold}\subset m$, then $\emptyset$ is a final object of $\Span(\scr C,m,\all)$ and the spans $X\amalg Y\hookleftarrow X\to X$ and $X\amalg Y\hookleftarrow Y\to Y$ exhibit $X\amalg Y$ as the product of $X$ and $Y$: the induced map
	$\Map(Z,X\amalg Y)\to\Map(Z,X)\times\Map(Z,Y)$ has inverse
	\[
	(Z\from W_X\to X,Z\from W_Y\to Y) \mapsto (Z\from W_X\amalg W_Y\to X\amalg Y).\qedhere
	\]
\end{proof}

If $\scr C$ is any $\infty$-category, let $\scr C^\amalg\to\Fin$ be the cartesian fibration classified by
\[
\Fin^\op\to\Cat_\infty,\quad I\mapsto \scr C^I,
\]
and let $\iota\colon\scr C\into\scr C^\amalg$ be the inclusion of the fiber over $*\in\Fin$. Then $\iota$ is the initial functor to an $\infty$-category with finite coproducts, and $\scr C^\amalg$ is clearly extensive.

\begin{lemma}\label{lem:fold-spans}
	Let $\scr C$ be an arbitrary $\infty$-category and let $\scr D$ be an $\infty$-category with finite products. Then the functor
	\[
	(\iota\times\id)^*\Theta^* \colon \Fun(\Span(\scr C^\amalg,\all,\fold),\scr D) \to \Fun(\scr C^\op\times \Span(\Fin),\scr D)
	\]
	restricts to an equivalence between:
	\begin{itemize}
		\item functors $\Span(\scr C^\amalg,\all,\fold)\to\scr D$ that preserve finite products;
		\item functors $\scr C^\op \to \CAlg(\scr D)$.
	\end{itemize}
	The inverse is given by right Kan extension.
\end{lemma}

\begin{proof}
	It is clear that this functor is conservative on the given subcategory.
	Let $M\colon\scr C^\op\times\Span(\Fin) \to \scr D$ be a functor that preserves finite products in its second variable.
	We will show that, for every $(X_i)_{i\in I}\in\scr C^\amalg$, the restriction of $M$ to
	\[
	\scr P=(\scr C^\op\times\Span(\Fin))\times_{\Span(\scr C^\amalg,\all,\fold)}\Span(\scr C^\amalg,\all,\fold)_{(X_i)_{i\in I}/}
	\]
	has limit $\prod_{i\in I}M(X_i)$: this implies that $(\Theta(\iota\times\id))_*M$ exists and preserves finite products, and that $(\Theta(\iota\times\id))_*$ is the desired inverse.
	The $\infty$-category $\scr P$ has a coreflective subcategory
	\[
	\scr Q=(\scr C^\op\times\Fin^\op)\times_{(\scr C^\amalg)^\op} ((\scr C^\amalg)^\op)_{(X_i)_{i\in I}/}.
	\]
	Since $M$ preserves finite products in its second variable, the restriction of $M$ to $\scr Q$ sends an object $\sigma=(Z,J,\pi\colon J\to I, \{Z\to X_{\pi(j)}\}_{j\in J})$ to $M(Z)^J$.
	We claim that this restriction is right Kan extended from the full subcategory
	\[
	\scr R= (\scr C^\op\times\{*\})\times_{(\scr C^\amalg)^\op} ((\scr C^\amalg)^\op)_{(X_i)_{i\in I}/} \simeq \coprod_{i\in I}(\scr C^\op)_{X_i/}.
	\]
	Indeed, the comma $\infty$-category $\scr R\times_{\scr Q}\scr Q_{\sigma/}$ is equivalent to $\coprod_{j\in J}(\scr C^\op)_{Z/}$, and so the claim follows from the formula for right Kan extensions. Hence, we find
	\[
	\lim_{\scr P}M \simeq \lim_{\scr Q}M \simeq\lim_{\scr R}M \simeq \prod_{i\in I}\lim_{(\scr C^\op)_{X_i/}}M \simeq \prod_{i\in I}M(X_i),
	\]
	as desired.
\end{proof}

\begin{proposition}\label{prop:fold-spans}
	Let $\scr C$ be an extensive $\infty$-category and let $\scr D$ be an $\infty$-category with finite products. Then the functor
	\[
	\Theta^* \colon \Fun(\Span(\scr C,\all,\fold),\scr D) \to \Fun(\scr C^\op\times \Span(\Fin),\scr D)
	\]
	restricts to an equivalence between:
	\begin{itemize}
		\item functors $\Span(\scr C,\all,\fold)\to\scr D$ that preserve finite products;
		\item functors $\scr C^\op \to \CAlg(\scr D)$ that preserve finite products.
	\end{itemize}
	The inverse is given by right Kan extension.
\end{proposition}

\begin{proof}
	Since $\scr C$ has finite coproducts, the fully faithful functor $\iota\colon\scr C\into\scr C^\amalg$ has a left adjoint $\Lambda\colon\scr C^\amalg\to \scr C$, $(X_i)_{\in I}\mapsto\coprod_{i\in I}X_i$, which induces
	\[
	\Lambda\colon \Span(\scr C^\amalg,\all,\fold) \to \Span(\scr C,\all,\fold).
	\]
	By Lemma~\ref{lem:fold-spans}, it suffices to show that
	\[
	\Lambda^*\colon \Fun(\Span(\scr C,\all,\fold),\scr D) \to \Fun(\Span(\scr C^\amalg,\all,\fold),\scr D)
	\]
	induces an equivalence between:
	\begin{itemize}
		\item functors $\Span(\scr C,\all,\fold)\to\scr D$ that preserve finite products;
		\item functors $\Span(\scr C^\amalg,\all,\fold)\to\scr D$ that preserve finite products and whose restriction to $\scr C^\op$ preserves finite products.
	\end{itemize}
	The functor $\Lambda^*$ is obviously conservative. 
	Let $M\colon\Span(\scr C^\amalg,\all,\fold) \to \scr D$ be a functor with the above properties.
	We will show that, for every $X\in\scr C$, $M$ is right Kan extended along the functor
	\[
	\Psi\colon\Span(\scr C^\amalg,\all,\fold)_{\iota X/} \to \Span(\scr C^\amalg,\all,\fold)\times_{\Span(\scr C,\all,\fold)}\Span(\scr C,\all,\fold)_{X/}.
	\]
	As $M$ inverts maps of the form $\iota(\coprod_{i\in I}X_i)\from (X_i)_{i\in I}$, this implies that $\Lambda_*M$ exists and preserves finite products, and that $\Lambda_*$ is the desired inverse.
	
	Fix an object $\sigma = ((Z_i)_{i\in I}, X\from Y\to \coprod_{i\in I}Z_i)$ in the target of $\Psi$, and let $\sigma/\Psi$ denote the comma $\infty$-category. We must then show that the restriction of $M$ to $\sigma/\Psi$ has limit $\prod_{i\in I}M(\iota Z_i)$. 
	 The $\infty$-category $\sigma/\Psi$ has a full subcategory $\scr A(\sigma)=\sigma/\Psi\times_{\Span(\Fin)_{I/}}(\Fin^\op)_{I/}$. Unraveling the definitions, we see that an object of $\scr A(\sigma)$ consists of
	\begin{itemize}
		\item an object $I\stackrel\pi\from J$ in $(\Fin^\op)_{I/}$;
		\item for each $j\in J$, an object $Z_{\pi(j)}\from W_j$ in $((\scr C^\amalg)^\op)_{Z_{\pi(j)}/}$;
		\item for each $j\in J$, a finite set $K_j$ and an equivalence $Y\times_{Z_{\pi(j)}}W_j\simeq \coprod_{K_j}W_j$ over $W_j$.
	\end{itemize}
	One can easily check that the subcategory $\scr A(\sigma)\subset \sigma/\Psi$ is coreflective and in particular coinitial, either directly or by showing that the projection $\sigma/\Psi\to\Span(\Fin)_{I/}$ is a cartesian fibration and using \cite[Proposition 4.1.2.15]{HTT}.
	We must therefore show that the restriction of $M$ to $\scr A(\sigma)$ has limit $\prod_{i\in I}M(\iota Z_i)$. For $i\in I$, let $\sigma_i=((Z_i,\{i\}), X\from Y_i \to Z_i)$. 
	Then there is an equivalence of $\infty$-categories $\prod_{i\in I}\scr A(\sigma_i)\simeq\scr A(\sigma)$, and
	\[
	\lim_{\tau\in\scr A(\sigma)}M(\tau) \simeq \lim_{(\tau_i)\in\prod_i\scr A(\sigma_i)}\prod_{i\in I}M(\tau_i)\simeq \prod_{i\in I}\lim_{\tau\in\scr A(\sigma_i)}M(\tau).
	\]
	The last equivalence uses that the projections $\prod_i\scr A(\sigma_i)\to\scr A(\sigma_i)$ are coinitial, since each $\scr A(\sigma_i)$ has a final object and hence is weakly contractible.
	 Without loss of generality, we can therefore assume that $I=*$. Similarly, since $Y\to Z$ is a sum of fold maps, we can assume that it is a fold map $Z^{\amalg n}\to Z$. Replacing $\scr C$ by $\scr C_{/Z}$, we may further assume that $Z$ is a final object of $\scr C$.
	 
	 We have arrived at the following situation. We have an extensive $\infty$-category $\scr C$ with a final object; we have a functor $M\colon\Span(\scr C^\amalg,\all,\fold)\to\scr D$ that preserves finite products and whose restriction to $\scr C^\op$ preserves finite products; we have an integer $n\geq 0$; we have the $\infty$-category $\scr A$ whose objects are tuples 
	 \[\tau=(J,(W_j)_{j\in J},(K_j)_{j\in J},(\phi_j)_{j\in J})\] 
	 where $J\in\Fin^\op$, $W_j\in\scr C^\op$, $K_j\in\Fin$, and $\phi_j$ is an equivalence $W_j^{\amalg n}\simeq W_j^{\amalg K_j}$ over $W_j$; and we seek to show that the forgetful functor $\scr A\to(\scr C^\amalg)^\op$ induces an equivalence $M(*)\simeq\lim_{(\scr C^\amalg)^\op} M \simeq \lim_{\scr A}M$. This forgetful functor has a section
	 \[
	 \Phi\colon (\scr C^\amalg)^\op \to \scr A,\quad (W_j)_{j\in J} \mapsto (J,(W_j)_{j\in J}, (n)_{j\in J}, (\id)_{j\in J}),
	 \]
	 which is fully faithful since the space of sections of a fold map is discrete. It suffices to show that $M$ is right Kan extended along $\Phi$.
	 Given $\tau\in\scr A$ as above, we want to show that the restriction of $M$ to the comma $\infty$-category $\tau/\Phi$ has limit $\prod_{j\in J}M(W_j)$. The $\infty$-category $\tau/\Phi$ decomposes in an obvious way as a product over $J$, allowing us to assume that $J=*$ and $\tau=(*,W,K,\phi)$.
	 Moreover, there exists a finite coproduct decomposition $W=\coprod_\alpha W_\alpha$ such that each $\tau_{\alpha}=(*,W_\alpha,K,\phi)$ is in the essential image of $\Phi$. As before, since each $\tau_\alpha/\Phi$ has initial object $\iota W_\alpha$, we obtain equivalences
	 \[
	 \lim_{\tau/\Phi}M\simeq \prod_{\alpha}\lim_{\tau_{\alpha}/\Phi}M\simeq \prod_{\alpha}M(\iota W_\alpha)\simeq M(\iota W),
	 \]
	  as desired.
\end{proof}

\begin{remark}
	The $\infty$-category $\Cat_\infty^\amalg$ of $\infty$-categories with finite coproducts (and functors that preserve finite coproducts) has a symmetric monoidal structure $\otimes$ such that $\scr A\times\scr B\to\scr A\otimes\scr B$ is the universal functor that preserves finite coproducts in each variable \cite[\sect4.8.1]{HA}. Proposition~\ref{prop:fold-spans} can be rephrased as follows: if $\scr C$ is an extensive $\infty$-category, then
	\[
	\scr C\otimes\Span(\Fin) \simeq \Span(\scr C,\fold,\all).
	\]
\end{remark}

\begin{remark}
	Proposition~\ref{prop:fold-spans} follows immediately from Lemma~\ref{lem:fold-spans} when $\scr C=\scr C_0^\amalg$, hence when $\scr C$ is a filtered colimit of such $\infty$-categories. This covers most cases of interest, like categories of quasi-compact locally connected schemes or of finitely presented schemes over a qcqs base.
\end{remark}

\begin{corollary}\label{cor:fold-spans}
	Let $\scr C$ be an extensive $\infty$-category, let $\scr A\colon \scr C^\op\to\CAlg(\Cat_\infty)$ be a functor that preserves finite products, and let $\hat{\scr A}\colon \Span(\scr C,\all,\fold)\to \Cat_\infty$ be the corresponding functor as in Proposition~\ref{prop:fold-spans}.
	Then $\Theta$ induces an equivalence of $\infty$-categories between sections of $\CAlg(\scr A(\ph))$ and sections of $\hat{\scr A}$.
\end{corollary}

\begin{proof}
	Apply Proposition~\ref{prop:fold-spans} to the $\infty$-category of cocartesian fibrations over $\Delta^1$, as in Corollary~\ref{cor:comm-algebras}.
\end{proof}

\begin{proposition}\label{prop:automatic-calg}
	Let $\scr C$ be an extensive $\infty$-category, let $m$ be a class of morphisms in $\scr C$ that is closed under composition, base change, and binary coproducts, and let $\scr D$ be an $\infty$-category with finite products. Suppose that $\fold\subset m$. Then the forgetful functor $\CAlg(\scr D)\to\scr D$ induces an equivalence of $\infty$-categories
	\[
	\Fun^\times(\Span(\scr C,\all,m),\CAlg(\scr D))\simeq \Fun^\times(\Span(\scr C,\all,m),\scr D).
	\]
\end{proposition}

\begin{proof}
	By Lemma~\ref{lem:Span-semiadditive}, the $\infty$-category $\Span(\scr C,\all,m)$ is semiadditive. The claim then follows from \cite[Remark 2.7]{GlasmanGoodwillie}.
\end{proof}

\begin{corollary}\label{cor:automatic-calg}
	Let $\scr C$ be an extensive $\infty$-category and let $m$ be a class of morphisms in $\scr C$ that is closed under composition, base change, and binary coproducts, and such that $\fold\subset m$. Let $\scr A\colon \Span(\scr C,\all,m)\to\Cat_\infty$ be a functor that preserves finite products and let $\hat{\scr A}\colon \Span(\scr C,\all,m)\to\CAlg(\Cat_\infty)$ be the corresponding functor as in Proposition~\ref{prop:automatic-calg}. Then there is an equivalence of $\infty$-categories between sections of $\scr A$ and sections of $\CAlg(\hat{\scr A}(\ph))$.
\end{corollary}

\begin{proof}
	Apply Proposition~\ref{prop:automatic-calg} to the $\infty$-category of cocartesian fibrations over $\Delta^1$, as in Corollary~\ref{cor:comm-algebras}.
\end{proof}

\subsection{Spans and descent}
\label{sub:spans-descent}

\begin{proposition}\label{prop:local-spans}
	Let $\scr C$ be an $\infty$-category equipped with a topology $t$ and let $m\subset n$ be classes of morphisms in $\scr C$ that contain the equivalences and are closed under composition and base change. 
	Suppose that:
	\begin{enumerate}
		\item If $p\in n$, $q\in m$, and $qp\in m$, then $p\in m$.
		\item Every morphism in $n$ is $t$-locally in $m$.
	\end{enumerate}
	 Then, for every $\infty$-category $\scr D$, the restriction functor
	\[
	\Fun(\Span(\scr C,\all,n),\scr D) \to \Fun(\Span(\scr C,\all,m),\scr D)
	\]
	induces an equivalence between the full subcategories of functors whose restrictions to $\scr C^\op$ are $t$-sheaves.
	The inverse is given by right Kan extension.
\end{proposition}

\begin{proof}
	Let $F\colon \Span(\scr C,\all,m) \to\scr D$ be a functor such that $F|\scr C^\op$ is a $t$-sheaf, and
	let
	\[
	\Psi\colon \Span(\scr C,\all,m)_{X/} \into \Span(\scr C,\all,m)\times_{\Span(\scr C,\all,n)}\Span(\scr C,\all,n)_{X/}
	\]
	be the inclusion. As in the proof of Proposition~\ref{prop:fold-spans}, it suffices to show that $F$ is right Kan extended along $\Psi$. Let $\sigma=(X\from Y\to Z)$ be an object in the target of $\Psi$, and let $\sigma/\Psi$ be the comma $\infty$-category.
	We must show that the canonical map $\lim_{\Span(\scr C,\all,m)_{Z/}} F \to \lim_{\sigma/\Psi} F$ is an equivalence.
	 By (1), $\sigma/\Psi$ is the full subcategory of $\Span(\scr C,\all,m)_{Z/}$ consisting of those spans $Z\from W\to T$ such that $W\times_ZY\to W$ is in $m$. The inclusion $(\scr C^\op)_{Z/}\into \Span(\scr C,\all,m)_{Z/}$ has a right adjoint sending $\sigma/\Psi$ onto the sieve $\scr R\subset \scr C_{/Z}$ ``where $Y\to Z$ is in $m$''. Hence, the previous map can be identified with $\lim_{(\scr C^\op)_{Z/}} F \to \lim_{\scr R^\op} F$. By (2), $\scr R$ is a $t$-covering sieve of $Z$, so this map is an equivalence.
\end{proof}

\begin{corollary}\label{cor:local-spans}
	Let $\scr C$ be an $\infty$-category equipped with a topology $t$ and let $m\subset n$ be classes of morphisms in $\scr C$ as in Proposition~\ref{prop:local-spans}.
	Let $\scr A\colon \Span(\scr C,\all,m)\to \Cat_\infty$ be a functor whose restriction to $\scr C^\op$ is a $t$-sheaf, and let $\hat{\scr A}\colon \Span(\scr C,\all,n)\to\Cat_\infty$ be its unique extension. Then the restriction functor
	\[
	\Sect(\hat{\scr A}) \to \Sect(\scr A)
	\]
	restricts to an equivalence between the full subcategories of sections that are cocartesian over $\scr C^\op$.
\end{corollary}

\begin{proof}
	Let $\scr D\subset (\Cat_\infty)_{/\Delta^1}$ be the full subcategory of cocartesian fibrations, and let $\scr D'\subset\scr D$ be the wide subcategory whose morphisms preserve cocartesian edges. If we write
	\[
	\Sect(\scr A) \simeq \{*\}\times_{\Fun(\Span(\scr C,\all,m),\Cat_\infty)} \Fun(\Span(\scr C,\all,m),\scr D) \times_{\Fun(\Span(\scr C,\all,m),\Cat_\infty)} \{\scr A\}
	\]
	as in Corollary~\ref{cor:comm-algebras},
	the full subcategory of sections that are cocartesian over $\scr C^\op$ corresponds to the full subcategory of the right-hand side consisting of functors $\Span(\scr C,\all,m)\to\scr D$ sending $\scr C^\op$ to $\scr D'$. Since the inclusion $\scr D'\subset \scr D$ preserves limits, such functors are automatically $t$-sheaves. Therefore, the result follows from Proposition~\ref{prop:local-spans}.
\end{proof}

\begin{corollary}\label{cor:automatic-norms}
	Let $\scr C$ be an extensive $\infty$-category with a topology $t$ and let $m$ be a class of morphisms in $\scr C$ that is closed under composition and base change. Suppose that $\Shv_t(\scr C)\subset\Shv_\amalg(\scr C)$, that $\fold\subset m$, and that every morphism in $m$ is $t$-locally in $\fold$.
	Then, for every $\infty$-category $\scr D$ with finite products and every $t$-sheaf $F\colon \scr C^\op \to \CAlg(\scr D)$, there is a unique functor $\hat F$ making the following triangle commute:
	\begin{tikzmath}
		\diagram{\scr C^\op & \CAlg(\scr D)\rlap. \\ \Span(\scr C,\all,m) & \\};
		\arrows (11-) edge node[above]{$F$} (-12) (11) edge[c->] (21) (21) edge[dashed] node[below right]{$\hat F$} (12);
	\end{tikzmath}
Moreover, $\hat F$ is the right Kan extension of its restrictions to $\scr C^\op\times\Fin_\pt$, $\scr C^\op\times\Span(\Fin)$, and $\Span(\scr C,\all,\fold)$.
\end{corollary}

\begin{proof}
	Since $\CAlg(\CAlg(\scr D))\simeq\CAlg(\scr D)$, this is a combination of Propositions \ref{prop:comm-monoids}, \ref{prop:fold-spans}, and~\ref{prop:local-spans}. To apply the latter, note that if $q$ and $qp$ are sums of fold maps in $\scr C$, so is $p$.
\end{proof}

Corollary~\ref{cor:automatic-norms} applies for instance when $\scr C$ is the category of schemes, $t$ is the finite étale topology, and $m$ is the class of finite étale maps: any finite étale sheaf of commutative monoids on schemes has canonical transfers along finite étale maps.

Let $\scr C$ be an $\infty$-category equipped with a topology $t$, and let $\scr A\colon \scr C^\op \to \Cat_\infty$ be a presheaf of $\infty$-categories. We say that a section $s\in\Sect(\scr A)$ \emph{satisfies $t$-descent} if, for every $c\in \scr C$ and $E\in\scr A(c)$, the presheaf 
\[
\scr C_{/c}^\op \to \scr S,\quad (f\colon c'\to c) \mapsto \Map_{\scr A(c')}(f^*(E), s(c')) 
\]
is a $t$-sheaf. We will say that an object $E\in\scr A(c)$ satisfies $t$-descent if the associated cocartesian section of $\scr A|\scr C_{/c}^\op$ satisfies $t$-descent.

\begin{remark}
	Let $\scr D$ be an arbitrary $\infty$-category and $\scr A\colon \scr C^\op\to\Cat_\infty$ the constant functor with value $\scr D$. Then a functor $\scr C^\op\to \scr D$ satisfies $t$-descent in the usual sense if and only if the corresponding section of $\scr A$ satisfies $t$-descent.
\end{remark}

Recall from \cite[\sect 6.2.2]{HTT} that the $t$-sheafification functor $\L_t\colon \PSh(\scr C)\to\PSh(\scr C)$ can be obtained as a transfinite iteration of the endofunctor $F\mapsto F^\dagger$ of $\scr P(\scr C)$ given informally by
	\[
	F^\dagger(c)\simeq \colim_{\scr R\subset \scr C_{/c}} \lim_{d\in\scr R} F(d),
	\]
where $\scr R$ ranges over all $t$-covering sieves of $c$. This remains true for presheaves valued in any compactly generated $\infty$-category, because any such $\infty$-category is a full subcategory of a presheaf $\infty$-category closed under limits and filtered colimits. In particular, if $\scr A$ is a presheaf of symmetric monoidal $\infty$-categories on $\scr C$, its $t$-sheafification $\L_t\scr A$ can be computed via this procedure.

\begin{lemma}\label{lem:sheafification}
	Let $\scr C$ be an $\infty$-category equipped with a topology $t$, $\scr A\colon \scr C^\op \to \CAlg(\Cat_\infty)$ a presheaf of symmetric monoidal $\infty$-categories, $\L_t{\scr A}$ its $t$-sheafification, and $\eta\colon\scr A\to\L_t{\scr A}$ the canonical map. 
	For every $c\in \scr C$ and $E,F\in\scr A(c)$, if $F$ satisfies $t$-descent, then $\eta$ induces an equivalence $\Map_{\scr A(c)}(E, F) \simeq \Map_{(\L_t{\scr A})(c)}(\eta E, \eta F)$.
\end{lemma}

\begin{proof}
	Denote by $\scr M\mathrm{ap}_{\scr A}(E,F)$ the presheaf
	 \[
	 \scr C_{/c}^\op \to \scr S,\quad (f\colon c'\to c) \mapsto \Map_{\scr A(c')}(f^*(E),f^*(F)).
	 \]
By the explicit description of sheafification recalled above, the presheaf $\scr M\mathrm{ap}_{\L_t{\scr A}}(\eta E,\eta F)$ is the $t$-sheafification of $\scr M\mathrm{ap}_{\scr A}(E,F)$. Since $F$ satisfies $t$-descent, the latter is already a $t$-sheaf, whence the result.
\end{proof}

\begin{corollary}\label{cor:automatic-norms2}
	Let $\scr C$ be an extensive $\infty$-category equipped with a topology $t$ and let $m$ be classes of morphisms in $\scr C$ that is closed under composition and base change. Suppose that $\Shv_t(\scr C)\subset\Shv_\amalg(\scr C)$, that $\fold\subset m$, and that every morphism in $m$ is $t$-locally in $\fold$.
	Let $\hat{\scr A}\colon \Span(\scr C,\all,m)\to \Cat_\infty$ be a functor that preserves finite products, and let $\scr A\colon \scr C^\op \to \CAlg(\Cat_\infty)$ be its restriction. Then the restriction functor
	\[
	\Sect(\hat{\scr A}) \to \Sect(\CAlg(\scr A))
	\]
	induces an equivalence between the full subcategories of sections that are cocartesian over $\scr C^\op$ and satisfy $t$-descent.
\end{corollary}

\begin{proof}
	Let $\scr A'\colon \scr C^\op\to\CAlg(\Cat_\infty)$ be the $t$-sheafification of $\scr A$. 
	Using Proposition~\ref{prop:fold-spans}, we can identify $\scr A$ and $\scr A'$ with finite-product-preserving functors $\Span(\scr C,\all,\fold)\to\Cat_\infty$. By Corollary~\ref{cor:fold-spans}, we can moreover identify sections of $\CAlg(\scr A)$ with sections of $\scr A$ over $\Span(\scr C,\all,\fold)$.
	
	By Proposition~\ref{prop:local-spans}, $\scr A'$ extends uniquely to a finite-product-preserving functor $\hat{\scr A}'\colon \Span(\scr C,\all,m)\to\Cat_\infty$. Moreover, since $\hat{\scr A}'$ is the right Kan extension of its restriction to $\Span(\scr C,\all,\fold)$, the natural transformation $\eta\colon \scr A \to \scr A'$ extends uniquely to a natural transformation $\eta\colon \hat{\scr A}\to\hat{\scr A}'$ over $\Span(\scr C,\all,m)$.
	
	Let $\scr E\subset \int\hat{\scr A}$ be the full subcategory spanned by the objects satisfying $t$-descent (in their fiber).
	Consider the commutative diagram
	\begin{tikzmath}
		\diagram{
		\Sect(\hat{\scr A})  & \Sect(\scr A) \\
		\Sect(\hat{\scr A}') & \Sect(\scr A')\rlap, \\
		};
		\arrows (11-) edge (-12) (21-) edge (-22) (11) edge node[left]{$\eta$} (21) (12) edge node[right]{$\eta$} (22);
	\end{tikzmath} 
	where the horizontal arrows are given by restriction. By Corollary \ref{cor:local-spans}, the lower horizontal arrow induces an equivalence between the sections that are cocartesian over $\scr C^\op$. It is thus enough to show that the vertical arrows are fully faithful when restricted to $\scr E$-valued sections. We will show that the functor $\scr E \to \int \hat{\scr A}'$ induced by $\eta$ is in fact fully faithful. Let $E,F\in\scr E$ lie over $c, d \in \scr C$, respectively, and consider the commutative diagram
	 \begin{tikzmath}
	 	\diagram{\Map_{\scr E}((c, E), (d, F)) & \Map_{\int \hat{\scr A}'}((c, \eta E), (d, \eta F)) \\
		\Map_{\Span(\scr C, \all, m)}(c, d) & \Map_{\Span(\scr C, \all, m)}(c, d)\rlap. \\};
		\arrows (11-) edge node[above]{$\eta$} (-12) (11) edge (21) (12) edge (22)
		(21-) edge[-,vshift=1] (-22) edge[-,vshift=-1] (-22);
	 \end{tikzmath}
To prove that the top horizontal arrow is an equivalence, it suffices to show that it induces an equivalence between the fibers over any point. Given a span $c \xleftarrow{f} c' \xrightarrow{p} d$ with $p \in m$, the induced map between the fibers is
\[
\Map_{\scr A(d)}(p_\otimes f^* F, E)\xrightarrow\eta \Map_{\scr A'(d)}(\eta p_\otimes f^* F, \eta E) \simeq \Map_{\scr A'(d)}( p_\otimes f^* \eta F, \eta E).
\] 
This map is an equivalence by Lemma~\ref{lem:sheafification}.
\end{proof}

\begin{remark}
	A section $s\in\Sect(\scr A)$ satisfies $t$-descent if and only if, for every $c\in \scr C$ and every $t$-covering sieve $\scr R\subset \scr C_{/c}$, the restriction of $s$ to $\scr R^\triangleright$ is a relative limit diagram.
	Using this observation, it is possible to prove Corollary~\ref{cor:automatic-norms2} more directly, by simply replacing right Kan extensions with relative right Kan extensions in the proof of Proposition~\ref{prop:local-spans}. Unfortunately, we do not know a reference that treats the theory of relative right Kan extensions in sufficient generality for this argument.
\end{remark}

\begin{proposition}\label{prop:span-RKE}
	Let $\scr C$ be an $\infty$-category, $m$ a class of morphisms in $\scr C$ containing the equivalences and closed under composition and base change, and $\scr D$ an arbitrary $\infty$-category.
	Let $\scr C_0\subset \scr C$ be a full subcategory such that, if $X\in\scr C_0$ and $Y\to X$ is in $m$, then $Y\in\scr C_0$.
	Then a functor $F\colon \Span(\scr C_0,\all,m)\to \scr D$ admits a right Kan extension to $\Span(\scr C,\all,m)$ if and only if $F|\scr C_0\colon\scr C_0^\op\to\scr D$ admits a right Kan extension to $\scr C$, and in this case the latter is the restriction to $\scr C$ of the former.
\end{proposition}

\begin{proof}
	Under the assumption on $\scr C_0$, the inclusion
	\[
	\scr C_0\times_{\scr C}\scr C_{X/}\into \Span(\scr C_0,\all,m) \times_{\Span(\scr C,\all,m)} \Span(\scr C,\all,m)_{X/}
	\]
	has a right adjoint and hence is coinitial. This proves the statement.
\end{proof}

A typical application of Proposition~\ref{prop:span-RKE} is the following: a Zariski sheaf with finite étale transfers on the category of affine schemes extends uniquely to a Zariski sheaf with finite étale transfers on the category of all schemes.

\begin{corollary}\label{cor:span-RKE}
	Let $\scr C$, $m$, and $\scr C_0\subset\scr C$ be as in Proposition~\ref{prop:span-RKE}, let $\scr A\colon\Span(\scr C,\all,m)\to\Cat_\infty$ be a functor, and let $\scr A_0$ be the restriction of $\scr A$ to $\Span(\scr C_0,\all,m)$.
	Suppose that $\scr A|\scr C^\op$ is the right Kan extension of its restriction to $\scr C_0^\op$. Then the restriction functor $\Sect(\scr A)\to\Sect(\scr A_0)$ induces an equivalence between the full subcategories of sections that are cocartesian over backward morphisms.
\end{corollary}

\begin{proof}
	Denote by $\Sect'\subset\Sect$ the full subcategory of sections that are cocartesian over backward morphisms, and let $u\colon \Sect'(\scr A)\to\Sect'(\scr A|\scr C^\op)$ and $u_0\colon \Sect'(\scr A_0)\to\Sect'(\scr A_0|\scr C_0^\op)$ be the restriction functors.
Let $\scr D$ be the full subcategory of $(\Cat_\infty)_{/\Delta^1}$ spanned by the cocartesian fibrations, and let $\scr D'\subset \scr D$ be the wide subcategory spanned by the functors that preserve cocartesian edges, so that $\scr D'\simeq \Fun(\Delta^1,\Cat_\infty)$. If we write
	 \[
	\Sect(\scr A) \simeq \{*\}\times_{\Fun(\Span(\scr C,\all,m),\Cat_\infty)} \Fun(\Span(\scr C,\all,m),\scr D) \times_{\Fun(\Span(\scr C,\all,m),\Cat_\infty)} \{\scr A\}
	\]
	as in Corollary~\ref{cor:comm-algebras}, the full subcategory $\Sect'(\scr A)$ corresponds to those functors $\Span(\scr C,\all,m)\to\scr D$ that send $\scr C^\op$ to $\scr D'$.
Since the inclusion $\scr D'\subset \scr D$ preserves limits, it commutes with right Kan extensions. By Proposition~\ref{prop:span-RKE}, any functor $F\colon \Span(\scr C_0,\all,m)\to\scr D$ sending $\scr C_0^\op$ to $\scr D'$ admits a right Kan extension to $\Span(\scr C,\all,m)$, whose restriction to $\scr C^\op$ is the right Kan extension of $F|\scr C_0^\op$. We deduce that the restriction functors $\Sect'(\scr A)\to\Sect'(\scr A_0)$ and $\Sect'(\scr A|\scr C^\op)\to\Sect'(\scr A_0|\scr C_0^\op)$ have right adjoints that commute with the functors $u$ and $u_0$.
	Since the latter adjunction is an equivalence by assumption, and since $u$ and $u_0$ are conservative, the result follows.
\end{proof}

\subsection{Functoriality of spans}

We now discuss the functoriality of $\infty$-categories of spans. 
Let $\Cat_\infty^+$ be the following $\infty$-category:
\begin{itemize}
	\item an object is an $\infty$-category equipped with classes of left and right morphisms that contain the equivalences and are closed under composition and pullback along one another;
	\item a morphism is a functor that preserves the left morphisms, the right morphisms, and pullbacks of left morphisms along right morphisms.
\end{itemize}
More precisely, $\Cat_\infty^+$ is a subcategory of $\scr M\Cat_\infty\times_{\Cat_\infty}\scr M\Cat_\infty$.
We would like to promote $\Cat_\infty^+$ to an $(\infty,2)$-category.
One way to promote an $\infty$-category $\scr C$ to an $(\infty,2)$-category is to construct a $\Cat_\infty$-module structure on $\scr C$ such that, for every $x\in\scr C$, the functor $(\ph)\tens x\colon\Cat_\infty\to\scr C$ has a right adjoint $\MAP(x,\ph)$ (see \cite[\sect7]{GH} and \cite[Theorem 7.5]{Haugseng}). For example, the $(\infty,2)$-category $\mathbf{Cat}_\infty$ of $\infty$-categories is associated with the $\Cat_\infty$-module structure on $\Cat_\infty$ given by the cartesian product.
 There is an obvious $\Cat_\infty$-module structure on $\Cat_\infty^+$ given by
\[
\scr E\otimes (\scr C,\lleft,\rright) = (\scr E\times\scr C,\mathrm{equiv}\times\lleft,\all\times\rright),
\]
which defines an $(\infty,2)$-category $\mathbf{Cat}_\infty^+$. By inspection, a $2$-morphism in $\mathbf{Cat}_\infty^+$ is a natural transformation $\eta\colon F\to G\colon\scr C\to\scr D$ whose components are right morphisms and that is cartesian along left morphisms, i.e., such that for every left morphism $x\to y$ in $\scr C$, the induced map $F(x) \to F(y) \times_{G(y)} G(x)$ is an equivalence.

\begin{proposition}\label{prop:spans-extended-functoriality}
	The construction $(\scr C,\lleft,\rright)\mapsto \Span(\scr C,\lleft,\rright)$ can be promoted to an $(\infty,2)$-functor
	\[
	\Span\colon \mathbf{Cat}_\infty^+ \to \mathbf{Cat}_{\infty}.
	\]
\end{proposition}

\begin{proof}
	Let $\Cat_{\infty}\subset \Fun(\Delta^\op,\scr S)$ be the fully faithful embedding identifying $\infty$-categories with complete Segal $\infty$-groupoids.
	On the underlying $\infty$-categories, the functor $\Span\colon \Cat_\infty^+ \to \Cat_{\infty}\subset \Fun(\Delta^\op,\scr S)$ is corepresentable by the cosimplicial object $\Tw(\Delta^\bullet)$, where $\Tw(\Delta^n)$ is the twisted arrow category of $\Delta^n$ \cite[\sect3]{BarwickMackey} with the obvious classes of left and right morphisms. Moreover, it is clear that $\Span$ is a $\Cat_\infty$-module functor, whence in particular an $(\infty,2)$-functor.
\end{proof}

\begin{corollary}\label{cor:span-adjunctions}
	Let $(\scr C,\lleft,\rright),(\scr D,\lleft,\rright)\in\Cat_\infty^+$ and let $F:\scr C\adj\scr D:G$ be an adjunction with unit $\eta\colon\id\to GF$ and counit $\epsilon\colon FG\to\id$. Suppose that $F$ and $G$ preserve left morphisms, right morphisms, and pullbacks of left morphisms along right morphisms.
	\begin{enumerate}
		\item If $\eta$ and $\epsilon$ consist of right morphisms and are cartesian along left morphisms, there is an induced adjunction
		\[
		F: \Span(\scr C,\lleft,\rright) \adj \Span(\scr D,\lleft,\rright): G.
		\]
		\item If $\eta$ and $\epsilon$ consist of left morphisms and are cartesian along right morphisms, there is an induced adjunction
		\[
		G: \Span(\scr D,\lleft,\rright) \adj \Span(\scr C,\lleft,\rright): F.
		\]
	\end{enumerate}
\end{corollary}

\begin{proof}
	The first adjunction is obtained by applying the $(\infty,2)$-functor $\Span\colon\mathbf{Cat}_\infty^+\to\mathbf{Cat}_{\infty}$ of Proposition~\ref{prop:spans-extended-functoriality} to the adjunction between $F$ and $G$ in $\mathbf{Cat}_\infty^+$. The second adjunction is obtained from the first by exchanging the roles of left and right morphisms and composing with the $(\infty,2)$-functor $(\ph)^\op\colon\mathbf{Cat}_{\infty}\to(\mathbf{Cat}_{\infty})^\textnormal{2-op}$.
\end{proof}

\section{Relative adjunctions}
\label{sec:app-adjunctions}

Let $F\colon\scr C\to \scr D$ and $G\colon\scr D\to\scr C$ be functors between $\infty$-categories. Recall that a natural transformation $\epsilon\colon FG\to\id_\scr D$ \emph{exhibits $G$ as a right adjoint of $F$} if,
for all $c \in \mathcal{C}$ and $d \in \mathcal{D}$, the composition 
\[
\Map_\mathcal{C}(c, G(d)) \xrightarrow{F} \Map_{\scr D}(F(c),FG(d)) \xrightarrow{\epsilon(d)_*} \Map_{\scr D}(F(c),d)
\]
 is an equivalence \cite[Proposition 5.2.2.8]{HTT}.
 Using \cite[Proposition 2.3]{glasman2016spectrum}, we can replace $c$ and $d$ by arbitrary functors $\scr K\to\scr C$ and $\scr K\to\scr D$. It is then easy to show that $\epsilon$ exhibits $G$ as a right adjoint of $F$ if and only if there exists a natural transformation $\eta\colon \id_\scr C\to GF$ such that the composites
 \[
F \xrightarrow{F\eta} FGF \xrightarrow{\epsilon F} F\quad\text{and}\quad G\xrightarrow{\eta G} GFG \xrightarrow{G\epsilon} G
 \]
 are homotopic to the identity.
 
We recall the notion of \emph{relative adjunction} from \cite[Definition 7.3.2.2]{HA}.
Let $\scr E$ be an $\infty$-category, let
\begin{tikzequation}\label{eqn:relative-adj}
	\def\colsep{1.5em}
	\diagram{
	\scr C & \relax & \scr D \\ & \scr E & \\
	};
	\arrows (11-) edge node[above]{$F$} (-13) (11) edge node[below left,pos=.4]{$p$} (22) (13) edge node[below right,pos=.4]{$q$} (22);
\end{tikzequation}
be a commutative triangle in $\Cat_{\infty}$, and let $G\colon\scr D\to\scr C$ be a functor. A natural transformation $\epsilon\colon FG\to\id_{\scr D}$ \emph{exhibits $G$ as a right adjoint of $F$ relative to $\scr E$} if it exhibits $G$ as a right adjoint of $F$ and if the natural transformation $q\epsilon\colon p G\to q$ is an equivalence.

\begin{remark}\label{rmk:relative-adj}
	Given the triangle~\eqref{eqn:relative-adj}, suppose that $F$ admits a right adjoint $G$ with counit $\epsilon\colon FG\to\id$.
	It then follows from Lemma~\ref{lem:swallowtail} that $(G,q\epsilon)$ is a right adjoint of $F$ in $\mathbf{Cat}_\infty^{\sslash\scr E}$.
	Since $\mathbf{Cat}_{\infty/\scr E}\subset \mathbf{Cat}_\infty^{\sslash\scr E}$ is a wide subcategory, we see that a right adjoint of $F$ relative to $\scr E$ is the same thing as a right adjoint of $F$ in the slice $(\infty,2)$-category $\mathbf{Cat}_{\infty/\scr E}$.
\end{remark}

By \cite[Proposition 7.3.2.5]{HA}, if $G$ is a right adjoint of $F$ relative to $\scr E$, then it is in particular a fiberwise right adjoint, i.e., $G_e$ is a right adjoint of $F_e$ for every $e\in\scr E$.
In case $p$ and $q$ are locally (co)cartesian fibrations, there is the following criterion to detect relative right adjoints:

\begin{lemma}\label{lemm:construct-relative-adjoint}
	Consider the triangle~\eqref{eqn:relative-adj}.
	\begin{enumerate}
		\item If $p$ and $q$ are locally cocartesian fibrations, then $F$ admits a right adjoint $G$ relative to $\scr E$ if and only if:
		\begin{itemize}
			\item for every $e\in\scr E$, the functor $F_e\colon \scr C_e\to\scr D_e$ admits a right adjoint $G_e$;
			\item $F$ preserves locally cocartesian edges.
		\end{itemize}
		\item If $p$ and $q$ are locally cartesian fibrations, then $F$ admits a right adjoint $G$ relative to $\scr E$ if and only if:
		\begin{itemize}
			\item for every $e\in\scr E$, the functor $F_e\colon \scr C_e\to\scr D_e$ admits a right adjoint $G_e$;
			\item for every morphism $f\colon e\to e'$ in $\scr E$, the canonical transformation $f^* G_{e'}\to G_{e}f^* \colon \scr D_{e'}\to \scr C_{e}$ is an equivalence.
		\end{itemize}
	\end{enumerate}
	In either case, the right adjoint $G$ satisfies $G|\scr D_e\simeq G_e$ for all $e\in\scr E$.
\end{lemma}

\begin{proof}
	This is \cite[Proposition 7.3.2.6]{HA} and \cite[Proposition 7.3.2.11]{HA}, respectively.
\end{proof}

\begin{remark}
	There is a dual version of Lemma~\ref{lemm:construct-relative-adjoint} for detecting relative left adjoints.
\end{remark}

\begin{remark}
	The full subcategory of $\mathbf{Cat}_{\infty/\scr E}$ spanned by the (locally) cocartesian fibrations is equivalent to the $(\infty,2)$-category of (left-lax) functors $\scr E\to\mathbf{Cat}_\infty$ and left-lax natural transformations. Thus, Lemma~\ref{lemm:construct-relative-adjoint}(1) says that a left-lax natural transformation between left-lax functors $\scr E\to\mathbf{Cat}_\infty$ has a right adjoint if and only if it is strict and has a right adjoint pointwise. More generally, one can show that the pointwise right adjoints of a right-lax transformation between (left-lax or right-lax) functors assemble into a left-lax transformation, and similarly with ``left'' and ``right'' exchanged.
\end{remark}

\begin{lemma} \label{lemm:adjoints-pass-to-sections}
Consider the triangle~\eqref{eqn:relative-adj}, and suppose that $F$ admits a right adjoint $G$ relative to $\scr E$. Then there is an induced adjunction
\[
F_*: \Fun_{\scr E}(\scr E,\scr C) \adj \Fun_{\scr E}(\scr E,\scr D): G_*,
\]
where $F_*(s)=F\circ s$ and $G_*(t)=G\circ t$.
\end{lemma}

\begin{proof}
	By Remark~\ref{rmk:relative-adj}, $G$ is right adjoint to $F$ in $\mathbf{Cat}_{\infty/\scr E}$. The desired adjunction is obtained by applying the corepresentable $(\infty,2)$-functor $\MAP(\scr E,\ph)\colon \mathbf{Cat}_{\infty/\scr E}\to\mathbf{Cat}_\infty$ \cite[\sect A.2.5]{GRderalg}. Alternatively, one can directly check that the induced transformation $\epsilon_*\colon F_*G_*\to\id_{\Fun_\scr E(\scr E,\scr D)}$ exhibits $G_*$ as a right adjoint of $F_*$ using \cite[Proposition 2.3]{glasman2016spectrum}.
\end{proof}

We will often use Lemmas \ref{lemm:construct-relative-adjoint}(1) and~\ref{lemm:adjoints-pass-to-sections} when $p$ and $q$ are the cocartesian fibrations classified by functors $\scr A,\scr B\colon \scr E\to\Cat_\infty$ and $F$ is classified by a natural transformation $\phi\colon\scr A\to\scr B$. Then $F$ has a right adjoint relative to $\scr E$ if and only if $\phi(e)\colon\scr A(e)\to\scr B(e)$ has a right adjoint for all $e\in\scr E$, in which case we have an induced adjunction
\[
\phi_* :\Sect(\scr A)\adj \Sect(\scr B):\phi^!.
\]

Recall that a functor $L\colon\scr C\to\scr C$ is a \emph{localization functor} if, when regarded as a functor to its essential image $L\scr C$, it is left adjoint to the inclusion $L\scr C\subset\scr C$ \cite[\sect5.2]{HTT}.

\begin{proposition}\label{prop:localization-abstract}
	Let $p\colon\scr C\to\scr E$ be a (locally) cocartesian fibration, and for each $e\in\scr E$, let $L_e\colon \scr C_e\to\scr C_e$ be a localization functor. Let $L\scr C\subset\scr C$ be the full subcategory spanned by the images of the functors $L_e$, and let $q\colon L\scr C\to\scr E$ be the restriction of $p$.	
	 Suppose that, for every $f\colon e\to e'$ in $\scr E$, the functor $f_*\colon \scr C_e\to\scr C_{e'}$ sends $L_e$-equivalences to $L_{e'}$-equivalences. 
	Then:
	\begin{enumerate}
		\item $q$ is a (locally) cocartesian fibration;
		\item the inclusion $L\scr C\subset\scr C$ admits a left adjoint $L$ relative to $\scr E$ such that $L|\scr C_e\simeq L_e$.
	\end{enumerate}
\end{proposition}

\begin{proof}
	Let $f\colon e\to e'$ be a morphism in $\scr E$, let $c\in L\scr C_e$, and let $c\to c'$ be a locally $p$-cocartesian edge over $f$. Then the composition $c\to c'\to L_{e'}(c')$ is a locally $q$-cocartesian edge. Indeed, for every $d\in L\scr C_{e'}$, both maps
	\[
	\Map_{L\scr C_{e'}}(L_{e'}(c'),d)\to\Map_{\scr C_{e'}}(c',d)\to \Map_{\scr C}(c,d)=\Map_{L\scr C}(c,d)
	\]
	are equivalences. This proves that $q$ is a locally cocartesian fibration. If $p$ is a cocartesian fibration, we use \cite[Lemma 2.4.2.7]{HTT} to prove that $q$ is also a cocartesian fibration: we must show that for $f\colon e\to e'$, $g\colon e'\to e''$, and $c\in L\scr C_e$, the canonical map $L_{e''}(g\circ f)_*(c) \to L_{e''} g_* L_{e'} f_*(c)$ is an equivalence. This map factorizes as
	\[
	L_{e''}(g\circ f)_*(c) \to L_{e''}g_*f_*(c) \to L_{e''} g_* L_{e'} f_*(c).
	\]
	The first map is an equivalence since $p$ is a cocartesian fibration, and the second map is an equivalence since $g_*$ sends $L_{e'}$-equivalences to $L_{e''}$-equivalences. This proves (1). To prove (2), by the dual of Lemma~\ref{lemm:construct-relative-adjoint}(2), it suffices to show that for every $f\colon e\to e'$ in $\scr E$, the canonical transformation $L_{e'}f_* \to L_{e'}f_* L_{e}$ is an equivalence. This follows from the assumption that $f_*$ sends $L_e$-equivalences to $L_{e'}$-equivalences.
\end{proof}

\begin{corollary}\label{cor:localization-sections}
	Under the assumptions of Proposition~\ref{prop:localization-abstract}, the inclusion $\Fun_{\scr E}(\scr E,L\scr C)\subset \Fun_{\scr E}(\scr E,\scr C)$ has a left adjoint sending a section $s$ to the section $e\mapsto L_e(s(e))$.
\end{corollary}

\begin{proof}
	Combine Proposition~\ref{prop:localization-abstract} and Lemma~\ref{lemm:adjoints-pass-to-sections}.
\end{proof}

\newpage

\section*{Table of notation}

\begin{longtable}[l]{ll}
 $\Set$ & sets  \\
 $\Set_\Delta$ & simplicial sets \\
 $\Ab$ &  abelian groups \\
 $\Fin$ & finite sets  \\
 $\Sch_S$ & $S$-schemes  \\
 $\Sch_S^\fp$ & finitely presented $S$-schemes  \\
 $\Aff_S$ & affine $S$-schemes  \\
 $\QP_S$ & quasi-projective $S$-schemes \\
 $\Sm_S$ & smooth $S$-schemes \\
 $\SmAff_S$ & smooth affine $S$-schemes \\
 $\SmSep_S$ & smooth separated $S$-schemes \\
 $\SmQP_S$ & smooth quasi-projective $S$-schemes \\
 $\FEt_S$ & finite étale $S$-schemes  \\
 $\Gpd$ & groupoids \\
 $\Fin\Gpd$ & finite groupoids \\
 $\scr S$ & $\infty$-groupoids/spaces  \\
 $\Sp$ & spectra \\
 $\Cat_\infty$ & $\infty$-categories, i.e., $(\infty,1)$-categories (not necessarily small)  \\
 $\Cat_n$ & $n$-categories, i.e., $(n,1)$-categories  \\
 $\Cat_{(\infty,2)}$ & $(\infty,2)$-categories  \\
 $\Cat_\infty^\mathrm{sift}$ & sifted-cocomplete $\infty$-categories and sifted-colimit-preserving functors  \\
 $\Pr^\mathrm{L}$ & presentable $\infty$-categories and left adjoint functors  \\
 $\Pr^\mathrm{R}$ & presentable $\infty$-categories and right adjoint functors  \\
 $\scr O\Cat_\infty$ & $\infty$-categories with a distinguished class of objects \\
 $\scr M\Cat_\infty$ & $\infty$-categories with a distinguished class of morphisms \\
 $\Map(X,Y)$ & $\infty$-groupoid of maps in an $\infty$-category \\
 $\Hom(X,Y)$ & internal mapping object in a symmetric monoidal $\infty$-category  \\
 $\Fun(\scr C,\scr D)$ & $\infty$-category of functors  \\
 $\emptyset$ & initial object  \\
 $*$ & final object  \\
 $\1$ & unit object in a symmetric monoidal $\infty$-category  \\
 $\eta$, $\epsilon$ & unit, counit of an adjunction \\
 $\int\scr A$ & source of the cocartesian fibration classified by $\scr A\colon\scr C\to\Cat_\infty$ \\
 $\Sect(\scr A)$ & $\infty$-category of sections of the cocartesian fibration classified by $\scr A\colon\scr C\to\Cat_\infty$ \\
 $\scr C_\pt$ & pointed objects \\
 $\scr C_+$ & pointed objects of the form $X\amalg *$ \\
 $\PSh(\scr C)$ & presheaves (of $\infty$-groupoids)  \\
 $\PSh_\Sigma(\scr C)$ & presheaves that transform finite coproducts into finite products  \\
 $\Shv_t(\scr C)$ & $t$-sheaves  \\
 $\Sp(\scr C)$ & spectrum objects \\
 $\scr C_{/X}$, $\scr C_{X/}$ &  overcategory, undercategory \\
 $\tau_{\leq n}\scr C$ & subcategory of $n$-truncated objects  \\
 $\scr C^\simeq$ & maximal subgroupoid  \\
 $\scr C^\heartsuit$ & heart of a $t$-structure in a stable $\infty$-category \\
 $\scr C^\omega$ & subcategory of compact objects  \\
 $\h\scr C$ & homotopy category  \\
 $\CAlg(\scr C)$ &  commutative algebras in a symmetric monoidal $\infty$-category \\
 $\CAlg^\gp(\scr S)$ & grouplike $\E_\infty$-spaces \\
 $\mathrm{Grp}(\scr C)$ & group objects in an $\infty$-category \\
 $M^\gp$ & group completion of a monoid \\
 $\Mod_A(\scr C)$ & $A$-modules in a symmetric monoidal $\infty$-category  \\
 $\Pic(\scr C)$ & $\infty$-groupoid of invertible objects in a symmetric monoidal $\infty$-category \\
 $\scr C\subset_\fet\Sch_S$ & full subcategory containing $S$ and closed under finite sums and finite étale extensions \\
 $\fet$, $\flf$ & finite étale morphisms, finite locally free morphisms \\
 $\fincov$, $\fp$ & finite covering maps, finitely presented maps \\
 $z^*(X)$ & cycles in a noetherian scheme \\
 $\CH^*(X)$ & Chow group, i.e., cycles modulo rational equivalence \\
 $\Vect(X)$ & groupoid of vector bundles on a scheme \\
 $\QCoh(X)$ & stable $\infty$-category of quasi-coherent sheaves \\
 $\Perf(X)$ & stable $\infty$-category of perfect complexes \\
 $\K(X)$ & $\K$-theory $\infty$-groupoid (in the sense of Thomason–Trobaugh)  \\
 $\K^\circ(X)$ & rank $0$ summand of $\K$-theory \\
 $\KH(X)$ & homotopy $\K$-theory spectrum (in the sense of Weibel) \\
 $\Weil_fX$ & Weil restriction along $f$ \\
 $\V(\scr E)$ & $\Spec(\Sym \scr E)$ \\
 $\rm T_f$ & tangent bundle \\
 $\rm N_i$ & normal bundle \\
 $\Th_X(\xi)$ & Thom space of $\xi\in\Vect(X)$ or Thom spectrum of $\xi\in\K(X)$ \\
 $\S^\xi$, $\Sigma^\xi$ & Thom space or Thom spectrum over the base, $\S^\xi\wedge (\ph)$ \\
 $\S^{p,q}$, $\Sigma^{p,q}$, $\pi_{p,q}$ & $\S^{p-q}\wedge \G_m^{\wedge q}$, $\S^{p,q}\wedge (\ph)$, $[\S^{p,q}, \ph]$ \\
 $\underline{\pi}_n(E)_*$ & $n$th homotopy module of a motivic spectrum  \\
 $\underline{\pi}_n^\eff(E)$ & $n$th effective homotopy module of a motivic spectrum  \\
 $\f_nE$, $\tilde\f_nE$ & $n$-effective cover, very $n$-effective cover \\
 $\s_n E$, $\tilde\s_n E$ & $n$th slice, generalized $n$th slice \\
 $\L_t$ & $t$-sheafification  \\
 $\L_{\A^1}$ & $\A^1$-localization, i.e., associated $\A^1$-invariant presheaf  \\
 $\L_\mot$ & motivic localization, i.e., associated $\A^1$-invariant Nisnevich sheaf  \\
 $\B G$ & delooping of group object \\
 $\B_tG$ & $t$-sheafification of $\B G$ \\
 $\Sigma_n$ & symmetric group on $n$ letters \\
\end{longtable}

\newcommand{\etalchar}[1]{$^{#1}$}
\providecommand{\bysame}{\leavevmode\hbox to3em{\hrulefill}\thinspace}

\end{document}